\documentclass[12pt, a4paper]{amsart}

\usepackage[T1]{fontenc}
\usepackage{amsfonts,amssymb,amsmath,amsthm}
\usepackage{mathtools,MnSymbol,accents}
\usepackage{bbm}  
\usepackage{xcolor}

\usepackage[margin=2.5cm]{geometry}
\usepackage{slashed}

\allowdisplaybreaks

\newtheorem{theorem}{Theorem}[section]
\newtheorem{lemma}[theorem]{Lemma}
\newtheorem{corollary}[theorem]{Corollary}
\newtheorem{prop}[theorem]{Proposition}
\newtheorem*{theorem*}{Theorem}
\newtheorem*{theoremA*}{Theorem A}
\newtheorem*{theoremB*}{Theorem B}
\newtheorem*{theoremC*}{Theorem C}
\newtheorem*{theoremD*}{Theorem D}
\newtheorem*{theoremE*}{Theorem E}
\newtheorem*{corollaryF*}{Corollary F}
\newtheorem*{theoremG*}{Theorem G}
\newtheorem*{theoremH*}{Theorem H}

\theoremstyle{definition}
\newtheorem{definition}[theorem]{Definition}
\newtheorem*{definition*}{Definition}

\theoremstyle{remark}
\newtheorem{remark}[theorem]{Remark}

\numberwithin{equation}{section}

\let\oldvarepsilon\varepsilon
\renewcommand{\epsilon}{\oldvarepsilon}
\renewcommand{\varepsilon}{\oldvarepsilon}
\let\oldphi\phi
\let\oldvarphi\varphi
\renewcommand{\phi}{\oldvarphi}
\renewcommand{\varphi}{\oldphi}
\newcommand{\Eta}{\mathrm{H}}
\newcommand{\Zeta}{\mathrm{Z}}
\newcommand{\Mu}{\mathrm{M}}
\newcommand{\Tau}{\mathrm{T}}

\newcommand{\C}{\mathbb{C}}
\newcommand{\N}{\mathbb{N}}
\newcommand{\R}{\mathbb{R}}
\newcommand{\Z}{\mathbb{Z}}
\newcommand{\F}{\mathbb{F}}

\renewcommand{\Im}{\operatorname{Im}}
\newcommand{\supp}{\operatorname{supp}}

\newcommand{\Dom}{\operatorname{Dom}}
\newcommand{\wid}{\operatorname{w}}
\newcommand{\heit}{\operatorname{h}}
\newcommand{\cent}{\operatorname{c}}

\newcommand{\Hn}{\mathbb{H}^{\cdim}}
\newcommand{\Cn}{\mathbb{C}^{\cdim}}
\newcommand{\cdim}{\nu}
\newcommand{\hdim}{\delta}
\newcommand{\dilat}{D}

\newcommand{\one}{^{(1)}}
\newcommand{\two}{^{(2)}}
\newcommand{\rivf}{\accentset{\leftharpoonup}}

\newcommand{\HLap}{\Lap_{(1)}}    
\newcommand{\riHLap}{\rivf{\Lap}_{(1)}}
\newcommand{\Hnabla}{\nabla_{(1)}}  
\newcommand{\Vnabla}{\nabla_{(2)}}  
\newcommand{\riHnabla}{\rivf{\nabla}_{(1)}}
\newcommand{\riVnabla}{\rivf{\nabla}_{(2)}}  
\newcommand{\VLap}{\Lap_{(2)}}    
\newcommand{\fullHLap}{\Lap\one }    
\newcommand{\fullVLap}{\Lap\two }    

\newcommand{\snabla}{\slashed{\nabla}}  
\newcommand{\fullHnabla}{\nabla\one }  
\newcommand{\fullVnabla}{\nabla\two }
\newcommand{\Hprod}{\cdot}
\newcommand{\conv}{\ast}
\newcommand{\Hconv}{\ast_{(1)}}
\newcommand{\Vconv}{\ast_{(2)}}

\newcommand{\lset}{\left\lbrace}
\newcommand{\rset}{\right\rbrace}
\newcommand{\biglset}{\bigl\lbrace}
\newcommand{\bigrset}{\bigr\rbrace}
\newcommand{\lpar}{\left(}
\newcommand{\rpar}{\right)}
\newcommand{\Biglpar}{\Bigl(}
\newcommand{\Bigrpar}{\Bigr)}
\newcommand{\bigglpar}{\biggl(}
\newcommand{\biggrpar}{\biggr)}
\newcommand{\lip}{\left<}
\newcommand{\rip}{\right>}
\newcommand{\labs}{\left|}
\newcommand{\rabs}{\right|}
\newcommand{\biglabs}{\bigl\vert}
\newcommand{\bigrabs}{\bigr\vert}

\def\abs|#1|{\left\vert #1 \right\vert}
\newcommand{\norm}[1]{\left\Vert#1\right\Vert}
\newcommand{\plainnorm}[1]{\Vert#1\Vert}
\newcommand{\bignorm}[1]{\bigl\Vert#1\bigr\Vert}

\newcommand{\biggnorm}[1]{\biggl\Vert#1\biggr\Vert}
\def\bigabs|#1|{\bigl\vert #1 \bigr\vert}
\def\biggabs|#1|{\biggl\vert #1 \biggr\vert}
\def\Bigabs|#1|{\Bigl\vert #1 \Bigr\vert}
\newcommand{\indifn}{\mathbbm{1}}
\newcommand{\fn}{\,}
\newcommand{\sfrac}[2]{#1 / #2}
\newcommand{\term}[1]{\mathrm{#1}}
\newcommand{\region}[1]{{#1}}

\newcommand{\probspace}{\Omega}
\newcommand{\wrt}[1]{\,\mathrm{d}#1}

\newcommand{\bsym}{\boldsymbol}

\newcommand{\fnspace}[1]{\mathsf{#1}}
\newcommand{\Leb}{\fnspace{L}}
\newcommand{\Schwartz}{\fnspace{S}}

\newcommand{\family}{\fnspace{F}}
\newcommand{\familyz}{\fnspace{F}_0}
\newcommand{\familycz}{\fnspace{F}_{00}}
\newcommand{\familyc}{\fnspace{F}_0}
\newcommand{\familyo}{\fnspace{F}_{1}}
\newcommand{\familyt}{\fnspace{F}_{2}}
\newcommand{\ancestor}[1]{\breve{#1}}

\newcommand{\Hil}{\mathcal{H}}
\newcommand{\oper}[1]{\mathcal{#1}}
\newcommand{\Lap}{\oper{L}}

\newcommand{\winverse}[1]{\mathring{#1}}

\newcommand{\set}[1]{\mathfrak{#1}}
\newcommand{\tile}{\set{T}}
\newcommand{\maxrect}{\set{M}}
\newcommand{\rect}{\set{R}}
\newcommand{\tube}{\set{T}}
\newcommand{\finiteset}{\set{F}}

\newcommand{\flag}{F}

\newcommand{\Hardy}{\fnspace{H}}
\newcommand{\HLSHardy}{\fnspace{H}^1_{HLS}}               
\newcommand{\FSCGHardy}{\fnspace{H}^1_{FSCG}}             
\newcommand{\atomHardy}{\fnspace{H}^1_{\flag,\mathrm{atom}}}
\newcommand{\matomHardy}{\fnspace{H}^1_{\flag,\mathrm{mom at}}}

\newcommand{\areaop}[1][]{\oper{S}_{\flag,\mathrm{area} #1}}
\newcommand{\areaopphi}{\areaop[,\bsym\phi]}
\newcommand{\areaopphiprime}{\areaop[,\bsym\phi]'}

\newcommand{\areaoppsi}{\areaop[,\bsym\psi]}
\newcommand{\areaPoisson}{\oper{S}_{\flag,\mathrm{area},\bsym{p}}}
\newcommand{\areaHardy}{\fnspace{H}^1_{\flag,\mathrm{area}}}
\newcommand{\areaHardyphi}[1][]{\fnspace{H}^1_{\flag,\mathrm{area},\bsym\phi #1}}

\newcommand{\ctssqfnopphi}{\oper{S}_{\flag,\mathrm{cts},\bsym\phi}}
\newcommand{\ctssqfnoppsi}{\oper{S}_{\flag,\mathrm{cts},\bsym\psi}}
\newcommand{\ctssqfnHardyphi}{\fnspace{H}^1_{\flag,\mathrm{cts},\bsym\phi}}

\newcommand{\ctssqfnHardy}{\fnspace{H}^1_{\flag,\mathrm{cts}}}
\newcommand{\dissqfnopphi}{\oper{S}_{\flag,\mathrm{dis},\bsym\phi}}
\newcommand{\dissqfnoppsi}{\oper{S}_{\flag,\mathrm{dis},\bsym\psi}}
\newcommand{\dissqfnHardyphi}{\fnspace{H}^1_{\flag,\mathrm{dis},\bsym\phi}}
\newcommand{\dissqfnHardypsi}{\fnspace{H}^1_{\flag,\mathrm{dis},\bsym\psi}}
\newcommand{\dissqfnHardy}{\fnspace{H}^1_{\flag,\mathrm{dis}}}

\newcommand{\nontanHardy}[1][]{\fnspace{H}^1_{\flag,\mathrm{nontan} #1}}
\newcommand{\nontanHardyphi}[1][]{\fnspace{H}^1_{\flag,\mathrm{nontan},\bsym\phi #1}}
\newcommand{\radialHardy}{\fnspace{H}^1_{\flag,\mathrm{radial}}}
\newcommand{\RieszHardy}{\fnspace{H}^1_{\flag,\mathrm{Riesz}}}
\newcommand{\siphiHardy}{\fnspace{H}^1_{\flag,\mathrm{singint},\boldsymbol\phi}}
\newcommand{\flagHardy}{\fnspace{H}^1_{\flag}}
\newcommand{\gmaxop}{\oper{M}_{\flag,\mathrm{gmax},\family}}

\newcommand{\gmaxopz}{\oper{M}_{\flag,\mathrm{gmax},\family_0}}
\newcommand{\gmaxopc}{\oper{M}_{\flag,\mathrm{gmax},\familyc}}
\newcommand{\gmaxopo}{\oper{M}_{\flag,\mathrm{gmax},\familyo}}
\newcommand{\gmaxopt}{\oper{M}_{\flag,\mathrm{gmax},\familyt}}
\newcommand{\gmaxHardy}{\fnspace{H}^1_{\flag, \mathrm{gmax}}}
\newcommand{\gmaxHardyfam}{\fnspace{H}^1_{\flag, \mathrm{gmax},\family}}
\newcommand{\singHardy}{\fnspace{H}^1_{\flag,\mathrm{sing}}}
\newcommand{\flagBMO}{\fnspace{BMO}_{\flag}}

\newcommand{\radmaxopphi}{\oper{M}^{+}_{\flag,\bsym\phi}}

\newcommand{\tanmaxopphi}{\oper{M}^{**}_{\flag,\bsym\phi}}
\newcommand{\nontanmaxopphi}[1][]{\oper{M}^{*}_{\flag,\bsym\phi #1}}

\newcommand{\radialmaxphi}{\oper{M}_{\flag,\mathrm{radial},\bsym\phi}}
\newcommand{\radialHardyphi}{\fnspace{H}^1_{\flag,\mathrm{radial},\bsym\phi}}

\newcommand{\Rdagger}{\ensuremath{R^\dagger}}

\newcommand{\expe}{\mathrm{e}}
\newcommand{\euli}{\mathrm{i}}

\makeindex

\begin{document}

\title[Flag Hardy Spaces]{Flag Hardy Space Theory
on Heisenberg Groups and Applications}

\author{Peng Chen}
\address{Department of Mathematics\\
Sun Yat-Sen University\\
Guangzhou 510275\\
China}
\curraddr{}
\email{chenpeng3@mail.sysu.edu.cn}
\thanks{P. Chen was supported by National Key R\&D Program of China 2022YFA1005700 and by NNSF of China 12171489}

\author{Michael G. Cowling}
\address{School of Mathematics and Statistics\\ University of New South Wales\\ Sydney 2052\\ Australia}
\curraddr{}
\email{m.cowling@unsw.edu.au}
\thanks{Cowling acknowledges support from the Australian Research Council, grants DP170103025 and DP220100285.}

\author{Ming-Yi Lee}
\address{Department of Mathematics, National Central University, Chung-Li 320, Taiwan   \&
National Center for Theoretical Sciences, 1 Roosevelt Road, Sec. 4, National Taiwan University, Taipei 106, Taiwan}
\curraddr{}
\email{mylee@math.ncu.edu.tw}
\thanks{}

\author{Ji Li}
\address{Department of Mathematics and Statistics,
Macquarie University,
NSW,  2109,
Australia}
\curraddr{}
\email{ji.li@mq.edu.au}
\thanks{Li acknowledges support from the Australian Research Council, grants DP170101060 and DP220100285.}

\author{Alessandro Ottazzi}
\address{School of Mathematics and Statistics\\ University of New South Wales\\ Sydney 2052\\ Australia}
\curraddr{}
\email{a.ottazzi@unsw.edu.au}
\thanks{Ottazzi acknowledges support from the Australian Research Council, grant DP220100285.}

\date{}

\subjclass[2020]{Primary 22E30, 43A15, Secondary 42B30}

\keywords{Heisenberg group, Flag structure, Hardy space}

\begin{abstract}
We develop a complete theory of the flag Hardy space $\flagHardy(\Hn)$ on the Heisenberg group $\Hn$ with characterisations by atomic decompositions, area functions, square functions, maximal functions and singular integrals.
We introduce new techniques to overcome the difficulties caused by the noncommutativity of the Heisenberg group and the lack of a suitable Fourier transformation and Cauchy--Riemann type equations.
Applications include the boundedness on $\flagHardy(\Hn)$ of various singular integral operators that arise in complex analysis, a sharp boundedness result on $\flagHardy(\Hn)$ of the Marcinkiewicz-type multipliers introduced by Müller, Ricci and Stein, and the characterisation of the flag BMO space by singular integrals.

\end{abstract}

\maketitle

\setcounter{tocdepth}{1}
\tableofcontents
\newpage

\setcounter{section}{-1}
\section{Introduction}\label{sec:intro}

The point of this paper is to completely characterise a flag Hardy space $\flagHardy(\Hn)$ on the Heisenberg group $\Hn$.
Our space is a proper subspace of the one-parameter Hardy space of Folland and Stein \cite{FolSte} that was developed by Christ and Geller \cite{ChrGel}, and is a simple modification of the Hardy space proposed by Han, Lu and Sawyer \cite{HanLuSaw}, who initiated the study of this question and characterised their proposed space by square functions.
Our work answers a question of Stein, who asked for a Hardy space theory in the flag setting.

Our space is useful in several applications:\\[2pt]
\noindent\textbf{1) endpoint boundedness of certain singular integrals}, including the Hilbert transform in the central variable, the homogeneous kernels considered first by Folland and Stein \cite{FolSte}, the Cauchy--Szegõ projection on $\Hn$ and the singular integrals considered by Phong and Stein \cite{PhoSte} in their study of the $\bar\partial$-Neumann problem (see also \cite{Ste3});\vspace{2pt}

\noindent\textbf{2) endpoint boundedness of two-parameter Marcinkiewicz multipliers} as studied by Müller, Ricci and Stein \cite{MulRicSte1, MulRicSte2}, and the one-parameter multipliers associated to the sub-Laplacian in the Heisenberg group considered by Hebisch \cite{Heb} and by Müller and Stein \cite{MS};\vspace{2pt}

\noindent\textbf{3) representation of functions} in the flag BMO space on $\Hn$.\vspace{2pt}

This work opens the door to characterising Hardy spaces on more general homogeneous Lie groups with implicit multiparameter structures, and has potential applications to the study of the $\bar\partial_b$-complex on different domains, continuing the path blazed by  Nagel, Ricci and Stein \cite{NagRicSte} and Nagel, Ricci, Stein and Wainger \cite{NagRicSteWai1, NagRicSteWai2}, and more generally to the development of the $\Leb^p$ and $\Hardy^p$ theory of flag singular integral operators on nilpotent Lie groups.

In this section, we present the background to our results and the difficulties,  state our results, outline the rest of the paper and explain what is new in our work.

\subsection*{Background and questions}\label{ssec:background}

Modern approaches to one-parameter harmonic analysis have been developed from the 1950s on; the Calderón--Zygmund theory of singular integrals and the related function spaces are central to this theory.
In the setting of Euclidean spaces $\R^n$, a notable property of standard Calderón--Zygmund singular integral operators and also of the Hardy--Littlewood maximal operator is that they commute with the classical dilations $\dilat_t$, given by $\dilat_t x = (t x_1,\dots,t x_n)$ for all $x \in \R^n$ and all $t\in\R^+$ (see, for example, \cite{Ste1, Ste2}).
Multi-parameter harmonic analysis, with emphasis on the special case of product spaces (that is, singular integrals commuting with multiparameter families of dilations on $\R^n$), has been studied extensively since the 1970s by Gundy, Stein, Chang, R. Fefferman, Journé, Pipher, Lacey, and many others (see, for example, \cite{ChaFef1, ChaFef2, ChaFef3, RFefSte82, RFef, FerLac, GunSte, Jou, Pip}).
To show the boundedness of Calderón--Zygmund singular integrals, there are corresponding function spaces, notably Hardy spaces and their dual spaces, which provide a natural way to do this easily.
These spaces were developed in \cite{CoiWei, CFefSte72} in the one-parameter setting, and later in \cite{ChaFef1, Jou} in the multiparameter product setting.

A more recent breakthrough is due to Müller, Ricci and Stein \cite{MulRicSte1}, who
introduced a new type of multiparameter structure, between the one-parameter and multiparameter cases, and studied the $\Leb^p$ boundedness of Marcinkiewicz multiplier operators $\Mu(\HLap/|\oper{T}|,i\oper{T})$ on the Heisenberg group $\Hn$, where $\HLap$ is the sub-Laplacian on $\Hn$, $\oper{T}$ is the central invariant vector field, and $m$ is a multiplier of Marcinkiewicz type.
They proved the $\Leb^p$ boundedness of $\Mu(\HLap/|\oper{T}|,i\oper{T})$ by using lifting and projection arguments when $m$ satisfies $\fnspace{C}^{\infty}$ regularity conditions.
Using the same approach, they also proved the $\Leb^p$ boundedness of certain singular integrals  that arise in the $\bar\partial$-Neumann problem (see \cite{PhoSte, Ste3}).
The new multiparameter structure, called a \emph{flag structure},
is implicit, that is, it cannot be written in terms of  explicit dilations, and this leads to completely new difficulties that do not appear in the one-parameter or product settings.
The implicit structure is obtained by taking the product space $\Hn \times \R$, and identifying points $((z,t),s)$ and $((z',t'),s')$ when $z = z'$ and $t+s = t'+s'$.
Under this identification, a product of balls $B_1 \times B_2$ in the product space $\Hn \times \R$ becomes a group theoretic product $B_1  B_2$ of sets in $\Hn$.
The identification gives rise to a projection from functions on $\Hn \times \R$ to functions on $\Hn$: let $F=F\bigl( (z,t),s\bigr)$ be a function on $\Hn \times \R$, with $(z,t)\in \Hn$ and $s \in \R$, and define the projected function $f$ by
\[
f(z,t) := \int_{\R} F\bigl( (z,t-s),s\bigr) \wrt s.
\]
See the survey paper by Stein \cite{Ste3}.
More recently, Nagel, Ricci and Stein \cite{NagRicSte} studied a class of operators on nilpotent Lie groups $G$ given by convolution with flag kernels and applied this theory to study the $\Box_b$-complex on certain CR submanifolds of $\Cn$.
Further, Nagel, Ricci, Stein and Wainger \cite{NagRicSteWai1, NagRicSteWai2} developed the theory of singular integrals with flag kernels in the more general setting of homogeneous groups.
They proved that singular integral operators on these groups with flag kernels are bounded on $\Leb^p$ when $1 < p < \infty$, and form an algebra.
See also the recent results of Street \cite{Stt}.

Motivated by Stein's 1999 question, and by \cite{MulRicSte1, Ste3, NagRicSte, NagRicSteWai1, NagRicSteWai2}, the following questions arise.
\vspace{2pt}

\noindent\textbf{Question 1.\enspace}
Is there a flag Hardy space $\flagHardy(\Hn)$ on the Heisenberg group that may be characterised in terms of square functions, maximal functions, atomic decompositions, and Riesz transforms?
\vspace{2pt}

\noindent\textbf{Question 2.\enspace}
What is the relationship of the space $\flagHardy(\Hn)$ to the one-parameter Hardy space $\FSCGHardy(\Hn)$ introduced by Folland and Stein \cite{FolSte} and then studied by Christ and Geller \cite{ChrGel}?
\vspace{2pt}

\noindent\textbf{Question 3.\enspace}
Various singular integral operators on the Heisenberg group appear in connection with boundary value problems in complex analysis in several variables and are known to be bounded on the spaces $\Leb^p(\Hn)$ when $1 < p < \infty$.
These include the Cauchy--Szegõ projection, which has a homogeneous kernel, and the operators introduced in \cite{PhoSte}, which have nonhomogeneous kernels.
Are these operators also bounded from $\flagHardy(\Hn)$ to $\Leb^1(\Hn)$?

\vspace{2pt}

\noindent\textbf{Question 4.\enspace}
Suppose that the Marcinkiewicz multiplier function $\Mu$ satisfies the sharp regularity conditions $\alpha>\cdim$ and $\beta>1/2$ identified by Müller, Ricci and Stein \cite{MulRicSte2}, who showed that the operator $\Mu(\HLap,i\oper{T})$ is bounded on $\Leb^p(\Hn)$ when $1<p<\infty$.
Is $\Mu(\HLap,i\oper{T})$ bounded on $\flagHardy(\Hn)$?

\vspace{2pt}

\noindent\textbf{Question 5.\enspace}
Does the dual space $\flagBMO(\Hn)$ of $\flagHardy(\Hn)$ have a decomposition like that of the one-parameter space $\mathrm{BMO}(\Hn)$ established in \cite{FolSte}?
\vspace{2pt}

\noindent\textbf{Answers.\enspace}
We answer Question 1 by describing a Hardy space $\flagHardy(\Hn)$ that may be characterised by atomic decompositions, square functions, area functions, maximal functions and flag Riesz transforms.

Our results imply that $\flagHardy(\Hn) \subseteq \FSCGHardy(\Hn)$.
We show that $\flagHardy(\Hn)$ is a proper subspace of $\FSCGHardy(\Hn)$, addressing Question 2.

Next, we give a positive answer to Question 3, by verifying that $\norm{ \oper{K} a}_{\Leb^1(\Hn)}$ is uniformly bounded when $\oper{K}$ is a singular integral operator and $a$ is a flag atom,  and hence confirm the boundedness of the Cauchy--Szegõ operator on $\flagHardy(\Hn)$.

We also show that, when $\Mu$ satisfies the sharp regularity condition above, the Marcinkiewicz multiplier operator $\Mu(\HLap,i\oper{T})$ is bounded on $\flagHardy(\Hn)$, thereby answering Question 4.
In particular, we see that the one-parameter Hörmander multiplier operator $\Mu(\HLap)$ is bounded on $\flagHardy(\Hn)$, and hence also on $\FSCGHardy(\Hn)$; apparently this was not known before.
By interpolation we also find another proof of the results of Müller and Stein \cite{MS} and of Hebisch \cite{Heb}.

Finally, we note that the characterisation of $\flagHardy(\Hn)$ by flag Riesz transforms and the duality of $\flagHardy(\Hn)$ with $\flagBMO(\Hn)$ imply that functions in $\flagBMO(\Hn)$ may be written as sums of singular integral operators applied to bounded functions.

\noindent\textbf{Remarks.\enspace}
Han, Lu and Sawyer \cite{HanLuSaw} defined a flag Hardy space $\HLSHardy(\Hn)$ using a discrete Little\-wood--Paley square function, with convolutions of the form $\psi_{m,n} * f$, and described the interpolation spaces between $\HLSHardy(\Hn)$ and $\Leb^2(\Hn)$.
There is a simple isomorphism between their space and ours, namely the reflection operator; this enables us to appeal to their work to deal with square function Hardy spaces, and to use their interpolation theorem.
They dealt with $\fnspace{H}^p$ spaces when $p < 1$, so their work is more general than ours; however, the heat and Poisson semigroups and Riesz transforms are not given by left convolutions, so their methods do not link immediately to these standard operators, and they do not obtain characterisations by maximal functions, singular integrals, or atomic decompositions.

Han, Lu and Sawyer proved the boundedness from $\HLSHardy(\Hn)$ to $\Leb^1(\Hn)$ of singular integral operators whose kernel satisfies a cancellation condition.
However, their methods cannot handle the radial and nontangential maximal functions or the operators $\Mu(\HLap)$ of Müller--Stein \cite{MS} and Hebisch \cite{Heb} and $\Mu(\HLap,i\oper{T})$ of Müller--Ricci--Stein \cite{MulRicSte1}, which do not even have an explicit kernel.
(Even for the classical one-parameter Mihlin--Hörmander multiplier, when we consider the sharp index of differentiation, there is no pointwise estimate for the kernel, see \cite{Ste2}.)
Our atomic decomposition shows that all these operators are bounded on $\flagHardy(\Hn)$.

Recently, Han, Lee, Li and Wick
\cite{HanLeeLiWic} gave a complete description of the flag Hardy space on the simplified model space $\R^m \times \R^n$; their main tools include Fourier transforms, the Cauchy--Riemann equations and the geometrical fact that a flag rectangle $R$ on $\R^m \times \R^n$ may be written as a product $R=I\times J$, where $I$ and $J$ are cubes in $\R^m$ and $\R^n$.
These tools and geometry do not apply in the flag setting on $\Hn$.

Further progress in the flag setting includes (but is not limited to) \cite{ChaHanWu, SteStt}.

In studying Hardy spaces, and in other areas of analysis on euclidean and more general spaces, ``dyadic decompositions'' play an important role.
In the context of the Heisenberg group, there are two possible strategies: one may use the fractal dyadic decompositions of Strichartz \cite{Str} and of Tyson \cite{Tys}, which have symmetries under translation and certain dilations, or the less regular decompositions of Christ \cite{Chr} and of Hytönen and Kairema \cite{HytKai}, which do not have these symmetries.
We have chosen to use the former, as it enables us to appeal to results of Han, Lu and Sawyer \cite{HanLuSaw}, where this approach is used, and enables us to make simplifying assumptions in many proofs.
However,  fractal dyadic decompositions do not exist on all nilpotent Lie groups, let alone in more general spaces.
Our results may also be couched in terms of the decompositions of Christ and of Hytönen and Kairema, and that this is essential for some further generalisations.

\subsection*{Statement of main results}\label{ssec:main-results}

In this section, we state our main results in more detail.

We begin with a little notation; the details and further notation are given in the next section.
The Heisenberg group $\Hn$ is a Lie group, with underlying manifold $\Cn \times \R$; a typical element is written $g$, or in coordinates as $(z,t)$ where $z \in \Cn$ and $t \in \R$.
We may also write $(z,t)$ as $(z_1, \dots, z_\cdim , t)$, or as $(x,y,t)$, or as $(x_1, \dots, x_\cdim , y_1, \dots, y_\cdim ,t)$.
The multiplication law is
\[
(z',t') \Hprod (z, t) = (z' + z, t' + t + S(z', z) ),
\]
where $S$ is the (slightly nonstandard) symplectic form on $\Cn$ given by
\[
S(z', z)
= 4\cdim \Im \Bigl(\sum_{j=1}^\cdim  z'_j \bar{z}_{j} \Bigr)
= 4\cdim \lpar\langle y', x\rangle - \langle x', y\rangle\rpar
= 4\cdim \Bigl(\sum_{j=1}^\cdim  y'_j x_j  - \sum_{j=1}^\cdim  x'_j y_j \Bigr).
\]
The identity of $\Hn$ is written $o$ or $(0,0)$; inversion is given by $(z, t)^{-1} = (-z,-t)$.
The dilations $\dilat_r$, where $r\in\R^+$, of $\Hn$, given by $\dilat_r(z,t) := (rz, r^2t)$, are automorphisms  of $\Hn$.

The Haar measure on $\Hn$ is the Lebesgue measure, which we write $\wrt g$ or $\wrt z\wrt t$.
The standard Heisenberg group convolution is given by the formula
\[
f_1 \Hconv f_2 (g)
= \int_{\Hn} f_1(g_1) f_2(g_1^{-1} g) \wrt g_1
= \int_{\Hn} f_1(gg_1) f_2(g_1^{-1}) \wrt g_1
\qquad\forall g \in \Hn.
\]

The flag structure on $\Hn$ involves the subgroup $\{ (0,t) : t \in \R \}$, which we identify with $\R$ in the obvious way (unfortunately this may be a little confusing).
We convolve a function $f_1$ on $\Hn$ with a function $f_2$ on $\R$ as follows:
\[
f_1 \Vconv f_2 (g) = \int_{\R} f_1(gg_2^{-1}) f_2(g_2) \wrt g_2
\qquad\forall g \in \Hn.
\]
Thus we think of functions on $\R$ as distributions on $\Hn$ which are supported in the centre of $\Hn$; by distributions, we mean elements of the dual of the usual Schwartz space $\Schwartz(\Hn)$.
Generally we write the group operation in $\Hn$ as juxtaposition when we use the symbol $g$ for a group element, and as $\Hprod$ when we use coordinates, while we write the group product in $\R$ multiplicatively when we use the symbol $g$ for a group element, and additively when we are working in coordinates.

The ball with centre $g_1$ and radius $r$ in the gauge metric on $\Hn$ (see Section \ref{ssec:notation} for the definition) is denoted by $B\one(g_1,r)$; it coincides with $g_1 B\one(o,r)$.
The interval $(t-s,t+s)$ is denoted by $B\two (t,s)$, and identified with a subset of the centre of $\Hn$.
As mentioned above, the basic geometric object in the flag structure, analogous to the ball in classical analysis and to the direct product of balls in analysis on product spaces, is the (group) product of balls $B\one(g_1,r) B\two (g_2,s)$.
Given an open subset $E$ of $\Hn$, we write $\maxrect(E)$ for the set of all maximal ``subshards'' of $E$ (see the next section for detailed notation and definitions).

We write $\HLap$ for the usual sub-Laplacian on $\Hn$ and $\VLap$ for the Laplacian on $\R$, both normalised to be positive operators, and $\oper{T}$ for the central invariant vector field on $\R$.
These operators, which are interpreted distributionally, and the associated geometry are key to our results.
We write $\Dom(\HLap^M \VLap^N)$ for the subspace of $\Leb^2(\Hn)$ of functions
$b$ such that $\HLap^M \VLap^N b$ also lies in $\Leb^2(\Hn)$.

Now we give seven different alternative definitions of flag Hardy spaces.

\begin{definition*}
Fix $M$ and $N$ in $\N^+$ and a real number $\kappa \in (1 + 1/(2\nu),\infty)$.
An \emph{atom} is a function $a$ in $\Leb^2(\Hn)$ such that there exists an open subset $E$ of $\Hn$ of finite measure $|E|$ and functions $a_R$ in $\Leb^2(\Hn)$, called \emph{particles}, and $b_R$ in $\Dom( \HLap^{M}\VLap^{N})$ for all $R \in \maxrect(E)$ such that
\begin{enumerate}
\item[(A1)] $a_R= \HLap^M \VLap^N b_R$ and $\supp b_R \subseteq R^*$, where $R^*$ is a $\kappa$-enlargement of $R$;
\item[(A2)] for all sign sequences $\sigma: \maxrect(E) \to \{\pm1\}$, the sum $\sum_{R \in \maxrect(E)} \sigma_R a_R$ converges in $\Leb^2(\Hn)$, to $a_{\sigma}$ say, and $\norm{ a_\sigma }_{\Leb^2(\Hn)}
\leq \abs|E|^{-1/2}$;
\item[(A3)] $ a = \sum_{R \in \maxrect(E)} a_R$.
\end{enumerate}

We say that $f \in \Leb^1(\Hn)$ has an \emph{atomic decomposition} if we may write $f$ as a sum $\sum_{j\in \N} \lambda_j a_j$, converging in $\Leb^1(\Hn)$, where $\sum_{j\in\N} |\lambda_j| < \infty$ and each $a_j$ is an atom; we write $f \sim \sum_{j\in\N} \lambda_j a_j$ to indicate that $\sum_{j\in\N} \lambda_j a_j$ is an atomic decomposition of $f$.
The \emph{atomic Hardy space $\atomHardy(\Hn)$} is defined to be the linear space of all $f \in \Leb^1(\Hn)$ that have atomic decompositions, with norm
\begin{equation*}
\norm{ f }_{\atomHardy(\Hn)}
:=\inf\biggl\{ \sum_{j\in\N} |\lambda_j| :
  f \sim \sum_{j\in\N} \lambda_j a_j \biggr\} .
\end{equation*}
\end{definition*}

We provide a more complete definition at the beginning of Section \ref{sec:atomic}, and in particular make the notion of a $\kappa$-enlargement $R^*$ of a shard $R$ precise.
We also discuss alternative definitions of particles.

Several characterisations of classical Hardy spaces involve integrals or maximal functions over cones.
Given $g \in \Hn$ and $\beta\in\R^+$, we define the cone $\Gamma_\beta(g)$ as follows:
\begin{equation}\label{Gamma cone}
\begin{aligned}
\Gamma_\beta(g):=\{ (g',r,s) \in \Hn \times \R^{+} \times \R^{+}:
     g' \in g B\one(o,\beta r) B\two (0,\beta^2 s) \} .
\end{aligned}
\end{equation}
For simplicity, we write $\Gamma_1(g)$ as $\Gamma(g)$.
In Section \ref{sec:Lusin} we shall see that changing the parameter $\beta$ does not change the Hardy space, though it changes the norm to an equivalent norm.

For our next versions of the flag Hardy space, we take \emph{normalised dilates} of functions $\phi\one$ on $\Hn$ and $\phi\two $ on $\R$, satisfying various smoothness, decay, moment  and ``invertibility'' conditions that we specify later:
\begin{gather}\label{eq:def-phi-1}
\phi\one_r(z,t)  := r^{-\hdim} \phi\one(z/r,t/r^2)
\qquad\text{and}\qquad
\phi\two _s(t) := s^{-1} \phi\two (t/s),
\end{gather}
where $z \in \Cn$ and $t \in \R$.
We also define the \emph{normalised characteristic functions}
\begin{gather}\label{eq:def-chi-1}
\chi\one_r := \frac{1}{|B\one(o,r)|} \indifn_{B\one(o,r)}
\qquad\text{and}\qquad
\chi\two _s := \frac{1}{|B\two (0,s)|} \indifn_{B\one(0,s)} \,,
\end{gather}
where $\indifn$ denotes a characteristic function.
We shorten many formulae by writing
\begin{gather}\label{eq:def-sub-rs}
\phi_{r,s}  := \phi\one_r \Vconv \phi\two _2
\qquad\text{and}\qquad
\chi _{r,s} := \chi\one_r \Vconv \chi\two _2 .
\end{gather}
The Poisson kernel is an important special case.
Let $p\one_{r}$ and $p\two_{s}$ be the convolution kernels of the operators $\expe^{-r\sqrt\HLap}$ on $\Hn$ and $\expe^{-s\sqrt\VLap}$ on $\R$; then when we take $\phi\one$ to be $p\one$ and $\phi\two$ to be $p\two$, we obtain the \emph{flag Poisson kernel} $p_{r,s}$, given by $p\one_{r} \Vconv p\two _{s}$.
The flag heat kernel $h_{r,s}$ arises similarly.

\begin{definition*}
Take $\phi\one$, $\phi\two $ and $\phi_{r,s}$ satisfying various smoothness, decay, moment  and ``invertibility'' conditions that we specify later, and $\chi\one$, $\chi\two$, and $\chi_{r,s}$ as above.
For $f \in \Leb^1(\Hn)$, we define the \emph{Lusin--Littlewood--Paley area function $\areaopphi(f)$ associated to $\phi\one$ and $\phi\two $} by
\begin{equation*}
\begin{aligned}
&\areaopphi (f)(g)
:= \biggl( \iint_{\R^{+} \times\R^{+}}
\bigl| f \Hconv \phi_{r,s}  \bigr|^2 \Hconv \chi_{r,s} (g)  \,\frac{dr}{r} \,\frac{ds}{s} \biggr)^{1/2}
\end{aligned}
\end{equation*}
for all $g \in \Hn$, and we define the Lusin--Littlewood--Paley Hardy space $\areaHardyphi(\Hn)$ to be the set of all $f\in \Leb^1(\Hn)$ for which $\areaopphi (f) \in \Leb^1(\Hn)$, with  norm
\begin{equation*}
\norm{ f }_{\areaHardyphi(\Hn)}:=\norm{ \areaopphi (f) }_{\Leb^1(\Hn)}.
\end{equation*}
\end{definition*}

In Section \ref{sec:Lusin}, we shall make the conditions on $\phi\one$ and $\phi\two $ precise, and see that $\areaHardyphi(\Hn)$ is largely independent of the choice of $\phi\one$ and $\phi\two $, so may be abbreviated to $\areaHardy(\Hn)$.
In particular, we may take $\phi\one$ and $\phi\two $ to be derivatives of the heat or Poisson kernels associated to $\HLap$ and $\VLap$.

\begin{definition*}
Suppose that $\phi\one$ and $\phi\two $ satisfy appropriate smoothness, decay, moment and ``invertibility'' conditions.
For $f \in \Leb^1(\Hn)$, we define the \emph{continuous and discrete Littlewood--Paley square functions $\ctssqfnopphi(f)$  and $\dissqfnopphi(f)$ associated to $\phi\one$ and $\phi\two $} by
\begin{equation*}
\begin{aligned}
&\ctssqfnopphi (f)(g)
:= \biggl( \iint_{\R^{+} \times\R^{+}}
\bigl| f \Hconv \phi_{r,s} (g)  \bigr|^2
    \,\frac{dr}{r} \,\frac{ds}{s} \biggr)^{1/2} \\
&\dissqfnopphi (f)(g)
:= \biggl( \sum_{(m,n) \in \Z \times \Z}
\bigl| f \Hconv \phi_{2^m,2^n}(g)  \bigr|^2
    \biggr)^{1/2}
\end{aligned}
\end{equation*}
for all $g \in \Hn$.
We define the \emph{square function Hardy spaces $\ctssqfnHardyphi(\Hn)$ and $\dissqfnHardyphi(\Hn)$} to be the sets of all $f\in \Leb^1(\Hn)$ for which $ \ctssqfnopphi (f) \in \Leb^1(\Hn)$ or $\dissqfnopphi (f) \in \Leb^1(\Hn)$, with  norms
\[
\norm{ f }_{\ctssqfnHardyphi(\Hn)}:=\norm{ \ctssqfnopphi (f) }_{\Leb^1(\Hn)}
\]
and
\[
\norm{ f }_{\dissqfnHardyphi(\Hn)}:=\norm{ \dissqfnopphi (f) }_{\Leb^1(\Hn)}. .
\]
\end{definition*}

Later we shall see that $\ctssqfnHardyphi(\Hn)$ and $\dissqfnHardyphi(\Hn)$ are essentially independent of the choice of $\phi\one$ and $\phi\two $, and so may be abbre\-viated to $\ctssqfnHardy(\Hn)$ and $\dissqfnHardy(\Hn)$.
In particular, we may take $\phi\one$ and $\phi\two $ to be gradients of the heat or Poisson kernels associated to $\HLap$ and $\VLap$.
The reflection operator identifies our space $\dissqfnHardy(\Hn)$ with the Hardy space of Han, Lu and Sawyer \cite{HanLuSaw}.

Our next definitions involve pairs $\bsym\phi$ of functions $\phi\one$ on $\Hn$ and $\phi\two $ on $\R$ that are Poisson bounded, that is, together with all their derivatives, they decay at least as fast as Poisson kernels and their corresponding derivatives at infinity.
The precise conditions are stated in Definition \ref{def:Poisson-bounds}.

\begin{definition*}
The \emph{radial maximal function} of $f \in \Leb^1(\Hn)$ is defined by
\begin{align*}
\radmaxopphi(f)(g)
:=\sup_{r,s \in \R^+} |f \Hconv \phi_{r,s} (g)|.
\end{align*}
The space $\radialHardyphi(\Hn)$ is the set of all $f\in \Leb^1(\Hn)$ such that $\radmaxopphi(f) \in \Leb^1(\Hn)$, with norm
$\norm{ f }_{\radialHardy(\Hn)} :=\norm{ \radmaxopphi(f) }_{\Leb^1(\Hn)}$.
\end{definition*}

\begin{definition*}
Fix $\beta\in\R^+$.
The \emph{nontangential maximal function $\nontanmaxopphi[,\beta](f)$} of $f \in \Leb^1(\Hn)$ is defined by
\begin{align}\label{eq:def:nontanmax}
\nontanmaxopphi[,\beta](f)(g)
:= \sup_{(g',r,s) \in\Gamma_\beta(g)} |f \Hconv \phi_{r,s} (g')|
\qquad\forall g \in \Hn.
\end{align}
By Lemma \ref{lem:apertures-maxfn} below, the space of all $f\in \Leb^1(\Hn)$ such that $\nontanmaxopphi[,\beta](f) \in \Leb^1(\Hn)$ is independent of $\beta \in \R^+$, so to simplify notation we take $\beta$ equal to $1$ and write $\nontanmaxopphi(f)$ rather than $\nontanmaxopphi[,\beta](f)$.
The Hardy space $\nontanHardyphi(\Hn)$ consisting of all such $f$ is well-defined.
We equip this space with the norm 
\[
\norm{f}_{\nontanHardyphi(\Hn)} := \norm{ \nontanmaxopphi(f) }_{\Leb^1(\Hn)}.
\]
\end{definition*}

An important special case of the space $\nontanHardyphi(\Hn)$ arises when we deal with the flag Poisson kernel $p_{r,s}$.
In this case, properties of harmonic functions may be exploited.

\begin{definition*}
Let $\family$ be a Poisson bounded family of pairs of functions, as in Definition \ref{def:Poisson-bounds} below.
The \emph{grand maximal function} $\gmaxop (f)$ of $f \in \Leb^1(\Hn)$ is defined by
\begin{align*}
\gmaxop (f) (g)
:= \sup_{\bsym\phi \in \family} \sup_{r,s \in \R^+}
\abs|f \Hconv \phi_{r,s} (g)|
\qquad\forall g \in \Hn.
\end{align*}
The space $\gmaxHardyfam(\Hn)$ is defined to be the set of all $f\in \Leb^1(\Hn)$ such that $\gmaxop (f)$ is in $\Leb^1(\Hn)$, with norm
$\norm{ f }_{\gmaxHardyfam(\Hn)} :=\norm{ \gmaxop (f) }_{\Leb^1(\Hn)}$.
\end{definition*}

One of our key results, which generalises a result of Fefferman and Stein \cite{CFefSte72} in the classical case, is that these Hardy spaces do not depend on $\bsym\phi$ or the family $\family$.
More precisely, we get the same radial and nontangential spaces for all pairs $\bsym\phi$ as long as $\phi\one$ and $\phi\two $ are suitably normalised and satisfy Poisson decay conditions, though the norms do depend on $\bsym\phi$.
Similarly, we get the same space if $\family$ is any nontrivial Poisson bounded family.

Our final definition brings in singular integral operators.

\begin{definition*}
The (tensor-valued) \emph{flag Riesz transformations} are defined by
\[
\oper{R}_{(1)} = \Hnabla \HLap^{-1/2} ,
\qquad
\oper{R}_{(2)} =  \Vnabla \VLap^{-1/2}
\qquad\text{and}\qquad
\oper{R}_{\flag} = \oper{R}_{(1)} \otimes \oper{R}_{(2)},
\]
and the space $\RieszHardy(\Hn)$ is the set of all $f\in \Leb^1(\Hn)$ such that $\oper{R}_{(1)}(f) $, $\oper{R}_{(2)}(f) $ and $\oper{R}_{\flag}(f) $ all lie in $\Leb^1(\Hn)$, with norm $\norm{ f }_{\RieszHardy(\Hn)}$ given by
\begin{equation*}
\norm{ \oper{R}_{\flag}(f) }_{\Leb^1(\Hn)}
+\norm{ \Hnabla \HLap^{-1/2}(f) }_{\Leb^1(\Hn)}
+\norm{ \Vnabla \VLap^{-1/2}(f) }_{\Leb^1(\Hn)}
+\norm{ f }_{\Leb^1(\Hn)}.
\end{equation*}
\end{definition*}

The main result of this paper is that the definitions above all agree.

\begin{theoremA*}
The spaces defined above all coincide, that is,
\[
\begin{gathered}
  \atomHardy(\Hn)
= \areaHardy(\Hn)
= \ctssqfnHardy(\Hn)
= \dissqfnHardy(\Hn)
= \nontanHardy(\Hn)\\
= \radialHardy(\Hn)
= \gmaxHardy(\Hn)
= \RieszHardy(\Hn).
\end{gathered}
\]
Moreover, these spaces have equivalent norms, and the constants in the equivalences depend only on the parameters involved in defining the Hardy spaces.
\end{theoremA*}

By the parameters involved in defining the Hardy spaces, we mean, for instance, the integers $M$ and $N$ in the definition of the atomic spaces, or the angle $\beta$ of a cone, or the functions $\phi\one$ and $\phi\one$ involved in the Lusin area function, the square function and the maximal function definitions.
In light of this theorem, we may use the notation $\flagHardy(\Hn)$ to denote any of these spaces.

\subsection*{Applications}\label{ssec:intro-applications}

First, to connect our results with previous work,  we relate our space $\flagHardy(\Hn)$ with the Hardy space $\HLSHardy(\Hn)$ of Han, Lu and Sawyer \cite{HanLuSaw}.
We use the reflection operator $\oper{R}$, which acts on functions on $\Hn$ by composing with the inversion: $\oper{R} f(g) = f(g^{-1})$ for all $g \in \Hn$.

\begin{theoremB*}
The reflection $\oper{R}$ is a bicontinuous linear isomorphism from $\flagHardy(\Hn)$ to $\HLSHardy(\Hn)$.
Hence the complex interpolation space $[\flagHardy(\Hn), \Leb^2(\Hn) ]_\theta$ may be identified with $\Leb^p (\Hn)$ when $1/p = 1 - \theta/2$, and the dual space of $\atomHardy(\Hn)$ may be identified with $\flagBMO(\Hn)$.
\end{theoremB*}

The space $\flagBMO(\Hn)$ will be defined, and this theorem will be proved, in Section 8.2.

Next, Folland and Stein \cite{FolSte} defined a one-parameter Hardy space on stratified Lie groups, and characterised it by Littlewood--Paley theory, atomic and molecular decompositions, and tangential and nontangential maximal functions.
Later Christ and Geller \cite{ChrGel} characterised this Hardy space by singular integrals, and in particular by Riesz transforms.
We denote this one-parameter Hardy space by $ \FSCGHardy(\Hn)$.

\begin{theoremC*}
The space $\flagHardy(\Hn)$ is a proper subspace of $\FSCGHardy(\Hn)$.
\end{theoremC*}

This will be proved in Section 8.3.

Third, the Phong--Stein singular integral operator $\oper{K}$ arises in solving the $\bar\partial$-Neumann problem in a bounded smooth domain in $\C^{\cdim+1}$ (see \cite{PhoSte}, \cite[Section 5]{MulRicSte1} or \cite{Ste3}).
It is a convolution operator, that is, $\oper{K}f = f* k_{PS}$, whose convolution kernel $k_{PS}$ is given by
\begin{equation}\label{Phong Stein kernel}
\begin{aligned}
k_{PS}(z,t)= \frac{\omega(z)}{(|z|^{2}+t^2)^\cdim} \frac{1}{|z|^2+ \euli t}
\qquad\forall (z,t)\in \Hn,
\end{aligned}
\end{equation}
where $\omega$ is homogeneous of degree $0$ on $\Cn$, smooth away from the origin and with mean-value $0$ on the unit sphere.
Stein \cite{Ste3} proved that $\oper{K}$ is bounded on $\Leb^p(\Hn)$ when $1<p<\infty$ by a lifting and projection argument.
The cancellation of $\omega$ is only used to show the $\Leb^2$ boundedness of $\oper{K}$ (using the Cotlar--Stein almost orthogonality estimate \cite[Theorem 4]{PhoSte}).

Another important singular integral operator on $\Hn$ is the Cauchy--Szegõ projection $\oper{C}$, which gives an analytic function in the Siegel upper half space in terms of its boundary value.
Its restriction to the boundary is a convolution operator, that is, $\oper{C}(f)= f * k_{CS}$, and the convolution kernel $k_{CS}$ is given by
\begin{equation}\label{Cauchy Szego}
\begin{aligned}
k_{CS}(z,t)= \frac{c}{(|z|^2+ \euli t)^{\cdim+1}}
\qquad\forall (z,t)\in \Hn
\end{aligned}
\end{equation}
(see \cite[Chapter 12, Section 2.4]{Ste2}).
These two operators are examples of general flag singular integral operators, which were discussed in much great generality by Nagel, Ricci, Stein and Wainger \cite{NagRicSteWai1, NagRicSteWai2}.

In Section 8.4, we use the atomic decomposition to show that general flag singular integral operators, which we define later, are bounded on $\flagHardy(\Hn)$.

\begin{theoremD*}
General flag singular integral operators are bounded on $\flagHardy(\Hn)$.
\end{theoremD*}


Fourth, we define the two-parameter Sobolev space $\Leb_{\alpha,\beta}^2(\R\times\R^+)$ to be the collection of tempered distributions $\Mu$ on $\R^2$ for which the following norm is finite:
\begin{equation}\label{e88.16}
\begin{aligned}
\norm{ \Mu }_{L_{\alpha,\beta}^2(\R^2)}
:= \lpar \iint_{\R^2} (1 + |\xi_1|)^\alpha (1+|\xi_1| + |\xi_2|)^\beta
        | \hat \Mu(\xi_1, \xi_2) |^2 \wrt \xi_1 \wrt \xi_2 \rpar^{1/2};
\end{aligned}
\end{equation}
here $\hat \Mu$ denotes the usual Fourier transform of $\Mu$.

Define $\Mu^{r,s}(\xi_1, \xi_2) := \Mu(r\xi_1, s\xi_2)$.
Choose $\Eta$ in $\fnspace{C}^{\infty}(\R^+)$ with support in $(1/2,2)$
such that $\sum_{j\in\Z} \Eta(\cdot /{2^j}) =1$, and set $\Eta_{1,1}(\xi_1,\xi_2):=\Eta(\xi_1)\Eta(|\xi_2|)$.
Then the following result on the endpoint boundedness of Marcinkiewicz multipliers on $\Hn$ is sharp.

\begin{theoremE*}
If $\alpha > \cdim$ and $\beta > 1/2$, and the function $\Mu$ satisfies the condition
\[
\sup_{r, s >0}  \norm{ \Eta_{1,1} \Mu^{r,s} }_{L_{\alpha,\beta}^2(\R^2)} <\infty, 
\]
then the joint spectral multiplier $\Mu(\HLap, i\oper{T})$ is bounded {on $\flagHardy(\Hn)$}.
\end{theoremE*}

The $\Leb^p$ boundedness of $\Mu(\HLap,i\oper{T})$ when $1<p<\infty$ was proved in \cite{MulRicSte2}.
The proof of this theorem will be given in Section 8.5.

Fix a nontrivial, smooth, nonnegative-real-valued function $\Eta$ on $\R^+$ with support in $[1/2,2]$, and an index $\alpha > (2\cdim+1)/2$.
Write $\Leb^2_\alpha(\R)$ for the usual Sobolev space on $\R$, whose norm is defined by
\[
\norm{ \Mu }_{L_{\alpha}^2(\R)}
:= \lpar \iint_{\R} (1 + |\xi|)^\alpha
        \hat \Mu(\xi_1) |^2 \wrt \xi  \rpar^{1/2}.
\]
Müller and Stein \cite{MS} and Hebisch \cite{Heb} showed that the Hörmander multiplier $\Mu(\HLap)$ is bounded on $\Leb^p(\Hn)$ for $1<p<\infty$, if $\sup_{t>0} \|\Eta (\cdot)\Mu(t \cdot)\|_{\Leb^2_\alpha(\R)}<\infty$.

An interesting corollary of Theorem E is that the one-parameter Hörmander multiplier $\Mu(\HLap)$ is also bounded on our multiparameter flag Hardy space and on the Folland--Stein--Christ--Geller Hardy space $\FSCGHardy(\Hn)$.

\begin{corollaryF*}
Suppose that $\alpha > (2\cdim+1)/2$,  and the multiplier $\Mu$ satisfies the condition $\sup_{t>0} \|\Eta (\cdot)\Mu(t \cdot)\|_{\Leb^2_\alpha(\R)}<\infty$; then $\Mu(\HLap)$ is bounded {on $\flagHardy(\Hn)$} and on  $\FSCGHardy(\Hn)$.
\end{corollaryF*}

By interpolation and duality, $\Mu(\HLap)$ is bounded on $\Leb^p(\Hn)$ when $1<p<\infty$.
This gives another proof of a result of Müller and Stein \cite{MS} and of Hebisch \cite{Heb}.

Next, from \cite[Definition 6 and Theorem 7]{HanLuSaw} and the isomorphism of our space $\flagHardy(\Hn)$ with $\HLSHardy(\Hn)$, it follows the dual space of $\flagHardy(\Hn)$ is the space $\flagBMO(\Hn)$, which is defined using a ``flag Carleson measure condition'' (see Section \ref{ssec:HLS}).
From Theorems A and B, we obtain a decomposition of $\flagBMO(\Hn)$.

\begin{theoremG*}
A function $b \in \flagBMO(\Hn)$ if and only if there exist (vector-valued) functions $h_0$, $h_1$, $h_2$ and $h_3$ in $\Leb^\infty(\Hn)$ such that
\[
b= h_0+ \Hnabla \HLap^{-1/2}(h_1) + \Vnabla \VLap^{-1/2}(h_2) +\oper{R}_{\flag}(h_3).
\]
\end{theoremG*}

Finally, we note that there are examples of convolution operators with kernels of the form \eqref{Phong Stein kernel} that are not of weak type $(1,1)$.
It is therefore natural to explore other weak type end-point estimates.

R. Fefferman \cite{RFef} showed that Calderón--Zygmund operators in the product setting are bounded from $\Leb\log^+\Leb$ to the weak $\Leb^1$ space $\Leb^{1,\infty}$.
This is related to Zygmund's conjecture for maximal operators in the general multiparameter setting \cite{Cor}.
Our result on the domination of area function by the non-tangential maximal function (Theorem \ref{thm:nontan-implies-area}) and our new techniques for establishing a flag atomic decomposition (Theorem \ref{thm:area-Hardy-implies-atom-Hardy}) enable us to prove a similar result in the flag context.
\begin{theorem*}
The singular integral operators of Theorem D and the Marcinkiewicz
 multiplier operators of Theorem E are bounded from $\Leb\log^+\Leb(\Hn)$ to $\Leb^{1,\infty}(\Hn)$.
\end{theorem*}
For reasons of space, we shall treat this case in another paper.

\subsection*{New tools and techniques}\label{ssec:intro-whats-new}

The Heisenberg group $\Hn$ is noncommutative, with a more complicated geometry and Fourier transformation than the product Euclidean setting $\R^m\times\R^n$ of previous results \cite{HanLeeLiWic} on flag Hardy spaces.
To overcome these obstacles, we have developed some new tools and techniques, which may be helpful in solving related problems on the Heisenberg groups or more general stratified groups.

First, in the classical theory of Hardy spaces, the Cauchy--Riemann equations are often used.
However, it is not clear whether such systems of equations are available in the context of the Heisenberg group, or more general contexts.
We find two different ways to avoid the use of these equations.
In the classical case of $\R^n$, Fefferman and Stein \cite{CFefSte72} used the Cauchy--Riemann equations to show the nontangential maximal function dominates the Lusin area function, while in the product setting $\R^m\times\R^n$, the usual tool is Merryfield's lemma \cite{Mer}, which states that for every even $\phi\in \fnspace{C}^{\infty}_0(\R)$ such that $\int \phi (x) \wrt x=1$, the function $\psi\in \fnspace{C}^{\infty}_0(\R)$, given by $\psi(x) = x \phi(x)$, has the same support as $\phi$, and further $\int \psi (x) \wrt x = 0$ and $\partial_t \phi_t(x) = \partial_x \psi_t(x)$, where $\phi_t(x) =t^{-1} \phi(x/t)$ and $\psi_t(x) = t^{-1}  \psi(x/t)$.
It is unclear whether this lemma extends to the Heisenberg group or more general contexts.
Our approach bypasses the use of this construction and of the Fourier transformation, and hence it may be used in more general settings such as stratified Lie groups.

Similarly, the standard proof of the characterisation of Hardy space by Riesz transforms uses the radial maximal function and a Cauchy--Riemann type equation.
As before, the availability of such an equation is unclear in our setting.
Our new method dominates the flag Littlewood--Paley square function by the flag Riesz transform, by combining the singular integral characterisation of Christ and Geller with a randomisation argument, based on the Khinchin inequality.
We expect that our method may be applied to singular integral characterisations of Hardy spaces in various settings.
Even in the classical setting of $\R^n$, these methods are new.

Again in the euclidean context, it is known that various maximal functions characterise the classical Hardy space $\fnspace{H}^1(\R^n)$.
These include maximal functions $\sup_t |\phi_t * f|$, where $\phi$ is normalised and has decay properties that are too restrictive to apply to the Poisson kernel, or when $\phi$ is the Poisson kernel.
To show that the maximal function defined using the Poisson kernel $p$ is connected to the maximal function using, say, a Schwartz function $\phi$ involves some special properties of the Poisson kernel, namely the analyticity of the Fourier transform $\hat p$ (see \cite[p.~187]{CFefSte72} or the existence of an asymptotic expansion at infinity for $p$ (see \cite[p.~99]{Ste2}).
Alternatively, a more brutal approach due to Uchiyama \cite{Uch} may be used.
We use a new method of showing that ``approximate identities are more or less alike'' (quoted from \cite[p.~184]{CFefSte72}) that allows us to do this with more general approximate identities that decay like the Poisson kernel, but do not have its special properties.
Since our approach does not involve the Fourier transform in an essential way, it should be applicable in other situations.

Second, in classical harmonic analysis on $\R^n$, dyadic decompositions play an important role; one key feature of these is that each dyadic cube may be written as a disjoint union of $2^n$ congruent cubes, each similar to the parent cube.
In product harmonic analysis on $\R^m \times \R^n$, an analogous role is played by dyadic rectangles.
In particular, in product Hardy space, maximal dyadic subrectangles of open sets are used to index the particles that make up an atom.
There is a similar structure on $\Hn$, as observed by Strichartz \cite{Str}, but the sets involved are rather irregular; further, there are nilpotent Lie groups where such a structure cannot exist.
To emphasize the analogy with the product space case, we defined atoms which are sums of particles associated to maximal ``subshards'' of an open set, but we also show that it would suffice to consider particles supported in arbitrary tubes and indexed arbitrarily.
This means that our approach is applicable in more general situations where ``maximal subshards'' are not at our disposition.

As part of our investigation into the definition of atoms, we clarify the sense in which an atom is a sum of particles, by requiring \emph{unconditional convergence in $\Leb^2(\Hn)$}.
Some previous works on Hardy spaces in the product or flag setting required cumbersome estimates on expressions such as
\[
\norm{ \lpar \sum_{R \in \rect} \left| r^{2m-2M} \HLap^{m} h^{2n-2N} \VLap^{n} b_R \right|^2 \rpar^{1/2}  }_{\Leb^2(\Hn)} ,
\]
whenever $0 \leq m \leq M$ and $0 \leq n \leq N$, but our approach enables us to avoid this.
In particular, we are able to give a very straightforward criterion (Proposition \ref{prop:particle-bounded-function-bounded}) for the boundedness of a linear (or sublinear) operator from $\flagHardy(\Hn)$ to $\Leb^1(\Hn)$ that involves the action of the operator on particles, not on atoms.

Third, we define the Heisenberg group in an unusual way.
In dealing with classical Hardy spaces, the basic geometry is determined by cubes rather than by euclidean balls, and in our situation a similar geometry is appropriate.
We use an $\ell^{\infty}$ rather than an $\ell^2$ distance to achieve this.
To obtain a true distance, rather than a quasidistance, we need an unusual parametrisation of the group structure.

Further, to handle the geometry of the flag structure, we have to deal simultaneously with several metrics.
For example, we deal with two heat kernels: one on the Heisenberg group, which may be estimated in terms of the Korányi distance, and another in just the central variable, which involves a Euclidean distance.
This leads to various complications, such as decompositions into sets defined using two metrics.
To characterise the boundedness of singular integral on an individual atom, we use the translation and dilation on $\Hn$ to simply and reduce the estimate to the case of a particle $a_R$ which is supported in a shard $R$ centred at the origin of $\Hn$ and of width $1$.
Then the decomposition of $\Hn$ into annuli is straightforward.
However, a direct decomposition of $\Hn$ into annuli for an arbitrary shard $R$ is also feasible by combining the Euclidean metric on $\Cn$ and the Korányi metric on $\Hn$.
This allows us to handle singular integrals that are not convolutions and further development on flag Hardy spaces associated with more general operators.
We discuss this in more depth in Section 8.1.

The fundamental tools to prove our results on Marcinkiewicz type multipliers,  are the atomic decomposition and  the auxiliary weight $w_{j,\ell}^{\epsilon}: \Hn \to \R^+$ introduced by Müller, Ricci and Stein \cite{MulRicSte2}:
\[
w_{j,\ell}^{\epsilon}(z,t)
:= 2^{-\cdim(j+\ell)} (1 + 2^{j+\ell}|z|^2)^{\cdim(1+\epsilon)}
    2^{-\ell} (1+2^{\ell}|t|)^{1+\epsilon}
\qquad\forall (z,t) \in \Hn.
\]
We need to see the interplay of $w_{j,\ell}^{\epsilon}$ with a particle $a_R$ associated to a shard $R$.
The difficulty here is that the decomposition of $\Hn$ into shards is adapted to the Korányi metric of $\Hn$ while the weight $w_{j,\ell}^{\epsilon}$ is described in terms of the Euclidean metrics in $\Cn$ and $\R$ (i.e., $|z|$ and $|t|$).
In \cite{MulRicSte2}, this was handled by using iterated one-dimensional maximal functions; these are bounded on $\Leb^p(\Hn)$ but not on our Hardy space.
We use refined decompositions of $\Hn$ to overcome this difficulty.

\subsection*{Plan of the paper}\label{ssec:intro-plan}
In Section \ref{sec:prelim}, we discuss the Heisenberg group, and prove or summarise some preliminary results.
In Section \ref{sec:atomic}, we examine the definition of the atomic Hardy space in  detail, and show that the (usually sublinear) operators that define the other Hardy spaces are all bounded on the atomic space, thereby proving one half of many of the equivalences of Theorem A.
In Section \ref{sec:Lusin}, we consider the definition of the area function Hardy space and conclude that $\areaHardy(\Hn) = \atomHardy(\Hn)$.
In Section \ref{sec:sqfn}, we study the square function Hardy spaces in detail and complete the proof that $\ctssqfnHardy(\Hn) = \dissqfnHardy(\Hn) = \atomHardy(\Hn)$.
In Section \ref{sec:maximal-Hardy-spaces}, we examine the maximal function Hardy spaces.
At this point, we know that
\[
\atomHardy(\Hn)\subseteq \gmaxHardy(\Hn) \subseteq \nontanHardy(\Hn) \subseteq \radialHardy(\Hn);
\]
we show that $\nontanHardy(\Hn) \subseteq \areaHardy(\Hn)$ and outline the proof that $\nontanHardy(\Hn)$ and $\radialHardy(\Hn)$ coincide; the proof of this last fact is similar to the corresponding proof in the product space setting of \cite[Subsection 3.2]{HanLeeLiWic}.
We complete the proof of Theorem A by showing that $\areaHardy(\Hn)=\RieszHardy(\Hn)$ and briefly discuss possible extensions in Section \ref{sec:Riesz-transform}.
Applications and their proofs appear in Section \ref{sec:applications}.

\subsection*{Acknowledgements:}

It is a pleasure to thank Alessio Martini, Jill Pipher, Fulvio Ricci and an anonymous referee for many helpful comments.


\section{Preliminaries on the Heisenberg group}\label{sec:prelim}

In this section, we summarise relevant facts on the Heisenberg group and its geometry.
The center of $\Hn$ is $\{ (0,t) : t \in \R \}$, and the canonical projection of $\Hn$ onto $\Cn$ ``forgets'' the variable $t$.

\subsection{Notation}\label{ssec:notation}

Recall that the Heisenberg group $\Hn$ is parametrised by $\Cn \times \R$; a typical element is written $g$, or in coordinates as $(z,t)$ where $z \in \Cn$ and $t \in \R$.
We may also write $(z,t)$ as $(z_1, \dots, z_\cdim , t)$, or as $(x,y,t)$, or as $(x_1, \dots, x_\cdim , y_1, \dots, y_\cdim ,t)$.
The multiplication law is
\begin{equation}\label{eq:Hn-multiplication}
(z',t') \Hprod (z, t) = (z' + z, t' + t + S(z', z) ),
\end{equation}
where
\[
S(z', z)
= 4\cdim \Im \biggl(\sum_{j=1}^\cdim  z'_j \bar{z}_{j} \biggr)
= 4\cdim \lpar\langle y', x\rangle - \langle x', y\rangle\rpar
= 4\cdim \biggl(\sum_{j=1}^\cdim  y'_j x_j  - \sum_{j=1}^\cdim  x'_j y_j \biggr).
\]
The identity of $\Hn$ is written $o$ or $(0,0)$; inversion is given by $(z, t)^{-1} = (-z,-t)$.
The reflection operator $\oper{R}$ is defined on functions, and we may write $(\oper{R}f)(g) = f(g^{-1})$ or $(\oper{R}f)(z,t) = f(-z,-t)$.
We say that $f$ is \emph{even} if $\oper{R}f = f$ and \emph{odd} if $\oper{R}f = -f$.

The Haar measure on $\Hn$ is the Lebesgue measure, which we write $\wrt g$ or $\wrt z\wrt t$.
The standard Heisenberg group convolution is given by the formula
\[
f_1 \Hconv f_2 (g)
= \int_{\Hn} f_1(g_1) f_2(g_1^{-1} g) \wrt g_1
= \int_{\Hn} f_1(gg_1) f_2(g_1^{-1}) \wrt g_1
\qquad\forall g \in \Hn.
\]
For reasonable input functions, convolution is associative but not necessarily commutative.
However, it is commutative on the space of \emph{radial functions}, by which we mean the subspace of functions that are radial in the $z$ variable (see \cite{HulRic}).
If $f_1$ and $f_2$ both have compact support, then $\supp(f_1 \Hconv f_2) \subseteq \supp(f_1) \supp(f_2)$.
We will also say that a function $f$ on $\R$ is radial if it is even.

We also convolve a function $f_1$ on $\Hn$ with a function $f_2$ on $\R$ as follows:
\[
f_1 \Vconv f_2 (g) = \int_{\R} f_1(gg_2) f_2(g_2^{-1}) \wrt g_2
\qquad\forall g \in \Hn.
\]
Observe that, for suitable functions $f$ and $f_1$ on $\Hn$ and $f_2$ on $\R$,
\[
\begin{aligned}
f \Hconv f_1 \Vconv f_2 (g)
&= \int_{\Hn} \int_\R f(g g_1) f_1 (g_1^{-1} g_2) f_2(g_2^{-1}) \wrt g_2 \wrt g_1 \\
&= \int_{\Hn} \int_\R f(g g_2 g_1) f_1 (g_1^{-1}) f_2(g_2^{-1}) \wrt g_2 \wrt g_1 \\
&= \int_{\Hn} \int_\R f(g g_1 g_2) f_1 (g_1^{-1}) f_2(g_2^{-1}) \wrt g_2 \wrt g_1,
\end{aligned}
\]
since $g_2$ is central.
If $f_1$ and $f_2$ are even functions, then we may write
\begin{equation}\label{eq:double-conv}
f \Hconv f_1 \Vconv f_2 (g)
= \int_{\Hn} \int_\R f(g g_1 g_2) f_1 (g_1) f_2(g_2) \wrt g_2 \wrt g_1
\qquad\forall g \in \Hn.
\end{equation}

For $f_1 \in \Leb^1(\Hn)$ and $f_2 \in \Leb^1(\R)$, we define the adjoints $f_1^*$ and $f_2^*$ by $f_1^*(g_1) = \bar{f}_1(g_1^{-1})$ and $f_2^*(g_2) = \bar{f}_2(g_2^{-1})$.
Equipped with their usual norms, convolution, and adjunction, $\Leb^1(\Hn)$ and $\Leb^1(\R)$, are Banach $^*$-algebras; in particular, the $\Leb^1$ convolution inequality $\norm{f \conv f'}_{\Leb^1} \leq \norm{f}_{\Leb^1} \norm{f'}_{\Leb^1}$ holds when the convolution is defined (for instance, if $f_1 \in \Leb^1(\Hn)$, $f_2 \in \Leb^1(\R)$ and the convolution is $\Vconv$).

We recall that the dilations $\dilat_r$, where $r\in\R^+$, of $\Hn$, given by $\dilat_r(z,t) := (rz, r^2t)$, are automorphisms  of $\Hn$.
So are the rotations in $\mathrm{U}(n)$, acting on the right in the $z$ variable and leaving $t$ fixed.
Rotations are measure preserving, while $|\dilat_r(E)| = r^{\hdim}|E|$ for every measurable subset $E$ of $\Hn$ and every $r\in\R^+$; here, $|E|$ denotes the measure of a subset $E$ of $\Hn$ and $\hdim$ the \emph{homogeneous dimension} $2\cdim +2$ of $\Hn$.
The normalised dilates $\phi\one_r$ and $\phi\two_s$ of functions $\phi\one$ on $\Hn$ and $\phi\two$ on $\R$ are given by
\[
\phi\one_r(z,t)  := r^{-\hdim} \phi\one(z/r,t/r^2)
\qquad\text{and}\qquad
\phi\two _2(t) := s^{-1} \phi\two (t/s),
\]
for all $r,s \in \R^+$, all $(z,t) \in \Hn$ and all $t \in \R$.
Then 
\[
[\phi\one\Hconv\psi\one]_r = \phi\one_r \Hconv \psi\one_r
\qquad\text{and}\qquad
\phi\two\Hconv\psi\two]_s = \phi\two_s \Hconv \psi\two_s
\] 
whenever the convolutions are well defined.

We identify vector fields with the associated first order differential operators.
The Lie algebra of left-invariant vector fields on $\Hn $ is spanned by the fields
\begin{equation}\label{XYT}
\begin{aligned}
\oper{X}_j
= \frac{\partial}{\partial x_j} + 4\cdim y_j \frac{\partial}{\partial t},
\qquad
\oper{Y}_j
= \frac{\partial}{\partial y_j} - 4\cdim x_j \frac{\partial}{\partial t},
\quad\text{and}\quad
\oper{T} = \frac{\partial}{\partial t},
\end{aligned}
\end{equation}
where $j=1,\dots,\cdim$.
To unify some of the formulae, when $\cdim +1 \leq j \leq 2\cdim$, we write $x_j$ for $y_{j-\cdim}$, $y_j$ for $-x_{j-\cdim}$, and $\oper{X}_j$ for $\oper{Y}_{j-\cdim}$ when $j=\cdim+1,\dots,2\cdim$.
The vector fields $\oper{X}_1$, \dots, $\oper{X}_{2\cdim}$ are called horizontal, while $\oper{T}$ is called vertical.

We note that $\oper{X}_j f$ may be considered as the convolution $f \Hconv \Upsilon_j$, where the distribution $\Upsilon_j$ is given by $\Upsilon_j (f) = \partial_j f (0)$.
The right invariant vector field $\rivf{\oper{X}}_j$ that coincides with $\oper{X}_j$ at the identity $o$ of $\Hn$ is given by \eqref{XYT}, but with $4\cdim$ and $-4\cdim$ exchanged.
Alternatively, $\rivf{\oper{X}}_j f = -\oper{R}\oper{X}_j \oper{R}f$.
Note that $\oper{X}(\phi * \psi) = \phi * \oper{X}\psi$ and
\[
(\oper{X}\phi) * \psi
= \phi * \Upsilon_j * \psi
= \phi * \rivf{\oper{X}}\psi.
\]
The vector field $\oper{T}$ is both left and right invariant, and we may write
$\oper{T}f$ as a convolution with a distribution on either side.

We define the sub-Laplacian $\HLap$ on $\Hn$ to be $-\sum_{j=1}^{2\cdim} \oper{X}_j^2$ and the Laplacian $\VLap$ to be $-\oper{T}^2$; the latter only involves the central variable.
Then $\HLap$ and $\VLap$ are radial distributions.
We write $\Hnabla$, $\riHnabla$ and $\Vnabla$ for the left and right invariant horizontal and vertical gradients, that is, $(\oper{X}_1, \dots, \oper{X}_{2\cdim})$, $(\rivf{\oper{X}}_1, \dots, \rivf{\oper{X}}_{2\cdim})$ and $\oper{T}$.
We use higher gradients, such as $\Hnabla^q$ and $\riHnabla^q$, where $q \in\N$, which are tensors; $\Hnabla^0$ means the identity operator.

Note that $\Hnabla$ and $\riHnabla$ are homogeneous of degree $1$ with respect to the dilations $\dilat_r$, while $\Vnabla$ and $\HLap$ are homogeneous of degree $2$, and $\VLap$ is homogeneous of degree $4$, in the sense that
\begin{align*}
\Hnabla \left( f \circ \dilat_r \right)
    &= r \, ( \Hnabla f) \circ \dilat_r,  &
\HLap \left( f \circ \dilat_r \right)
    &= r^2 \, ( \HLap f) \circ \dilat_r,  \\
\Vnabla \left( f \circ \dilat_r \right)
    &= r^2 \, ( \Vnabla f) \circ \dilat_r,  &
\VLap \left( f \circ \dilat_r \right) &= r^4 \, ( \VLap f) \circ \dilat_r,
\end{align*}
when $r \in \R^+$ and $f \in \fnspace{C}^2(\Hn)$.

The space of Schwartz functions on $\Hn$, written $\Schwartz(\Hn)$, is the space of functions $f$ such that $p \oper{D} f$ vanishes at infinity for all polynomials $p$ on $\Hn$ and all left-invariant differential operators $\oper{D}$.

In the context of classical Hardy spaces, it is natural to focus on cubes rather than balls, and an analogous comment applies in our setting.
There are various left-invariant distances in use on $\Hn$; the most common are the \emph{control (or Carnot--Carathéodory) distance} $d_c$, and the \emph{Korányi distance} $d_K$.
We shall use the \emph{gauge distance} $d_{(1)}$, which is defined by setting
\begin{equation}\label{eq:gauge-distance}
\begin{aligned}
d_{(1)}(g, g') := \norm{ g'^{-1} g } = \norm{ g^{-1} g'}
\qquad\forall g, g' \in \Hn ,
\end{aligned}
\end{equation}
where $\norm{{}\cdot{}}$ is given by
\begin{equation}\label{eq:gauge-norm}
\begin{aligned}
\norm{ (z,t) }
:= \max\lset |x_1|, |y_1|, \dots, |x_\cdim |, |y_\cdim |, |t|^{1/2} \rset
\qquad\forall (z,t)\in\Hn .
\end{aligned}
\end{equation}
See \cite[Section 2.2]{Tys} for a discussion, and note that we have defined the group multiplication in an unusual way to ensure that our norm satisfies the triangle inequality; our norm also behaves like the usual $\ell^{\infty}$ norm in $\R^{2\cdim+1}$.
We write $B\one(g, r)$ for the ball in $\Hn$ with centre $g$ and radius $r$ constructed using the distance $d$.
We also use balls in the centre of $\Hn$, which may be identified with $\R$: we define $B\two (t, s) := \{ t' \in \R : |t-t'| < s \}$.
Sets of the form $g B\one(o, r) B\two (0, s)$ occur often in our work: these are images of products of balls in $\Hn \times \R$ under the identification mentioned in the introduction.
We call them \emph{tubes}, and define $T(g,r,s)$ by the formula
\begin{equation}\label{eq:def-tube}
T(g,r,s)
:= g B\one(o, r) B\two (0, s)
= B\one(g, r) B\two (0, s).
\end{equation}
Thus $T(o,r,s)$ may be identified with the set $\{ z \in \C^\cdim: |z|_\infty < r \} \times (-r^2 - s, r^2+s)$.

By definition, the following conditions are equivalent: first, $g' \in  T(g,r,s)$; second,  $g' \in g B\one(o,r) B\two (0,s)$; third, $g \in g'  B\one(o,r) B\two (0,s)$; and fourth, $g \in T(g',r,s)$.
\begin{remark}\label{rem:tube-geometry}
Consider the question of finding the smallest closed tube $\bar T(o, r, s)$ centred at $o$ that contains a point $g$ in $\Hn$.
In general, if $g = (z,t)$, then we require that $r \geq |z|_\infty$, $s \geq 0$  and $s +r^2  \geq  |t|$; when $|t| > |z|_\infty$, there are infinitely many minimal tubes with this property, but the smallest by volume is well-determined:  if $g = (z,t)$, then $r = |z|_\infty$ and $s = \max\{ |t| - r^2, 0\}$.
\end{remark}

A smooth curve $\gamma$ in $\Hn$ is said to be \emph{horizontal} if its tangent vector is a linear combination of the $\oper{X}_j$ at each point, and the \emph{control (or Carnot--Carathéodory) distance} $d_c(g,g')$ between points $g$ and $g'$ in $\Hn$ is defined to be the infimum of the lengths of horizontal curves joining $g$ and $g'$.
We define the \emph{control norm} $\norm{{}\cdot{}}_c$ on $\Hn$ by
\[
\norm{g}_c := d_c(e,g)
\qquad\forall g \in \Hn.
\]
The distance $d_c$ is left-invariant, that is, $d_c(gg',gg'') = d_c(g',g'')$ for all $g, g', g'' \in \Hn$, whence
\[
\begin{aligned}
d_c((z,t),(z',t'))
&= \norm{ (z,t)^{-1} \Hprod (z',t') }_{c} .
\end{aligned}
\]
Working with the control norm and distance is not easy, and so we often replace it with equivalent norms and distances that are computationally easier.
One such is the \emph{Korányi distance} $d_K$, given by
\begin{equation}\label{dk}
\begin{aligned}
d_K(g, g') = \norm{ g'^{-1}g }_K = \norm{ g^{-1}g'}_K
\qquad\forall g, g' \in \Hn ,
\end{aligned}
\end{equation}
where the \emph{Korányi norm} $\norm{{}\cdot{}}_K$ (with our definitions) is given by
\begin{equation}\label{kn}
\begin{aligned}
\norm{ (z,t) }_K := (\|z\|^4+ 4\cdim^2 t^2)^{1/4}
\qquad\forall (z,t)\in\Hn .
\end{aligned}
\end{equation}
We do not need much about all these distances on $\Hn$, other than their equivalence.

Because our vector fields and distances are left-invariant, it is necessary to use right convolutions with, for example, the flag Poisson kernel in the definition of the nontangential and radial maximal functions.
This creates a small but important difference between our work and that of Han, Lu and Sawyer \cite{HanLuSaw}, who used left convolutions.
Reflecting functions (that is, composing with the inversion) exchanges left and right convolutions, so that in the end, the differences are minor, and we may relate our Hardy spaces to theirs by a reflection.

\subsection{Tilings and shards}\label{ssec:tiling}

Following \cite{HanLuSaw}, we use the work of \cite{Str, Tys} on self-similar tilings to find a ``nice'' decomposition of $\Hn$, analogous to the decomposition of $\R^n$ into dyadic cubes in classical harmonic analysis.
We identify $\Cn$ with $\R^{2\cdim}$, and $|z|_\infty$ denotes $\max\{ |x_1|, |y_1|, \dots |x_\cdim |, |y_\cdim | \}$, $Q_0$ denotes the cube $[-1/2,1/2)^{2\cdim}$, while $\Hn_{\Z}$ denotes the subgroup $\{ (z,t) \in \Hn : z \in \Z^{2\cdim}, t \in (2\cdim)^{-1} \Z\}$.
In this subsection, we write $P$ for the canonical projection of $\Hn$ onto $\Cn$.

\begin{theorem}[\cite{Str, Tys}]\label{thm:Str-Tys}
There is a Borel measurable function $f: Q_0 \to \R$ such that $f(0) = 1/4\cdim$ and
\[
\frac{1}{4\cdim (\cdim + 1)}
\leq f(z)
\leq \frac{2\cdim + 1}{4\cdim (\cdim + 1)}
\qquad\forall z \in Q_0,
\]
such that the set $T_o$, defined by
\[
T_o := \lset (z,t) : z \in Q_0, f(z) - \frac{1}{2\cdim} \leq t < f(z) \rset,
\]
has the property that
\[
\dilat_{2\cdim+1} (T_o) = \bigsqcup_{g \in \Delta} g T_o ,
\]
where $\Delta := \{ (z,t) \in \Hn_{\Z} : |z|_\infty \leq \cdim : |t| \leq \cdim + 1 \}$.
\end{theorem}

The proof of this is essentially the content of \cite[Lemma 3.3]{Tys}.

The definitions of $T_o$ and the metrics that we use show that
\begin{equation}\label{eq:size-of-basic-tile}
\begin{gathered}
T_o
\subset \lset(z,t) \in \Hn: |z|_\infty \leq \frac{1}{2},
        |t| \leq \frac{2\cdim + 1}{4\cdim (\cdim + 1)} \rset
\subset \lset(z,t) \in \Hn: |z|_\infty \leq \frac{1}{2},
        |t| \leq \frac{3}{8} \rset  \\
\subseteq \bar B\one \left(o, \frac{1}{2}\right) \bar B\two \left(0, \frac{1}{8}\right)
= \bar T\left(o,\frac{1}{2}, \frac{1}{8}\right) \,,
\end{gathered}
\end{equation}
where the bars indicate closures.
We may improve this to $T_o \subset \bar B\one(o,1/2)$, which is optimal, when $\cdim \geq 2$.
Similarly,
\begin{equation}\label{eq:size-of-basic-tile-2}
T_o \supset
\lset(z,t) \in \Hn: |z|_\infty < \frac{1}{2}, |t| < \frac{1}{4\cdim(\cdim+1)} \rset \supset B\one\left(o, \frac{1}{2(\cdim+1)}\right);
\end{equation}
the size of the largest open ball inside $T_o$ cannot be controlled independently of $\cdim$.

We observe that $(0,t) \in T_0$ if and only if $t \in [-1/4\cdim , 1/4\cdim)$, and in some respects we may think of $T_o$ as a modified version of $Q_0 \times [1/4\cdim , 1/4\cdim )$.

\begin{definition}\label{def:tiles}
We define
\[
\tile_0 := \{ g T_o : g \in \Hn_{\Z} \},
\qquad
\tile_j := \dilat_{(2\cdim+1)^j} \tile_0
\quad\text{and}\quad
\tile := \bigsqcup_{j \in \Z} \tile_j .
\]
We call the sets $T \in \tile$ \emph{tiles}.
Then $T = \dilat_{(2\cdim+1)^j} (g) \dilat_{(2\cdim+1)^j} (T_o)$ if $j \in \Z$, $g \in \Hn_{\Z}$ and $T = \dilat_{(2\cdim+1)^j} (g T_o)$.
We further define
\[
\cent(T) := \dilat_{(2\cdim+1)^j} (g),
\qquad
\wid(T) := (2\cdim+1)^j
\quad\text{and}\quad
\heit(T) := \frac{(2\cdim+1)^{2j}}{2\cdim} \,.
\]
\end{definition}

Recall that $\{0\} \times [-1/4\cdim, 1/4\cdim) \subset T_o$.
If $g = (m, n/2\cdim) \in \Hn_\Z$, where $m \in \Z^\cdim$ and $n = (2\cdim)^{-1}l$, where $l \in \Z$, then 
\[
g  (\{0\} \times [-1/4\cdim, 1/4\cdim) ) = \{m\} \times [(2l-1)/4\cdim, (2l+ 1)/4\cdim).
\]
We may view $gT_0$ as a modified version of $(m+Q_0) \times  [(2l-1)/4\cdim, (2l+ 1)/4\cdim)$, at least in some respects; for instance, the measure of $gT_0$ is equal to $1/2\cdim$.
However, the projection of $gT_0$ onto the vertical axis is an interval whose length may be much larger than $1/2\cdim$, due to the term $S(z, m)$, where $z$ varies over $[-1/2,1/2)^{2\cdim}$, in the Heisenberg group multiplication (see \eqref{eq:Hn-multiplication}).

\begin{theorem}[\cite{HanLuSaw, Tys, Str}]\label{thm:Heisenberg-grid}
Let $\tile_j$ and $\tile$ be defined as above.
Then the following hold:
\begin{enumerate}
  \item for each $j \in \Z$, $\tile_j$ is a partition of $\Hn$, that is, $\Hn = \bigsqcup_{T \in \tile_j} T$;
  \item $\tile$ is nested, that is, if $T, T' \in \tile$, then either $T$ and $T'$ are disjoint or one is a subset of the other;
  \item $B\one(g, C_1 q) \subseteq T \subseteq B\one(g, C_2 q)$, where $g = \cent(T)$ and $q = \wid(T)$ for each $T \in \tile$; the constants $C_1$ and $C_2$ depend only on $\cdim$;
  \item if $T \in \tile_j$, then $g T \in \tile_j$ for all $g \in \dilat_{(2\cdim+1)^j} \Hn_{\Z}$, and $\dilat_{(2\cdim+1)^k} T \in \tile_{j+k}$ for all $k \in \Z$.
\end{enumerate}
\end{theorem}

Every tile is a dilate and translate of the basic tile $T_o$, so all tiles are similar geometrically.
Hence each tile in $\tile_j$ is a fractal set---its boundary is a set of Lebesgue measure $0$ and (euclidean Hausdorff) dimension $2\cdim$---and is ``approximately'' a Heisenberg ball of radius $(2\cdim+1)^{j}$.
The decompositions are \emph{product-like} in the sense that the tiles project
onto cubes in the factor $\Cn$, and their centers form a product set.
They are like the usual dyadic decomposition into cubes in $\R^n$ in the sense that each tile may be written as the disjoint union of $(2\cdim+1)^{\hdim}$ congruent tiles.
If two tiles in $\tile_j$ are ``horizontal neighbours'', then the distance between their centres is $(2\cdim+1)^{j}$, while if they are ``vertical neighbours'', then the (Heisenberg) distance is $(2\cdim+1)^{2j}/2\cdim$.

We say that a tile $T$ is the $j$th ancestor of a tile $T'$, or that $T'$ is a $j$th descendant of $T$, when $T' \subseteq T$ and $|T|/|T'| = (2\cdim+1)^{j\hdim}$, 

Han, Lu and Sawyer \cite{HanLuSaw} used unions of tiles to pursue the analogy with rectangles in the plane $\R^2$.
We follow them, but with different nomenclature to reflect the fact that our objects are fractal, and are not based on powers of $2$, but of $2\cdim+1$.
Given tiles $T$ and $T'$, such that $T \subset T'$, the projection $P(T)$ of $T$ onto $\Cn$ is a cube, $Q$ say; let $R = P^{-1}(Q) \cap T'$.
Then $R$ is the (finite) union of the tiles $T'' \in \tile_j$ such that $T'' \subset T'$ and $P(T'') = Q$.

\begin{definition}\label{def:jagged-rect}
The \emph{shard} $R$ determined by tiles $T$ and $T'$, where $T \subseteq T'$, is defined to be the set $P^{-1}(P(T)) \cap T'$.
Then $R$ consists of $\heit(T')/\heit(T)$ copies of $T$ stacked vertically.
The centre of $R$, written $\cent(R)$, is the centre of the middle tile in the stack; the width of $R$, written  $\wid(R)$, is $\wid(T)$ and the height of $R$, written $\heit(R)$, is $\heit(T')$.
The collection of all shards is denoted $\rect$.
\end{definition}

We note that the collection of all shards is countable.
If two shards are not disjoint, their intersection is also a shard.
Shards are called vertical dyadic rectangles by Han, Lu and Sawyer \cite{HanLuSaw}; they also define horizonal dyadic rectangles, but we do not use these.

Recall that the tube $T(g,r,s)$ is the set $g B\one(o,r) B\two (0,s)$; evidently,
\begin{equation}\label{eq:size-of-tube}
|T(g,r,s)| = 2^{2\cdim+1} r^{2\cdim} (r^2 + s).
\end{equation}
Tubes are easier to handle than shards in some respects; for instance, the product of two tubes is a tube, but the product of two shards is not a shard.
However, there are no simple nested  decompositions of space into disjoint tubes.
So we need to be able to compare tubes and shards.

\begin{lemma}\label{lem:tubes-and-rectangles-1}
If $R \in \rect$ and $\cent(R) = g$, $\wid(R) = q$ and $\heit(R) = h$, then
\begin{equation} \label{eq:rect-near-tube}
R \subset \bar T(g, q/2, (4h+q^2)/8).
\end{equation}
\end{lemma}

\begin{proof}
By part (4) of Theorem \ref{thm:Heisenberg-grid}, \eqref{eq:rect-near-tube} is invariant under certain dilations and translations.
Hence we may assume that $\cent(R) = o$, $\wid(R) = 1$, and $\heit(R) = (2\cdim+1)^{2j}/2\cdim$ for some $j \in \N^+$.
In this case, $R$ is made up of $(2\cdim+1)^j$ copies of $T_o$ stacked vertically on each other, each of height $1/2\cdim$; more precisely,
\begin{equation*}
R
=  T_o \left[-\frac{(2\cdim+1)^j - 1}{4\cdim},\frac{(2\cdim+1)^j - 1}{4\cdim} \right) .
\end{equation*}
Now \eqref{eq:rect-near-tube} follows easily from \eqref{eq:size-of-basic-tile}.
\end{proof}

In the case where $R$ is not a tile, the result of Lemma \ref{lem:tubes-and-rectangles-1} may be sharpened to
\[
T(g, q/2, h/4) \subset R \subset \bar T(g, q/2, h/2)  .
\]

When we consider atomic Hardy spaces, we will want to consider functions supported in open sets, which may be written as sum of functions supported in enlargements of maximal shards of the open set.

\begin{definition}\label{def:enlargement}
Fix $\kappa \in (1 + 1/(2\nu),\infty)$.
The \emph{enlargement} $R^{*,\kappa}$, often written $R^*$, of a shard $R$ is the tube $T( g,\kappa q/2, \kappa^2(4h + q^2)/8)$, where $g = \cent(R)$, $q = \wid(R)$, and $h = \heit(R)$.
\end{definition}

\begin{lemma}\label{lem:tubes-and-rectangles-2}
For all $R \in \rect$ and $\kappa \in (1 + 1/(2\nu),\infty)$,
\begin{equation}\label{eq:rect-in-tube}
R \subset R^{*,\kappa}.
\end{equation}
Further, when $\kappa \geq (2\cdim+1)^2$, given any tube $T$, there exists $R \in \rect$ such that
\begin{equation}\label{eq:tube-near-rect}
R \subseteq T \subseteq R^{*,\kappa}.
\end{equation}
\end{lemma}

\begin{proof}
Lemma \ref{lem:tubes-and-rectangles-1} and the definition of an enlargement imply \eqref{eq:rect-in-tube} immediately.

The assertion \eqref{eq:tube-near-rect} is the $\Hn$ version of the fact that every interval $[a,b)$ in $\R$ contains an interval with centre $(2\cdim+1)^j k $ and width $(2\cdim+1)^j$ such that $[a,b)$ is contained in the interval with the same centre and width$(2\cdim+1)^{j+2}$; it suffices to take $j = \lfloor \log_{2\cdim+1} (b-a)\rfloor - 1$ and $k = \lceil \log_{2\cdim+1} a \rceil +1$.
\end{proof}

For future reference, we note that, if $R$ is a  shard, $\cent(R) = g$, $\wid(R) = q$ and $\heit(R) = h$, and $R^\ddagger$ is the shard with $\cent(R^\ddagger) = g$, $\wid(R^\ddagger) = q$ and $\heit(R^\ddagger) = (2\cdim+1)^2 h$, then
\begin{equation}\label{eq:size-of-rectangle}
\begin{gathered}
R \subset \bar T(g, q/2, (4h + q^2)/8) \subset R^\ddagger, \\
R^* =  T(g, \kappa q/2, \kappa^2 (4h + q^2)/8),\\
|R| = q^{2\cdim} h, \quad |R^\ddagger| = (2\cdim+1)^2|R| \\
\abs| R^{*} |
= \kappa^{\hdim} q^{2\cdim} (h+q^2/4) \leq \kappa^{\hdim} (1+\cdim/2) |R|,
\end{gathered}
\end{equation}
(the last inequality holds since $h \geq q^2/2\cdim$).

We need one more geometric result to help us  pass from tubes to shards.

\begin{lemma}\label{lem:rectangles-control-tubes}
Suppose that $\kappa \in (1 + 1/(2\nu),\infty)$, and that $\ell \in \N$ is such that $(2\cdim+1)^{\ell-1} \geq \kappa$.
For all $R \in \rect$, let \( R^{\mathrm\dagger} \) denote the unique shard that contains $R$ and is a translate of $\dilat_{(2\cdim+1)^\ell}R$.
Then
\[
R \subset R^{*} \subset  \Rdagger.
\]
\end{lemma}

\begin{proof}
By translating and dilating if necessary, we may suppose that $\cent(R) = o$, $\wid(R) = 1$ and $\heit(R) = (2\cdim+1)^{2k}$, where $k \in \N$.
In this case, 
\[
R^* =  T(o, \kappa /2, \kappa^2 ((2\cdim+1)^{2k} + 1)/8), 
\]
and it is easy to check that $\dilat_{(2\cdim+1)^{-\ell}}R^* \subset R$, whence $R^* \subset  \Rdagger$.
\end{proof}

\subsection{Maximal functions and Journé's lemma}\label{ssec:Journe}
In light of the results in the previous subsection, controlling averages over shards is essentially the same as controlling averages over tubes, and we use a maximal operator to do this.
Our  ``flag maximal operator'' is ``bigger'' than the usual Hardy--Littlewood maximal operator, but ``smaller'' than the ``strong maximal operator'' used in \cite{MulRicSte1, MulRicSte2}.
Recall that $\chi\one_r$ and $\chi\two _s$ denote the normalised characteristic functions of the unit balls in $\Hn$ and in $\R$, and that $\chi_{r,s}$ is short for $\chi\one_r \Vconv \chi\two _s$.

\begin{definition}\label{def-Ms}
We define the \emph{flag maximal operator} $\oper{M}_{\flag} $, the \emph{shard maximal operator} $\oper{M}_{\mathrm{sh}}$,  and the \emph{iterated maximal operator} $\oper{M}_{\mathrm{it}}$ by
\begin{align}
\oper{M}_{\flag}(f)(g)
&:= \sup_{r,s \in \R^+} \frac{1}{|T(g, r, s)|}
 \int_{T(g, r, s)} |f(g')|\wrt g'
\qquad\forall g \in \Hn \label{MF} \\
\oper{M}_{\mathrm{sh}}(f)(g)
&:= \sup_{g \in R \in \rect} \frac{1}{|R|}
 \int_{R} |f(g')|\wrt g'
\qquad\forall g \in \Hn \label{Mshard} \\
\label{eq:def-iterated-max}\phantom{\bigg|}
\oper{M}_{\mathrm{it}}(f) (g)
&:= \sup_{r,s \in \R^+} \abs|f| \Hconv \chi_{r,s}(g)
\qquad\forall g \in \Hn.
\end{align}
\end{definition}

\begin{lemma}\label{lem:flag-and-iterated-maximal-fns}
Suppose that $f$ is a Lebesgue measurable function on $\Hn$.
Then
%
$\oper{M}_{\flag}(f) \eqsim \oper{M}_{\mathrm{sh}}(f) \eqsim \oper{M}_{\mathrm{it}}(f)$.
All three maximal operators are bounded on $\Leb^p(\Hn)$ when $1 < p \leq \infty$.
\end{lemma}

\begin{proof}
We write $g \in \Hn$ as $(z,t)$.
By definition,
\begin{equation}\label{eq:max-fns-1}
\begin{aligned}
&\frac{1}{|T(g, r, s)|}
 \int_{T(g, r, s)} \abs|f(g')| \wrt g' \\
&\qquad=  \iint_{\Cn\times\R} \abs|f((z,t) \Hprod (z',t'))|
 \fn v_{(r,s)}(z', t')  \wrt t'\wrt z',
\end{aligned}
\end{equation}
where
\[
v_{(r,s)}(z', t')
= \frac{1}{|B\one(o, r) B\two (0,s)|}
\fn\indifn_{Q(0, r)}(z')\fn\indifn_{B\two (0,r^2+s)}(t'),
\]
and similarly
\begin{equation}\label{eq:max-fns-2}
\begin{aligned}
&|f| \Hconv \chi\one_{r} \Vconv \chi\two _{s}) (g) \\
&\qquad= \int_{B\one(o,r)} \int_{B\two (0,s)}
    \frac{1}{ \abs|B\one(o, r)| \abs| B\two (0,s)|}
        \abs|f(gg'g'')| \wrt g'' \wrt g' \\
&\qquad= \iint_{\Cn \times \R}
\abs|f((z,t)\Hprod (z',t') )| \fn w_{(r,s)}(t')  \wrt t' \wrt z' ,
\end{aligned}
\end{equation}
and
\[
w_{(r,s)}(t')
=
\frac{1}{ \abs|B\one(o, r)| \abs| B\two (0,s)|}
\,\indifn_{Q(0, r)}(z') \int_{\R} \indifn_{B\two (0,r^2)}(t'+t'')
    \indifn_{B\two (0, s)} (t'') \wrt t'' .
\]
Now
\[
\begin{aligned}
\min\{r^2,s\} \indifn_{B\two (0, (r^2+ s)/2 )}(t')
&\leq \int_{\R} \indifn_{B\two (0,r^2)}(t' + t'')
        \indifn_{B\two (0, s)}(t'')\wrt t'' \\
&\leq 2 \min\{r^2,s\} \indifn_{B\two (0, r^2+ s)}(t')
\end{aligned}
\]
and
\[
\begin{aligned}
\frac{\min\{r^2,s\} }{\abs| B\one(o, r)| \abs| B\two (0, s)|}
&\eqsim \frac{\min\{r^2,s\} }{r^\hdim s}
= \frac{1 }{\max\{r^2,s\} r^{2\cdim}} \\
&\eqsim \frac{1 }{(r^2+s) r^{2\cdim}}
\eqsim \frac{1 }{\abs| B\one(o, r) B\two (0, s)|},
\end{aligned}
\]
whence
$v_{(r,s/2)} \lesssim w_{(r,s)} \lesssim v_{(r,s)},$
and we complete the proof of the equivalence of $\oper{M}_{\flag}(f)$ and $\oper{M}_{\mathrm{it}}(f)$   by substituting these inequalities into \eqref{eq:max-fns-1} and \eqref{eq:max-fns-2}, and then taking suprema.

The equivalence of $\oper{M}_{\flag}$ and $\oper{M}_{\mathrm{sh}}$ follows from the second part of Lemma \ref{lem:tubes-and-rectangles-2}.

Since $\abs|f| \Hconv \chi\one_r \Vconv \chi\two _s = \abs|f| \Vconv \chi\two _s \Hconv \chi\one_r$, the maximal operator $\oper{M}_{\mathrm{it}}$ may be dominated by the composition of a Hardy--Littlewood maximal operator in the central variable with a Hardy--Littlewood maximal operator on $\Hn$, in either order.
It is now evident that $\oper{M}_{\mathrm{it}}$ and hence also $\oper{M}_{\flag}$ and $\oper{M}_{\mathrm{sh}}$ are bounded on $\Leb^p(\Hn)$ when $1 < p \leq \infty$, and all are unbounded on $\Leb^1(\Hn)$.
\end{proof}

\begin{lemma}\label{lem:where-Ms-is-big}
Let $R$ be a shard.
Then $ \oper{M}_{\flag} \indifn_{R}(g')
\geq 2(5\cdim+2)^{-1}\kappa^{-\hdim}$ for all $g' \in R^{*,\kappa}$.
\end{lemma}
\begin{proof}
Write $g = \cent(R)$, $q = \wid(R)$, and $h = \heit(R)$.
Take $g', g'' \in R^*$.
By definition, 
\[
g' \in g B\one(o,\kappa q/2) B\two (0,\kappa^2(q^2+4h)/8), 
\]
whence 
\[
g \in g' B\one(o,\kappa q/2) B\two (0,\kappa^2(q^2+4h)/8), 
\]
and
\[
\begin{aligned}
g''
&\in g B\one(o,\kappa q/2)  B\two (0,\kappa^2(q^2+4h)/8) \\
&\subseteq g'  B\one(o,\kappa q/2)  B\two (0,\kappa^2(q^2+4h/8))  B\one(o,\kappa q/2)  B\two (0,\kappa^2(q^2+4h)/8) \\
&= g'  B\one(o,\kappa q)  B\two (0,\kappa^2(q^2+4h)/4).
\end{aligned}
\]
Hence $R \subseteq R^* \subseteq g'  B\one(o,\kappa q)  B\two (0,\kappa^2(q^2+4h)/4)$, whence
\[
\begin{aligned}
&\oper{M}_{\flag} \indifn_{R}(g') \\
&\qquad\geq \frac{1}{| B\one(o,\kappa q)  B\two (0,\kappa^2(q^2+4h)/4) |}
\int_{B\one(g',\kappa q)  B\two (0,\kappa^2(q^2+4h)/4)}
    \indifn_{R} (g'') \wrt g'' \\
&\qquad= \frac{|R|}{| B\one(o,\kappa q)  B\two (0,\kappa^2(q^2+4h)/4) |}
= \frac{q^{2\cdim} h}{ 2^{2\cdim +1} (\kappa q)^{2\cdim} \kappa^2(5q^2+4h)/4 } \\
&\qquad\geq \frac{1}{2^{2\cdim}\kappa^\hdim (5\cdim+2)},
\end{aligned}
\]
since $h  \geq q^2/2\cdim$.
\end{proof}

\begin{lemma}\label{lem:expanding-sets-by-max-fn}
Suppose that $E$ is an open subset of $\Hn$ of finite measure  $|E|$, and fix $\alpha \in (0,1)$.
Define
\begin{equation*}
\tilde{E}_\alpha
:= \biglset g \in \Hn : \oper{M}_{\flag}(\indifn_{E})(g) > \alpha \bigrset.
\end{equation*}
Then
\[
\abs| \tilde{E}_\alpha|
\lesssim  \frac{\abs|E|}{\alpha^2} \,.
\]
\end{lemma}
\begin{proof}
Since the flag maximal function $\oper{M}_{\flag}$ is bounded on $\fnspace{L}^2(\Hn)$,
\begin{equation*}
\abs| \tilde{E}_\alpha|
\leq \frac{\norm{ \oper{M}_{\flag} \indifn_E }_{\Leb^2(\Hn)}}{\alpha^2}
\lesssim \frac{\norm{ \indifn_{E} }_2^2}{\alpha^2}
= \frac{\abs|E|}{\alpha^2} \,,
\end{equation*}
where the implicit constant depends only on $\cdim$.
\end{proof}

We use this inequality to control the sizes of various supersets of a given set.

\begin{definition}\label{def:maxrect}
Suppose that $E$ is an open subset of $\Hn$ of finite measure.
We write $\rect(E)$ for the collection of all $R$ in $\rect$ whose interior is a subset of $E$, and $\maxrect(E)$ for the collection of all maximal such $R$ in $\rect(E)$.
\end{definition}

We abuse language a little and call shards $R \in \rect(E)$ subshards of $E$.
Each such subshard $R$ has an enlargement $R^*$, and we can control the measure of $\bigcup_{R\in \rect(E)} R^*$.

\begin{corollary}\label{cor:measure-of-enlargement of open set}
Suppose that $E$ is an open subset of $\Hn$ of finite measure.
Then
\begin{equation}\label{omega enlarge}
\begin{aligned}
\Bigabs| \bigcup_{R \in \rect(E)} R^* | \lesssim |E|.
\end{aligned}
\end{equation}
The implicit constant depends only on the enlargement parameter $\kappa$ and $\cdim$.
\end{corollary}

\begin{proof}
We take $\alpha$ to be $2/(5\cdim+2)$ in Lemma \ref{lem:expanding-sets-by-max-fn}, 
By the definitions and Lemma \ref{lem:where-Ms-is-big}, $\oper{M}_{\flag} \indifn_{E}(g')
\geq \oper{M}_{\flag} \indifn_{R}(g')
\geq 2(5\cdim+2)^{-1}\kappa^{-\hdim}$;
now \eqref{omega enlarge} follows.
\end{proof}

For a shard $R$, there is a unique shard, $ \Rdagger$ say, that contains $R$ and is a translate of $\dilat_{2\cdim+1} R$.
It is easy to check that, if the enlargement parameter $\kappa$ in the definition of $R^*$ is large enough, then $ \Rdagger \subseteq R^*$.

\begin{corollary}\label{cor:R-and-R-dagger}
Suppose that $E$ is an open subset of $\Hn$ of finite measure.
Then $\bigl|  \bigcup_{R \in \rect(E)}  \Rdagger  \bigr| \lesssim |E|$.
The implicit constant depends only on $\cdim$.
\end{corollary}

For $R \in \maxrect(E)$ and $\alpha \in (0,1)$, there may be several $R' \in \maxrect(\tilde{E}_\alpha)$ that contain $R$.
However, for a given width, there is a unique such shard; indeed, if $R \subseteq R'\subseteq \tilde{E}_\alpha$ and $R \subseteq R''\subseteq \tilde{E}_\alpha$, and $\wid(R') = \wid(R'')$, then the projections $P R'$ and $P R''$ of $R'$ and $R''$ onto $\Cn$ are $(2\cdim+1)$-adic cubes of the same size that contain $P R$, so they must coincide, and then $R \subseteq R' \cup R'' \subseteq \tilde{E}$; if $R'$ and $R''$ are both maximal, then $R' = R''$.
Similarly, there is a unique maximal shard of maximal height; indeed, if there were two, one would be wider than the other, and since the heights coincide, the narrower shard would be properly contained in the wider shard.

We now recall Journé's covering lemma, which was first proved by Journé \cite{Jou} in $\R\times \R$, and later by Fefferman \cite{RFef86, RFef} and Pipher \cite{Pip} in higher dimensions and with more factors.
It has been extended to products of spaces of homogeneous type; see, for example \cite{HanLiLin}.
The preliminary version of \cite{HanLuSaw} stated a flag version of this covering lemma.
We give some more notation to state the version.

\begin{definition}\label{def:E-one-E-two}
Let $E$ be an open subset of $\Hn$ of finite measure, and let $\alpha_1$ and $\alpha_2$ be constants in $(0,1)$.
We define sets $E\one$ and $E\two$ (which also depend on $\alpha_1$ and $\alpha_2$) as follows:
\begin{equation}\label{eq:def-E-one-E-two}
\begin{aligned} 
E\one &= \lset g \in G : \oper{M}_{\mathrm{sh}} \indifn_E(g) > {\alpha_1} \rset \\
E\two &= \lset g \in G : \oper{M}_{\mathrm{sh}} \indifn_{E\one}(g) > {\alpha_2} \rset.
\end{aligned}
\end{equation}
Given a shard $R$ in $\maxrect(E)$, we define $R\one$ to be the widest shard  $S \in \maxrect(E\one)$ such that $R \subseteq S$, and $R\two$ to be the highest shard $S \in \maxrect(E\two)$ such that $R\one \subseteq S$.
We write $\maxrect\one(E\one)$ for the collection of $R\one$ that arise as $R$ varies over $\maxrect(E)$.

If $R$ and $S$ are shards and $R \subseteq S$, and $\epsilon_1, \epsilon_2 \in \R^+$, then we define
\begin{equation}\label{eq:def-rho}
\begin{aligned}
\rho_{\boldsymbol\epsilon}(R,S)
:= \lpar \frac{\wid(R)}{ \wid(S)} \rpar^{\epsilon_1}
+ \lpar \frac{\heit(R)}{ \heit(S) }\rpar^{\epsilon_2} .
\end{aligned}
\end{equation}
\end{definition}

\begin{lemma}\label{lem:Journe}
Suppose that $E$ is an open subset of $\Hn$ of finite measure, and that $C_1, C_2 > 1$.
Then it is possible to choose the constants $\alpha_1$ and $\alpha_2$ in Definition \ref{def:E-one-E-two} small enough that
\[
\begin{aligned}
\frac{\wid(R\one)}{\wid(R)} \geq C_1
\qquad&\text{and}\qquad   \frac{\heit(R\one)}{\heit(R)} \geq 1\\
\frac{\heit(R\two)}{\heit(R\one)} \geq C_2
\qquad&\text{and}\qquad
\frac{\heit(R\two)}{\heit(R\one)} \geq 1
\end{aligned}
\]
for all $R \in \maxrect(E)$.
Further, for all $\epsilon_1, \epsilon_2\in\R^+$,
\begin{equation}\label{eq:main-Journe-inequality}
\sum_{R\in \maxrect(E)}
\rho_{\boldsymbol\epsilon}(R,R\two) |R|
\leq c |E|,
\end{equation}
where $c$ does not depend on $E$, but may depend on $\epsilon_1$, $\epsilon_2$, $C_1$ and $C_2$.
\end{lemma}

\begin{proof}
First, take $\alpha_1$ to be at most $(C_1+2)^{-\hdim}$ and $\alpha_2$ to be at most $(C_2+2)^{-1}$.
Then it is clear that
\[
\wid(R\two) \geq \wid(R\one) \geq C_1 \wid(R)
\qquad\text{and}\qquad
\heit(R\two) \geq C_2 \heit(R\one) \geq \heit(R) .
\]

Next, there is no loss of generality in assuming that $E$ is a  finite union $\bigcup_{R\in \set{F}} R$ of shards, each of which is maximal in $E$ (provided that our estimate is independent of the cardinality of the collection $\set{F}$ of shards).
Then the finitely many projected sets $P(R)$, as $R$ varies over $\set{F}$, are cubes in $\C^\cdim$, and by the translation and dilation invariance of \eqref{eq:main-Journe-inequality}, we may and shall suppose that all these cubes are subcubes of the cube $Q_0 := [-1/2,1/2)^{2\cdim}$.

Next, given such a  set $E$ and $\alpha \in (0,1)$, let $\tilde E_\alpha$ be $\{ g \in \Hn: \oper{M}_{\mathrm{sh}} \indifn_{E}(g) > \alpha\}$, let $R^{(w)}$ be the widest  shard $S \in \maxrect(\tilde E_\alpha)$ such that $R \subseteq S$ and $R^{(h)}$ be the highest shard $S \in \maxrect(\tilde E_\alpha)$ such that $R \subseteq S$.
We claim that, for all $\epsilon$ in $\R^+$,
\begin{equation}\label{eq:simpler-Journe}
\sum_{R \in \maxrect(E)}
\lpar \frac{\wid(R)}{\wid(R^{(w)})} \rpar^{\epsilon}
|R|  \lesssim_{\epsilon}  |E|
\qquad\text{and}\qquad
\sum_{R \in \maxrect(E)}
\lpar \frac{\heit(R)}{\heit(R^{(h)})} \rpar^{\epsilon}
|R|  \lesssim_{\epsilon}  |E| .
\end{equation}

One the one hand, $R^{(w)} \subseteq R\two$ (as defined in Definition \ref{def:E-one-E-two}), and so the first inequality of our claim \eqref{eq:simpler-Journe} implies that
\[
\sum_{R \in \maxrect(E)}
\lpar \frac{\wid(R)}{\wid(R\two)} \rpar^{\epsilon_1}
|R|  \lesssim_{\epsilon_1}  |E| .
\]
On the other hand, since shards in $\maxrect(E)$ of the same height are either disjoint or coincide, the second inequality of our claim, applied to $E\one$ rather than to $E$, shows that
\[
\begin{aligned}
&\sum_{R \in \maxrect(E)}
\lpar \frac{\heit(R)}{\heit(R\two)} \rpar^{\epsilon_2} |R| \\
&\qquad=  
\sum_{S \in \maxrect(E\one)} 
\sum_{\substack{R \in \maxrect(E)\\ R\one = S}}
\lpar \frac{\heit(S)}{\heit(R\two)} \rpar^{\epsilon_2}
\lpar \frac{\heit(R)}{\heit(S)} \rpar^{\epsilon_2}
|R| \\
&\qquad=
\sum_{S \in \maxrect(E\one)} 
\lpar \frac{\heit(S)}{\heit(R\two)} \rpar^{\epsilon_2}
\sum_{j_2 \in \N}
(2\cdim+1)^{- \epsilon_2 j_2}
\sum_{\substack{R \in \maxrect(E)\\ R\one = S \\ \heit(R) = (2\cdim+1)^{-j_2} \heit(S) }}
|R| \\
&\qquad\leq  
\sum_{S \in \maxrect(E\one)} 
\lpar \frac{\heit(S)}{\heit(S^{(h)})} \rpar^{\epsilon_2}
\sum_{j_2 \in \N}
(2\cdim+1)^{- \epsilon_2 j_2}
|S| \\
&\qquad\lesssim_{\epsilon_2}  
\sum_{S \in \maxrect(E\one)} 
\lpar \frac{\heit(S)}{\heit(S^{(h)})} \rpar^{\epsilon_2}
\abs|S| \\
&\qquad\lesssim_{\epsilon_2} \abs|E\one|  \\
&\qquad\lesssim_{\alpha_1} \abs|E|,
\end{aligned}
\]
Hence it suffices to prove \eqref{eq:simpler-Journe}, and we shall do this shortly.

We define a \emph{$(2\cdim+1)$-adic flag rectangle in $\C^\cdim \times \R$} to be a set of the form $Q \times I$, where $Q$ is a \emph{$(2\cdim+1)$-adic cube} in $\C^\cdim$, which we identify with $\R^{2\cdim}$; more precisely,
\[
Q = \prod_{i=1}^{2n} \left[(m_i-\tfrac{1}{2} ) (2\cdim+1)^j, (m_i+\tfrac{1}{2}) (2\cdim+1)^j\right)
\] 
for some $m_i$ and $j$ in $\Z$, while $I$ is a \emph{$(2\cdim+1)^2$-adic interval} in $\R$, that is,
\[
I = [(m-\tfrac{1}{2}) (2\cdim+1)^{2k}, (m+\frac{1}{2}) (2\cdim+1)^{2k})
\] 
for some $k$ and $n$ in $\Z$, where $k \geq j$.  
We write $\set{Q}$ and $\set{I}$ for the collections of all $(2\cdim+1)$-adic cubes in $\C^\cdim$ and all $(2\cdim+1)^2$-adic intervals in $\R$.

Recall the function $f:Q_0 \to \R$ of Theorem \ref{thm:Str-Tys}, and let $F: P^{-1}(Q_0) \to P^{-1}(Q_0)$ be the invertible Borel measurable function
\[
F(z,t) = (z, 2\cdim(t - f(z)) ) 
\qquad \forall z \in Q_0 \quad\forall t \in \R.
\]
Then $R$ is a shard if and only if $F(R)$ is a $(2\cdim+1)$-adic flag rectangle in $\C^\cdim \times \R$. 
The base of this rectangle in $\C^\cdim$ is exactly the cube $P(R)$, while the projection of $F(R)$ onto $\R$ is a $(2\cdim+1)^2$-adic interval, which we label $P_\R(R)$. 

We first prove that 
\begin{align}\label{Journe 2}
\sum_{R\in \maxrect(E)}
\lpar\frac{\heit(R)}{\heit(R^{(h)})} \rpar^{\epsilon} |R|
\lesssim_\epsilon  \abs|E| .
\end{align}

If $R\in \rect$, then there exist a unique tile $T$ and a unique cube $Q$ such that $R = T \cap P^{-1}Q$.
We define the $j$th vertical ancestor $(R)_j$ of $R$ to be the shard $(T)_j  \cap P^{-1}Q$, where $(T)_j$ is the $j$th ancestor of $T$, and $\sigma: \maxrect(E) \to \N^+$ as follows:
\[
\sigma(R) = \max\lset k \in \N :  \abs|(R)_{k} \cap E| > {\alpha} \abs|(R)_k| \rset + 1.
\]

Now we set, for each $I \in \set{I}$ and $k\in \N_+$, 
\begin{align*} 
\set{A}_{I,k} 
=\Bigl\{ R \in \maxrect(E)  :  P_{\R}(R) =  I , \ \sigma( R) = k \Bigr\}
\qquad\text{and}\qquad
A_{I,k} = \bigcup_{R \in \set{A}_{I,k}} R.
\end{align*}
Note that $\set{A}_{I,k}$ and hence also $A_{I,k}$ may be empty.
Clearly, if $R, R' \in \set{A}_{I,k}$, then $R, R' \in \maxrect(E)$ and $\heit(R) = \heit(R')$, so $R$ and $R'$ coincide or are disjoint, whence
\begin{equation}\label{sum Qk AQK}
\begin{aligned}
\sum_{R\in \maxrect(E)}
\lpar\frac{\heit(R)}{\heit(R^{(w)})}\rpar^{\epsilon} |R|
&=   \sum_{ I \in \set{I} } \sum_{k=1}^\infty \sum_{R\in \set A_{I,k} } (2\cdim+1)^{-2k\epsilon} \abs |R| \\
&= \sum_{ I \in \set{I} } \sum_{k=1}^\infty(2\cdim+1)^{-2k\epsilon}\  |A_{I,k} | .
\end{aligned}
\end{equation}

We now estimate $ \abs| A_{I,k} |$. 
For all $R$ in $\maxrect(E)$, all $I \in \set{I}$, and all $j \in \N$, define
\begin{gather*}
\set{B}_{R,j} 
= \{ S \in \maxrect(E) :  P( S) \subset P(R) , \ P_{\R}(S)   \supseteq (P_{\R}(R))_j     \} ,
\\
B_{R,j} 
= \bigcup_{S \in\set{B}_{R,j}} S ,
\qquad
\set{B}_{I,j} 
= \bigcup_{R \in \set{A}_{I,j} } \set{B}_{R,j} 
\qquad\text{and}\qquad
B_{I,j} =\bigcup_{ S \in \set{B}_{I,j} } S .
\end{gather*}
It is evident that $B_{R,j+1} \subseteq B_{R,j}$ so $B_{I,j+1} \subseteq B_{I,j}$, and $B_{I,j} \subseteq E$ for all $j$; furthermore, $\set{B}_{R,j+l}  = \set{B}_{(R)_j, l} $ for all $j$ and $l$ in $\N$.

For all $I \in \set{I}$ and $k \in \N^+$, every $R\in \set A_{I,k} $ lies in $\maxrect(E)$, and
by definition, 
\begin{align*} 
| (R)_k \cap B_{R,k} |
\leq  | (R)_k \cap B_{I,k} | 
\leq |(R)_{k} \cap E | 
\leq { \alpha} \abs|(R)_{k} | ,
\end{align*}
which implies that 
$\abs| R \cap B_{I,k} |  \leq  \alpha \abs| R  | $,  
and hence, since $R = (R\cap B_{I,0})$, that
\begin{align*} 
\abs| R \cap ( B_{I,0} \setminus  B_{I,k} )|>  (1- { \alpha} ) \abs|R|,
\end{align*}
that is,
\begin{align*} 
\oper{M}_{\mathrm{sh}} ( \indifn_{B_{I,0}\setminus B_{I,k} } )(g)  >  1 - { \alpha} 
\qquad\forall g \in R.
\end{align*}
Since $A_{I,k} = \bigcup_{R \in \set{A}_{I,k}} R$, the inequality above holds for all $g \in A_{I,k}$, that is,
\begin{align*} 
A_{I,k} 
\subseteq \Bigl\{g\in G: \oper{M}_{\mathrm{sh}} ( \indifn_{ B_{I,0}\setminus B_{I,k}})(g) >  1-{ \alpha }\Bigr\}.
\end{align*}
The $L^2(G)$-boundedness of the maximal function $\oper{M}_{\mathrm{sh}}$ now implies that 
\begin{align*} 
|A_{I,k}|
\leq \Big| \Big\{ g\in G: 
\oper{M}_{\mathrm{sh}} ( \indifn_{ B_{I,0} \setminus B_{I,k} } ) (g) > 1 - { \alpha}
	\Big\}\Big|
\lesssim_{\alpha} |B_{I,0} \setminus B_{I,k}| ,
\end{align*}
and hence
\begin{align} \label{eq:size of AIk2}
|A_{I,k}|
\lesssim_{\alpha}   |B_{I,0}\setminus B_{I,1}| + \cdots + |B_{I,{k-1}}\setminus B_{I,k}|   .
\end{align}

Finally, we substitute the estimate \eqref{eq:size of AIk2} into the right-hand side of \eqref{sum Qk AQK}, change the order of summation, and deduce that
\begin{equation}\label{eq:last-step-Journe}
\begin{aligned}
\sum_{R\in \maxrect(E)}
\lpar \frac{\heit(R)}{\heit(R^{(h)})} \rpar^{\epsilon} |R|
&\lesssim_{\alpha}  \sum_{k=1}^\infty \sum_{j = 0}^{k-1} 2^{-2k\epsilon} \sum_{ I \in \set{I} } \abs|B_{I,j}\setminus B_{I,j+1}|\\
&=  \sum_{k=1}^\infty \sum_{j = 0}^{k-1} 2^{-2k\epsilon} \sum_{ I \in \set{I} } \abs|B_{(I)_j,0}\setminus B_{(I)_j,1} |\\
&=  \sum_{k=1}^\infty \sum_{j = 0}^{k-1} 2^{-2k\epsilon} \sum_{ I \in \set{I} } \abs|B_{I,0}\setminus B_{I,1} |\\
&= \sum_{k=1}^\infty 2^{-2k\epsilon} k \abs|E | \\
&\lesssim_{\epsilon} |E|,
\end{aligned}
\end{equation}
and thus \eqref{Journe 2} holds.

To conclude, we prove that 
\begin{align}\label{Journe 1}
\sum_{R\in \maxrect(E)}
\lpar\frac{\wid(R)}{\wid(R^{(w)})} \rpar^{\epsilon} |R|
\lesssim_\epsilon  \abs|E| .
\end{align}

We start by changing notation, and define the horizontal ancestors of a rectangle $R$.
If $R\in \rect$, then there exist a unique tile $T$ and a unique cube $Q$ such that $R = T \cap P^{-1}Q$.
If $R \subset T$,  then the horizontal parent $(R)_1$ of $R$ is the rectangle $T \cap P^{-1}((Q)_1)$, where $ (Q)_1 $ is the parent cube of $Q$, while if $R = T$, then $(R)_1$ is the tile $(T)_1$.
The $j$th ancestor $(R)_j$ of $R$ is then defined to be the horizontal parent of the  $(j-1)$st ancestor $(R)_{j-1}$. 
By definition, the width of $(R)_j$ is  $(2\cdim+1)^{j}$ times the width of $R$, while the height of $(R)_j$ need not be the same as the height of $R$.

We now define,  for each $R \in \maxrect(E)$, 
\[
\sigma(R) = \max\lset k \in \N :  \abs|(R)_{k} \cap E| \geq {\alpha} \abs|(R)_k| \rset +1 ,
\]
and for each $R \in \rect$,
\[
\tau(R) = \tfrac{1}{2} \log_{2\cdim+1} \heit(R) - \log_{2\cdim+1} \wid(R).
\]
Then $\tau((R)_1) = \tau(R) -1$ if $R \in \maxrect(E)$ is not a tile.
The horizontal parent $(R)_1$ is defined differently when $\tau(R) > 0$ and when $\tau(R) = 0$, so we shall treat these cases separately.

Let $\maxrect_<(E) := \{ R \in \maxrect(E): \tau(R) < \sigma(R)\}$ and $\maxrect_\geq(E) := \{ R \in \maxrect(E): \tau(R) \geq \sigma(R)\}$. 
Let $\Sigma(E) := \{ (R)_{\sigma(R)} : R \in \maxrect(E) \}$ and $\Tau(E) := \{ (R)_{\tau(R)} : R \in \maxrect(E) \}$.
If $R \in \Sigma(E)$ and $g \in R$, then $\abs| (R)_{\sigma(R)-1} \cap E| \geq \alpha \abs| (R)_{\sigma(R)-1}|$, which implies that
\[
\oper{M}_{\mathrm{sh}} (\indifn_{E}) (g)
>  \frac{\abs| (R)_{\sigma(R)-1} \cap E| }{  \abs| (R)_{\sigma(R)-1} | }
\geq\frac{ \alpha}{(2\cdim+1)^\hdim} \,,
\]
so that 
\[
\bigcup_{R \in \Sigma(E)}  \subseteq E^*_{(2\cdim+1)^{-\hdim} \alpha} 
:= \{ g \in G : \oper{M}_{\mathrm{sh}}(\indifn_E)(g) > (2\cdim+1)^{-\cdim} \alpha \}, 
\]
and the measure of $\abs| E^*_{(2\cdim+1)^{-\hdim} \alpha}|  \lesssim_\alpha \abs|E|$ because $\oper{M}_{\mathrm{sh}}$ is $\Leb^2$-bounded.
It is easy to see that the tiles in $\Sigma(E)$ are pairwise disjoint, from the definition of $\sigma$, whence
\begin{equation}\label{eq:size of outer tiles}
\sum_{R \in \Sigma(E)}\abs| R |  \lesssim_\alpha \abs |E| .
\end{equation}

Thus, on the one hand,
\[
\begin{aligned}
&\sum_{ R \in \maxrect_<(E) }  
	\lpar\frac{\wid(R)}{\wid(R^{(w)})}\rpar^{\epsilon} \abs|R|  \\
&\qquad= \sum_{  R \in \maxrect_<(E)  } 
	(2\cdim+1)^{-\epsilon\sigma(R)} \abs|R|   \\
&\qquad= \sum_{S \in \Sigma(E)} \sum_{T \in \Tau(E)}  \sum_{ \substack{ R \in \maxrect_<(E) \\ (R)_{\tau(R)} = T \\  (R)_{\sigma(R)} = S}}  
	(2\cdim+1)^{-\epsilon(\sigma(R)-\tau(R))} (2\cdim+1)^{-\epsilon\tau(R)} \abs|R|  \\
&\qquad= \sum_{i \in\N} (2\cdim+1)^{-\epsilon i} \sum_{j \in\N} ( 2\cdim+1))^{-\epsilon j}
	\sum_{S \in \Sigma(E)} \sum_{T \in \Tau(E)}  
	\sum_{ \substack{ R \in \maxrect_<(E) \\ (R)_{i+j} = S \\  (R)_{i} = T}}  \abs|R|   \\
&\qquad\lesssim_\epsilon  \sum_{j \in\N} (2\cdim+1)^{-\epsilon j} 
	\sum_{S \in \Sigma(E)} \sum_{\substack{ T \in \Tau(E)\\ (T)_j = S}}  \abs|T|   \\
&\qquad\lesssim_\epsilon \sum_{S \in \Sigma(E)}  \abs|S|   \\ 
&\qquad\lesssim \abs|E| .
\end{aligned}
\]
The first inequality holds because when $i$ is fixed, the distinct shards $R \in \maxrect(E)$ such that $(R)_i = T$ have the same width, and are therefore disjoint, while the second inequality holds because the distinct tiles $T$ such that $(T)_j = S$ have fixed width and are therefore disjoint.
The last inequality is the estimate \eqref{eq:size of outer tiles}.

On the other hand, the sum 
\begin{equation}\label{eq:sum-some-rects}
\sum_{ R \in \maxrect_\geq(E) }  \lpar\frac{\wid(R)}{\wid(R^{(w)})}\rpar^{\epsilon} \abs|R| 
\end{equation} 
may be treated by an argument like that used to prove \eqref{Journe 2}.
The key definitions are that
\begin{align*} 
\set{A}_{Q,k} 
= \Bigl\{ R \in \maxrect(E)_{\geq}  :  P(R) =  Q , \ \sigma( R) = k \Bigr\}
\qquad\text{and}\qquad
A_{Q,k} = \bigcup_{R \in \set{A}_{Q,k}} R.
\end{align*}
for each $Q \in \set{I}$ and $k\in \N_+$, and,
for all $R$ in $\maxrect(E)$, all $Q \in \set{Q}$, and all $j \in \N$, 
\begin{gather*}
\set{B}_{R,j} 
= \{ S \in \maxrect(E) :  P(S)   \supseteq (P(R))_j   ,\ P_{\R}( S) \subset P_{\R}(R)   \} ,
\\
B_{R,j} 
= \bigcup_{S \in\set{B}_{R,j}} S ,
\qquad
\set{B}_{Q,j} 
= \bigcup_{R \in \set{A}_{Q,j} } \set{B}_{R,j} 
\qquad\text{and}\qquad
B_{Q,j} =\bigcup_{ S \in \set{B}_{Q,j} } S ,
\end{gather*}
and the key steps of the proof are estimating the sum \eqref{eq:sum-some-rects} in terms of the measures of the $A_{Q,k}$, much as in \eqref{sum Qk AQK}, estimating the measures of the $A_{Q,k}$ in terms of the measures of the  $A_{Q,k}$, much as in \eqref{eq:size of AIk2}, and arguing with sums as in \eqref{eq:last-step-Journe}.
\end{proof}

We say that a constant is \emph{geometric} if it depends on inherent properties of the Heisenberg group and its geometry (including its decompositions into tiles and shards), the apertures of the cones that appear in the various definitions, and the enlargement factors that connect shards and supports of particles; a \emph{geometric multiple} is defined similarly.
The constants mentioned above are geometric, except that some depend on $p$, others on $M$ and $N$, and others on positive parameters $\alpha$ and $\epsilon$.
We use the notation $A \lesssim B$ to mean that there is a constant $C$ such that $A \leq C B$, and $A \eqsim B$ to mean that $A \lesssim B$ and $A \lesssim B$.
If the constant is geometric, we do not necessarily point this out explicitly.
However, when the implicit constant in one of these inequalities depends on a nongeometric constant, we often indicate this explicitly, to make the proof more transparent, for example, we might write $\sum_{k \in \N} 2^{-\epsilon k} \eqsim_\epsilon 1$.
Many of the constants that we use in our proofs are geometric.

\subsection{Flag Sobolev inequalities}\label{ssec:Sobolev}

Because we are dealing with two Laplacians, there are various Sobolev-type inequalities possible.
A comprehensive account of many of these is due to Folland \cite{Fol}, to which the reader should refer for unexplained estimates below.
We present several that will be useful for us.

\begin{lemma}\label{lem:flag-Sobolev}
Suppose that $Q$ is a cube in $\Cn$ of side-length $q$, with sides parallel to the axes, and that $f: Q \to \R$ is measurable. 
Suppose also that $b \in  \Leb^2(\Hn)$ and $\supp(b) \subseteq \{ (z,t) \in \Hn : z \in Q, f(z) \leq t \leq f(z) + h \}$.
\begin{enumerate}
\item[(a)]
If also $\oper{T} b \in \Leb^2(\Hn)$, then
\[
\lpar \int_{\Hn} \abs| b(g) |^2 \wrt g \rpar^{1/2}
\lesssim h  \lpar \int_{\Hn} \abs| \oper{T} b(g) |^2 \wrt g \rpar^{1/2}.
\]
\item
If also $\HLap b \in \Leb^2(\Hn)$, then
\[
\lpar\int_{\Hn} \abs| b(g) |^2 \wrt g \rpar^{1/2}
\lesssim q^2 \lpar\int_{\Hn} \abs| \HLap b(g) |^2 \wrt g \rpar^{1/2}.
\]
\end{enumerate}
\end{lemma}

\begin{proof}
By calculus in $\R$,
\[
\int_\R \abs| c(t) |^2 \wrt t
\lesssim h^2
  \int_\R \Bigl| \frac{d}{dt} c(t) \Bigr|^2 \wrt t
\]
for a function $c$ on $\R$ supported in an interval of length $h$.
The first estimate is proved by integrating the inequality above in $z$.

To prove the second estimate, by translating and dilating, it suffices to suppose that $Q$ is the cube with centre $0$ in $\Cn$ and side-length $1$.
Again we denote by $P$ the canonical projection from $\Hn$ to $\C^\cdim$.
The fundamental solution $k$ of the sub-Laplacian $\HLap$ on $\Hn$ is well-known to be a power of the Korányi norm \cite{Fol73}, and
\[
\begin{aligned}
b(g')
&= \int_{\Hn} \HLap b(g) \fn k(g^{-1} g') \wrt g \\
&= \int_{\Hn} \HLap b(g) \fn \indifn_{Q-Q}(P(g^{-1}g')) \fn k(g^{-1} g') \wrt g
\qquad\forall g' \in \Hn,
\end{aligned}
\]
because of the support restrictions on $b$ and hence on $\HLap b$.
Hence
\[
\norm{b}_{\Leb^2(\Hn)}
\leq \norm{\HLap b}_{\Leb^2(\Hn)}
\norm{ (\indifn_{Q-Q}\circ P) k }_{\Leb^1(\Hn)}
\leq C_\cdim \norm{\HLap b}_{\Leb^2(\Hn)},
\]
by an easy calculation.
\end{proof}

\begin{lemma}\label{lem:L2-Linfty-Sobolev}
Suppose that $b \in  \Leb^2(\Hn)$ is supported in $T(g,r,s)$.
If $\HLap^M b$, $\VLap b$, $\HLap^M\VLap b$ are also in $\Leb^2(\Hn)$, where $M > \hdim/4$, then $b \in \Leb^\infty(\Hn)$, and $\norm{b}_{\Leb^\infty}$ is controlled by
\[
\frac{1}{|T(g,r,s)|^{1/2}}
\lpar \norm{b}_{\Leb^2(\Hn)} + \norm{\HLap^M b}_{\Leb^2(\Hn)}
+ \norm{\VLap b}_{\Leb^2(\Hn)} + \norm{\HLap^M \VLap b}_{\Leb^2(\Hn)} \rpar .
\]
\end{lemma}

\begin{proof}
By translation and dilation, it suffices to suppose that $g = o$ and $r =1/2$.
Let $h$ be the height of $T(o,r,s)$.

There are smooth nonnegative-valued functions $\eta_j$ and $\tilde\eta_j$ on $\R$ such that $\supp \eta_j \subseteq [j-1, j+1]$, $\eta_j = \eta_0(\cdot -j)$, and $\sum_{j\in\Z} \eta_j = 1$ while $\supp (\tilde\eta_j) \subseteq [j-2, j+2]$, $\tilde\eta_j = \tilde\eta_0(\cdot -j)$, and $\tilde\eta_j = 1$ on $\supp(\eta_j)$.
Abusing notation, we consider these as functions on $\Hn$ that depend on $t$ but not $z$.

Given $b$ as in the enunciation, we write $b_j$ for $\eta_j b$.
Then
\[
\begin{aligned}
\norm{b}_{\Leb^\infty(\Hn)}
&\lesssim \sup_{j \in \Z} \norm{b_j}_{\Leb^\infty(\Hn)}
\lesssim \sup_{j \in \Z} \bignorm{\HLap^M b_j}_{\Leb^2(\Hn)} \\
&\lesssim \sup_{j \in \Z} \lpar
\bignorm{\tilde\eta_j \HLap^M b}_{\Leb^2(\Hn)}
+ \bignorm{\tilde\eta_j b}_{\Leb^2(\Hn)} \rpar ,
\end{aligned}
\]
by an argument similar to that used to prove the second part of the previous lemma.

From calculus in $\R$,
\[
\begin{aligned}
& \sup_{j \in \Z}
 \lpar \int_{\R} \abs| \tilde\eta_j \HLap^m b(z,t) |^2 \wrt t \rpar^{1/2} \\
&\qquad\lesssim h^{-1/2} \lpar
\lpar \int_{\R} \abs| \VLap \HLap^m b(z,t) |^2 \wrt t \rpar^{1/2}
+ \lpar \int_{\R} \abs| \HLap^m b(z,t) |^2 \wrt t \rpar^{1/2} \rpar
\end{aligned}
\]
for $m$ equal to $0$ or $M$.
We integrate in $z$ to deduce that
\[
 \sup_{j \in \Z} \bignorm{ \tilde\eta_j \HLap^m b }_{\Leb^2(\Hn)}
\lesssim h^{-1/2} \lpar
\bignorm{ \HLap^m \VLap b }_{\Leb^2(\Hn)}
+ \bignorm{ \HLap^m b }_{\Leb^2(\Hn)} \rpar.
\]
The desired result follows easily.
\end{proof}

\subsection{Vanishing moments on $\Hn$}\label{ssec:moments}

We say that a function $f$ on $\Hn$ is homogeneous of degree $d$ if $f\circ \dilat_r = r^d f$ and that a differential operator $\oper{D}$ on $\Hn$ is homogeneous of degree $e$ if $\oper{D}(f\circ \dilat_r)= r^e \oper{D}(f) \circ \dilat_r$.
If $f$ is homogeneous of degree $d$ and $\oper{D}$ is homogeneous of degree $e$, then the function $\oper{D}f$ is homogeneous of degree $d - e$.
The coordinate functions $z_j$ (and their real and imaginary parts $x_j$ and $y_j$) are homogeneous of degree one, while the coordinate $t$ is homogeneous of degree $2$, and each homogeneous polynomial $p$ has a positive degree, which is the product of the degrees of the coordinate functions involved in each of the monomial terms of $p$.
A monomial is an expression $p(x,t) = x_1^{\alpha_1} \dots x_{2n}^{\alpha_{2n}} t^k$, where each $\alpha_j \in \N$ and $k \in \N$.

Our next results are the outcome of reflections on a remark of Fulvio Ricci, who pointed out to us that if all moments of homogeneous order up to $k$ of a radial Schwartz function $f$ vanish, then there are radial Schwartz functions $g_{a,b}$, where $2a+b = k+1$, such that $f = \sum_{a,b} \HLap^a \oper{T}^b g_{a,b}$.
See \cite{FisRicYak} (especially Lemma 5.2) for more on this type of result.
These authors use a higher order extension of Hadamard's lemma to prove this.
We are interested in functions with compact support, and for us a different version of these ideas, based on a lemma of de Rham \cite[Lemme II]{dRh}, is useful.

\begin{lemma}\label{deRham-Rn}
Suppose that $\phi \in \fnspace{C}^\infty(\R^n)$, that $\supp(\phi) \subseteq [-1,1]^n$ and that $m \in \N$.
If
\[
\int_{\R^n} \phi(x) \fn p(x) \wrt x = 0
\]
for all homogeneous polynomials $p$ on $\R^n$ of homogeneous degree at most $m$, then there are $\fnspace{C}^\infty(\R^n)$ functions $f_\alpha$ supported in $[-1,1]^n$, for $\alpha \in I^n_{m+1}$, such that
\begin{equation}\label{eq:deriv-0}
\phi = \sum_{\alpha \in I^n_{m+1} } \partial^\alpha f_\alpha ,
\end{equation}
where $I^n_{m+1}$ is the set of multi-indices $(\alpha_1, \dots, \alpha_n) \in \N^n$ such that $|\alpha| =m+1$.
The $f_\alpha$ in \eqref{eq:deriv-0} may be chosen to depend smoothly and linearly on $\phi$.
\end{lemma}

\begin{proof}
In this proof, $J^n_{m+1}$ denotes the set of multi-indices $(\alpha_1, \dots, \alpha_n) \in \N^n$ such that $|\alpha|  \leq m+1$.

We apply induction on the dimension $n$ of the ambient space $\R^n$.
If $n =1$, we take
\[
f_{m+1}(x) = \int_{-\infty}^{x} \frac{(x-t)^{m}}{m!} \fn \phi(t) \wrt t,
\]
and the result is a routine verification.

Suppose that the result holds in $\R^{n-1}$, and take $\phi$ on $\R^n$ that satisfies the hypotheses of the lemma.
We write $x \in \R^n$ as $(\hat x, x_n)$, where $\hat x \in \R^{n-1}$.

Take smooth functions $\hat\phi_\beta$ on $\R^{n-1}$, supported in $[-1,1]^{n-1}$, such that, for all  $\alpha, \beta \in \N^{n-1}$, of order at most $m$,
\begin{equation}\label{eq:phi-hat-props}
\int_{[-1,1]^{n-1}} \hat{x}^\alpha \fn\hat\phi_\beta(\hat{x}) \wrt \hat{x} = \delta_{\alpha,\beta}
\end{equation}
(this is the Kronecker delta).
For such $\beta$, define
\[
g_\beta(x_n) = \int_{\R^{n-1}} x^\beta \fn \phi(\hat x, x_n) \wrt \hat x
\qquad\forall x_n \in \R,
\]
so that
\[
\int_{\R^{n-1}} x^\alpha \Bigl[ \phi(\hat{x}, x_n) - \sum_{\beta \in J^{n-1}_{m}} \hat\phi_\beta(\hat{x}) \fn g_\beta(x_n)  \Bigr] \wrt \hat x = 0
\]
for all $x \in \R$ and all $\alpha \in J^{n-1}_m$.
By the inductive hypothesis, we may write
\[
\phi(\hat{x}, x) - \sum_{\beta \in J^{n-1}_m} \hat\phi_\beta(\hat{x}) \fn g_\beta(x_n)
= \sum_{\alpha \in I^{n-1}_{m+1}} \partial^\alpha \psi_\alpha(\hat x, x_n) ,
\]
where each $\psi_\alpha$ is smooth and supported in $[-1,1]^{n}$.
So it suffices to consider the terms $\hat\phi_\beta(\hat{x}) \fn g_\beta(x_n)$.

By definition,
\[
\int_{\R} x_n^j \fn g_\beta(x_n) \wrt x_n = 0
\]
when $0 \leq j \leq m - |\beta|$, so that
\[
g_\beta(x_n) = \frac{d^{m+1- |\beta|}}{dx_n^{m+1- |\beta|}} G_\beta(x_n),
\]
where
\[
G_\beta(x_n)
= \int_{-\infty}^{x_n}
        \frac{(x_n-t)^{m-|\beta|}}{(m-|\beta|)!} g_\beta(t) \wrt t.
\]
Moreover, from \eqref{eq:phi-hat-props} and the inductive hypothesis, there are functions $\psi_{\beta,\alpha}$ on $\R^{n-1}$ such that
\[
\hat\phi_\beta
= \sum_{\alpha \in I^{n-1}_{m - |\beta| +1} } \partial^\alpha \psi_{\beta,\alpha} .
\]
The lemma follows.
\end{proof}

\begin{definition}\label{def:higher-divergence}
Fix $m,n\in \N$, and let $\bar{T}$ be a closed tube.
We write $\Hnabla^m \Vnabla^n \cdot \fnspace{C}^\infty(\bar{T})^{\otimes}$ for the space of all linear combinations  of expressions $\oper{D} \oper{T}^n f$, where $\oper{D}$ is a product of $m$ vector fields, each chosen from $\{\oper{X}_1, \dots, \oper{X}_{2\cdim}\}$, while $f \in \fnspace{C}^\infty(\Hn)$ and $\supp(f) \subseteq \bar{T}$.
We define $\riHnabla^m \Vnabla^n \cdot \fnspace{C}^\infty(\bar{T})^{\otimes}$ analogously, with right invariant vector fields.
\end{definition}

In the next proposition, we identify $\Hn$ with $\R^{2\cdim} \times \R$, and write elements of $\Hn$ as $(x,t)$, where $x \in \R^{2\cdim}$ and $t\in\R$.
We write $p_{\alpha, k}$ for the monomial $x_1^{\alpha_1} \dots x_{2n}^{\alpha_{2n}} t^k$.

\begin{prop}\label{prop:moments}
Let $m,n \in \N$, and $\phi \in \fnspace{C}^\infty(\bar{T})^{\otimes}$.
Then
\begin{equation}\label{eq:moment-1}
\int_{\R} p_{\alpha,k}(x,t) \fn\phi(x,t) \wrt t =0 
\qquad \forall x \in \R^{2\cdim}
\end{equation}
for all monomials $p_{\alpha,k}$ when $k < n$ if and only if 
$\phi \in \Vnabla^n \cdot \fnspace{C}^\infty(\bar{T})^{\otimes}$.
Further, 
\begin{equation}\label{eq:moment-2}
\int_{\Hn} p_{\alpha,k}(g) \fn\phi(g) \wrt g  =0
\end{equation}
for all monomials  $p_{\alpha,k}$ when $ |\alpha| + 2k < m$
if and only if $\phi \in \Hnabla^m \cdot \fnspace{C}^\infty(\bar{T})^{\otimes}$.

Consequently, $\phi \in \Hnabla^m \Vnabla^n \cdot \fnspace{C}^\infty(\bar{T})^{\otimes}$ if and only if \eqref{eq:moment-1} holds for all monomials $p_{\alpha,k}$ when $k < n$, and 
\[
\int_{\Hn} p_{\alpha,k}(g) \fn\psi(g) \wrt g  =0
\]
for all monomials  $p_{\alpha,k}$ when $k \geq n$ and $ |\alpha| + 2(k-n) < m$, 
where $\psi \in \fnspace{C}^\infty(\bar{T})$ is such that $\phi = \Vnabla^n \psi$.
\end{prop}

\begin{proof}
The proposition is invariant under translations and dilations, so without loss of generality we take $\bar{T}$ to be $\bar{B}\one(o,1)  \bar{B}\two(0,h)$.
It is evident that \eqref{eq:moment-1} implies \eqref{eq:moment-2} (for the same $p_{\alpha,k}$).

Suppose that $\phi =  \Vnabla^n \psi$.
Integration by parts shows that, if $n>k$, then
\[
\int_{\R} p(x,t) \fn\phi(x,t) \wrt t 
= (-1)^n\int_{\R} \Vnabla^n p(x,t) \fn\psi(x,t) \wrt t 
=0
\qquad \forall x \in \R^{2\cdim} .
\]
Conversely, if \eqref{eq:moment-1} holds when $n < k$, then we may define
\[
\psi(x,t) = \int_{-\infty}^{t} \frac{(t-s)^{n-1}}{(n-1)!} \fn \phi(x,s) \wrt s
\qquad\forall x \in \R^{2\cdim},
\]
and it is clear that $\psi \in \fnspace{C}^\infty(\bar{T})$ and $\phi = \Vnabla^{n} \psi$.

Similarly, if  $\phi = \Hnabla^m \cdot \psi$ (here $\psi$ is tensor valued and we apply the summation convention), then
\[
\int_{\Hn} p(g) \fn(\Hnabla^m \cdot \psi)(g) \wrt g
=  (-1)^{m} \int_{\Hn} (\Hnabla^m p)(g) \cdot \psi(g) \wrt g .
\]
When $ |\alpha| + 2k < m$, the homogeneous degree of the polynomial $\Hnabla^m p$ is 
$|\alpha| + 2k - m$, which is negative, so $\Hnabla^m p = 0$ and \eqref{eq:moment-2} holds.

Conversely, suppose that \eqref{eq:moment-2} holds.
Much as in the proof of the previous lemma, we take smooth functions $\hat\phi_{\alpha}$ on $\R^{2\cdim}$, supported in $[-1,1]^{2\cdim}$, such that, for all  $\alpha, \beta \in \N^{2\cdim}$, of order at most $m-1$,
\begin{equation*}
\int_{[-1,1]^{2\cdim}} x ^\alpha \fn\hat\phi_\beta(x ) \wrt x  = \delta_{\alpha,\beta}
\end{equation*}
(this is the Kronecker delta).
For such $\beta$, define
\[
g_\beta(t) = \int_{\R^{2\cdim}} x^\beta \fn\phi(x, t) \wrt x
\qquad\forall t \in \R,
\]
so that
\[
\int_{\R^{2\cdim}} x^\alpha \Bigl[ \phi(x , t) - \sum_{\beta \in J^{2\cdim}_{m-1}} \hat\phi_\beta(x ) \fn g_\beta(t)  \Bigr] \wrt x = 0
\]
for all $t \in \R$ and all $\alpha \in J^{2\cdim}_{m-1}$.
By the previous lemma, we may write
\[
\phi(x , t) - \sum_{\beta \in J^{2\cdim}_{m-1}} \hat\phi_\beta(x ) \fn g_\beta(t)
= \sum_{\alpha \in I^{2\cdim}_{m}} \partial^\alpha \phi_\alpha(x, t) ,
\]
where each $\phi_\alpha$ is smooth and supported in $\bar{T}$.
Now we consider the  $\hat\phi_\beta(x ) \fn g_\beta(t)$.

By definition,
\[
\int_{\R} t^j \fn g_\beta(t) \wrt t = 0
\]
when $0 \leq 2j \leq m - |\beta|$, so that
\[
g_\beta(t) = \frac{d^{\bar m+1}}{dt^{\bar m+1}} G_\beta(t),
\]
where $\bar m = \lfloor (m - |\beta|)/2 \rfloor$ and
\[
G_\beta(t)
= \int_{-\infty}^{t}
        \frac{(t-s)^{\bar m}}{\bar m!} g_\beta(s) \wrt s.
\]
Moreover, from \eqref{eq:phi-hat-props} and the previous lemma, there are functions $\phi_{\beta,\alpha}$ on $\R^{2\cdim}$ such that
\[
\hat\phi_\beta
= \sum_{\alpha \in I^{2\cdim}_{|\beta| +1} } \partial^\alpha \phi_{\beta,\alpha} .
\]
Hence we may write
\[
\phi
= \sum_{(\alpha,k) \in I} \partial^{\alpha, k} \phi_{\alpha,k},
\]
where $I$ is the collection of multi-indices $(\alpha,k)$, where $\alpha \in \N^{2\cdim}$, $k \in \N$, and $k = \lfloor (m + 2 - |\alpha|)/2 \rfloor$.

To conclude, we show that each summand belongs to $\Hnabla^m \cdot \fnspace{C}^\infty(\bar{T})^{\otimes}$.
We write the partial derivatives on $\Hn$ in terms of the left invariant vector fields; then
\[
\partial^{\alpha, k}
= (\oper{X}_{\alpha_1} + c y_{\alpha_1} \oper{T}) \dots (\oper{X}_{\alpha_j} + c y_{\alpha_j} \oper{T}) \oper{T}^{k} ,
\]
where $j = |\alpha|$ and $k = \lfloor (m + 2 - j)/2 \rfloor$.
We write $\oper{T}$ as the commutator $[\oper{X}_1, \oper{X}_{\cdim+1}]$, and expand the expression, obtaining a weighted sum of (not necessarily commuting) products of vector fields $\oper{X}_i$ and monomials $y_j$.
Each summand is homogeneous, of degree $j+2k$, and so the number of vector fields that occur, minus the number of monomials, is equal to $j+2k$.

Observe that $y_i \oper{X}_j = \oper{X}_j y_i  - \delta_{i,j}$.
By iterating this, we may move all the monomials to the right, and end with a weighted sum of terms of the form $\oper{X}_{i_1} \dots \oper{X}_{i_a}  p_{\mathbf{i}}$, where the number of vector fields that occur, minus the homogeneous degree of the polynomial, is equal to $j+2k$.
Hence we realise $\partial^{\alpha, k} \phi_{\alpha,k} $ as a sum of terms of the required form
\[
\oper{X}_{j_1} \dots \oper{X}_{j_b} q_{\mathbf{j}} \phi_{\alpha,k},
\]
where $b \geq j+2k \geq m+1$.
We group together all the terms that begin with the same invariant differential operator $\oper{X}_{j_1} \dots \oper{X}_{j_{m+1}}$, which shows that  $\phi \in \Hnabla^m \cdot \fnspace{C}^\infty(\bar{T})^{\otimes}$, as claimed.

The last part of the lemma is proved by combining the first two results.
\end{proof}

\begin{corollary}\label{cor:null-moments-are-derivs}
Suppose that $\phi \in \fnspace{C}^\infty(\bar{T})$ and $m,n \in \N$.
Then $\phi \in \Hnabla^m\Vnabla^n\cdot \fnspace{C}^\infty(\bar{T})^{\otimes}$ if and only if $\phi \in \riHnabla^m\riVnabla^n \cdot \fnspace{C}^\infty(\bar{T})^{\otimes}$.
\end{corollary}

\begin{proof}
The following are equivalent: first,  $\phi \in \riHnabla^q \cdot \fnspace{C}^\infty(\bar{T})^{\otimes}$; second, $\oper{R}\phi \in \fnspace{C}^\infty(\bar{T})$ (as $\oper{R}$ exchanges left invariance with right invariance); third, certain moments of $\oper{R}\phi$ vanish as in Proposition \ref{prop:moments}; fourth, certain moments of $\phi$ vanish as in Proposition \ref{prop:moments} (because $\oper{R}$ preserves homogeneity); and finally,  $\phi \in \riHnabla^q \cdot \fnspace{C}^\infty(\bar{T})^{\otimes}$ by the proposition.
\end{proof}

\begin{remark}
It is not clear whether $\HLap^m\VLap^n \fnspace{C}^\infty(\bar{T}) = \riHLap^m \VLap^n \fnspace{C}^\infty(\bar{T})$.
\end{remark}

\subsection{Estimates for the heat and Poisson kernels}\label{ssec:heat-Poisson}

Let $h\one_r$ (where $r\in\R^+$) be the heat kernel for $\HLap$, that is, the convolution kernel of $\expe^{-r \HLap}$ on $\Hn$.
By homogeneity,
\begin{equation}\label{hkp1}
\begin{aligned}
h\one_r(z, t)
= r^{-\hdim/2} h\one_1 \circ \dilat_{1/{\sqrt{r}}}
\qquad\forall r\in\R^+ .
\end{aligned}
\end{equation}
The following Gaussian upper bound for the heat kernel in terms of the  control norm $\norm{\cdot}_{c}$ holds:
\begin{equation}\label{ht kernel}
\begin{aligned}
| \Hnabla^q h\one_r(z,t)|
&\lesssim_{q,\epsilon}
r^{-q/2-\hdim/2}\exp\Bigl( -\frac{ \norm{(z,t)}_{c} ^2}{4(1+\epsilon) r} \Bigr)
\end{aligned}
\end{equation}
for all $\epsilon \in \R^+$.
This is proved in \cite[p.~48]{VarS-CCou}.
There is a similar lower bound (see \cite[p.~61]{VarS-CCou}) for the heat kernel (but not its derivatives), namely,
\[
 h\one_r(z,t)
\gtrsim_\epsilon
r^{-\hdim/2}\exp\Bigl( - \frac{ \norm{(z,t)}_{c}^2}{4(1-\epsilon) r} \Bigr).
\]

The subordination formula
\[
\expe^{-\lambda}
= \frac{1}{\sqrt{\pi}} \int_{\R^+} \frac{\expe^{-v}}{\sqrt{v}} \fn\expe^{-\lambda^2/4v} \wrt v
\qquad\forall \lambda \in \R^+
\]
leads us to corresponding estimates for the Poisson kernel $p\one$.
By functional calculus, for all $r \in \R^+$,
\begin{align*}
\expe^{-r \sqrt\HLap}
&= \frac{1}{\sqrt{\pi}} \int_{\R^+} \frac{\expe^{-v}}{\sqrt{v}} \fn\expe^{-r^2 \HLap/4v} \wrt v\\
\noalign{\noindent{\text{and}}}
 \Hnabla^q \expe^{-r \sqrt\HLap}
&= \frac{1}{\sqrt{\pi}} \int_{\R^+} \frac{\expe^{-v}}{\sqrt{v}} \Hnabla^q \expe^{-r^2 \HLap/4v} \wrt v
\end{align*}
so from \eqref{ht kernel},
\[
\begin{aligned}
 \abs| \Hnabla^q p\one_{r}(g) |
&\lesssim
\int_{\R^+} \frac{\expe^{-v}}{\sqrt{v}} \lpar\frac{r^2}{4v}\rpar^{ -(q + \hdim)/2}
  \expe^{ -  v \norm{g}_{c}^2/(1+\epsilon) r^2 }  \wrt v \\
&= 2^{q+\hdim} r^{-q - \hdim} \int_{\R^+}  v^{(q+\hdim -1)/2}
  \expe^{ -v -  v \norm{g}_{c}^2 /(1+\epsilon) r^2}  \wrt v  \\
&= 2^{q+\hdim} r^{-q -\hdim}
\frac{ \Gamma((q + \hdim + 1)/2)}{ (1 + \norm{g}_{c}^2/ (1+\epsilon)r^2) ^{(q + \hdim + 1)/2}} \\
&\eqsim \frac{ r }{ (r^2 + \norm{g}_{c}^2)^{(q + \hdim + 1)/2}}
\end{aligned}
\]
for all $g \in \Hn$; the implicit constants depend on $q$, $\cdim$ and $\epsilon$.
Analogously,
\[
p\one_{r}(g)
\gtrsim  \frac{ r }{ (r^2 + \norm{g}_{c}^2)^{\hdim/2 + 1/2}}
\qquad\forall g \in G.
\]
For computation, we may and shall replace the control norm by the gauge norm, and we shall use the estimates
\begin{equation}\label{eq:poisson-est-1}
\abs| \Hnabla^q p\one_{r}(g) |
\lesssim_{m,n} {\frac{ r }{ \max\{ r, \norm{g}\}^{q + \hdim + 1}}}
\qquad\forall g \in G.
\end{equation}

Similar estimates hold for the usual Poisson kernel $p\two _s$ on $\R$, namely,
\begin{equation}\label{eq:poisson-est-2}
\abs| \Vnabla^q p\two _s(t) |
\lesssim_q  \frac{ s }{\max\{s, |t|\}^{q+2}}
\end{equation}
for all $s,t \in \R$  and $q \in \N$.
Hence
\[
\int_{\Hn} \abs|  \Hnabla^q p\one_r(g) | \wrt g
\lesssim_q r^{-q}
\qquad\text{and}\qquad
\int_{\R} \abs|  \Vnabla^q p\two _s(t) | \wrt t
\lesssim_q s^{-q}  .
\]

\subsection{Poisson boundedness}\label{ssec:Poisson-boundedness}
We are going to be interested in smooth functions on $\Hn$ that are supported in $B\one(o,1)$, and in smooth functions that behave much like the Poisson kernel or its derivatives.
The following family of norms works naturally in both situations.

\begin{definition}\label{def:Poisson-bounds}
Fix $\theta \in \R^+$.
The \emph{Poisson norms} on the space of smooth (possibly vector-valued) functions on $\Hn$ are given by
\begin{equation}\label{eq:poisson-like-decay-1}
\bignorm{\phi\one }_{(Q)}
=   \sup\lset \max\{1, \norm{g}\}^{q + \hdim + \theta} \bigabs|\Hnabla^q \phi\one(g)|:
    g \in \Hn, 0 \leq q \leq Q\rset .
\end{equation}
We write $\fnspace{P}^\theta(\Hn)$ for the Fréchet space of functions $\phi\one$ such that $\plainnorm{\phi\one }_{(Q)}$ is finite for all $Q \in \N$.
We say that a family  $\family\one(\Hn)$ of functions on $\Hn$ is \emph{$\theta$-Poisson-bounded} if $\plainnorm{\phi\one }_{(Q)}$ is bounded for all $\phi\one \in \family\one(\Hn)$, for sufficiently many $Q \in \N$.

Similarly, the \emph{Poisson norms} on the space of smooth (possibly vector-valued) functions on $\R$ are given by
\begin{equation}\label{eq:poisson-like-decay-2}
\bignorm{\phi\two }_{(Q)}
=   \sup \left\{ \max\{1, |t| \}^{q + 1 + \theta} \bigabs|\Vnabla^q \phi\two (t)| :
    t \in \R, 0 \leq q \leq Q \right\} ;
\end{equation}
we write $\fnspace{P}^\theta(\R)$ for the Fréchet space of functions $\phi\two $ on $\R$ such that $\plainnorm{\phi\two }_{(Q)}$ is finite for all $Q \in \N$.
We say that a family $\family\two (\R)$ of functions on $\R$ is \emph{$\theta$-Poisson-bounded} if  $\plainnorm{\phi\two }_{(Q)}$ is bounded for all $\phi\two \in \family\two (\R)$, for sufficiently many $Q \in \N$.

We often write $\bsym\phi$ for a pair $(\phi\one ,\phi\two )$ of functions, where $\phi\one$ is defined on $\Hn$ and $\phi\two $ is defined on $\R$.
A family $\family$ of such pairs is said to be $\theta$-Poisson bounded if the families of first and second components are $\theta$-Poisson-bounded.

Finally, $\theta$-Poisson bounded functions $\phi\one : \Hn \to \C$ or $\phi\two : \R \to \C$ are said to be \emph{normalised} if their total integral is $1$.
\end{definition}

Mostly we deal with the case where $\theta = 1$, and omit $\theta$ from the notation.

In this definition, ``sufficiently many'' is left unspecified; in any particular calculation involving Poisson bounded families of functions, it is only necessary to control finitely many derivatives, but the number of derivatives varies from one calculation to another, and depends on $\cdim$; we do not bother keeping tabs on these dependencies.

By \eqref{eq:poisson-est-1} and \eqref{eq:poisson-est-2}, the Poisson kernels $p\one$ on $\Hn$ and $p\two $ on $\R$ lie in $\fnspace{P}(\Hn)$ and $\fnspace{P}(\R)$, and evidently $\fnspace{P}(\Hn) \subseteq \Leb^1(\Hn)$ while $\fnspace{P}(\R) \subseteq \Leb^1(\R)$.


The closed subspaces of $\fnspace{P}(\Hn)$ and of $\fnspace{P}(\R)$ of functions supported in the closed balls $\bar B\one(o,1)$ and $\bar B\two (0,1)$ coincide with the closed subspaces of the Fréchet spaces $\fnspace{C}^\infty(\Hn)$ and $\fnspace{C}^\infty(\R)$ supported in $\bar B\one (o,1)$ and $\bar B\two (0,1)$.

\begin{remark}\label{rem:either-side-will-do}
There is another family of Poisson norms on the space of smooth functions on $\Hn$ , namely, 
\begin{equation}\label{eq:poisson-like-decay-1-r}
\bignorm{\phi\one }_{(\rivf Q)}
=   \sup\lset \max\{1, \norm{g}\}^{q + \hdim + \theta} \bigabs|\riHnabla^q \phi\one(g)|:
    g \in \Hn, 0 \leq q \leq Q\rset ,
\end{equation}
where we consider a right translation invariant gradient rather than a left invariant gradient.
These norms give rise to the same Fréchet space as the norms used in the definition of $\fnspace{P}(\Hn)$.
Indeed, $\riHnabla = \Hnabla \phi\one + 8 \cdim y \oper{T}$ (here $y$ is a vector), so that
\[
\begin{aligned}
\max\{1, \norm{g}\}^{2 + \hdim}  \bigabs|\riHnabla \phi\one(g)| 
&\lesssim \max\{1, \norm{g}\}^{2 + \hdim} \lpar \bigabs| \Hnabla \phi\one(g)| 
+ \norm{g} \bigabs| \oper{T} \phi\one(g)| \rpar \\
&\lesssim \norm{ \phi\one }_{(1)} 
+  \norm{ \phi\one }_{(2)} ,
\end{aligned}
\]
so that $\plainnorm{ \phi\one }_{(\rivf 1)}  \lesssim  \plainnorm{ \phi\one }_{(2)}  $.
Similarly, $\plainnorm{ \phi\one }_{(1)}  \lesssim  \plainnorm{ \phi\one }_{(\rivf 2)} $.
Inductively we may estimate all the norms from one family by sums of norms from the other family.
\end{remark}

\begin{definition}\label{def:grand-max-fn}
Given a Poisson bounded family $\family$ of pairs of functions, the  associated \emph{grand maximal function} $\gmaxop (f)$ of $f \in \Leb^1(\Hn)$ is defined by
\begin{align*}
\gmaxop (f) (g)
:= \sup_{\bsym\phi \in \family} \sup_{r,s \in \R^+}
\abs|f \Hconv \phi_{r,s} (g)|
\qquad\forall g \in \Hn,
\end{align*}
where $\phi_{r,s}$ denotes the convolution product $ \phi\one_{r} \Vconv \phi\two _{s}$ of normalised dilates.
\end{definition}

In the next lemma, we estimate the size of $\phi_{r,s}$, where $\bsym\phi$ is Poisson-bounded, and hence compare two maximal functions.
We recall that $\chi_{r,s} = \chi\one_r \Vconv \chi\two _s$, where $\chi\one_r$ is the normalised characteristic function of the ball $B\one(o, r)$ and $\chi\two _s$ is the normalised characteristic function of the ball $B\two (0, s)$.

\begin{lemma}\label{lem:poisson-and-flag-maximal-functions}
Fix $\beta,\gamma\in\R^+$, and a Poisson bounded family $\family$ of pairs of functions on $\Hn$ and on $\R$.
Then
\begin{enumerate}
  \item $p_{r,s} \simeq p_{\beta r, \gamma s}$ for all $r,s \in \R^+$;
  \item $\chi_{r, s} \lesssim p_{r,s} \lesssim \sum_{i,j \in \N} 2^{-i-j} \chi_{2^i r, 2^j s}$ for all $r, s \in \R^+$;
  \item $\abs|\phi_{r,s}| \lesssim p_{r,s}$ for all $\bsym\phi \in \family$;
  \item $ \gmaxop (\abs|f|) \eqsim \oper{M}_{\flag} (\abs|f| ) $ for all measurable functions $f$ on $\Hn$.
\end{enumerate}
\end{lemma}

\begin{proof}
The first and second parts follow immediately from \eqref{eq:poisson-est-1}  and \eqref{eq:poisson-est-2}, and the third from \eqref{eq:poisson-est-1}, \eqref{eq:poisson-est-2} and Definition \ref{def:Poisson-bounds}.
The final part follows from the definitions.
\end{proof}

In particular, the grand maximal operator is bounded on $\Leb^p(\Hn)$ when $1 < p \leq \infty$.
A similar result is shown in \cite{CFefSte72}.
The following lemma extends the above computation, and enables us to reduce many calculations involving Poisson bounded functions to calculations involving smooth functions with compact support.

\begin{lemma}\label{lem:molecules-are-sums}
Fix $N \in \N^+$ and $\theta \in (0,1)$.
There are continuous linear maps $L^{1,i}$ on $\fnspace{P}^\theta(\Hn)$ such that $\supp L^{1,0} \phi \subset B\one(o,1)$ and $\supp L^{1,i} \phi \subset B\one(o,2^j) \setminus B\one(o, 2^{j-2})$ for all $i \in \N^+$, and
\[
\phi = \sum_{i \in \N} L^{1,i} \phi
\qquad\forall \phi \in \fnspace{P}^\theta(\Hn).
\]
Moreover, we may choose the $L^{1,i}$ such that $\norm{L^{1,i} \phi}_{(Q)} \lesssim_{N,Q} \norm{\phi}_{(Q)}$ and the sum converges in $\fnspace{P}^\theta(\Hn)$.
Further, if $\phi$ has mean $0$, then we may choose the $L^{1,i}$ such that all moments of $L^{1,i}\phi$ of homogeneous order $m$ vanish if $m \leq N$.
\end{lemma}

\begin{proof}
By composing a suitable smooth partition of unity on $[0,\infty)$ with a smooth norm, such as the Kóranyi norm, we may easily find smooth functions $\eta^i$ such that $\sum_{i\in\N}\eta^i = 1$,  $\supp \eta^0 \subseteq B\one(o,1)$,  $\supp \eta^1 \subseteq B\one(o,2) \setminus B\one(o,1/2)$, and $\eta^i = \eta^1(2^{-i} \cdot)$ when $i \geq 2$.
Clearly $\sum_{i\in\N} \eta^i \phi = \phi$ when $\phi \in \fnspace{P}^\theta(\Hn)$ and $\norm{\eta^i\phi}_{(Q)} \lesssim \norm{\phi}_{(Q)}$ for all $Q \in \N$.

For the rest of the proof, we suppose that $\phi$ has mean $0$.
In this case, there is no reason for the $\eta^i\phi$ to have the required vanishing moments, so we apply a correction.

Take a basis $\{b_1, \dots, b_J\}$ for the vector space of polynomials on $\Hn$ of homogeneous degree at most $N$ such that $b_j$ is homogeneous of degree $d_j$, and a dual set of $\fnspace{C}^\infty(\Hn)$ functions $\{f_1, \dots, f_J\}$, supported in $B\one(o,1) \setminus B\one(o,1/2)$,  such that
\[
\int_{\Hn} f_j(g) \fn b_k(g) \wrt g = \delta_{j,k}
\]
($\delta_{j,k}$ is the Kronecker delta) when $j,k \in \{1, \dots, J\}$.
Adding a nonzero multiple of a dilate of $f_j$ to $\eta^i\phi$ changes its integral against $b_j$ but does not change other moments in the range of interest.

We define the coefficient $c_k^I(\phi)$, where $K \in \N$, by
\[
c_k^I (\phi)
=
\int_{\Hn} \sum_{i = 0}^{I} (\eta^i \phi)(g) \fn b_k(g)\wrt g.
\]
Then for $I \in \N$, the correction term $C_k^I(\phi)$ is defined by
\[
C_k^{I} (\phi)
=  c_k^{I}(\phi) 2^{-I d_k} [f_k]_{2^I}
 - c_k^{I-1}(\phi) 2^{(1-I) d_k} [f_k]_{2^{I-1}}.
\]
If $d_k > 0$, then
\[
\abs| c_k^I (\phi) |
\lesssim \int_{B\one(0,2^{I+1})} \abs|\phi(g) \fn b_k(g)| \wrt g
\lesssim \int_{0}^{2^{I+1}} \frac{r^{d_k}}{1+r^{\theta+ \hdim}} r^{\hdim-1} \wrt r
\lesssim 2^{I(d_k-\theta)}.
\]
And if $d_k = 0$, then $\int_{\Hn} \phi(g) \fn b_k(g) \wrt g = \int_{\Hn} \phi(g) \wrt g =  0$, whence similarly
\[
\begin{aligned}
\abs| c_k^I (\phi) |
&= \abs| -\int_{\Hn} \sum_{i = I+1}^{\infty} (\eta^i\phi)(g) \wrt g|
\lesssim \int_{\Hn \setminus B\one(0,2^{I-1})} \abs|\phi(g) | \wrt g
\lesssim 2^{I(d_k-\theta)}.
\end{aligned}
\]
Hence
\[
\begin{aligned}
\abs| c_k^{I}(\phi) 2^{-I d_k} [f_k]_{2^I}|
& = \abs| \frac{c_k^{I}(\phi)} {2^{I(\hdim + d_k)}} f_k \circ\dilat_{2^{-I}} | \\
&\lesssim
2^{-I(\theta + \hdim)} \fn \indifn_{B\one(o,2^{I})\setminus B\one(o,2^{I-1})},
\end{aligned}
\]
and so the correction terms all lie in $\fnspace{P}^{\theta}(\Hn)$ with controlled Poisson norms.
Further, the correction terms form a telescopic series with sum $0$.
Finally,
\[
\begin{aligned}
&\int_{\Hn} \left( \eta^k\phi(g) - C_i^K(\phi)(g) \right) b_i(g) \wrt g \\
&\qquad= \int_{\Hn} \left( \eta^j\phi(g) - c_i^{k}(\phi) 2^{-kd_i} [f_i]_{2^k}
    + c_i^{m-1}(\phi) 2^{(1-m)d_i} [f_i]_{2^{k-1}} \right) b_i(g) \wrt g \\
&\qquad= \int_{\Hn} \eta^j\phi(g) b_i(g) \wrt g
    - \int_{\Hn} c_i^{j}(\phi) 2^{-jd_i} [f_i]_{2^j}  b_i(g)\wrt g \\
&\qquad\qquad     + \int_{\Hn} c_i^{j-1}(\phi) 2^{(1-j)d_i} [f_i]_{2^{j-1}}  b_i(g) \wrt g \\
&\qquad= \int_{\Hn} \eta^j\phi(g) b_i(g) \wrt g
    - c_i^{k}(\phi) + c_i^{k-1}(\phi) \\
&\qquad=0.
\end{aligned}
\]
Clearly, subtracting off all the correction terms $C^{k}_i(\phi)$ from $L^{1,i}\phi$ produces the desired effect.
\end{proof}

If we had taken $\theta$ to be $1$, we would have had an extra logarithmic term.

If $\theta<1$, then $\fnspace{P}(\Hn)$ may be injected continuously in $\fnspace{P}^\theta(\Hn)$.
Further, analogous results hold for Poisson bounded functions on $\R$.
This has the following consequence.

\begin{corollary}\label{cor:Poisson decomposition}
Suppose that $\family$ is a family of pairs of functions $\phi\one$ on $\Hn$ and $\phi\two$ on $\R$, and that $M, N \in \N$ and $\theta \in (0,1)$.

If $\family$ is Poisson bounded, then for all $\bsym\phi \in \family$, we may write $\phi\one$ and $\phi\two $ as sums of normalised dilates:
\begin{equation}\label{eq:decomposition as sum}
\phi\one = \sum_{i\in \N} [\phi^{(1,i)}]_{2^i}
\qquad\text{and}\qquad
\phi\two = \sum_{j\in \N} [\phi^{(2,j)}]_{2^j} ,
\end{equation}
where $\phi^{(1,i)} \in \fnspace{C}^\infty(\Hn)$ and $\supp \phi^{(1,i)} \subseteq \bar B\one(o,1)$ and moreover $\phi^{(2,j)} \in \fnspace{C}^\infty(\R)$ and $\supp \phi^{(2,j)} \subseteq \bar B\two (0,1)$; further,
\begin{equation}\label{eq:norm estimate of summands}
\bignorm{ \phi^{(1,i)} }_{(Q)} \lesssim_{\family, Q} 2^{-\theta i}
\qquad\text{and}\qquad
\bignorm{ \phi^{(2,j)} }_{(Q)} \lesssim_{\family, Q} 2^{-\theta j}
\end{equation}
for all $i, j, Q \in \N$.
Further, if $\phi\one$ and $\phi\two $ have mean $0$, then $\phi^{(1,i)}$ and $\phi^{(2,j)}$ have all moments of homogeneous order at most $M$ and $N$ equal to $0$.

Conversely, if we may write every $\phi\one$ and $\phi\two$ as sums of normalised dilates as in \eqref{eq:decomposition as sum}, where estimates of the form \eqref{eq:norm estimate of summands} hold, then $\family$ is $\theta$-Poisson bounded.
\end{corollary}

\begin{proof}
Let $L^{(1,i)}$ be the maps constructed in the previous lemma, and $L^{(2,j)}$ be the analogous maps on functions on $\R$.
We define $\phi^{(1,i)}$ to be the normalised dilate $[L^{(1,i)}\phi\one ]_{2^{-i}}$, which is supported in $B\one(o,1)$,
and $\phi^{(2,j)}$ to be the normalised dilate $[L^{(2,j)} \phi\two ]_{2^{-j}}$, which is supported in $B\two (0,1)$.

Conversely, if $\phi\one$ and $\phi\two$ decompose as in \eqref{eq:decomposition as sum}, where estimates of the form \eqref{eq:norm estimate of summands} hold, then the $\theta$-Poisson boundedness of $\family$ follows by summation.
\end{proof}

Many of the results that we are going to prove involve operators that are sublinear in more than one sense.
For example, we consider the maximal function $\gmaxop(f)$, given by
\[
\gmaxop (f) (g)
:= \sup_{\bsym\phi \in \family} \sup_{r,s \in \R^+}
\abs|f \Hconv \phi\one_{r} \Vconv \phi\two _{s}(g)|
\qquad\forall g \in \Hn.
\]
The expression $\sup_{r,s \in \R^+} \bigabs|f \Hconv \phi\one_{r} \Vconv \phi\two _{s}(g)|$ is evidently sublinear in $f$, $\phi\one$ and $\phi\two $.
The point of the next proposition is that is suffices to handle functions $\phi\one$ and $\phi\two $ with compact support.

\begin{prop}\label{prop:compact-support-enough}
Suppose that $\family$ is a Poisson bounded family of pairs of functions.
There there is a Poisson bounded family $\familyc$ of pairs of functions, supported in $B\one(o,1)$ and $B\two (0,1)$, such that
\[
\norm{ \gmaxop(f) }_{\Leb^1(\Hn)} \leq 4 \norm{ \gmaxopz(f) }_{\Leb^1(\Hn)}
\qquad\forall f \in \Leb^1(\Hn).
\]
\end{prop}

\begin{proof}
By Corollary \ref{cor:Poisson decomposition}, given  $(\phi\one , \phi\two ) \in \family$, we may decompose $ \phi\one =  \sum_{i\in \N} 2^{-i} [\phi^{(1,i)}]_{2^i}$ and $ \phi\two = \sum_{j\in \N} 2^{-j} [\phi^{(2,j)}]_{2^j}$, where the functions $\phi^{(1,i)}$ and $\phi^{(2,j)}$ are compactly supported and have uniformly bounded Poisson norms.
We take $\familyc$ to be the set of all such $\phi^{(1,i)}$ and $\phi^{(2,j)}$.
Then
\[
\begin{aligned}
\sup_{r,s>0} \abs| f \Hconv \phi_r\one \Vconv \phi_s\two |
&\leq \sum_{i,j \in \N} \sup_{r,s>0} \abs| f \Hconv [\phi^{(1,j)}]_{2^jr} \Vconv [\phi^{(2,k)}]_{2^ks} | \\
&= \sum_{i,j \in \N} \sup_{r,s>0} \abs| f \Hconv [\phi^{(1,j)}]_{r} \Vconv [\phi^{(2,k)}]_{s} |  \\
&\lesssim \gmaxopc (f),
\end{aligned}
\]
and the result follows by taking the supremum over all $\bsym\phi \in \family$.
\end{proof}

\begin{remark}\label{rem:compact-support-enough}
We will appeal several times to the following principle: when we deal with functionals that are sublinear and dilation invariant in Poisson bounded functions $\phi\one$ and $\phi\two $, such as maximal functions, square functions and area functions, the estimation may be reduced to the case where $\supp(\phi\one ) \subseteq \bar B\one(o,1)$ and $\supp(\phi\two ) \subseteq \bar B\two (0,1)$ in a similar manner to that of the proof above.
Further, when we are dealing with square functions and area functions when the hypotheses include that $\phi\one$ and $\phi\two $ both have mean $0$, we may additionally assume that $\phi\one$ and $\phi\two $ have moments of homogeneous order at most $M$ and $N$ equal to $0$.
In light of Corollary \ref{cor:null-moments-are-derivs}, this means that we may assume that $\phi\one$ and $\phi\two $ are derivatives.
\end{remark}

An algebra that is also a Fréchet space is called a Fréchet algebra when multiplication is continuous in the Fréchet topology.

\begin{prop}\label{prop:Poisson is algebra}
The spaces $\fnspace{P}^\theta(\Hn)$ and $\fnspace{P}^\theta(\R)$ are Fréchet algebras for all $\theta \in \R^+$.
\end{prop}

\begin{proof}
We prove only the first assertion, as the other proof is very similar.
Let $\eta$ be a $[0,1]$-valued function on $\Hn$ that takes the value $1$ on $B\one(o,1/2)$ and vanishes off $B\one(o,1)$.

Take $\phi\one$ and $\psi\one$ in $\fnspace{P}(\Hn)$.
It is evident that, if $\norm{g} \leq 2$, then
\[
\begin{aligned}
\bigabs|\Hnabla^q ( \phi\one \Hconv\psi\one)(g)|
&= \bigabs| (\phi\one \Hconv \Hnabla^q \psi\one)(g)| \\
&\leq \norm{ \phi\one }_{\Leb^1(\Hn)} \bignorm{ \Hnabla^q \psi\one }_{\Leb^\infty(\Hn)} 
\lesssim \norm{ \phi\one }_{(0)} \norm{ \psi\one }_{(q)} ,
\end{aligned}
\]
so we only need worry about estimating $\bigabs|\Hnabla^q ( \phi\one \Hconv\psi\one)(g)|$ when $\norm{g} \geq 2$.
Fix $r \in [2,\infty)$, and set 
\[
\phi\one_1(g) := \eta( D_{1/r} g) \phi\one(g) 
\qquad\text{and}\qquad 
\phi\one_2(g) :=  \phi\one(g) - \phi\one_1(g) 
\]
for all $g \in \Hn$; define $\psi\one_1(g)$ and $\psi\one_2$ analogously.
Observe that, when $\norm{g} \geq 2$,
\[
\norm{g}^{q + \hdim + \theta} \bigabs|\Hnabla^q \psi\one_2(g)|
\lesssim_q \norm{ \psi\one }_{(q)},
\]
because $\Hnabla^q \psi\one_2(g) = 0$ when $\norm{g} \leq r/2$ while $\Hnabla^q \psi\one_2(g) = \Hnabla^q \psi\one(g)$ when $\norm{g} \geq r$ by our choice of $\eta$, and because differentiating $\eta(D_{1/r} g)$ $k$ times introduces a factor of $r^{-k}$ which is comparable in size to $\norm{g}^{-k}$ when $r/2 < \norm{g} < r$.
The implicit constant does not depend on either $g$ or $r$.

We now suppose that $\norm{g} = r$.
Because of our choice of $\eta$, $\riHnabla^q \phi\one \Hconv \psi\one(g)$ vanishes.
Much as before, if $i$ is either $1$ or $2$, then
\[
\begin{aligned}
\bigabs|\Hnabla^q ( \phi\one_i \Hconv\psi\one_2)(g)|
&\leq \bignorm{ \phi\one_i }_{\Leb^1(\Hn)} \bignorm{ \Hnabla^q \psi\one_2 }_{\Leb^\infty(\Hn)} \\
&\lesssim \norm{g}^{-q - \hdim - \theta}\norm{ \phi\one }_{(0)} \norm{ \psi\one }_{(q)} .
\end{aligned}
\]
Finally, by arguing as in Remark \ref{rem:either-side-will-do}, we see that
\[
\begin{aligned}
\bigabs|\Hnabla^q ( \phi\one_2 \Hconv\psi\one_1)(g)|
\lesssim \norm{g}^{-q - \hdim - \theta}\norm{ \phi\one }_{(2q)} \norm{ \psi\one }_{(0)} ,
\end{aligned}
\]
as required.
\end{proof}

Similar results were proved in \cite{HanLuSaw}.

\subsection{Geometry of cones}\label{ssec:cone-geometry}
In this section, we prepare for later results on the nontangential maximal function and on the area function that show that changing the apertures of cones does not affect the corresponding Hardy space, and for the proof that the maximal function controls the area function.
We start by noting that when $f \in \Leb^1(\Hn)$ and $\bsym\phi$ is a Poisson bounded pair, then $(g,r,s) \mapsto f \Hconv \phi\one_r \Vconv \phi\two _s$ is continuous in $\Hn \times\R^+ \times \R^+$.

First, we take a continuous function $F: \Hn \times \R^+ \times R^+ \to [0,\infty)$, and define
\begin{equation}\label{eq:def-F}
F_{\alpha,\beta}(g) = \sup\{ F(g',r,s) : g' \in T(g,\alpha r, \beta s) \} .
\end{equation}

\begin{lemma}\label{lem:apertures-maxfn}
For all $F$ as in \eqref{eq:def-F}, all $\alpha, \alpha', \beta, \beta' \in \R^+$ such that $\alpha < \alpha'$ and $\beta <  \beta'$, and all $\epsilon \in \R^+$,
\begin{align}\label{eq:aperture a b}
\| F_{\alpha,\beta} \|_{\Leb^1(\Hn)}
\leq \|F_{\alpha',\beta'} \|_{\Leb^1(\Hn)}
\lesssim_{\epsilon} \max \biggl\{ \Bigl(\frac{\alpha'}{\alpha}\Bigr)^{\hdim} , \Bigl(\frac{\alpha'}{\alpha}\Bigr)^{2\cdim}\frac{\beta'}{ \beta} \biggr\}^{1+\epsilon}   \|F_{\alpha,\beta}\|_{\Leb^1(\Hn)}  .
\end{align}
\end{lemma}
\begin{proof}
The left hand inequality is trivial.
To prove the right hand inequality, we take $\lambda \in \R^+$, and write $L(\lambda)$ for the set $\{ (g,r,s) \in \Hn \times \R^+ \times \R^+ : F(g,r,s) \geq \lambda \}$.
Since $g' \in  T(g,r,s)$ if  and only if $g \in T(g' ,r, s)$, we see that
\[
\{ g \in \Hn : F_{\alpha,\beta}(g) \geq \lambda \}
= \bigcup_{ (g',r,s) \in L(\lambda) } T(g', \alpha r, \beta s).
\]

Suppose that $\alpha' \geq \alpha$ and $\beta' \geq \beta$.
Then for all $g \in T(g', \alpha' r, \beta' s)$,
\[
\begin{aligned}
\oper{M}_{\flag} \indifn_{ T(g', \alpha r, \beta s)} (g)
&\geq \frac{\abs| T(g', \alpha r, \beta s)|}{ \abs| T(g', \alpha' r, \beta' s)|}
= \left(\frac{\alpha}{\alpha'}\right)^{2\cdim} \frac{(\alpha r)^2 + \beta s}{(\alpha'r)^2 + \beta' s} \\
&\geq \Bigl(\frac{\alpha}{\alpha'}\Bigr)^{2\cdim}
\min \left\{ \Bigl( \frac{\alpha }{\alpha'}\Bigr)^2 , \frac{\beta}{ \beta'} \right\} ,
\end{aligned}
\]
(the equality follows from \eqref{eq:size-of-tube}), and so
\[
\oper{M}_{\flag} \indifn_{ \bigcup_{ (g',r,s) \in L(\lambda) } T(g', \alpha r, \beta s) }
\geq \left(\frac{\alpha}{\alpha'}\right)^{2\cdim}
\min \biggl\{ \left(\frac{\alpha }{\alpha'}\right)^2 , \frac{\beta}{ \beta'} \biggr\}
\indifn_{\bigcup_{ (g',r,s) \in L(\lambda) } T(g', \alpha' r, \beta' s)} .
\]
That is,
\[
\oper{M}_{\flag} \indifn_{ F_{\alpha,\beta} \geq\lambda }
\geq
\min \biggl\{ \left(\frac{\alpha}{\alpha'}\right)^{\hdim} , \left(\frac{\alpha}{\alpha'}\right)^{2\cdim} \frac{\beta}{ \beta'} \biggr\}
\indifn_{ F_{\alpha',\beta'} \geq \lambda } .
\]
Since $\oper{M}_{\flag}$ is bounded on $\Leb^{1+\epsilon}(\Hn)$,
\[
\begin{aligned}
\abs| \{ g \in \Hn : F_{\alpha',\beta'}(g) \geq \lambda \} |
&\lesssim_\epsilon
\max \biggl\{ \left(\frac{\alpha'}{\alpha}\right)^{\hdim} , \left(\frac{\alpha'}{\alpha}\right)^{2\cdim} \frac{\beta'}{ \beta} \biggr\}^{1+\epsilon}
\abs| \{ g \in \Hn : F_{\alpha',\beta'}(g) \geq \lambda \} |.
\end{aligned}
\]
Finally, by integrating the distribution functions, we conclude that
\[
\norm{ F_{\alpha',\beta'} }_{\Leb^1(\Hn)}
\lesssim_\epsilon \max \left\{ \biggl(\frac{\alpha'}{\alpha}\biggr)^{\hdim} , \biggl(\frac{\alpha'}{\alpha}\biggr)^{2\cdim}\frac{\beta'}{ \beta} \right\}^{1+\epsilon}
\norm{ F_{\alpha,\beta} }_{\Leb^1(\Hn)} ,
\]
as required.
\end{proof}

The next two lemmas will be used to study the area integral.

\begin{lemma}\label{lem:equivalent-area-fns}
For all $\beta, \gamma \in \R^+$ and all continuous functions $F :\Hn \times \R^+ \times \R^+ \to [0, +\infty)$, define
\begin{align}
\term{I}_{\beta,\gamma}(F) &:=
\int_{\Hn} \lpar \iint_{\R^+ \times \R^+}
    (F(\cdot, r, s) \Hconv
    \chi\one_{\beta r} \Vconv \chi\two _{\gamma s})(g)
    \,\frac{dr}{r} \,\frac{ds}{s} \rpar^{1/2} \wrt g \notag \\ \label{eq:def-J}
\term{J}_{\beta,\gamma}(F) &:=
\int_{\Hn}  \lpar \iiint_{\Gamma_{\beta,\gamma}(g)} \frac{F(g', r, s)}{\abs| T(o,\beta r,\gamma s)|}
\wrt g' \,\frac{dr}{r} \,\frac{ds}{s} \rpar^{1/2} \wrt g .
\end{align}
Then $\term{I}_{\beta,\gamma}(F) \eqsim_{\beta,\gamma,\beta'\gamma'} \term{I}_{\beta',\gamma'}(F)$ and $\term{I}_{\beta,\gamma/2}(F) \lesssim_{\beta,\gamma} \term{J}_{\beta,\gamma}(F) \lesssim_{\beta,\gamma} \term{I}_{\beta,\gamma}(F)$.
\end{lemma}

\begin{proof}
First,
\begin{equation}\label{eq:tent-equivalent}
\begin{aligned}
&\iint_{\R^{+}\times\R^{+}} ( F(\cdot, r, s)
            \Hconv \chi\one_{\beta r} \Vconv \chi\two _{\gamma s}) (g)
            \,\frac{dr}{r} \,\frac{ds}{s}  \\
&\qquad = \iiint_{\Hn \times \R^+ \times \R^+} F(gg_1,r,s)
            (\chi\one_{\beta r} *_{(2)} \chi\two _{\gamma s}) (g_1^{-1})
            \wrt g_1 \,\frac{dr}{r} \,\frac{ds}{s} \\
&\qquad = \frac{C}{\beta^\hdim \gamma} \iiiint_{\Hn \times \R \times \R^+ \times \R^+} F(gg_1g_2,r,s) \\
&\qquad\qquad \times \indifn_{B\one(o,\beta r)} (g_1) \indifn_{B\two (0, \gamma s)} (g_2)
\wrt g_1 \wrt g_2 \frac{dr}{r^{\hdim+1}} \,\frac{ds}{s^2} \,,
\end{aligned}
\end{equation}
where $C$ depends on $\cdim$.
The last expression may be rewritten as either
\begin{align*}
&\frac{C}{\beta^\hdim \gamma}
\iint_{\Gamma\one_{\beta}(o)} \lpar \iint_{ \Gamma\two _{\gamma}(0)}
    F(gg_1g_2, r, s) \,\frac{ds}{s^2} \wrt g_2\rpar \,\frac{dr}{r^{D+1}} \wrt g_1
\\
\noalign{\noindent{or}}
&\frac{C}{\beta^\hdim \gamma}
\iint_{\Gamma\two _{\gamma}(0)} \lpar \iint_{ \Gamma\one_\beta(o)}
    F(gg_1g_2, r, s)  \,\frac{dr}{r^{D+1}} \wrt g_1 \rpar
    \,\frac{ds}{s^2} \wrt g_2.
\end{align*}
Hence $\term{I}_{\beta,\gamma}$ is equivalent to either
\begin{align*}
&\frac{C}{\beta^\hdim \gamma}  \int_{\Hn}\lpar \iint_{\Gamma\one_{\beta}(o)}
    \lpar \iint_{ \Gamma\two _{\gamma}(0)}
    F(gg_1g_2, r, s) \,\frac{ds}{s^2} \wrt g_2\rpar
    \,\frac{dr}{r^{D+1}} \wrt g_1
    \rpar^{1/2} \wrt g
\\
\noalign{\noindent{or}}
&\frac{C}{\beta^\hdim \gamma}  \int_{\Hn}\lpar \iint_{\Gamma\two _{\gamma}(0)}
    \lpar \iint_{ \Gamma\one_\beta(o)} F(gg_1g_2, r, s)
    \,\frac{dr}{r^{D+1}} \wrt g_1 \rpar\,\frac{ds}{s^2} \wrt g_2 \rpar^{1/2} \wrt g.
\end{align*}
It may be argued, as in the classical case (see \cite[pp. 125--126]{Ste2}) that
changing the parameter $\beta$ changes the first of these two expressions to an equivalent expression, and that changing the parameter $\gamma$ changes the second to an equivalent expression (in this case, we need to write $g$ as $(z,t)$ and integrate with respect to first $t$ and then $z$).
We conclude that changing $\beta$ and $\gamma$ changes $\term{I}_{\beta,\gamma}$ to an equivalent integral.

Finally, from Lemma \ref{lem:flag-and-iterated-maximal-fns},
\begin{equation}\label{eq:tent-equivalent-2}
\term{I}_{\beta, \gamma/2}
\lesssim \term{J}_{\beta, \gamma}
\lesssim \term{I}_{\beta, \gamma},
\end{equation}
and the lemma follows.
\end{proof}

Before our next lemma, we recall that $J_{\beta,\gamma}(F)$ was defined in \eqref{eq:def-J}.

\begin{lemma}\label{lem:area-fn-cts}
For all $\beta, \gamma \in \R^+$ and all continuous functions $F :\Hn \times \R^+ \times \R^+ \to \C$ such that $J_{\beta,\gamma}(F) < \infty$,
\[
\lim_{g'' \to o}
\int_{\Hn}  \lpar \iiint_{\Gamma_{\beta,\gamma}(g)} \frac{\abs|F(g''g', r, s) - F(g', r, s)|}{\abs| T(o,\beta r,\gamma s)|}
\wrt g' \,\frac{dr}{r} \,\frac{ds}{s} \rpar^{1/2} \wrt g =0.
\]
\end{lemma}

\begin{proof}
First, by several applications of Lebesgue's convergence theorem,
\[
\lim_{b \to \infty}\int_{\Hn}  \lpar \int_{1/b}^{b} \int_{1/b}^{b} \int_{T(g,\beta r,\gamma s)} \frac{\abs|F(g''g', r, s)| }{\abs| T(o,\beta r,\gamma s)|}
\wrt g' \,\frac{dr}{r} \,\frac{ds}{s} \rpar^{1/2} \wrt g =0,
\]
and then, for each $b \in \R^+$,
\[
\lim_{a \to \infty}\int_{\Hn \setminus B\one(o,a)}  \lpar \int_{1/b}^{b} \int_{1/b}^{b} \int_{T(g,\beta r,\gamma s)} \frac{\abs| F(g'' g', r, s)| }{\abs| T(o,\beta r,\gamma s)|}
\wrt g' \,\frac{dr}{r} \,\frac{ds}{s} \rpar^{1/2} \wrt g =0,
\]
uniformly for $g'' \in B\one(o,1)$ in both cases.
On the other hand, by continuity and compactness,
\[
\lim_{g'' \to o} \abs| F(g'' g', r, s) - F(g', r, s)| = 0
\]
uniformly for $g'' \in \bar B\one(o,1)$, $g \in \bar B\one(o,a)$, and $g ' \in T(0,r,s)$ where $1/b \leq r,s \leq b$.
The lemma follows.
\end{proof}

\begin{lemma}
There exists a positive geometric constant $C_1$ such that, if $E$ is a subset of $\Hn$ of finite measure, and $\Omega = \bigcup_{g \in E} \Gamma(g)$, then $\indifn_{E} \Hconv p_{r,s} (g') \geq C_1$ for all $(g', r, s) \in \Omega$.
\end{lemma}

\begin{proof}
By Lemma \ref{lem:poisson-and-flag-maximal-functions}, $ \indifn_{E} \Hconv p_{r,s}(g') \gtrsim_{\beta,\gamma} \indifn_{E} \Hconv \chi_{\beta r,\gamma s} (g')$ for all $(g',r,s) \in \Hn \times \R^+ \times \R^+$ and all $\beta, \gamma \in \R^+$, and we are done.
\end{proof}

\subsection{Spectral theory}\label{ssec:spectral-theory}

In much of the literature on analysis on the Heisenberg group, such as \cite{GelMay} or \cite{HanLuSaw}, the closed subalgebra $\fnspace{A}(\Hn)$ of $\Leb^1(\Hn)$ consisting of limits of linear combinations of the heat kernels $h\one_r$ plays an important role.
The Hulanicki functional calculus \cite{Hul} is used in place of the Fourier transformation.
Unfortunately, $\fnspace{A}(\Hn)$ is not closed under pointwise multiplication and ``cut-off arguments'' are problematic.
Fortunately, the functional calculus of \cite{AstDiBRic} allows us to deal with ``radial'' functions, and ``cut-off arguments'' are straightforward in this context.
It turns out that radial functions provide a natural setting for the joint spectral calculus of $\HLap$ and $\oper{T}$.

On the Heisenberg group, radial functions (and radial distributions), that is, those that are invariant under the action of the rotations in $\mathrm{U}(n)$,  form a commutative algebra.
This was noticed in \cite{HulRic} (though as is often the case with analysis on $\Hn$, the necessary calculations to show this may be found in \cite{G77}, but the result is not stated explicitly there).
In particular the differential operators $\HLap$ and $i\oper{T}$ are radial and essentially self-adjoint.
Spectral theory, combining the general features of abstract theory, as in, for instance, \cite{McI}, with the additional features available because we are working on a Lie group, has been used to develop a functional calculus for these operators (see \cite{AstDiBRic} and the references cited there), so that expressions $\Phi(\HLap,i\oper{T})$ and $\Phi(\HLap,\VLap)$ are defined for functions $\Phi$ in $\Schwartz(\R^2)$, and these operators are given by right convolution with kernels $k_\Phi$ in $\Schwartz(\Hn)$.
The functional calculus may be extended to other classes of functions, such as bounded continuous functions, or rational functions, at the cost of using distributional kernels.

Let $\F$ denote the Heisenberg fan, that is,
\[
\{ ((2d + n)|\lambda|, \lambda) \in \R^2 :
 d \in \N , \lambda \in \R \setminus\{0\} \}
\cup \{ (\xi_1,0) : \xi_1 \in \R^+ \}.
\]
We shall use the following features of this theory.

\begin{theorem}\label{thm:Plancherel}
There exists an isometry $\oper{F}: L^2(\Hn) \to \bigoplus_{d\in\N} L^2(\R;\Hil_d)$, where each $\Hil_d$ is a Hilbert space, such that
\begin{enumerate}
\item
$\displaystyle \norm{ f }_{ L^2(\Hn) }
= \lpar \sum_{d \in \N} \int_{\R} \norm{ \oper{F} f(d, \lambda) }_{\Hil_d}^2 \lambda^{\hdim-1} \wrt \lambda \rpar^{1/2}$
for all $f \in L^2(\Hn)$;
\item
$\oper{F} (\Phi(\HLap, i\oper{T}) f)(d, \lambda) = \Phi((2d+n)|\lambda|,\lambda) \oper{F} (f)(d, \lambda)$ for all $d \in \N$, all $\lambda \in \R$, all $f \in \Dom \Phi(\HLap, i\oper{T})$,  and all Schwartz functions or polynomials $\Phi$;
\item
let $\oper{A}$ be a continuous linear operator from $\Schwartz(\Hn)$ to its dual space $\Schwartz'(\Hn)$ that commutes with (left) translations.
Then there exists $\Phi \in \Schwartz(\R^2)$ such that $\oper{A} = \Phi(\HLap, i\oper{T})$ if and only if $\oper{A}$ has a radial convolution kernel $k_\Phi$ in $\Schwartz(\Hn)$.
\end{enumerate}
\end{theorem}

If $\Phi: \R^2 \to \R$ is even in the second variable, we may interpret the theorem as giving information about $\Psi(\HLap, \VLap)$, where $\Psi(\mu, \lambda) = \Phi(\mu, \lambda^2)$.

In particular, the space of kernels of convolution operators $\Phi(\HLap)$, where $\Phi \in \Schwartz(\R)$, is a subalgebra of the convolution algebra of radial Schwartz functions on $\Hn$, and the space of kernels of convolution operators $\Phi(\HLap)$, where $\Phi \in \Schwartz(\R)$, is an algebra that may be identified with the convolution algebra of even Schwartz functions on the centre $\R$ of $\Hn$.

For the reader who may be interested in the extension of our results, we mention that this result has been extended to more general contexts, including stratified nilpotent groups, by Martini~\cite{Mar}.

\subsection{Homogeneous  singular integrals and distributions}\label{ssec:singular-integrals}

In the context of euclidean space $\R^n$, it is well known that the following are equivalent (see, e.g, \cite[Section XIII.5.3]{Ste2}):  
\begin{enumerate}
\item[(a)] $K$ is a distribution that is smooth away from $0$ and is homogeneous of degree $-n$;
\item[(b)] $K$ is a linear combination of the Dirac delta at $0$ and a principal value distribution $K_0$, where $K_0$ is homogeneous of degree $-n$, is smooth away from $0$, and has mean $0$ on spheres with centre $0$;
\item[(c)] $K$ may be expressed as an integral $\int_{\R^+} \omega_t \wrt t/t$, where $\omega$ is in $\fnspace{S}(\R^n)$ and has mean $0$.
\end{enumerate}

The same holds in the Heisenberg group (and indeed on more general stratified groups).
The following lemma may be extracted from  \cite[Section II.8.19]{Ste2} (for one direction of the proof).

\begin{lemma}\label{lem:homogeneous-integral}
Suppose that $\omega$ on $\Hn$ is Poisson bounded and has mean $0$, that $a,b \in \R^+$, that $c,d \in \Z$ and that $\alpha > 1$.

If $a \to 0$ and $b \to \infty$, then 
\[
k_{a,b} := \int_{a}^{b} \omega\one_t \,\frac{dt}{t}
\]
converges in $\Schwartz(\Hn)'$ and in $\fnspace{C}^\infty(\Hn \setminus \{o\})$ to a distribution $k$ that is homogeneous of degree $-\hdim$, that is, $(k, \phi) = (k, \phi\circ\dilat_r)$ for all $r \in \R^+$.  Further, away from $o$, $k$ is given by integration against a smooth function that is homogeneous of degree $-\hdim$ and has mean $0$ on the unit sphere of any smooth homogeneous norm on $\Hn$.
The associated convolution operators are uniformly bounded in $a$ and $b$ on $\Leb^p(\Hn)$ when $p \in (1,\infty)$, with a bound that depends on $p$ and on $\norm{\omega}_{(1)}$, and hence converge strongly as $a \to 0$ and $b \to \infty$.
If $a$ is fixed and $b \to \infty$, then the distributions $k_{a,b}$ converge in $\Schwartz(\Hn)'$ to a distribution $k_{a,\infty}$ that is given by integration against a smooth function, and the associated convolution operators are uniformly bounded in $a$ on $\Leb^p(\Hn)$ when $p \in (1, \infty)$.
If $a \to 0$ and $b$ is fixed, then the distributions $k_{a,b}$ converge in $\Schwartz(\Hn)'$ to a distribution $k_{0,b}$ that is given by integration against a smooth function away from $o$, and the associated convolution operators are uniformly bounded in $b$ on $\Leb^p(\Hn)$ when $p \in (1, \infty)$.

Likewise, if $c \to -\infty$ and $d \to \infty$ in $\Z$, then
\[
\tilde k_{c,d} := \sum_{m=c}^{d} \omega\one _{\alpha^m}
\]
converges in $\Schwartz(\Hn)'$ and in
$\fnspace{C}^\infty(\Hn \setminus \{o\})$ to a distribution $\tilde k$ that is discretely homogeneous of degree $-\hdim$, in the sense that $(\tilde k, \phi) = (\tilde k, \phi\circ\dilat_{\alpha^m})$ for all $m \in \Z$, and is given by integration against a smooth function away from $o$.
The associated convolution operators are uniformly bounded in $c$ and $d$ on $\Leb^p(\Hn)$ when $p \in (1,\infty)$, with a bound that depends on $p$ and on $\norm{\omega}_{(1)}$, and hence converge strongly as $c \to -\infty$ and $d \to \infty$.

The mappings $\omega \mapsto k$ and $\omega \mapsto \tilde k$ from $\fnspace{P}(\Hn)$ to $\Schwartz'(\Hn)$ are continuous.

Conversely, every homogeneous distribution that is smooth away from $o$ 
arises in this way, for some $\omega \in \fnspace{C}^\infty(\Hn)$.

Analogous results hold on $\R$.
\end{lemma}

\begin{proof}
We sketch the forward direction of the proof, treating only $\tilde k_{c,d}$; the other case is similar.

It is easy to check that $\oper{D} k_{c,d}$ converges in $\Leb^2(\Hn \setminus B\one(o,\epsilon))$ for all $\epsilon \in \R^+$ and all left-invariant differential operators $\oper{D}$.
Consider the sum
\[
\sum_{m=c}^{d} \int_{\Hn} \omega_{\alpha^m}(g) \fn \psi(g) \wrt g ,
\]
where $\psi\in \Schwartz(\Hn)$; if $\psi$ vanishes at $0$, this converges absolutely, and if $\psi$ is constant near $o$, then we may use the cancellation of $\omega$ to show convergence.

The Cotlar--Stein lemma \cite[p.~280]{Ste2} and the (Heisenberg group version of) the Hörmander cancellation condition \cite{Hor} may be used to show that convolution with $\tilde k_{c,d}$ is bounded uniformly in $c$ and $d$ on $\Leb^p(\Hn)$ when $1<p< \infty$; the uniform bound for the operator norms depends on $p$ and on $\norm{\omega}_{(1)}$.

The continuity claimed may be established by checking continuity at each step of the argument above.

For the converse, we treat the continuous case only;   it is useful to use the Korányi norm (or another smooth norm).
Let $k$ be a homogeneous distribution that is smooth away from $0$.
Then necessarily $k = c\delta_o + k_0$, where $c$ is a constant and $k_0$ is a smooth principal value distribution; $k_0$ has integral $0$ over the unit sphere relative to the Korányi norm.
It is not hard to find a nonnegative-real-valued $\fnspace{C}_c^\infty(\R^+)$-function $\eta$ such that
\[
\int_{\R^+} \eta\left( \frac{\cdot}{t} \right) \frac{\wrt t}{t} = 1.
\]
Set $\omega = k_0 \eta(|\cdot|)$; then $\omega \in \fnspace{C}_c^\infty(\Hn)$ and 
\[
\int_{\R^+} \omega_t(\cdot) \,\frac{dt}{t}
= \int_{\R^+} t^{-\hdim} \eta\Bigl( \frac{|\cdot|}{t} \Bigr) k_0\left(\frac{\cdot}{t}\right) \frac{\wrt t}{t} 
= k_0\left(\cdot\right).
\]
Thus the principal value distribution $k_0$ arises in this way.

The unit sphere $S$ is a smooth submanifold, and every $x \in \Hn \setminus \{o\}$ may be written uniquely in the form $D_r \sigma$, where $r \in \R^+$ and $\sigma \in S$ (see \cite[Proposition 1.15]{FolSte} for more information about polar coordinates).
For $f \in \Schwartz(\Hn)$, define the function $F$ on $\R \times S$ by $F(u,\sigma) := f(D_{e^u}\sigma)$.
Take $\Phi \in \fnspace{C}^\infty_c(\R)$ such that $\int_{\R} \Phi(u) \wrt u = 1$, and define $\zeta \in \fnspace{C}^\infty_c(\R^+)$ by $s^\delta \zeta(s) = \Phi'(\log(s))$.
Then
\begin{align*}
\int_{\R^+} \int_{\Hn} f(x) \zeta_t (|x|) \wrt x \,\frac{dt}{t} 
&= \int_{\R^+} \int_{S} \int_{\R^+} f(D_s\sigma) t^{-\hdim} \zeta(s/t)  s^{\hdim-1} \wrt s \wrt \sigma \,\frac{dt}{t} \\
&= \int_{S} \int_{\R^+} \int_{\R^+} f(D_s\sigma) (s/t)^{\hdim} \zeta(s/t)  \frac{\wrt s}{s} \,\frac{dt}{t} \wrt \sigma  \\
&= \int_{S} \int_{\R^+} \int_{\R^+} f(D_{st}\sigma) s^{\hdim} \zeta(s)  \frac{\wrt s}{s} \,\frac{dt}{t} \wrt \sigma  \\
&= \int_{S} \int_{\R^+} \int_{\R^+} f(D_{e^{u+v}}\sigma) \expe^{\hdim u} \zeta(e^u)  \wrt u \wrt v  \wrt \sigma  \\
&= \int_{S} \int_{\R} \int_{\R} F(u+v,\sigma) \Phi'(u)  \wrt u \wrt v \wrt \sigma  \\
&= -\int_{S} \int_{\R} \int_{\R} F'(u+v,\sigma) \Phi(u)  \wrt u \wrt v \wrt \sigma  \\
&= -\int_{S} \int_{\R} \int_{\R} F'(u+v,\sigma) \Phi(u)  \wrt v \wrt u \wrt \sigma  \\
&= -\int_{S} \int_{\R} [F(\infty,\sigma) - F(-\infty,\sigma)] \Phi(u) \wrt u \wrt \sigma  \\
&= \int_{S} \int_{\R} f(o) \Phi(u) \wrt u \wrt \sigma  \\
&= \int_{S} \wrt \sigma f(o),
\end{align*}
that is,
\[
\int_{\R^+} \zeta_t (|\cdot|) \,\frac{dt}{t} = \lpar\int_{S} \wrt \sigma \rpar \delta_o,
\]
and so $\delta_o$ may also be represented in the required form.
\end{proof}

For more on principal value convolution with such kernels, see \cite{Fol}, \cite{KnaSte}, or \cite{Ste2}.

\begin{definition}\label{def:singular-integrals}
A distribution $k$ as constructed in the previous lemma is called a \emph{simple homogeneous singular integral kernel} on $\Hn$ (of type $0$); a \emph{simple homogeneous singular integral operator} is a convolution (on the right) by such a kernel.
Simple homogeneous singular integrals on $\R$ are defined analogously.

A \emph{homogeneous flag singular integral kernel} on $\Hn$ is the convolution of a \emph{simple  homogeneous singular integral kernel} on $\Hn$ with a \emph{simple  homogeneous singular integral kernel} on $\R$.
\end{definition}

In this section, for simplicity, we omit the word homogeneous.

Flag singular integral operators may be realised as expressions of the form 
\[
f \mapsto \iint_{\R^+ \times \R^+} f \Hconv \phi_{r,s} \,\frac{\wrt r}{r} \frac{\wrt s}{s} \,,
\]
where the convergence of the double integral is much as in Lemma \ref{lem:homogeneous-integral}.

The simple singular integral operators on $\Hn$ form an algebra under composition; likewise the simple singular integral operators on $\R$ form an algebra.  
Since $\R$ is central in $\Hn$, it is easy to see that the flag singular integral operators on $\Hn$ also form an algebra under composition.
Nagel, Ricci, Stein and Wainger \cite{NagRicSteWai2} showed that the flag singular integral operators form an algebra in much more general circumstances.

Simple singular integral operators on $\Hn$ are bounded on $\Leb^p(\Hn)$ when $1 < p < \infty$ and on the Folland--Stein--Christ--Geller Hardy space $\FSCGHardy(\Hn)$.
In Section 8 we consider more general flag singular integrals studied by Phong and Stein \cite{PhoSte} and by Nagel, Ricci, Stein and Wainger \cite{NagRicSteWai1, NagRicSteWai2}.

We recall that the (tensor-valued) flag Riesz transformations are defined by
\[
\oper{R}_{(1)} = \Hnabla \HLap^{-1/2} ,
\qquad
\oper{R}_{(2)} =  \Vnabla \VLap^{-1/2}
\qquad\text{and}\qquad
\oper{R}_{\flag} = \oper{R}_{(1)} \otimes \oper{R}_{(2)}.
\]
It is well known that $\oper{R}_{(2)}$ is the Hilbert transformation, which is a convolution with an odd singular kernel.
Likewise $\oper{R}_{(1)}$ is a convolution with an odd singular kernel.
Indeed, $\oper{R}_{(1)}$ is $\Leb^2$ bounded by spectral theory, and is a convolution with a distribution that is homogeneous of degree $-\hdim$. 
By the subellipticity of $\Hnabla$, this distribution may be identified with a smooth kernel away from the identity. 
Any additional component of the distribution must be a multiple of the Dirac delta distribution at the group identity, but since $\oper{R}_{(1)}$ is odd, the multiple must be $0$.
For further information, the reader should consult the cited references.

\subsection{The Calderón reproducing formula} \label{ssec:Calderon}

Before our next definitions, we recall that Poisson boundedness is defined in Definition \ref{def:Poisson-bounds}.

\begin{definition}\label{def:higher-divergence-Poisson}
Fix $m,n\in \N$.
We write $\Hnabla^m \Vnabla^n \cdot \fnspace{P}(\Hn)^{\otimes}$ for the space of all linear combinations  of expressions $\oper{D} \oper{T}^n f$, where $\oper{D}$ is a product of $m$ vector fields, each chosen from $\{\oper{X}_1, \dots, \oper{X}_{2\cdim}\}$, while $f \in \fnspace{P}(\Hn)$.
We define $\riHnabla^m \Vnabla^n \cdot \fnspace{P}(\Hn)^{\otimes}$ analogously, with right invariant vector fields.
\end{definition}

\begin{definition}\label{def:w-invertible}
A Poisson-bounded, possibly vector-valued, function $\phi$ on $\Hn$ is said to be continuous or discretely w-invertible if there exists a Poisson bounded function $\psi$, called a continuous or discrete w-inverse of $\phi$, such that
\begin{equation}\label{eq:Calderon-reproducing-formula}
\int_{\R^{+}} \phi_r \Hconv \psi_r  \,\frac{dr}{r} = \delta
\qquad\text{or}\qquad
\sum_{m\in\Z} \phi_{\alpha^m} \Hconv \psi_{\alpha^m} = \delta
\end{equation}
(here $\delta$ is the delta of Dirac and $\alpha \in (1,\infty)$), as in Lemma \ref{lem:homogeneous-integral} (where $\omega = \phi \Hconv \psi$).
If  $\phi$ is vector-valued, then $\psi$ must be dual-vector-valued and the convolution $\phi \Hconv \psi$  taken as scalar-valued.

Analogous definitions apply to Poisson bounded functions on $\R$, and to pairs of Poisson bounded functions on $\Hn$ and on $\R$.
\end{definition}

It is advantageous to have w-inverses with rapid decay and lots of cancellation.
We do this by finding w-inverses in $\Hnabla^m \Vnabla^n \cdot \fnspace{P}(\Hn)^{\otimes}$, when $m$ and $n$ are large.

The letter w stands for weak or for wavelet; \eqref{eq:Calderon-reproducing-formula} (or its analogue on $\R$) is called the Calderón reproducing formula.
It is sometimes said that $\phi$ and $\psi$ satisfy the Calderón condition.

Gradients may be dealt with using integrations by parts.
As we have already noted, there are distributions $\Upsilon^j$ on $\Hn$ such that $\oper{X}^j f = f \Hconv \Upsilon^j$ for all differentiable functions $f$ on $\Hn$.
It follows that
\[
 (\HLap \phi ) \Hconv \psi
= \sum_{j=1}^{2\cdim}\phi \Hconv\Upsilon^j \Hconv \Upsilon^j \Hconv \psi
= \Hnabla \phi \Hconv \riHnabla \winverse\psi,
\]
where $\riHnabla$ denotes a left gradient rather than a right gradient.
Similar considerations hold for gradients in the central variable, and we conclude that if $(\winverse\psi\one , \winverse\psi\two )$ is a w-inverse for $(\HLap \phi\one , \VLap \phi\two )$, then $(\riHnabla \winverse\psi\one , \riVnabla \winverse\psi\two )$ is a w-inverse for $(\Hnabla \phi\one , \Vnabla \phi\two )$.

In $\R$, or more generally $\R^n$, the Fourier transformation is used to understand the Calderón condition.
In $\R$ all nonzero Poisson bounded functions are w-invertible; a Poisson bounded function $\phi$ on $\R^n$ is w-invertible if and only if its Fourier transform $\hat \phi$ does not vanish on any ray emanating from $0$.

In the literature that we have found on the Calderón condition in the context of the Heisenberg group, such as \cite{GelMay} or \cite{HanLuSaw}, the closed subalgebra $\fnspace{A}(\Hn)$ of $\Leb^1(\Hn)$ mentioned at the start of Section \ref{ssec:spectral-theory} plays an important role.
The functions $\phi\one$ and their w-inverses belong to $\fnspace{A}(\Hn)$, or are gradients of such functions.
However these results extend to more general $\phi$.
Indeed, the functional calculus of \cite{AstDiBRic} allows us to deal with the case where $\phi\one$ is radial.

The set of invertible elements in a Banach algebra is open.
A similar argument applies to the set of w-invertible elements.

\begin{lemma}
Suppose that $\phi\one \in \fnspace{P}(\Hn)$ has continuous or discrete w-inverse $\psi\one$.
All $\tilde\phi\one$ that are sufficiently close to $\phi\one$ in $\fnspace{P}(\Hn)$ have a w-inverse $\tilde\psi\one$ that is close to $\psi\one$ in $\fnspace{P}(\Hn)$.
An analogous result holds for Poisson bounded functions $\phi\two$ on $\R$.
\end{lemma}

\begin{proof}
We consider the discrete case on $\Hn$ only; the continuous case is similar, and the analysis on $\R$ is simpler.
For notational simplicity, we write $\phi$ instead of $\phi\one$, $\psi$ instead of $\psi\one$, and so on.

If $\tilde\phi$ is close to $\phi$ in $\fnspace{P}(\Hn)$, then $\tilde\omega := \tilde\phi \Hconv \psi$ is close to $\phi \Hconv \psi$ in $\fnspace{P}(\Hn)$.
Then $\sum_{m \in \Z} \tilde\omega_{\alpha^m}$ is close to $\delta$ as a homogeneous distribution, so convolution with $\sum_{m \in \Z} \tilde\omega_{\alpha^m}$ is close to the identity as an operator on $\Leb^2(\Hn)$, and hence is invertible as an operator on $\Leb^2(\Hn)$ with an inverse that is close to the identity operator.  
The inverse operator is also convolution with a homogeneous distribution by \cite[Theorem B]{ChrGel}, and is necessarily close to $\delta$ in the space of homogeneous distributions that are smooth away from $o$.
\end{proof}

The final results of this section connect w-invertibility to square functions.
Suppose that $\phi\one$ and $\phi\two $ are Poisson bounded and have mean $0$ on $\Hn$ and $\R$ respectively.
Recall that the sublinear operator $\ctssqfnopphi$ is defined on $\Leb^1(\Hn)$ by
\[
\ctssqfnopphi(f)
= \lpar \iint_{\R^+ \times \R^+} \abs| f \Hconv \phi\one_r \Vconv \phi\two _s(g) |^2 \frac{dr}{r} \,\frac{ds}{s} \rpar^{1/2}.
\]
Similar operators that involve only one convolution are more standard.

\begin{corollary}\label{cor:L2-norm-of-sqfnop-bounded}
Suppose that $\phi\one$ and $\phi\two $ are Poisson bounded and have mean $0$ on $\Hn$ and $\R$ respectively.
Then
\begin{gather}
\norm{ \lpar \int_{\R^+} \abs| f \Hconv \phi\one_r(\cdot) |^2 \,\frac{dr}{r} \rpar^{1/2}}_{\Leb^2(\Hn)}
\lesssim_{\phi\one} \norm{f}_{\Leb^2(\Hn)} \label{eq:LP-1}\\
\norm{ \lpar \int_{\R^+ } \abs| f \Vconv \phi\two _s(\cdot) |^2 \,\frac{ds}{s} \rpar^{1/2}}_{\Leb^2(\Hn)}
\lesssim_{\phi\two} \norm{f}_{\Leb^2(\Hn)} \label{eq:LP-2}\\
\norm{ \lpar \iint_{\R^+ \times \R^+} \abs| f \Hconv \phi_{r,s}(\cdot) |^2 \,\frac{dr}{r}\,\frac{ds}{s}  \rpar^{1/2}}_{\Leb^2(\Hn)}
\lesssim_{\bsym\phi} \norm{f}_{\Leb^2(\Hn)} \label{eq:LP-3}
\end{gather}
for all $f \in \Leb^2(\Hn)$.

Moreover,  $\phi\one$ has a w-inverse $\winverse\phi\one$ with mean $0$, or $\phi\two $ has a w-inverse $\winverse\phi\two$ with mean $0$, if and only if 
converse inequalities to \eqref{eq:LP-1}, \eqref{eq:LP-2} or \eqref{eq:LP-3}  hold, with constants that depend on $\winverse\phi\one$ or $\winverse\phi\two$.
We may find w-inverses  $\winverse\phi\one$ and $\winverse\phi\two$ belonging to the spaces $\Hnabla^m \cdot \fnspace{P}(\Hn)^{\otimes}$ and $\Vnabla^n \cdot \fnspace{P}(\R)$ for arbitrarily large $m$ and $n$.
\end{corollary}

\begin{proof}
We treat only one of these sublinear operators, as all are similar.
Observe that
\[
\begin{aligned}
&\norm{\ctssqfnopphi(f)}_{\Leb^2(\Hn)}^2 \\
&\qquad= \iint_{\R^+ \times \R^+} \int_{\Hn} \abs| f \Hconv \phi_{r,s} (g) |^2 \wrt g \,\frac{dr}{r} \,\frac{ds}{s} \\
&\qquad= \iint_{\R^+ \times \R^+} \int_{\Hn} f \Hconv \phi_{r,s} \Hconv [\phi_{r,s}]^* (g) \bar f(g) \wrt g \,\frac{dr}{r} \,\frac{ds}{s} \\
&\qquad= \int_{\Hn} f \Hconv \lpar \iint_{\R^+ \times \R^+}  \phi_{r,s} \Hconv [\phi_{r,s}]^*  \,\frac{dr}{r} \,\frac{ds}{s} \rpar(g) \bar f(g) \wrt g \\
&\qquad= \int_{\Hn} \oper{K}f(g)  \bar f(g) \wrt g,
\end{aligned}
\]
where $\oper{K}$ is a convolution with a flag singular integral kernel as in Definition \ref{def:singular-integrals} (where we take $\omega\one$ to be $\phi\one \Hconv [\phi\one]^*$ and define $\omega\two $ similarly), and $\oper{K}$ is bounded on $\Leb^2(\Hn)$.
Hence
\[
\begin{aligned}
&\norm{\ctssqfnopphi(f)}_{\Leb^2(\Hn)}^2
\leq \norm{\oper{K}f}_{\Leb^2(\Hn)} \norm{f}_{\Leb^2(\Hn)}
\lesssim_{\bsym\phi} \norm{f}_{\Leb^2(\Hn)}^2,
\end{aligned}
\]
as required.
If $\oper{K}$ is the identity operator, then $\ctssqfnopphi$ is isometric.
A polarisation argument shows that the converse also holds.

Suppose that $\phi\one$ and $\phi\two$ have w-inverses $\winverse{\phi}\one$ and $\winverse{\phi}\two$ with mean $0$, and write $\psi\one$ and $\psi\two$ for $[\winverse\phi\one]^*$ and $[\winverse\phi\two]^*$.
Then $\psi\one$ and $\psi\two$ are Poisson bounded and have mean $0$, and
\[
\begin{aligned}
\norm{f}_{\Leb^2(\Hn)}^2
&= \iint_{\R^+ \times \R^+} \lip f \Hconv \phi\one_r \Vconv \phi\two_s \Hconv \winverse\phi\one_r \Vconv \winverse\phi\two_s, f \rip \,\frac{dr}{r} \,\frac{ds}{s} \\
&= \iint_{\R^+ \times \R^+} \lip f \Hconv \phi\one_r \Vconv \phi\two_s ,
f \Hconv \psi\one_r \Vconv \psi\two_s \rip \,\frac{dr}{r} \,\frac{ds}{s} \\
&\leq \lpar \iint_{\R^+ \times \R^+} \norm{ f \Hconv \phi\one_r \Vconv \phi\two_s }_{\Leb^2(\Hn)}^2 \,\frac{dr}{r} \,\frac{ds}{s} \rpar^{1/2} \\
&\qquad\times
\lpar \iint_{\R^+ \times \R^+} \norm{ f \Hconv \psi\one_r \Vconv \psi\two_s }_{\Leb^2(\Hn)}^2 \,\frac{dr}{r} \,\frac{ds}{s} \rpar^{1/2} \\
& =   \norm{ \lpar \iint_{\R^+ \times \R^+}  f \Hconv \phi\one_r \Vconv \phi\two_s \,\frac{dr}{r} \,\frac{ds}{s} \rpar^{1/2}}_{\Leb^2(\Hn)} \\
&\qquad\times
\norm{ \lpar \iint_{\R^+ \times \R^+} f \Hconv \psi\one_r \Vconv \psi\two_s \,\frac{dr}{r} \,\frac{ds}{s} \rpar^{1/2} }_{\Leb^2(\Hn)}  \\
&\lesssim_{\bsym{\winverse{\phi}}}
 \norm{ \ctssqfnopphi (f)}_{\Leb^2(\Hn)} \norm{ f }_{\Leb^2(\Hn)} ,
\end{aligned}
\]
by the first part of the lemma, and $\norm{f}_{\Leb^2(\Hn)} \lesssim \norm{ \ctssqfnopphi (f)}_{\Leb^2(\Hn)} $, as claimed.

Conversely, suppose that $\norm{f}_{\Leb^2(\Hn)} \lesssim \norm{ \ctssqfnopphi (f)}_{\Leb^2(\Hn)} $.
Then from the first part of the lemma,
\begin{equation}\label{eq:S-one-bounded-below}
\norm{f}_{\Leb^2(\Hn)} \lesssim 
\norm{ \lpar \int_{\R^+} \abs| f \Hconv \phi\one_r(\cdot) |^2 \,\frac{dr}{r} \rpar^{1/2}}_{\Leb^2(\Hn)}
\end{equation}
and
\[
\norm{f}_{\Leb^2(\Hn)} 
\lesssim \norm{ \lpar \int_{\R^+ } \abs| f \Vconv \phi\two _s(\cdot) |^2 \,\frac{ds}{s} \rpar^{1/2}}_{\Leb^2(\Hn)}
\]
for all $f \in \Leb^2(\Hn)$.
The second inequality implies that
\[
\norm{f}_{\Leb^2(\R)} 
\lesssim \norm{ \lpar \int_{\R^+ } \abs| f \Vconv \phi\two _s(\cdot) |^2 \,\frac{ds}{s} \rpar^{1/2}}_{\Leb^2(\R)}
\]
for all $f \in \Leb^2(\R)$.

From \eqref{eq:LP-1} and \eqref{eq:S-one-bounded-below}, the singular integral operator 
\[
f \mapsto \int_{\R^+}  
f \Hconv \phi\one_r \Hconv [\phi\one]^*_r \,\frac{dr}{r}  
\]
is invertible as an operator on $\Leb^2(\Hn)$.
It was shown by Christ and Geller \cite[Theorem B]{ChrGel} that if a simple singular integral operator is invertible as an operator on $\Leb^2(G)$, then the inverse operator is also a simple singular integral operator.
Thus there is a simple singular integral kernel $k$ such that
\[
\int_{\R^+}  
f \Hconv \phi\one_r \Hconv [\phi\one]^*_r  \Hconv k \,\frac{dr}{r}  = f  
\]
for all $f \in \Leb^2(\Hn)$.
Now $[\phi\one]^*_r  \Hconv k = [ [\phi\one]^* \Hconv k ]_r$, and we may easily check that $\psi\one := [\phi\one]^* \Hconv k$ has mean $0$, is Poisson bounded, and is a w-inverse for $\phi\one$.
We may find a w-inverse for $\phi\two$ similarly.

To find a w-inverse with rapid decay, we take a w-inverse $\winverse\phi\one$ and modify it.
Fix $\theta \in (0,1)$.
By applying Corollary \ref{cor:Poisson decomposition}, we may approximate $\winverse\phi\one$ in $\fnspace{P}^\theta(\Hn)$ by a function $\psi\one$ in the space $\Hnabla^m \cdot \fnspace{P}(\Hn)^{\otimes}$. 
Since $\fnspace{P}^\theta(\Hn)$ is a Fréchet algebra, by Proposition \ref{prop:Poisson is algebra}, $\phi\one \Hconv \psi\one$ approximates $\phi\one \Hconv \winverse\phi\one$ in $\fnspace{P}^\theta(\Hn)$, and then from Lemma \ref{lem:homogeneous-integral}, $\int_{\R^+} (\phi\one \Hconv \psi\one)_t \,dt/t$ approximates $\delta$ in the space of distributions, and if the approximation is good, the associated operator is invertible. 
Now we can find a simple singular integral kernel $k$ such that   $\bigl(\int_{\R^+} (\phi\one \Hconv \psi\one)_t \,dt/t\bigr) \Hconv k = \delta$, that is, $\psi\one \Hconv k$ is a w-inverse for $\phi\one$. 
Since $\psi\one = \riHnabla^m \rho$ for some $\rho \in \fnspace{P}^\otimes$, 
$\psi\one \Hconv k = \riHnabla^m \rho \Hconv k \in \riHnabla^m \fnspace{P}^\otimes$.
\end{proof}

\begin{lemma}\label{lem:approx-identities-w-invertible}
Suppose that $\phi \in \fnspace{P}(\Hn)$ and $\int_{\Hn} \phi(g) \wrt g = 1$.
Then $\phi$ is w-invertible, and for all $M \in \N^+$ there exists $\psi \in \fnspace{P}(\Hn)^\otimes$ such that
\[
\int_{\R^+} (\phi \Hconv \Hnabla^M \cdot \psi)_r  \,\frac{dr}{r} = \delta.
\]
\end{lemma}

\begin{proof}
This proof is similar to the last part of the proof of Corollary \ref{cor:L2-norm-of-sqfnop-bounded}, so we shall be brief.

Let $q$ be $\HLap^M p_1$.
Then $q \in \riHnabla^{2M} \fnspace{P}(\Hn)$, so, by the functional calculus of $\HLap$,
\[
\int_{\R^+} q_t \,\frac{dt}{t} = c \delta,
\]
 for a suitable nonzero constant $c$.

It is easy to check that when $s$ is small, $\phi_s \Hconv q$ is close to $q$ in $\fnspace{P}(\Hn)$, since $(\phi_s)_{s\in \R^+}$ is an approximate identity for convolution, whence $\int_{\R^+} (\phi_s \Hconv q)_t \,dt/t$ is close to $c \delta$ in $\fnspace{S}(\Hn)'$, and so, by \cite[Theorem B]{ChrGel}, there exists a simple singular integral kernel $k$ such that 
\[
\Biglpar\int_{\R^+} (\phi \Hconv q_{1/s})_t \,\frac{dt}{t} \Bigrpar \Hconv k = c\delta.
\]
Now $c^{-1} q_{1/s} \Hconv k$ has the required properties.
\end{proof}

Finally we show that if $\phi$ is w-invertible continuously, then a related function $\psi$ is w-invertible discretely.
This will enable us to connect discrete square functions and continuous square functions.

\begin{lemma}\label{lem:ctssqfn-dominates-dissqfn}
Suppose that $\phi$ is Poisson bounded on $\Hn$, with w-inverse $\winverse{\phi}$, as in Definition \ref{def:w-invertible}.
For $\alpha\in (1,\infty)$, define
\[
\psi = \int_{1}^{\alpha} \phi_r \,\frac{dr}{r} \,.
\]
If $\alpha$ is close enough to $1$, then $\psi$ is w-invertible discretely.
Further,
\[
\dissqfnoppsi(f) \leq \log^{1/2}(\alpha) \ctssqfnopphi(f)
\qquad\forall f \in \Leb^1(\Hn).
\]
An analogous result holds for functions on $\R$.
\end{lemma}

\begin{proof}
By hypothesis,
\[
\sum_{m\in\Z} \int_{1}^{\alpha} \phi_{r \alpha^m} \Hconv \winverse{\phi}_{r\alpha^m} \,\frac{dr}{r}
=\int_{\R^+} \phi_r \Hconv \winverse{\phi}_r \,\frac{dr}{r}
= \delta.
\]
Hence
\[
\begin{aligned}
\sum_{m\in\Z} [\psi \Hconv \winverse\phi]_{\alpha^m}
&=\delta - \sum_{m\in\Z} \int_{1}^{\alpha} \phi_{r \alpha^m} \Hconv
[\winverse{\phi}_{r\alpha^m} - \winverse{\phi}_{\alpha^m} ] \,\frac{dr}{r} \\
&= \delta -\sum_{m\in\Z} \left[\int_{1}^{\alpha} \phi_{r} \Hconv
[\winverse{\phi}_{r} - \winverse{\phi} ] \,\frac{dr}{r} \right]_{\alpha^m} \\
&= \delta - \sum_{m\in\Z} [\omega]_{\alpha^m} ,
\end{aligned}
\]
say.
If $\alpha$ is close to $1$, then $\omega$ is small in $\fnspace{P}(\Hn)$ , so $\sum_{m\in\Z} [\omega]_{\alpha^m} $ is small in the space of homogeneous distributions, and $\delta - \sum_{m\in\Z} [\omega]_{\alpha^m} $ is invertible.
Let $k$ be the kernel of the inverse operator.
Much as in the previous lemma, we may write
\[
\begin{aligned}
  \sum_{m\in\Z} \left[\psi \Hconv \winverse\phi \Hconv k \right]_{\alpha^m}
&=   \left[\sum_{m\in\Z} [\psi \Hconv \winverse\phi]_{\alpha^m} \right] \Hconv k
= \delta,
\end{aligned}
\]
and $\winverse\phi \Hconv k$ is the desired discrete w-inverse of $\psi$.

The second claim of the lemma is true for all positive $\alpha$: by definition and Hölder's inequality,
\begin{align*}
\dissqfnoppsi(f)
&= \lpar \sum_{m \in \Z} \abs| f \Hconv \psi_{\alpha^m} |^2 \rpar^{1/2} \\
&\leq  \lpar \sum_{m \in \Z} \lpar \int_{1}^{\alpha}  \abs| f \Hconv \phi_{r\alpha^m} |
\,\frac{dr}{r} \rpar^2 \rpar^{1/2} \\
&\leq \log(\alpha)^{1/2}
    \lpar \sum_{m \in \Z} \int_{1}^{\alpha}
\abs| f \Hconv \phi_{r\alpha^m} |^2
\,\frac{dr}{r}   \rpar^{1/2} \\
&= \log(\alpha)^{1/2}
    \lpar \int_{\R^+}
\abs| f \Hconv \phi_{r} |^2
\,\frac{dr}{r}  \rpar^{1/2} \\
&= \log(\alpha)^{1/2} \ctssqfnopphi(f) ,
\end{align*}
as required.
\end{proof}

It is evident that if
\[
\lpar \sum_{m \in \Z} \abs| f \Hconv \phi_{\alpha^m} |^2 \rpar^{1/2}
\lesssim  \norm{f}_{\Leb^2(\Hn)}
\qquad\forall f \in \Leb^2(\Hn),
\]
then
\[
\lpar \sum_{m \in \Z} \abs| f \Hconv \phi_{\beta^m} |^2 \rpar^{1/2}
\lesssim  \norm{f}_{\Leb^2(\Hn)}
\qquad\forall f \in \Leb^2(\Hn).
\]
when $\beta$ is a positive integral power of $\alpha$.


\section{The atomic Hardy space}\label{sec:atomic}

In this section we focus on the Hardy space defined using atomic decompositions.
We first examine the definition and properties of the space in more detail than most previous studies, even in the product case, and then show that if $f \in \atomHardy(\Hn)$, then $\areaopphi(f) \in \Leb^1(\Hn)$, $\gmaxop(f) \in \Leb^1(\Hn)$, and $\oper{K}(f) \in \Leb^1(\Hn)$ when $\oper{K}$ is a simple singular integral operator.

\subsection{A precise definition of the atomic Hardy space}\label{ssec:atomic-space-details}

We recall the definition of the atomic Hardy space, and make it more precise.

\begin{definition}\label{def:atomic-Hardy-space}
Fix  $M,N \in \N^+$ and a real number $\kappa \in (1 + 1/(2\nu),\infty)$.
An \emph{atom} is a function $a\in \Leb^2(\Hn)$ such that there exist an open subset $E$ of $\Hn$ of finite measure $|E|$ and functions $a_R$ in $\Leb^2(\Hn)$, called \emph{particles}, and $b_R$ in $\Dom(\HLap^{M} \VLap^{N})$ for all $R \in \maxrect(E)$ such that
\begin{enumerate}
\item[(A1)] $a_R= \HLap^M \VLap^N b_R$ and $\supp b_R \subseteq R^* = T(g, \kappa q, \kappa^2 h)$, where $g = \cent(R)$, $q = \wid(R)$ and $h = \heit(R)$;
\item[(A2)] for all sign sequences $\sigma: \maxrect(E) \to \{\pm1\}$, the sum $\sum_{R \in \maxrect(E)} \sigma_R a_R$ converges in $\Leb^2(\Hn)$, to $a_{\sigma}$ say, and
\begin{equation}\label{eq:random-signs-in-atoms}
\norm{ a_\sigma }_{\Leb^2(\Hn)} \leq \abs|E|^{-1/2} ;
\end{equation}
\item[(A3)] $ a = \sum_{R \in \maxrect(E)} a_R$.
\end{enumerate}

We say that $f \in \Leb^1(\Hn)$ has an atomic decomposition if we may write $f$ as a sum $\sum_{j\in \N} \lambda_j a_j$, converging in $\Leb^1(\Hn)$, where $\sum_{j\in\N} |\lambda_j| < \infty$ and each $a_j$ is an atom; we write $f \sim \sum_{j\in\N} \lambda_j a_j$ to indicate that $\sum_{j\in\N} \lambda_j a_j$ is an atomic decomposition of $f$.
The space $\atomHardy(\Hn)$ is defined to be the linear space of all $f \in \Leb^1(\Hn)$ that have atomic decompositions, with norm
\begin{equation}\label{def:atomHardynorm}
\norm{ f }_{\atomHardy(\Hn)}
:=\inf\biggl\{ \sum_{j\in\N} |\lambda_j| :
  f \sim \sum_{j\in\N} \lambda_j a_j \biggr\} .
\end{equation}
The constant $\kappa$ is called the enlargement parameter. 
\end{definition}

A few comments are in order.

First, $\norm{ a }_{\Leb^1(\Hn)} \leq 1$ for each atom $a$, so the sum representing $f \in \atomHardy(\Hn)$ as a weighted sum of atoms converges in $\Leb^1(\Hn)$; hence $\atomHardy(\Hn) \subseteq \Leb^1(\Hn)$.

Next, our norm condition (A2) implies the usual condition that $\norm{ a }_{\Leb^2(\Hn)} \leq |E|^{-1/2}$, and, by a randomisation argument,  that
\begin{equation}\label{eq:L^2-condition-particles-in-atoms}
\sum_{R \in \maxrect(E)} \norm{ a_R }_{\Leb^2(\Hn)}^2 \leq |E|^{-1} ;
\end{equation}
This condition is standard in much of the literature.
Our condition is just as easy to verify, and is more useful; it is a quantitative form of unconditional convergence.
In particular, it shows that the sum converges in any order, and with any regroupings, and implies a similar inequality to \eqref{eq:random-signs-in-atoms} when $\sigma$ is just a bounded sequence.
See, for instance, \cite[Section 1.c]{LinTza} for information about unconditional convergence.

In particular, it is known that for all small positive $\epsilon$, we can find a finite subset $\finiteset$ of $\maxrect(E)$ such that for any subset $\finiteset_1$ of $\maxrect(E)$ that is disjoint from $\finiteset$, $ \norm{ \sum_{R \in \finiteset_1} a_R } _{\Leb^2(\Hn)} < \epsilon$.
Hence if $\sigma$ is any sign function, and $\finiteset_+$ and $\finiteset_-$ are the subsets of $\maxrect(E) \setminus \finiteset$ where $\sigma$ is positive and negative, then 
\[
\biggnorm{ \sum_{R \in \maxrect(E) \setminus \finiteset } \sigma_R a_R } _{\Leb^2(\Hn)} 
\leq \biggnorm{ \sum_{R \in \finiteset_+} a_R } _{\Leb^2(\Hn)}  
+ \biggnorm{ \sum_{R \in \finiteset_-} a_R } _{\Leb^2(\Hn)}  
< 2\epsilon.
\]

Third, the enlargement $R^*$ of $R$ has the same centre as $R$, but is ``bigger and smoother'' than $R$.

If we need to be more precise, we say that a particle is an $(M,N)$ particle, and that an atom $a$ and the Hardy space as above are a $(1,2,M,N,\kappa)$ atom and the $(1,2,M,N,\kappa)$ atomic Hardy space.
We will discuss the dependence of the atomic Hardy space on $\kappa$ in the following section.

We note that if $a \in \Leb^2(\Hn)$, and $a = \HLap^{M} \VLap^{M} b$, where $b \in \Leb^2(\Hn)$ and $\supp b \subseteq R$ for some $R \in \rect$, then $a = \norm{a}_{\Leb^2(\Hn)} |R|^{1/2} c$, where $c = \norm{a}_{\Leb^2(\Hn)}^{-1} |R|^{-1/2} a$, and $c$ is an atom, so that
\begin{equation}\label{eq:simple-atom}
\norm{a}_{\atomHardy(\Hn)} \leq |R|^{1/2} \norm{a}_{\Leb^2(\Hn)} .
\end{equation}

\subsection{Properties of the atomic Hardy space}\label{ssec:atomic-space-props}
In this section, we clarify some properties of the atomic Hardy space in the flag setting.
We examine the dependence on the enlargement parameter, we show that there is a dense subspace with atomic decompositions that converge in $\Leb^2(\Hn)$ as well as in $\Leb^1(\Hn)$, and we show that linear or nonnegative-real-valued sublinear operators that are bounded on $\Leb^2(\Hn)$ and which send particles into $\Leb^1(\Hn)$ and satisfy certain decay estimates are bounded from $\atomHardy(\Hn)$ into $\Leb^1(\Hn)$.

In the classical theory, Hardy spaces may be defined using atoms which are naturally $\Leb^p$ functions, where $p$ may be any index greater than $1$, including $\infty$, and the convergence considered was a simple $\ell^1$ sum.
In later versions of the theory, matters became more complicated, and in particular, in some Hardy spaces associated to rough differential operators, it is not at all clear whether the norm is given by the expression \eqref{def:atomHardynorm}.
One aim of this section is to establish that we do have such a representation.

The first point is that the $\Leb^2(\Hn)$ norm conditions on the particles $a_R$ imply norm conditions on the functions $b_R$ and certain of their derivatives.

\begin{lemma}\label{lem:a-is-enough}
Let $R \in \rect$ and $R^* = T(g,\kappa q/2, \kappa^2(4 h+q^2)/8)$, where $g = \cent(R)$, $q = \wid(R)$ and $h = \heit(R)$.
Suppose moreover that $a_R \in \Leb^2(\Hn)$, $b_R \in \Dom( \HLap^{M} \VLap^{N})$, $a_R= \HLap^M \VLap^N b_R$ and $\supp b_R \subseteq R^*$, as in Definition \ref{def:atomic-Hardy-space}.
Choose $m$ and $n$ such that $0 \leq m \leq M$ and $0 \leq n \leq N$, and define $c_R$ to be $q^{2m-2M} \HLap^{m} h^{2n-2N} \VLap^{n} b_R$.
Then
\begin{equation}\label{eq:c_R-inequality}
\norm{c_R}_{\Leb^2(\Hn)}
\lesssim_{M,N,\kappa} \norm{a_R}_{\Leb^2(\Hn)}.
\end{equation}
\end{lemma}

\begin{proof}
Iterate the second inequality of Lemma \ref{lem:flag-Sobolev} $M-m$ times and
the first $N-n$ times, noting that $h \gtrsim q^2$ since
$R \in \rect$, so the height of the tube is controlled by a multiple of $h$.
\end{proof}

\begin{remark}
This lemma means that it is not necessary for us to be concerned about $\norm{c_R}_{\Leb^2(\Hn)}$.
Some previous work on atomic Hardy spaces in the product and flag setting has imposed conditions on $\norm{c_R}_{\Leb^2(\Hn)}$ (or the equivalent), but this is only sometimes necessary.
\end{remark}

Note that subelliptic regularity (see, for instance, \cite[Theorem 6.1]{Fol}) implies that, for any atom $a$, there exists a function $b$ in $\Leb^2_{\mathrm{loc}}(\Hn)$ such that $a = \HLap^M\VLap^Nb$.
If $E$ is bounded or if $N = 0$, then it is evident that $b \in \Leb^2(\Hn)$, but if $E$ is unbounded and $N = 0$, then this is not so.

Our second remark is that in the definition of an atom, we may parametrise particles by general, not necessarily maximal, shards, and we may group together the particles in many ways.

\begin{lemma}\label{lem:grouping-particles}
Fix $M,N \in \N^+$ and a real number $\kappa$ in $(1 + 1/(2\nu),\infty)$.
Let $E$ be an open subset of $\Hn$ of finite measure and $\rect(E)$ be the set of subshards of $E$.
Suppose that there exist functions $a_R$ in $\Leb^2(\Hn)$ and $b_R$ in $\Dom( \HLap^{M} \VLap^{N})$ for all $R \in \rect(E)$ such that
\begin{enumerate}
\item[(A1)] $a_R= \HLap^M \VLap^N b_R$ and $\supp b_R \subseteq R^*$;
\item[(A2)] for all sign sequences $\sigma: \rect(E) \to \{\pm1\}$, the sum $\sum_{R \in \rect(E)} \sigma_R a_R$ converges in $\Leb^2(\Hn)$, to $a_{\sigma}$ say, and
$\norm{ a_\sigma }_{\Leb^2(\Hn)}
\leq \abs|E|^{-1/2} ;$
\item[(A3)] $ a = \sum_{R \in \rect(E)} a_R$.
\end{enumerate}
Suppose that $\set{S}$ is a subcollection of $\rect(E)$ and $\dagger: \rect(E)\to\set{S}$ is a mapping such that $R \subseteq  \Rdagger$.
Then
for each $S \in \set{S}$, the sum $\sum_{R \in \rect(E):  \Rdagger = S} a_R$ converges in $\Leb^2(\Hn)$, to $\tilde a_S$, say, and $\sum_{R \in \rect(E):  \Rdagger = S} b_R$ converges in $\Leb^2(\Hn)$, to $\tilde b_S$, say.
Further,
\begin{enumerate}
\item[(B1)] $\tilde a_S = \HLap^M \VLap^N \tilde b_S$ and $\supp b_S \subseteq S^*$;
\item[(B2)] for all sign sequences $\sigma: \set{S} \to \{\pm1\}$, the sum $\sum_{S \in \set{S}} \sigma_S \tilde a_S$ converges in $\Leb^2(\Hn)$, to $\tilde a_{\sigma}$ say, and
\begin{equation*}
\norm{\tilde a_\sigma }_{\Leb^2(\Hn)}
\leq \abs|E|^{-1/2} ;
\end{equation*}
\item[(B3)] $ a = \sum_{S \in \set{S}} \tilde a_S$.
\end{enumerate}
In particular, any function $a$ for which (A1) to (A3) hold is an atom.
\end{lemma}

\begin{proof}
By unconditional convergence, the sum $\sum_{R \in \rect(E):  \Rdagger = S} a_R$ converges in $\Leb^2(\Hn)$, for each $S \in \set{S}$, to $\tilde a_S$, say, and by Lemma \ref{lem:a-is-enough}, $\sum_{R \in \rect(E):  \Rdagger = S} b_R$ converges in $\Leb^2(\Hn)$, to $\tilde b_S$, say, and $\HLap^M \VLap^N \tilde{b}_S = \tilde{a}_S$.
Now $R \subseteq  \Rdagger$, so $R^* \subseteq ( \Rdagger)^*$, and $\supp(\tilde b_S) \subseteq S^*$.
Moreover, any sign function $\sigma: \set{S} \to \{\pm1\}$ determines a sign function $\tilde\sigma: \rect(E) \to \{\pm1\}$ by the rule $\tilde\sigma_R  = \sigma_{ \Rdagger}$.
Again by unconditional convergence, $\sum_{S \in \set{S}} \sigma_S \tilde a_S$ converges to $\sum_{R \in \maxrect(E)} \tilde\sigma_R a_R$ in $\Leb^2(\Hn)$ and
\[
\norm{ \sum_{S \in \set{S}} \sigma_S \tilde a_S }_{\Leb^2(\Hn)}
\leq \abs|E|^{-1/2}
\]
for all sign sequences $\sigma: \set{S} \to \{\pm1\}$, and  $a = \sum_{S \in \maxrect(E)\set{S}} \tilde a_S$ in $\Leb^2(\Hn)$.

In particular, for each shard $R \in \rect(E)$, there is a unique widest shard $S$ of maximal volume in $\maxrect(E)$ such that $R \subseteq S$.
Indeed, if $R \subseteq S_1 \subseteq E$ and $R \subseteq S_2 \subseteq E$, where $S_1$ and $S_2$ are shards and $\wid(S_1) = \wid(S_2)$, then $R \subseteq S_1 \cup S_2 \subseteq E$, and $S_1 \cup S_2$ is also a shard.
We denote this widest maximal shard $S$ by $R^\uparrow$.
As the mapping $R \mapsto R^\uparrow$ has the property $R \subseteq R^\uparrow$, we may take $\set{S}$ to be $\maxrect(E)$, and deduce that $a$ is an atom.
\end{proof}

Our third remark is that the set $E$ involved in the definition of an atom need not be open.

\begin{corollary}\label{cor:Omega-not-open}
Fix $M,N \in \N^+$ and a real number $\kappa > 1 + 1/(2\nu)$.
Let $\set{S}$ be a countable set of shards, let $E = \bigcup_{R \in \set{S}} R$, and suppose that $E$ has finite measure.
Suppose that there exist functions $a_R$ in $\Leb^2(\Hn)$ and $b_R$ in $\Dom( \HLap^{M} \VLap^{N})$ for all $R \in \set{S}$ such that
\begin{enumerate}
\item[(A1)] $a_R= \HLap^M \VLap^N b_R$ and $\supp b_R \subseteq R^*$;
\item[(A2)] for all sign sequences $\sigma: \maxrect(E) \to \{\pm1\}$, the sum $\sum_{R \in \maxrect(E)} \sigma_R a_R$ converges in $\Leb^2(\Hn)$, to $a_{\sigma}$ say, and
$\norm{ a_\sigma }_{\Leb^2(\Hn)}
\leq \abs|E|^{-1/2} ;$
\item[(A3)] $ a = \sum_{R \in \maxrect(E)} a_R$.
\end{enumerate}
Then $(1+\epsilon)^{-1} a$ is an atom for all $\epsilon \in \R^+$.
\end{corollary}

\begin{proof}
Since $|E|$ is finite, for all positive $\epsilon$, there is an open subset $E_\epsilon$ of $\Hn$ such that $|E_\epsilon| < (1+\epsilon) |E|$.
The shards $R \in \set{S}$ all lie in $\rect(E_\epsilon)$, and the previous lemma implies that $(1+\epsilon)^{-1} a$ is an atom.
\end{proof}

Hence we may replace the requirement that $E$ be an open set by the assumption that it is a countable union of shards.

Fourth, our atomic Hardy space does not depend on the enlargement parameter $\kappa$ (though the norm may well do so).

\begin{lemma}\label{lem:enlargement-size}
Suppose that $\kappa > \kappa_1 > 1 + 1/(2\nu)$ and $\kappa_1 - 1 \geq \kappa / (2\cdim+1) $.
Then every $(1,2,M,N,\kappa)$ atom is the product of a geometric constant and a $(1,2,M,N,\kappa_1)$ atom.
Consequently, the $(1,2,M,N,\kappa)$ atomic Hardy space and the $(1,2,M,N,\kappa_1)$ atomic Hardy space coincide, with equivalence of norms.
\end{lemma}

\begin{proof}
In this proof, we include the enlargement parameter $\kappa$ in the notation, and write $R^{*,\kappa}$ instead of $R^*$ for $g B\one(o, \kappa q/2)  B\two (0, \kappa^2(q^2+ 4h)/8)$, where $g = \cent(R)$, $q = \wid(R)$, and $h = \heit(R)$.

Suppose that $a$ is a $(1,2,M,N,\kappa)$ atom associated to the open set $E$.
Then we may find functions $a_R$ and $b_R$ in $\Leb^2(\Hn)$ for all $R \in \maxrect(E)$ such that
\begin{enumerate}
\item[(A1)] $a_R= \HLap^M \VLap^N b_R$ and $\supp b_R \subseteq R^{*,\kappa}$;
\item[(A2)] for all sign sequences $\sigma: \maxrect(E) \to \{\pm1\}$, the sum $\sum_{R \in \maxrect(E)} \sigma_R a_R$ converges in $\Leb^2(\Hn)$, to $a_{\sigma}$ say, and $\norm{ a_\sigma }_{\Leb^2(\Hn)}
    \leq \abs|E|^{-1/2}$;
\item[(A3)] $ a = \sum_{R \in \maxrect(E)} a_R$.
\end{enumerate}

For $R \in \maxrect(E)$, write $\Rdagger$ for the unique shard that contains $R$ and is a translate of $\dilat_{2\cdim+1}(R)$, and let $E^\dagger$ be the set $\bigcup_{R \in \maxrect(E)} \Rdagger$.
Then, as shown in Corollary \ref{cor:R-and-R-dagger}, $\abs|E^{\dagger}| \lesssim |E|$.
Further, for such $R$ and $ \Rdagger$, write $g = \cent(R)$, $g^\dagger = \cent( \Rdagger)$, $q = \wid(R)$ and $h = \heit(R)$; then $g$ is in the interior of $ \Rdagger$, and \eqref{eq:size-of-rectangle} implies that
\begin{align*}
R^{*,\kappa}
&= g  B\one(o,\kappa q/2)  B\two (0,\kappa^2(q^2+4h)/8) \\
&\subseteq g^\dagger  B\one(o,(2\cdim+1) q/2)
 B\two (0,(2\cdim+1)^2(q^2+4h)/8) \\
&\qquad B\one(o,\kappa q/2)  B\two (0,\kappa^2(q^2+4h)/8) \\
&\subseteq g^\dagger  B\one(o,(2\cdim+1 +\kappa) q/2)
 B\two (0,((2\cdim+1)^2+\kappa^2)(q^2+4h)/8)  \\
&\subseteq g^\dagger  B\one(o,\kappa_1 (2\cdim+1) q/2)
 B\two (0,\kappa_1^2(2\cdim+1)^2(q^2+4h)/8) \\
&= ( \Rdagger)^{*,\kappa_1},
\end{align*}
since $\kappa_1 - 1  \geq \kappa / (2\cdim+1) $.

We take $\set{S}$ to be the collection $\{ \Rdagger : R \in \maxrect(E)\}$, and for $S \in \set{S}$, we write
\[
\tilde{a}_S = \sum_{R \in \maxrect(E):  \Rdagger = S} a_R
\qquad\text{and}\qquad
\tilde{b}_S = \sum_{R \in \maxrect(E):  \Rdagger = S} b_R,
\]
as in the previous lemma.
Clearly, $\supp b_R \subseteq ( \Rdagger)^{*,\kappa_1}$, whence $\supp \tilde b_S \subseteq S^{*, \kappa_1}$.
By Lemma \ref{lem:grouping-particles}, $(|E|/|\tilde{E}|)^{1/2} a$ is a $(1,2,M,N,\kappa_1)$ atom.

It follows that the $(1,2,M,N,\kappa)$ atomic Hardy space is a subspace of the $(1,2,M,N,\kappa_1)$ atomic Hardy space; since the converse is trivial, these spaces coincide.
\end{proof}

\begin{corollary}
As a topological vector space, the $(1,2,M,N,\kappa)$ atomic Hardy space is independent of the parameter $\kappa$ in $(1 + 1/(2\nu),\infty)$.
\end{corollary}

In light of this corollary, when we are only interested in boundedness or convergence in $\atomHardy(\Hn)$, we may abbreviate $(1,2,M,N,\kappa)$ atom to $(1,2,M,N)$ atom.
However, when we make precise statements about $\atomHardy(\Hn)$ norms, we need to specify $\kappa$ (and $M$ and $N$).

Fifth, there is no loss of generality in supposing that atoms are \emph{finite}, by which we mean that they are sums of finitely many particles.

\begin{lemma}\label{lem:finite-atoms}
For every atom $a$ and every $\epsilon \in \R^+$, there exist finite atoms $a_j$ and $\lambda_j \in \R^+$ such that
\[
a = \sum_{j} \lambda_j a_j
\qquad\text{and}\qquad
\sum_{j} \lambda_j \leq 1 + \epsilon.
\]
The sum $\sum_{j} \lambda_j a_j$ converges absolutely in $\Leb^2(\Hn)$.
\end{lemma}

\begin{proof}
Take an atom $a$, and write $a = \sum_{R\in \maxrect(E)} a_R$, as in Definition \ref{def:atomic-Hardy-space}.
The unconditional convergence of the sum implies that we can find recursively finite subsets $\finiteset_j$ of $\maxrect(E)$ such that $\finiteset_j \subseteq \finiteset_{j+1}$ for all $j$, $\maxrect(E) = \bigcup_{j=0}^{\infty} \finiteset_j$ and, for all sign sequences $\sigma: \rect(E) \to \{\pm1\}$, the sum $\sum_{R \in \maxrect(E) \setminus  \finiteset_j} \sigma_R a_R$ converges in $\Leb^2(\Hn)$, to $a_{\sigma, j}$ say, and
\[
\norm{ a_{\sigma, j } }_{\Leb^2(\Hn)}
\leq 2^{-2 - j} \epsilon \abs|E|^{-1/2} .
\]
Let $\lambda_0 = 1$ and $\lambda_j = 2^{-j} \epsilon$ when $j \geq 1$.
It now follows that $\lambda_j^{-1} \sum_{R \in \finiteset_{j+1} \setminus \finiteset_j} a_R$ is a finite atom associated to $E$, and clearly $\sum_j \lambda_j = 1+ \epsilon$ and $a = \sum_j \lambda_j a_j$.
\end{proof}

\begin{corollary}\label{cor:finite-atomic-decomposition}
Fix $M,N \in \N^+$ and a real number $\kappa > 1 + 1/(2\nu)$.
Suppose that $f \in \atomHardy(\Hn)$.
Then for all $\epsilon \in \R^+$, we may write $f \sim \lambda_j a_j$, where $a_j$ is a finite atom, and $\lambda_j \in \R^+$ and $\sum_j \lambda_j < (1+\epsilon) \norm{f}_{\atomHardy(\Hn)} $.
\end{corollary}

\begin{proof}
This follows from Lemma \ref{lem:finite-atoms} and the definition of $\atomHardy(\Hn)$.
\end{proof}

Each finite atom is associated to a bounded set $E$.

Finally, we used the family of shards $\rect$ in defining atoms.
Since a dilate or a left translate of a shard need not be a shard, it is possible that $\atomHardy(\Hn)$ might not be dilation or (left) translation invariant.
Since the other Hardy spaces defined in the Introduction  are evidently dilation and left translation invariant, this would mean that these spaces could not coincide with the atomic space.
Fortunately, our next result implies that this is not the case.

We write $\tube(E)$ for the collection of all tubes contained in an open set $E$.

\begin{definition}\label{def:tube-atom}
Fix  $M,N \in \N^+$.
A \emph{$(1,2,M,N)$ tube atom} associated to an open set $E$  of finite measure $|E|$ is a function $a\in \Leb^2(\Hn)$ such that there exist functions $a_T$ in $\Leb^2(\Hn)$, called \emph{tube particles}, and $b_T$ in $\Dom(\HLap^{M} \VLap^{N})$ for all $T \in \tube(E)$ such that
\begin{enumerate}
\item[(A1)] $a_T = \HLap^M \VLap^N b_T$ and $\supp b_T \subseteq T$;
\item[(A2)] for all sign sequences $\sigma: \tube(E) \to \{\pm1\}$, the sum $\sum_{T \in \tube(E)} \sigma_T a_T$ converges in $\Leb^2(\Hn)$, to $a_{\sigma}$ say, and
\begin{equation*}
\norm{ a_\sigma }_{\Leb^2(\Hn)}
\leq \big|E\big|^{-1/2} ;
\end{equation*}
\item[(A3)] $a = \sum_{T \in \tube(E)} a_T$.
\end{enumerate}

We say that $f \in \Leb^1(\Hn)$ has a decomposition into tube atoms if we may write $f$ as a sum $\sum_{j\in \N} \lambda_j a_j$, converging in $\Leb^1(\Hn)$, where $\sum_{j\in\N} |\lambda_j| < \infty$ and each $a_j$ is a tube atom; we also write $f \sim \sum_{j\in\N} \lambda_j a_j$ to indicate that $\sum_{j\in\N} \lambda_j a_j$ is a decomposition of $f$ into tube atoms.
The tube atom Hardy space is defined to be the linear space of all $f \in \Leb^1(\Hn)$ that have decompositions into tube atoms, with norm
\begin{equation*}
\norm{ f }
:=\inf\biggl\{ \sum_{j\in\N} |\lambda_j| :
  f \sim \sum_{j\in\N} \lambda_j a_j \biggr\} .
\end{equation*}
\end{definition}

We note that the set $\tube(E)$ is uncountable. 
Only countably many $a_T$ can be nonzero in the definition of a tube atom.
We also note that the proof of Lemma \ref{lem:finite-atoms} also works for tube atoms, and hence in the definition of the Hardy space based on tube atoms, we may suppose that all atoms are finite, that is, are finite sums of particles.

\begin{corollary}\label{cor:atoms-based-on-tubes}
Fix $M, N \in \N^+$.
Every $(1,2,M,N)$ tube atom associated to a set $E$ is a $(1,2,M,N,(2\cdim+1)^2))$ atom associated to a set $E_1$, where $|E_1| \leq |E|$.
Conversely, every $(1,2,M,N,\kappa)$ atom associated to a set $E_1$ is a geometric multiple of a $(1,2,M,N)$ tube atom associated to a set $E$,  where $|E_1| \eqsim |E|$.
Hence the tube atom Hardy space coincides with the Hardy space based on shards, with equivalence of norms.
\end{corollary}

\begin{proof}
Let $a$ be a tube atom associated to a set $E$.
By Lemma \ref{lem:tubes-and-rectangles-2}, for each $T \in \tube(E)$, there is a shard $R_T$ such that $R_T \subseteq T \subseteq R_T^*$, where the enlargement parameter $\kappa$ is taken to be $(2\cdim+1)^2$.
We take $E_1$ to be $\bigcup_{T \in \tube(E)} R_T$; then 
\[
|E_1| \leq |E| \leq \abs| \bigcup_{k \in \N} R_k^* | \lesssim |E_1| .
\]
For all sign functions $\sigma: \tube(E) \to \{\pm1\}$, 
\[
\norm{ \sum_{T  \in \tube(E)} \sigma_T a_T }_{\Leb^2(\Hn)} 
\leq \big|E\big|^{-1/2} 
\leq \big|E_1\big|^{-1/2}.
\]

As in the proof of Lemma \ref{lem:grouping-particles}, we group together all $a_T$ contained in the same maximal subshard $R$ of $E_1$.
Let $R_T ^\dagger$ be the widest maximal subshard of $E_1$ such that $T \subseteq R_T ^\dagger$.
For $\tilde\sigma: \maxrect(E_1) \to \{\pm1\}$, we define $\sigma: \tube(E) \to \{\pm1\}$ by $\sigma(T) = \tilde\sigma(R_T^\dagger)$, and for $R \in \maxrect(E_1)$, we set 
\[
a_R = \sum_{\substack{T \in \tube(E)\\ R_T^\dagger = R}} a_T.
\]
It is now immediate that
\[
\biggnorm{ \sum_{R \in \maxrect(E_1)} \tilde\sigma_R \tilde{a}_R }_{\Leb^2(\Hn)} \leq \big|E_1\big|^{-1/2},
\]
and $a$ is a $(1,2,M,N, (2\cdim+1)^2)$ atom associated to $E_1$.

The converse statement follows from the definition: if $a$ is a $(1,2,M,N, \kappa)$ atom associated to $E_1$, then $a = \sum_{R \in \maxrect(E_1)} a_R$, and each $a_R$ is supported in the tube $R^*$.
Define $E = \bigcup_{R \in \maxrect(E_1)} R^*$.
Then 
\[
\abs| \bigcup_{R \in \maxrect(E_1)} R^*| \leq |E_1|^{-1/2} \lesssim |E|^{-1/2}, 
\]
so $a$ is a geometric multiple of a $(1,2,M,N)$ tube atom. 
\end{proof}

For a function $f$ on $\Hn$ and $g \in \Hn$, we define the \emph{left translate} $_{g}f$ of $f$ by $g$ by
\[
{}_gf(g') := f(g^{-1} g')
\qquad\forall g' \in \Hn.
\]

\begin{corollary}\label{cor:invariance-of-atomHardy}
Suppose that $f \in \atomHardy(\Hn)$ (defined using finite tube atoms).
Then the following hold:
\begin{enumerate}
  \item  the translate ${}_gf$ is in $\atomHardy(\Hn)$ and 
  \[
  \norm{ {}_gf }_{\atomHardy(\Hn)} 
  = \norm{ f }_{\atomHardy(\Hn)}  
  \qquad\forall g \in \Hn ,
  \] 
  \item as $g \to o$ in $\Hn$, ${}_gf \to f$ in $\atomHardy(\Hn)$;
  \item the normalised dilate $f_r$ is in $\atomHardy(\Hn)$ and 
  \[
  \norm{ f_r }_{\atomHardy(\Hn)} 
  = \norm{ f }_{\atomHardy(\Hn)} 
  \qquad\forall r \in \R^+;
  \]
  \item as $r \to 1$ in $\R^+$, $f_r \to f$ in $\atomHardy(\Hn)$.
\end{enumerate}
\end{corollary}

\begin{proof}
It suffices to consider what happens with atoms.

Suppose that $a$ is a finite tube atom associated to a bounded open set $E$.
Then $a = \sum_{T \in \finiteset} a_T$, where $\finiteset$ is a finite subset of $\tube(E)$; further, $\supp(a_T) \subseteq T$ and $E = \bigcup_{T \in \finiteset}T$.
Then ${}_g a = \sum_{T \in \finiteset} {}_g a_T$; clearly $\supp({}_g a_T) \subseteq gT$ and $gE = \bigcup_{T \in \finiteset}gT$.
If $a_T = \HLap^M \VLap^N b_T$ for some $b_T \in \Dom(\HLap^M \VLap^N)$, then ${}_ga_T = \HLap^M \VLap^N {}_gb_T$.
For any sign function $\sigma: \finiteset \to \{\pm 1\}$,
\[
\biggnorm{ \sum_{T \in \finiteset} \sigma_T {}_g a_T }_{\Leb^2(\Hn)}
= \biggnorm{ \sum_{T \in \finiteset} \sigma_T a_T }_{\Leb^2(\Hn)}
\leq \abs| E |^{-1/2} 
= \abs| g E |^{-1/2} .
\]
Hence ${}_ga$ is a tube atom associated to $gE$, and (1) holds.

For all $g$ in $B\one(o, t)$, the function ${}_ga_T - a_T$ is supported in $B\one(o, t) T$, and
\[
\biggabs| \bigcup_{T \in \finiteset} B\one(o, t) T | 
\leq \biggabs| \bigcup_{T \in \finiteset} T |
+ \sum_{T \in \finiteset} \biggabs| B\one(o, t) T  \setminus T|
\leq 4 \biggabs| \bigcup_{T \in \finiteset} T | 
\]
when $t$ is small enough, by the continuity of the group operations.

Indeed, suppose that $T = g B\one(o,r) B\two(0,h)$.
For all $t \in \R^+$, there exists $t' \in \R^+$ (which depends on $t$ and on $g$) such that $B\one(o,t')g \subseteq g B\one(o,t)$, and so
\[
B\one(o,t') T \subseteq g B\one(o,t) B\one(o,r) B\two(0,h) = g B\one(o,r+t) B\two(0,h),
\]
and 
\[
\lim_{t \to 0}\abs| g B\one(o,r+t) B\two(0,h) \setminus g B\one(o,r) B\two(0,h) | = 0  .
\]

Further,  $\norm{{}_ga_R - a_R}_{\Leb^2(\Hn)} \to 0$ as $g \to o$, so for any sign function $\sigma: \finiteset \to \{\pm1\}$,
\[
\lim_{g \to o} \biggnorm{ \sum_{T \in \finiteset}  \sigma_T ( {}_ga_T - a_T) }_{\Leb^2(\Hn)} 
\leq \lim_{g \to o} \sum_{T \in \finiteset} \norm{  {}_ga_T - a_T }_{\Leb^2(\Hn)}  = 0,
\]
and (2) holds.

We leave to the reader the task of treating dilations in a similar manner.
\end{proof}

\begin{corollary}\label{cor:approx-identities}
Suppose that $M, N \in \N^+$ and $f \in \atomHardy(\Hn)$.
Suppose also that $\phi: \Hn \to \C$ is Poisson bounded and $\int_{\Hn} \phi(g) \wrt g = 1$, and let $\phi_r$ be the normalised dilate of $\phi$.
Then $\phi_r \Hconv f  \in \atomHardy(\Hn)$ and $\bignorm{ \phi_r \Hconv f  }_{ \atomHardy(\Hn)}$ is uniformly bounded for all $r \in \R^+$, and $\phi_r \Hconv f  \to f$ in $\atomHardy(\Hn)$ as $r \to 0$.
\end{corollary}

\begin{proof}
This follows from the previous lemma.
\end{proof}

To prove Journé's lemma (in Section \ref{ssec:moment-atom-Hardy}), we need a slightly stronger result.

\begin{corollary}\label{eq:atomic-decomposition-of-convolution}
Suppose that $M, N \in \N^+$ and that $\phi$ is smooth,  normalised, that is, $\int_{\Hn} \phi(g) \wrt g = 1$, and $\supp(\phi) \subset B\one(o,1)$.
Then there exists a constant $C$ such that, for all finite atoms $a$, 
we may write $a - \phi_r \Hconv a $ as a finite sum $\sum_{k} \lambda_{k,r} a_{k,r}$, 
where $\lambda_{k,r} \in \R^+$, $\sum_{k} \lambda_{k,r} \leq C$ for all $r \in \R^+$ and $\limsup_{r \to 1} \sum_{k} \lambda_{k,r} = 0$, while each $a_{k,r}$ is a finite atom.
\end{corollary}

\begin{proof}
Suppose that $a = \sum_{T \in \finiteset} a_T$, where each $a_T$ is a particle, $\supp(T) \subseteq T$, and for all sign functions $\sigma: \finiteset \to \{\pm1\}$, 
\[
\biggnorm{ \sum_{T \in \finiteset} \sigma_T a_T } \leq \abs| E |^{-1/2},
\]
where $E = \bigcup_{T \in \finiteset} T$.

As argued in the proof of Corollary \ref{cor:invariance-of-atomHardy}, we may choose $\epsilon \in \R^+$ such that $\abs| B\one(o,\epsilon)E| \leq 4 \abs|E|$. 
When $g \in B\one(o,\epsilon)$, ${}_g a$ is supported in the set $B\one(o,\epsilon)E$.
Hence if $r \leq \epsilon$, then $a_T - \phi_r \Hconv a_T $ is supported in $B\one(o,\epsilon) T$ and $a - \phi_r \Hconv a $ is supported in $B\one(o,\epsilon) E$.
Further, $\norm{ a_T - \phi_r \Hconv a_T }_{\Leb^2} \to 0$ as $r \to 0+$, and hence
for any sign function $\sigma: \finiteset \to \{\pm1\}$,
\[
\biggnorm{ \sum_{T \in \finiteset}  \sigma_T ( a_T - \phi_r \Hconv a_T ) }_{\Leb^2(\Hn)} 
= \sum_{T \in \finiteset} \norm{ a_T - \phi_r \Hconv a_T }_{\Leb^2(\Hn)}  \to 0
\]
as $r \to 0+$.
Then $a - \phi_r \Hconv a$ is a multiple of a single atom associated to $B\one(o,\epsilon) E$, and the multiple tends to $0$ as $r$ does.

When $r$ is large, it suffices to write $\phi_r \Hconv a$ as a finite weighted sum of finite atoms, with control of the sum of the weights.
Take finitely many $g_j \in \bar B\one(o,r)$ such that the open sets $g_j B\one(o,\epsilon)$ cover the compact set $\bar B\one(o,r)$, and then take closed subsets $F_j$ of $g_j B\one(o,\epsilon)$ that are pairwise disjoint (up to null sets) and cover $\bar B\one(o,r)$.
Define $\lambda_j = 2 \int_{F_j} \abs|\phi_r(g)| \wrt g$.
Then
\begin{align*}
\phi_r \Hconv a(g') 
&= \sum_{j} \int_{\Hn} (\indifn_{F_j} \phi_r)(g) a( g^{-1} g') \wrt g \\
&= \sum_{j} \int_{\Hn} (\indifn_{F_j}\phi_r) (g_j g) a(g^{-1} g_j^{-1} g') \wrt g \\
&= \sum_{j} \lambda_j \frac{1}{\lambda_j}  \int_{\Hn} {}_{g_j}(\indifn_{F_j}\phi_r) (g) a(g^{-1} g_j^{-1} g') \wrt g.
\end{align*}
By construction, ${}_{g_j}(\indifn_{F_j}\phi_r)$ is supported in $g_j^{-1}F_j \subseteq B\one(o,\epsilon)$, and the convolution 
\[
g' \mapsto \frac{1}{\lambda_j}  \int_{\Hn} {}_{g_j}(\indifn_{F_j}\phi_r) (g) a(g^{-1} g') \wrt g
\]
is a finite atom supported in $B\one(o,\epsilon) E$.
Indeed,
\[
{}_{g_j}(\indifn_{F_j} \phi_r) \Hconv a = \sum_{T \in \finiteset} {}_{g_j}(\indifn_{F_j} \phi_r) \Hconv a_T 
\]
and, for any sign function $\sigma: \finiteset \to \{\pm 1\}$,
\[
\begin{aligned}
\biggnorm{ \frac{1}{\lambda_j}  \sum_{T \in \finiteset} \sigma_T {}_{g_j}(\indifn_{F_j} \phi_r) \Hconv a_T }_{\Leb^2(\Hn)}
&\leq \frac{ \norm{(\indifn_{F_j} \phi_r)}_{\Leb^1(\Hn)} }{\lambda_j}
\biggnorm{ \sum_{T \in \finiteset} \sigma_T  a_T }_{\Leb^2(\Hn)} \\
&\leq \frac{1}{2} \abs|E|^{-1/2}  \leq \abs| B\one(0,\epsilon) E|^{-1/2}  .
\end{aligned}
\]
Finally, the translate
\[
g' \mapsto \frac{1}{\lambda_j}  \int_{\Hn} {}_{g_j}(\indifn_{F_j}\phi_r) (g) a(g^{-1} g_j^{-1} g') \wrt g
\]
is also a finite atom.
\end{proof}

\subsection{The moment atom Hardy space}\label{ssec:moment-atom-Hardy}

There is a closely related atomic space that we define at this point.

\begin{definition}\label{def:moment-Hardy-space}
Fix positive integers $M$ and $N$ and a real number $\kappa$ in $(1 + 1/(2\nu),\infty)$.
A \emph{moment atom} is a function $a\in \Leb^2(\Hn)$ such that there exists an open subset $E$ of $\Hn$ of finite measure $|E|$ and functions $a_R$ in $\Leb^2(\Hn)$, called \emph{moment particles}, for all $R \in \maxrect(E)$ such that
\begin{enumerate}
\item[(A1)]  the moments $\int_{\Hn} p(z,t) \fn a_R(z,t) \wrt z\wrt t$ vanish for all monomials $p$ of degree $d_1$ in $z$ and degree $d_2$ in $t$, where $d_2 < 2N$ or $d_1 + 2d_2 < 2M + 4N$, and $\supp a_R \subseteq R^* = T(g, \kappa q, \kappa^2 h)$, where $g = \cent(R)$, $q = \wid(R)$ and $h = \heit(R)$; 
\item[(A2)] for all sign sequences $\sigma: \maxrect(E) \to \{\pm1\}$, the sum $\sum_{R \in \maxrect(E)} \sigma_R a_R$ converges in $\Leb^2(\Hn)$, to $a_{\sigma}$ say, and
\begin{equation*}
\norm{ a_\sigma }_{\Leb^2(\Hn)} \leq \abs|E|^{-1/2} ;
\end{equation*}
\item[(A3)] $ a = \sum_{R \in \maxrect(E)} a_R$.
\end{enumerate}

We say that $f \in \Leb^1(\Hn)$ has a moment atomic decomposition if we may write $f$ as a sum $\sum_{j\in \N} \lambda_j a_j$, converging in $\Leb^1(\Hn)$, where $\sum_{j\in\N} |\lambda_j| < \infty$ and each $a_j$ is a moment atom; we write $f \sim \sum_{j\in\N} \lambda_j a_j$ to indicate that $\sum_{j\in\N} \lambda_j a_j$ is an atomic decomposition of $f$.
The space $\matomHardy(\Hn)$ is defined to be the linear space of all $f \in \Leb^1(\Hn)$ that have moment atomic decompositions, with norm
\begin{equation}\label{def:momentatomHardynorm}
\norm{ f }_{\matomHardy(\Hn)}
:=\inf\biggl\{ \sum_{j\in\N} |\lambda_j| :
  f \sim \sum_{j\in\N} \lambda_j a_j \biggr\} .
\end{equation}
\end{definition}

We note that, from Proposition \ref{prop:moments}, the atomic Hardy space $\atomHardy(\Hn)$ is a subspace of the moment atomic Hardy space $\matomHardy(\Hn)$, and there is a corresponding norm inequality.
Later we show that these Hardy spaces coincide.

Next, the results of the previous section about the atomic Hardy space extend to the moment atomic Hardy space, with almost identical proofs.

\subsection{Boundedness on particles and atoms}\label{ssec:atomic-boundedness-thm}

We recall \eqref{eq:size-of-rectangle} and \eqref{eq:def-rho}. 
First, if $R$ is a  shard, $\cent(R) = g$, $\wid(R) = q$ and $\heit(R) = h$, then $R^*$ is defined to be $T(g, \kappa q/2, \kappa^2 (4h + q^2)/8)$.
Second, if $R$ and $S$ are shards and $R \subseteq S$, and $\epsilon_1, \epsilon_2 \in \R^+$, then
\begin{equation*}
\begin{aligned}
\rho_{\boldsymbol\epsilon}(R,S)
:= \lpar \frac{\wid(R)}{ \wid(S)} \rpar^{\epsilon_1}
+ \lpar \frac{\heit(R)}{ \heit(S) }\rpar^{\epsilon_2} .
\end{aligned}
\end{equation*}

\begin{prop}\label{prop:particle-bounded-function-bounded}
Suppose that $\oper{A}$ is a linear operator, or a nonnegative sublinear operator, that satisfies a strong type $(2,2)$ bound $\norm{ \oper{A} f}_{\Leb^2(\Hn)} \le C_1 \norm{ f }_{\Leb^2(\Hn)}$ or a weak-type $(2,2)$ bound
\begin{align*}
\abs| {\lset g \in \Hn: {\abs| \oper{A} f(g)|} >\lambda \rset} |
\le \frac{C_2}{\lambda^{2}} \norm{ f }_{\Leb^2(\Hn)}^2
\qquad\forall \lambda \in \R^+
\end{align*}
for all $f \in \Leb^2(\Hn)$.
Let $M,N \in \N^+$.

(a) Suppose also that there exist $C$, $\epsilon_1$, and $\epsilon_2$ in $\R^+$ and $I, J \in \N^+$ such that
\begin{equation}\label{eq:hypothesis}
\int_{(S^*)^c} \abs| \oper{A} a_R(g) | \wrt g
\leq C
\rho_{\boldsymbol\epsilon}(R,S)
\abs|R|^{1/2} \norm{a_R}_{\Leb^2(\Hn)}
\end{equation}
for all $R, S \in \rect$ such that $R \subseteq S \in \rect$ and $\wid(S)/\wid(R)\geq (2\cdim+1)^I$ and $\heit(S)/\heit(R) \geq (2\cdim+1)^{2J}$, and for all particles $a_R \in \Leb^2(\Hn)$ such that $a_R =\HLap^M \VLap^N b_R$ where $b_R \in \Leb^2(\Hn)$ and $\supp b_R \subseteq R^*$.
Then there is a constant $C_0$ such that
\begin{equation}\label{e4.111}
\begin{aligned}
 \norm{ \oper{A}a }_{\Leb^1(\Hn)}\leq C_0
\end{aligned}
\end{equation}
for all $(1,2,M,N)$ atoms $a$.

(b) Suppose also that \eqref{e4.111} holds.
Then $\oper{A}$ maps $\atomHardy(\Hn) \cap \Leb^2(\Hn)$ into $\Leb^1(\Hn)$, and
\begin{equation}\label{eq:norm-inequality}
 \norm{ \oper{A}f }_{\Leb^1(\Hn)} \leq C_0 \norm{ f }_{\atomHardy(\Hn)}
 \qquad\forall f \in \atomHardy(\Hn) \cap \Leb^2(\Hn).
\end{equation}
Hence $\oper{A}$ extends uniquely by continuity to a bounded operator from $\atomHardy(\Hn)$ to $\Leb^1(\Hn)$ that satisfies the same inequality for all $f \in \atomHardy(\Hn)$.
\end{prop}

\begin{proof}
Take a $(1,2,M,N,\kappa)$ atom $a$ associated to an open set $E$ of finite measure, and, as in Definition \ref{def:atomic-Hardy-space}, write $a \sim \sum_{R \in \maxrect(E)} a_R$.
As in Definition \ref{def:E-one-E-two}, we define sets $E\one$ and $E\two$ as follows:
\[
\begin{aligned}
E\one &= \lset g \in G : \oper{M}_{\mathrm{sh}}( \indifn_E)(g) > {\alpha_1} \rset \\
E\two &= \lset g \in G : \oper{M}_{\mathrm{sh}} (\indifn_{E\one})(g) > {\alpha_2} \rset.
\end{aligned}
\]
Given a shard $R$ in $\maxrect(E)$, we define $R\one$ to be the widest shard $S$ such that $R \subseteq S$ and $S \in \maxrect(E\one)$, and $R\two$ to be the highest shard $S$ such that $R\one \subseteq S$ and $S \in \maxrect(E\two)$.
Then $|E \two | \lesssim_{I,J} |E|$ by two applications of Lemma \ref{lem:expanding-sets-by-max-fn}.

Now
\[
\norm{ \oper{A} (a) }_{\Leb^1(\Hn)}
= \int_{{E\two}} |\oper{A} (a)(g)|\wrt g
+ \int_{({E\two})^c} |\oper{A} (a)(g)|\wrt g.
\]

On the one hand, by the (weak) $\Leb^2$ boundedness of $\oper{A}$, and Hölder's inequality,
\begin{align*}
\int_{{E\two}}|\oper{A} (a)(g)|\wrt g
&\lesssim   |{E\two} |^{1/2} \norm{ \oper{A} (a) }_{\Leb^{2,\infty}(\Hn)}
\lesssim_{I,J} |E|^{1/2} \norm{ a }_{\Leb^2(\Hn)}
\leq 1.
\end{align*}
On the other hand, $\oper{A} a = \sum_{R\in \maxrect(E)} \oper{A} a_R$, where the sum converges unconditionally  in the Lorentz space $\Leb^{2,\infty}(\Hn)$ by the boundedness hypothesis, so
\begin{align*}
\int_{({E\two})^c} |\oper{A} (a)(g)|\wrt g
&\leq \sum_{R\in \maxrect(E) } \int_{({E\two})^c} |\oper{A} (a_R)(g)|\wrt g \\
&\leq \sum_{R\in \maxrect(E) } \int_{(R\two)^c}
\abs|\oper{A} (a_R)(g)|\wrt g \\
&\leq \sum_{R\in \maxrect(E) }
\rho_{\boldsymbol\epsilon} (R,R\two)
        \abs|R|^{1/2} \norm{ a_R}_{\Leb^2(\Hn)}. 
\end{align*}
By the Cauchy--Schwarz inequality, Journé's lemma (Lemma \ref{lem:Journe}), and \eqref{eq:L^2-condition-particles-in-atoms}, 
\[
\begin{aligned}
\norm{ \oper{A} (a)}_{\Leb^1(\Hn)}
&\lesssim
\lpar \sum_{R\in \maxrect(E) } \rho_{\boldsymbol\epsilon}(R,R\two)^{2}  |R| \rpar^{1/2}
\lpar \sum_{R\in \maxrect(E)} \norm{a_R}^2_{\Leb^2(\Hn)} \rpar^{1/2} \\
&\lesssim |E|^{1/2} |E|^{-1/2} =1 ,
\end{aligned}
\]
as required.

To prove (b), it will suffice to prove \eqref{eq:norm-inequality}.
Take $f \in \atomHardy(\Hn) \cap \Leb^2(\Hn)$, and a smooth radial function $\phi$ with support in $B\one(o,1)$ such that $\int_{\Hn} \phi(g) \wrt g = 1$; let $\phi_r$ be the normalised dilate of $\phi$, as in \eqref{eq:def-phi-1}.
Now $f = \lim_{r\to 0+} \phi_r \Hconv f $ in $\Leb^2(\Hn)$ and so $\oper{A} f = \lim_{r\to 0+} \oper{A} (\phi_r \Hconv f )$ in $\Leb^2(\Hn)$.
If we can show that 
\begin{equation}\label{eq:want-to-prove}
\limsup_{r \to 0+} \norm{ \oper{A}(\phi_r \Hconv f ) }_{\Leb^1(\Hn)} 
\leq C_0 \norm{ f }_{\atomHardy(\Hn)} ,
\end{equation}
then it will follow that
\[
\int_E \abs| \oper{A}(f)(g) | \wrt g 
= \lim_{r\to 0+} \int_E \abs| \oper{A}(\phi_r \Hconv f )(g) | \wrt g 
\leq C_0 \norm{ f }_{\atomHardy(\Hn)}
\]
for all subsets $E$ of $\Hn$ of finite measure, so $\norm{ \oper{A} f }_{\atomHardy(\Hn)} \leq C_0 \norm{ f }_{\atomHardy(\Hn)}$.

Given $\epsilon \in \R^+$, we can write $f \sim \sum_j \lambda_j a_j$, where the $a_j$ are finite atoms, $\lambda_j > 0$, and $\sum_j \lambda_j  < \norm{f}_{\atomHardy(\Hn)} + \epsilon$.
The sum $\sum_{j} \lambda_j \phi_r \Hconv a_j $ converges in $\Leb^2(\Hn)$ since the sum $\sum_{j} \lambda_j a_j$ converges in $\Leb^1(\Hn)$, so 
\[
\abs| \oper{A}(\phi_r \Hconv f  ) | 
= \Bigabs| \oper{A} (\sum_{j} \lambda_j \phi_r \Hconv a_j ) |
\leq \sum_{j} \lambda_j \abs| \oper{A} (\phi_r \Hconv a_j) |, 
\]
whence
\[
\begin{aligned}
\norm{ \oper{A} (\phi_r \Hconv f ) }_{\Leb^1(\Hn)} 
& \leq  \sum_{j\in\N^+}  \lambda_j \norm{ \oper{A} (\phi_r \Hconv a_j )}_{\Leb^1(\Hn)} .
\end{aligned}
\]
If we can show that 
\begin{equation}\label{eq:want-to-prove-2}
\norm{ \oper{A} (\phi_r \Hconv a_j )}_{\Leb^1(\Hn)} \leq C_1 
\qquad\forall r \in \R^+ 
\end{equation}
for some constant $C_1$, and
\[
\limsup_{r \to 0+} \norm{ \oper{A} (\phi_r \Hconv a_j )}_{\Leb^1(\Hn)} \leq C_0,
\] 
then \eqref{eq:want-to-prove}, and hence the proposition, will follow.

By Corollary \ref{eq:atomic-decomposition-of-convolution}, there exists a constant $C$ such that, for all finite atoms $a$ and all $r \in \R^+$, 
we may write $a - \phi_r \Hconv a $ as a finite sum of finite atoms $\sum_{k} \lambda_{k,r} a_{k,r}$, 
where $\lambda_{k,r} \in \R^+$, $\sum_{k} \lambda_{k,r} \leq C$ for all $r \in \R^+$ and $\lim_{r \to 0} \sum_{k} \lambda_{k,r} = 0$.
Then
\[
\oper{A}(\phi_r \Hconv a) \leq \oper{A}(a) + \sum_{k} \lambda_{k,r} \oper{A}( a_{k,r} ), 
\]
and so
\[
\begin{aligned}
\norm{ \oper{A}(\phi_r \Hconv a)  }_{\Leb^1(\Hn)} 
&\leq \norm{  \oper{A}(a)  }_{\Leb^1(\Hn)}  + \sum_{k} \lambda_{k,r} \norm{  \oper{A}( a_{k,r} )  }_{\Leb^1(\Hn)}  \\
&\leq C_0  + \sum_{k} C_0 \lambda_{k,r}  , 
\end{aligned}
\]
and the desired result \eqref{eq:want-to-prove-2} follows.
\end{proof}

A similar result holds for the moment atomic Hardy space of Section \ref{ssec:moment-atom-Hardy}.

\subsection{The inclusion $\atomHardy(\Hn) \subseteq \areaHardy(\Hn)$}
\label{ssec:atom-implies-area}
In this section, we consider the Lusin--Littlewood--Paley operator $\areaopphi$, and show that it is bounded from the atomic Hardy space $\atomHardy(\Hn)$ to $\Leb^1(\Hn)$, whence the inclusion of the title follows.
We first recall the definition of this operator.
For the definition of Poisson bounded functions, see Definition \ref{def:Poisson-bounds}, and for normalised dilates, see \eqref{eq:def-phi-1} and \eqref{eq:def-chi-1}.

\begin{definition*}
Suppose that Poisson bounded functions $\phi\one$ and $\phi\two $ have mean $0$.
We define the Lusin--Littlewood--Paley area function $\areaopphi(f)$ of $f \in \Leb^1(\Hn)$ associated to $\bsym\phi$ by
\begin{equation}\label{eq:def-Lusinpp}
\begin{aligned}
&\areaopphi (f)(g)
:= \biggl( \iint_{\R^{+} \times\R^{+}}
\bigl| f \Hconv \phi_{r,s} \bigr|^2 \Hconv \chi_{r,s}(g)  \,\frac{dr}{r} \,\frac{ds}{s} \biggr)^{1/2}
\end{aligned}
\end{equation}
for all $g \in \Hn$, and the Hardy space $\areaHardyphi(\Hn)$, usually shortened to $\areaHardy(\Hn)$, to be the set of all $f\in \Leb^1(\Hn)$ for which $\norm{ \areaopphi (f) }_{\Leb^1(\Hn)}<\infty$, with  seminorm
\begin{equation*}
\norm{ f }_{\areaHardyphi(\Hn)}:=\norm{ \areaopphi (f) }_{\Leb^1(\Hn)}.
\end{equation*}
\end{definition*}

\begin{theorem}\label{thm:Atom-implies-area}
Suppose that $\bsym\phi$ is a Poisson bounded pair of functions, both of which have mean $0$.
Then $\areaopphi(f) \in \Leb^1(\Hn)$ and
\[
\norm{\areaopphi(f)}_{\Leb^1(\Hn)} \leq C(\bsym\phi)\norm{f}_{\atomHardy(\Hn)}
\qquad\forall f \in \atomHardy(\Hn).
\]
where, for some integer $M $ greater than $\hdim/4$,
\[
C(\bsym\phi)
\lesssim
\bignorm{\oper{R} \phi\one }_{(2M)} \bignorm{\phi\two }_{(2)} .
\]
\end{theorem}

\begin{proof}
Since $2M > \hdim/2$ and $2M$ and $\delta/2$ are integers,  $2M \geq 1 + \hdim/2$.

We appeal to Proposition \ref{prop:particle-bounded-function-bounded} and the principle enunciated in Remark \ref{rem:compact-support-enough}, which show that to prove the theorem, it suffices to consider smooth functions $\phi\one$ and $\phi\two $, supported in the unit balls in $\Hn$ and in $\R$ and with vanishing means, and show that
\begin{equation}\label{eq:Lusin-on-particles}
\begin{aligned}
\int_{(S^*)^c}\areaopphi(a_R)(g) \wrt g
\leq C(\bsym\phi)
\fn\rho_{1,3/2}(R,S)
\abs|R|^{1/2} \norm{a_R}_{\Leb^2(\Hn)}
\end{aligned}
\end{equation}
for all particles $a_R$ associated to all $R \in \rect$, such that $\cent(R) = o$, $\wid(R) =1$ and $\heit(R) = h$, and for all $S \in \rect$ such that $R \subseteq S$ and $1 < \wid(S) =: r^*$ and $h < \heit(S) =: s^*$.

It follows from the definition just before the statement of this theorem that
\[
\begin{aligned}
\areaopphi(a_R)(g)
&\eqsim \sum_{k=1}^{4} \lpar \iint_{\Omega_k}
\abs|a_R \Hconv \phi_{r, s}(g)|^2 \Hconv \chi_{r,s}(g)
\,\frac{dr}{r} \, \frac{ds}{s} \rpar^{1/2} \\
&:= \sum_{k=1}^{4} \areaopphi^k(a_R)(g) ,
\end{aligned}
\]
say, where
\[
\begin{aligned}
\Omega_1 &= \{ (r,s) \in \R^+ \times \R^+ : r < 1, s > h \} \qquad&
\Omega_2 &= \{ (r,s) \in \R^+ \times \R^+ : r < 1, s < h \} \\
\Omega_3 &= \{ (r,s) \in \R^+ \times \R^+ : r > 1, s < r^2 + h\} &
\Omega_4 &= \{ (r,s) \in \R^+ \times \R^+ : r > 1, s > r^2 + h\}.
\end{aligned}
\]
We treat these four summands separately.

The key to our estimation of these terms is the observation that
\begin{equation}\label{eq:Lusin-support}
\begin{aligned}
\supp \abs| a_R \Hconv \phi_{r,s}|^2 \Hconv \chi_{r,s}
&\subseteq
T(o,1, h)
 T(o, r, s)
 T(o, r, s) \\
&\subseteq T(o, 2r+1, 2s+h ) .
\end{aligned}
\end{equation}

First, \eqref{eq:Lusin-support} implies that
$\supp (\areaopphi^{1}(a_R)) \cap (S^*)^c = \emptyset$ and so no estimation is needed.

Second, to treat $\areaop[,\bsym\phi]^{2}(a_R)$ and $\areaop[,\bsym\phi]^{3}(a_R)$ we use a well known argument, which we now sketch.
We first choose an exponentially increasing family of sets $E_j$ such that, from \eqref{eq:Lusin-support},
\[
\supp(\areaopphi^{k}(a_R))
\subseteq \bigcup_{(r,s) \in \Omega_k} T(o,1+2r,h+2s)
\subseteq \bigcup_{j\in\N} E_j.
\]
We then take $j^*$ to be the smallest $j$ such that $E_j \cap (S^*)^c \neq \emptyset$ and $\Omega_{k,j}$ to be a subset of $\Omega_{k}$ such that $(r,s) \notin \Omega_{k.j}$ implies that $T(o,1+2r,h+2s) \subseteq E_{j}$.
Then, by the definitions and basic results,
\begin{equation}\label{eq:basic-calculation}
\begin{aligned}
&\int_{(S^*)^c} |\areaopphi^{k}(a_R)(g)| \wrt g \\
&\qquad= \sum_{j \geq j^*} \int_{E_{j+1}\setminus E_j} |\areaopphi^{k}(a_R))(g)| \wrt g \\
&\qquad\leq \sum_{j \geq j^*} \abs|E_{j+1}\setminus E_j|^{1/2}
\lpar \int_{E_{j+1}\setminus E_j} |\areaopphi^{k}(a_R))(g)|^2 \wrt g \rpar^{1/2} \\
&\qquad\leq \sum_{j \geq j^*} \abs|E_{j+1}|^{1/2} \lpar
\iint_{\Omega_{k,j}} \norm{ a_R \Hconv \phi_{r,s} }_{\Leb^2(\Hn)} ^2 \,\frac{dr}{r} \,\frac{ds}{s} \rpar^{1/2}.
\end{aligned}
\end{equation}
since $\abs|E_{j+1}\setminus E_j| \leq \abs|E_{j+1}|$ and
\begin{align*}
&
\int_{E_{j+1}\setminus E_j} |\areaopphi^{k}(a_R))(g)|^2 \wrt g \\
&\qquad= 
\iint_{\Omega_k}  \int_{E_{j+1}\setminus E_j}
\abs|a_R \Hconv \phi_{r,s}|^2 \Hconv \chi_{r,s}(g) 
\wrt g \,\frac{dr}{r} \,\frac{ds}{s} \\
&\qquad= 
\iint_{\Omega_{k,j}} \int_{E_{j+1}\setminus E_j}
\abs|a_R \Hconv \phi_{r,s}|^2 \Hconv \chi_{r,s}(g) 
\wrt g \,\frac{dr}{r} \,\frac{ds}{s}  \\
&\qquad\leq 
\iint_{\Omega_{k,j}} \int_{\Hn}\abs|a_R \Hconv \phi_{r,s}|^2 \Hconv \chi_{r,s}(g)
\wrt g \,\frac{dr}{r} \,\frac{ds}{s}  \\
&\qquad=
\iint_{\Omega_{k,j}} \norm{ a_R \Hconv \phi_{r,s} }_{\Leb^2(\Hn)} ^2 \,\frac{dr}{r} \,\frac{ds}{s} .
\end{align*}
We then compute the final integral, and find that the sum converges geometrically, so is of comparable size to the term when $j = j^*$.

To treat $\areaop[,\bsym\phi]^{2}(a_R)$, we define $E_j := T(o, 3, 3 \times 2^{j}h)$ when $j \in \N$ and $\Omega_{2,j} = (0, 1) \times (2^j h, \infty)$; then $\abs| E_j | \eqsim 2^jh$ and $j^* \eqsim \log_2(s^*/h)$.
Moreover,
\begin{equation}\label{eq:S2-function}
\begin{aligned}
a_R \Hconv \phi_{r,s}
&= \HLap^M b_R \Hconv \phi\one_r \Vconv \VLap [\phi\two _s] \\
&= \frac{1}{s^{2}} \HLap^M b_R \Hconv \phi\one_r \Vconv [\VLap \phi\two ]_s,
\end{aligned}
\end{equation}
so that
\[
\begin{aligned}
\norm{a_R }_{\Leb^2(\Hn)}
&\leq \frac{1}{s^{2}} \bignorm{\HLap^M b_R \Hconv \phi\one_r }_{\Leb^2(\Hn)} \bignorm{[\VLap \phi\two ]_s }_{\Leb^1(\Hn)} \\
&\eqsim_{\bsym\phi} \frac{1}{s^{2}} \bignorm{\HLap^M b_R \Hconv \phi\one_r }_{\Leb^2(\Hn)} .
\end{aligned}
\]
From these considerations, \eqref{eq:basic-calculation}, Littlewood--Paley theory (as in Section \ref{ssec:spectral-theory}) and Lemma \ref{lem:a-is-enough},
\begin{equation}\label{eq:S2-estimate}
\begin{aligned}
&\int_{(S^*)^c}
\areaopphi^{2}(a_R)(g) \wrt g \\
&\qquad\lesssim_{\bsym\phi} \sum_{j=j^*}^{\infty} |E_{j_1}|^{1/2}
\lpar \iint_{[2^{j-1}h, \infty) \times (0,1] }
\norm{\HLap^M b_R \Hconv \phi\one_r }_{\Leb^2(\Hn)}^2
\,\frac{dr}{r} \,\frac{ds}{s^5} \rpar^{1/2} \\
&\qquad\lesssim_{\bsym\phi} \sum_{j=j^*}^{\infty} |E_{j+1}|^{1/2}
\lpar \int_{[2^{j}h, \infty) }
 \norm{\HLap^M b_R }_{\Leb^2(\Hn)}^2
 \,\frac{ds}{s^5} \rpar^{1/2} \\
 &\qquad \eqsim 2^{-3j^*/2}  h^{1/2} \norm{a_R }_{\Leb^2(\Hn)} \eqsim \Bigl( \frac{h}{s^*} \Bigr)^{3/2}  |R|^{1/2}
\norm{a_R }_{\Leb^2(\Hn)}  .
\end{aligned}
\end{equation}

To deal with $\areaopphi^{3}(a_R)$, we define $E_j :=T(o, 2^{j+1} + 1, 3h)$  and $\Omega_{3,j}:= (2^j, \infty) \times (0,h)$ for all $j \in \N$; then $\abs| E_j| \lesssim 2^{\hdim j} + 2^{2\cdim j} h$ and $j^* \eqsim \log_2(r^*)$.
Further,
\begin{equation}\label{eq:S3-function}
\begin{aligned}
a_R \Hconv \phi_{r,s}
&= \VLap b_R \Vconv \phi\two _s \Hconv \riHLap^M [\phi\one_r] \\
&= \frac{1}{r^{2M}} \VLap b_R \Vconv \phi\two _s \Hconv [\riHLap^M \phi\one]_r ,
\end{aligned}
\end{equation}
so that
\[
\begin{aligned}
\bignorm{a_R \Hconv \phi_{r,s}}_{\Leb^2(\Hn)}
&\leq  \frac{1}{r^{2M}} \bignorm{\riHLap^M \phi\one }_{\Leb^1(\Hn)}
\bignorm{ \VLap b_R \Vconv \phi\two _s  }_{\Leb^2(\Hn)} \\
&\lesssim_{\bsym\phi} \frac{1}{r^{2M}} \bignorm{ \VLap b_R\Vconv\phi\two _s }_{\Leb^2(\Hn)} .
\end{aligned}
\]
The parameter $M$ does not matter in this calculation.

From these considerations, \eqref{eq:basic-calculation}, Littlewood--Paley theory (as in Section \ref{ssec:spectral-theory}) and Lemma \ref{lem:a-is-enough},
\begin{equation}\label{eq:S3-estimate}
\begin{aligned}
&\int_{(S^*)^c} \areaopphi^3(a_R)(g) \wrt g \\
&\qquad\lesssim_{\bsym\phi} \sum_{j=j^*}^{\infty} \abs|E_{j+1}|^{1/2}
\lpar \int_{2^j}^{\infty} \int_{0}^{r^2 + h}
\norm{ \VLap b_R\Vconv\phi\two _s }_{\Leb^2(\Hn)}^2
\,\frac{ds}{s} \,\frac{dr}{r^{4M+1}} \rpar^{1/2} \\
&\qquad\lesssim_{\bsym\phi} \sum_{j=j^*}^{\infty} \abs|E_{j+1}|^{1/2}
\lpar \int_{2^{j}}^{\infty}
 \norm{\VLap b_R }_{\Leb^2(\Hn)}^2
 \,\frac{dr}{r^{4M+1}} \rpar^{1/2} \\
&\qquad\eqsim_{\bsym\phi} \sum_{j=j^*}^{\infty} (2^{\hdim j/2} + 2^{\cdim j} h^{1/2} ) 2^{-2Mj}
 \norm{\VLap b_R }_{\Leb^2(\Hn)}  \\
&\qquad\lesssim_{M}   (2^{j^*(\hdim - 4M)/2} + 2^{j^*(\cdim - 2M)} h^{1/2})
\norm{a_R }_{\Leb^2(\Hn)} \\
&\qquad\lesssim \Bigl( \frac{1}{r^*} \Bigr)^{2M-\hdim/2}
\abs|R|^{1/2} \norm{a_R }_{\Leb^2(\Hn)}  .
\end{aligned}
\end{equation}
We require that $4M > \hdim$, whence $2M-\hdim/2 \geq 1$, to ensure that the sum converges.

Finally, we use a pointwise estimate to control the term $\areaop[,\bsym\phi]^{4}(a_R)$, that is,
\begin{equation}\label{eq:S-star-5}
 \biggl( \int_{1}^\infty \int_{h}^{\infty}
\bigl| a_R \Hconv \phi_{r,s}  \bigr|^2 \Hconv \chi_{r,s}
\,\frac{ds}{s} \,\frac{dr}{r} \biggr)^{1/2} .
\end{equation}
From \eqref{eq:Lusin-support}, the definition of $\chi_{r,s}$, and the Cauchy--Schwarz inequality,
\begin{equation}\label{eq:S5-est1}
\begin{aligned}
\abs|a_R \Hconv \phi_{r,s} |^2  \Hconv \chi_{r,s}
&\leq \bignorm{\abs|a_R \Hconv \phi_{r,s} |^2  \Hconv \chi_{r,s} }_{\Leb^\infty(\Hn)} \indifn_{T(o, 2r+1, 2s+h) } \\
&\leq \bignorm{\abs|a_R \Hconv \phi_{r,s} |^2  }_{\Leb^\infty(\Hn)}
\indifn_{T(o, 2r+1, 2s+h) }  \\
&= \bignorm{ b_R \Hconv \riHLap^M \VLap \phi_{r,s} }_{\Leb^\infty(\Hn)} ^2 \indifn_{T(o, 2r+1, 2s+h) } \\
&\leq \bignorm{ b_R }_{\Leb^1(\Hn)} ^2
\bignorm{ \riHLap^M \VLap \phi_{r,s} }_{\Leb^\infty(\Hn)} ^2
\indifn_{T(o, 2r+1, 2s+h) } .
\end{aligned}
\end{equation}
From Hölder's inequality and Lemma~\ref{lem:a-is-enough}, $\norm{b_R}_{\Leb^1(\Hn)} \leq |R^*|^{1/2} \norm{b_R}_{\Leb^2(\Hn)} \lesssim h^{5/2} \norm{a_R}_{\Leb^2(\Hn)}$.
Moreover,
\begin{align*}
\riHLap^M \VLap \phi_{r,s} (z,t)
&= \riHLap^M [\phi\one_{r}] \Hconv \VLap [\phi\two _{s}] (z,t) \\
&= \frac{1}{r^{2M}} \frac{1}{s^{2}}
    [\riHLap^M \phi\one ]_{r} \Hconv [\VLap \phi\two ]_{s} (z,t) \\
&= \frac{1}{r^{2M}} \frac{1}{s^{2}} \int_{\R}
    [\riHLap^M \phi\one ]_{r}(z,t') [\VLap \phi\two ]_{s}(t - t') \wrt t' ;
\end{align*}
the integrand vanishes off an interval of length $O(r^2)$ because $[\phi\one ]_r$ is supported in $B\one(o,r)$, so
\begin{equation}\label{eq:S5-est2}
\begin{aligned}
\bignorm{ \riHLap^M \VLap \phi_{r,s} }_{\Leb^\infty(\Hn)}
&\lesssim r^{2-2M} \frac{1}{s^{2}}
\bignorm{[\riHLap^M \phi\one]_{r}}_{\Leb^\infty(\Hn)}
\bignorm{(\VLap \phi\two )_{s}}_{\Leb^\infty(\R)} \\
&=  r^{2-2M-\hdim} s^{-3} \bignorm{\riHLap^M \phi\one }_{\Leb^\infty(\Hn)}
\bignorm{\VLap \phi\two }_{\Leb^\infty(\R)} .
\end{aligned}
\end{equation}
Combining the estimates from \eqref{eq:S5-est1} to \eqref{eq:S5-est2}, we deduce that
\begin{equation}\label{eq:S5-est3}
 \abs|a_R \Hconv \phi_{r,s} |^2  \Hconv \chi_{r,s}
\lesssim_{\bsym\phi} r^{-4M-4\cdim} s^{-6} h^{5} \norm{ a_R }_{\Leb^2(\Hn)} ^2
\indifn_{T(o, 2r+1, 2s+1) } .
\end{equation}

It follows that
\begin{align*}
&\left( \int_{1}^\infty \int_{h}^{\infty}
\bigl| a_R \Hconv \phi_{r,s}  \bigr|^2 \Hconv \chi_{r,s}
\,\frac{ds}{s} \,\frac{dr}{r} \right)^{1/2} \\
&\qquad\leq\bigglpar \sum_{i,j \in \N}\int_{\alpha^i}^{\alpha^{i+1}}  \int_{\alpha^{2j}h}^{\alpha^{2(j+1)}h}
\norm{\abs|a_R \Hconv \phi_{r,s} |^2  \Hconv \chi_{r,s} }_{\Leb^\infty(\Hn)} \\
&\qquad\qquad \times \indifn_{T(o, 2r+1, 2s+h) } \,\frac{ds}{s} \,\frac{dr}{r} \biggrpar^{1/2} \\
&\qquad\leq h^{-1/2} \norm{ a_R }_{\Leb^2(\Hn)} \bigglpar \sum_{i,j \in \N}
\alpha^{-4(M+\cdim)i - 12j}
\indifn_{T(o, 3 \alpha^{i+1}, 3 \alpha^{2(j+1)}h) } \biggrpar^{1/2} \\
&\qquad\leq h^{-1/2} \norm{ a_R }_{\Leb^2(\Hn)} \sum_{i,j \in \N}
\alpha^{-2(M+\cdim)i-6j}
\indifn_{T(o, 3 \alpha^{i+1}, 3 \alpha^{2(j+1)}h) }  .
\end{align*}
Clearly $\abs|T(o,3\alpha^{i+1}, 3\alpha^{2(j+1)}h)| \eqsim \alpha^{2\cdim i} (\alpha^{2i} + \alpha^{2j}h)$.
If $(z,t) \in T(o,3\alpha^{i+1}, 3\alpha^{2(j+1)}h)$, where $r \geq 1$ and $s \geq h$, then $|z| \leq 3\alpha^{i+1}$ and $|t| \leq 9\alpha^{2(i+1)} + 3\alpha^{2(j+1)}h$;
we write $I(r^*, s^*)$ for the set of all $(i,j) \in \N^2$ for which $r^* \leq 3\alpha^{i+1}$ or $s^* \leq 9\alpha^{2(i+1)} + 3\alpha^{2(j+1)}h$; this last inequality implies that $s^* \leq 18\alpha^{2(i+1)}$ or $s^* \leq 6\alpha^{2(j+1)}h$.
We conclude that if $(i,j) \in I(r^*, s^*)$, then $\alpha^i \geq r^*/3\alpha$ or $\alpha^{2i} \geq s^*/18\alpha^2$ or $\alpha^{2j} \geq s^*/6\alpha^2 h$.
Thus
\begin{equation}\label{eq:S4-estimate}
\begin{aligned}
&\int_{(S^{*})^c} \areaop[,\bsym\phi]^{4}(a_R)(g) \wrt g \\
&\qquad\lesssim_{\bsym\phi}
h^{-1/2} \norm{ a_R }_{\Leb^2(\Hn)}\sum_{(i,j) \in I(r^*, s^*)}
\frac{\alpha^{2i} + \alpha^{2j}h}{\alpha^{2Mi+6j} }    \\
&\qquad\lesssim_{M}
h^{-1/2} \norm{ a_R }_{\Leb^2(\Hn)}\biggl( \sum_{i: \alpha^i \geq r^*/3\alpha}
\frac{ \alpha^{2i} + h}{ \alpha^{2Mi}} \\
&\qquad\qquad +\sum_{i: \alpha^{2i} \geq s^*/18\alpha^2}
\frac{ \alpha^{2i} + h}{ \alpha^{2Mi}}
+\sum_{j: \alpha^{2j} \geq s^*/9\alpha^2}
\frac{1 + \alpha^{2j}h} {\alpha^{6j}}  \biggr)   \\
&\qquad\lesssim_{M}
h^{-1/2}\norm{ a_R }_{\Leb^2(\Hn)} \biggl(
 \frac{ (r^*)^2 + h }{(r^*)^{2M}}
+ \frac{ s^* + h }{(s^*)^M }
+ \frac{h}{(s^*)^2} \biggr)   \\
&\qquad\lesssim
\lpar \frac{1}{r^*} \rpar^{2(M-1)} \abs|R|^{1/2} \norm{ a_R }_{\Leb^2(\Hn)}
 + \lpar \frac{h}{ s^* }\rpar^{2} \abs|R|^{1/2} \norm{ a_R }_{\Leb^2(\Hn)}.
\end{aligned}
\end{equation}
We require that $M > 1$ to ensure that the sums above converge, whence $2(M-1) \geq 2$.

The quotients preceding $\abs|R|^{1/2} \norm{ a_R }_{\Leb^2(\Hn)}$ on the right hand sides of \eqref{eq:S2-estimate}, \eqref{eq:S3-estimate} and \eqref{eq:S4-estimate} are dominated by $\rho_{1,3/2}(R,S)$, as defined in \eqref{eq:def-rho}.
Further, examination of the argument shows that the implicit constants depend on $\bsym\phi$ through products such as $\bignorm{\riHLap^M \phi\one }_{\Leb^1(\Hn)} \bignorm{\VLap\phi\two }_{\Leb^\infty(\R)}$, which, because of the support restriction, may all be controlled by the product $\norm{\oper{R} \phi\one }_{(2M)} \norm{\phi\two }_{(2)}$ .
\end{proof}

Before the next corollary, we recall that $h$ and $p$ denote the heat and Poisson kernels.

\begin{corollary}
Suppose that $\phi\one$ and $\phi\two $ are either $\HLap h\one_1$ and $\VLap h\two _1$ or $\HLap p\one_1$ and $\VLap p\two _1$ on $\Hn$ and on $\R$, and that $\areaopphi(f)$ is as in \eqref{eq:def-Lusinpp}.
Then $\areaopphi(f) \in \Leb^1(\Hn)$ and
\[
\bignorm{ \areaopphi(f)} _{\Leb^1(\Hn)}
\lesssim \norm{f}_{\atomHardy(\Hn)}
\qquad\forall f \in \atomHardy(\Hn) .
\]
\end{corollary}

\begin{proof}
This result follows from the theorem above and the estimates for the heat and Poisson kernels in Section \ref{ssec:heat-Poisson}.
\end{proof}

\begin{corollary}\label{cor:max-area-fn}
Suppose that $\bsym\phi$ is a Poisson bounded pair of functions, both of which have mean $0$.
Then, if $\theta \in (0,1)$,
\[
\lpar \iint_{\R^+ \times \R^+}   \sup_{g \in T(o,\theta r, \theta s)} \abs|f \Hconv \phi_{r,s}(\cdot g)|^2 \,\frac{dr}{r} \,\frac{ds}{s} \rpar^{1/2} \in \Leb^1(\Hn)
\]
for all $f \in \atomHardy(\Hn)$,  and a corresponding norm inequality holds.
\end{corollary}

\begin{proof}
Define $\psi\one = \HLap^M \phi\one$ and $\psi\two = \VLap \phi\two$, where $M > \hdim/4$.
In light of Theorem \ref{thm:Atom-implies-area}, the area function $\areaopphi(f)$ is in $\Leb^1(\Hn)$, as are the area functions associated to the pairs $(\phi\one,\psi\two)$, $(\psi\one,\phi\two)$ and $(\psi\one,\psi\two)$.
Hence by the flag Sobolev inequalities of Lemma \ref{lem:L2-Linfty-Sobolev},
\[
\begin{aligned}
&\biggl(\iint_{\R^+ \times \R^+} \sup_{g' \in T(g,r,s)} \abs|f \Hconv \phi_{r,s}(g')|^2 \,\frac{dr}{r} \,\frac{ds}{s} \biggr)^{1/2} \\
&\qquad\lesssim \biggl(\iint_{\R^+ \times \R^+} \frac{1}{|T(o,r,s)|} \int_{T(g,r,s)}
   \abs|f \Hconv \phi\one_{r} \Vconv\phi\two_{s}(g')|^2 \\
&\qquad\qquad +  \abs|f \Hconv \phi\one_{r} \Vconv\psi\two_{s}(g')|^2  
 +  \abs|f \Hconv \psi\one_{r} \Vconv\phi\two_{s}(g')|^2 \\
&\qquad\qquad  +  \abs|f \Hconv \psi\one_{r} \Vconv\psi\two_{s}(g')|^2
       \,\frac{dr}{r} \,\frac{ds}{s} \biggr)^{1/2} ,
\end{aligned}
\]
and the right hand side function is in $\Leb^1(\Hn)$, as claimed.
\end{proof}

\subsection{The inclusion $\atomHardy(\Hn) \subseteq \ctssqfnHardy(\Hn)\cap \dissqfnHardy(\Hn)$}\label{ssec:atom-implies-sqfn}

Now we treat the square function Hardy spaces.
We begin with the definitions.

\begin{definition*}
Suppose that $\bsym\phi$ is a w-invertible, Poisson bounded pair of functions, both of which have mean $0$ and whose w-inverses have mean $0$.
For $f \in \Leb^1(\Hn)$, we define the \emph{continuous Littlewood--Paley square function $\ctssqfnopphi(f)$ associated to $\bsym\phi$} by
\begin{equation*}
\begin{aligned}
&\ctssqfnopphi (f)(g)
:= \biggl( \iint_{\R^{+} \times\R^{+}}
\bigl| f \Hconv \phi\one_r \Vconv \phi\two _s(g)  \bigr|^2
    \,\frac{dr}{r} \,\frac{ds}{s} \biggr)^{1/2} \\
\end{aligned}
\end{equation*}
for all $g \in \Hn$.
We define the continuous square function Hardy space $\ctssqfnHardyphi(\Hn)$, often written $\ctssqfnHardy(\Hn)$, to be the set of all $f\in \Leb^1(\Hn)$ for which $\norm{ \ctssqfnopphi (f) }_{\Leb^1(\Hn)}$ is finite, with  norm
\begin{equation*}
\begin{aligned}
\norm{ f }_{\ctssqfnHardyphi(\Hn)}&:=\norm{ \ctssqfnopphi (f) }_{\Leb^1(\Hn)} .
\end{aligned}
\end{equation*}
\end{definition*}

\begin{definition*}
Suppose that $\phi\one$ and $\phi\two $ are as above.
For $f \in \Leb^1(\Hn)$, we define the \emph{discrete Littlewood--Paley square function} $\dissqfnopphi(f)$ associated to $\phi\one$ and $\phi\two $ by
\begin{equation*}
\begin{aligned}
&\dissqfnopphi (f)(g)
:= \biggl( \sum_{(m,n) \in \Z \times \Z}
\bigl| f \Hconv \phi\one_{2^m} \Vconv \phi\two _{2^n}(g)  \bigr|^2
    \biggr)^{1/2}
\end{aligned}
\end{equation*}
for all $g \in \Hn$.
We define the \emph{discrete square function Hardy space} $\dissqfnHardyphi(\Hn)$, often written $\dissqfnHardy(\Hn)$, to be the set of all $f\in \Leb^1(\Hn)$ for which $\norm{ \dissqfnopphi (f) }_{\Leb^1(\Hn)}$ is finite, with  norm
\begin{equation*}
\begin{aligned}
\norm{ f }_{\dissqfnHardyphi(\Hn)}
&:=\norm{ \dissqfnopphi (f) }_{\Leb^1(\Hn)}.
\end{aligned}
\end{equation*}
\end{definition*}

\begin{theorem}
Suppose that the functions $\phi\one$ and $\phi\two $ on $\Hn$ and $\R$ are as above.
Then the operators $\ctssqfnopphi$ and $\dissqfnopphi$ are bounded from $\atomHardy(\Hn)$ to $\Leb^1(\Hn)$.
\end{theorem}

\begin{proof}
From Corollary \ref{cor:max-area-fn}, $\ctssqfnopphi(f) \in \Leb^1(\Hn)$ and $\norm{ \ctssqfnopphi(f) }_{\Leb^1(\Hn)} \lesssim \norm{ f } _{\atomHardy(\Hn)}$.
By Lemma \ref{lem:ctssqfn-dominates-dissqfn}, $\dissqfnopphi(f) \in \Leb^1(\Hn)$ and $\norm{ \dissqfnopphi(f) }_{\Leb^1(\Hn)} \lesssim \norm{ f } _{\atomHardy(\Hn)}$.
\end{proof}

We may also prove this theorem directly using simplified versions of the proof of Theorem \ref{thm:Atom-implies-area} and Corollary \ref{cor:Poisson decomposition} to ensure that $\phi\one$ and $\phi\two$ have sufficiently many vanishing moments.

\subsection{The inclusion $\atomHardy(\Hn) \subseteq \gmaxHardy(\Hn)$}\label{ssec:atom-implies-gmax}
We recall the key points of Definition \ref{def:Poisson-bounds}.
A family $\family$ of pairs  of functions $(\phi\one , \phi\two )$,
where $\phi\one :\Hn \to \C$ and $\phi\two :\R \to \C$, is said to be \emph{Poisson-bounded} if there are constants $C_m$ and $C_n$ for all $m, n \in \N$ such that
\begin{equation*}\label{eq:poisson-like-decay-1a}
\begin{aligned}
\abs| \oper{D} \phi\one(g) |
&\leq \frac{ C_m }{ (1 + \norm{g})^{m + \hdim + 1}}
\qquad\forall g \in \Hn
\end{aligned}
\end{equation*}
for all differential operators $\oper{D}$ that are products of $m$ vector fields, each chosen from $\oper{X}_1, \dots, \oper{X}_{2\cdim}$, and
\begin{equation*}\label{eq:poisson-like-decay-2a}
\begin{aligned}
\abs| \oper{T}^n \phi\two (t) |
&\leq  \frac{ C_n }{(1 + \norm{t})^{n+2}}
\qquad\forall t \in \R
\end{aligned}
\end{equation*}
for all pairs $(\phi\one ,\phi\two ) \in \family$.

\begin{definition}\label{def:gmax-Hardy}
The  \emph{grand maximal function} $\gmaxop(f)$ of $f \in \Leb^1(\Hn)$, for a Poisson bounded family $\family$, is defined by
\begin{align*}
\gmaxop(f)(g)
:= \sup_{\bsym\phi \in \family} \sup_{r,s \in \R^+}
\abs|f \Hconv \phi_{r,s} (g)|
\qquad\forall g \in \Hn.
\end{align*}
The space $\gmaxHardy(\Hn)$ is defined to be the linear space of all $f\in \Leb^1(\Hn)$ such that $\gmaxop (f) \in \Leb^1(\Hn)$, with norm
\begin{equation}\label{eq:gmax-Hardy}
\norm{ f }_{\gmaxHardy(\Hn)}:=\norm{ \gmaxop (f) }_{\Leb^1(\Hn)}.
\end{equation}
\end{definition}

\begin{theorem}\label{thm:Atom-implies-gmax}
Suppose that $\family$ is Poisson bounded.
Then $\atomHardy(\Hn) \subseteq \gmaxHardyfam(\Hn)$, and
\[
\norm{f}_{\gmaxHardyfam(\Hn)}
\lesssim_{\family} \norm{f}_{\atomHardy(\Hn)}
\qquad\forall f \in \atomHardy(\Hn).
\]
\end{theorem}

\begin{proof}
The proof resembles that of Theorem \ref{thm:Atom-implies-area}.
Let $M$ be the smallest integer such that $2M > \hdim/2$.

By Remark \ref{rem:compact-support-enough} (or the preceding result), it is enough to consider a Poisson bounded family $\familyc$ of functions  $\phi\one$ and $\phi\two $ that are supported in the unit balls of $\Hn$ and $\R$.
Next, by Proposition \ref{prop:particle-bounded-function-bounded}, it suffices to show that
\begin{equation}\label{eq:gmax-on-particles}
\begin{aligned}
&\int_{(S^{*})^c} \gmaxopc(a_R)(g) \wrt g \\
&\qquad\lesssim \sup_{\bsym\phi \in \familyc} \bignorm{\oper{R}\phi\one }_{(2M)} \bignorm{\phi\two }_{(2)}
\rho_{1,3/2}(R,S)\abs|R|^{1/2} \norm{a_R}_{\Leb^2(\Hn)}
\end{aligned}
\end{equation}
for all particles $a_R$ associated to a shard $R$ such that $\cent(R) = o$, $\wid(R) =1$ and $\heit(R) = h$ and all shards $S$ such that $\cent(S) = o$, $\wid(S) =r^* > 1$ and $\heit(S) = s^* > h$, where $\rho_{1,3/2}$ is as defined in \eqref{eq:def-rho}.

From the definition just before the statement of this theorem,
\[
\gmaxopc(a_R)
\eqsim \max_{j=1, \dots,4} \gmaxopc^{j}(a_R) ,
\]
where
\[
\gmaxopc^{j}(a_R)(g) =
\sup_{\bsym\phi \in \familyc} \sup_{(r,s) \in \Omega_j}
\abs|a_R \Hconv \phi_{r,s} (g)| ,
\]
and the regions $\Omega^j$, which partition $\R^+ \times \R^+$ into four, are defined by
\[
\begin{aligned}
\Omega_1 &= \{ (r,s) \in \R^+ \times \R^+ : r < 1, s < h \} \qquad&
\Omega_2 &= \{ (r,s) \in \R^+ \times \R^+ : r < 1, s > h \} \\
\Omega_3 &= \{ (r,s) \in \R^+ \times \R^+ : r > 1, s < h\} &
\Omega_4 &= \{ (r,s) \in \R^+ \times \R^+ : r > 1, s > h\}.
\end{aligned}
\]
We treat these four summands separately.
The key to our estimation is the fact that for all $\bsym\phi$ in $\familyc$,
\begin{equation}\label{eq:gmax-support}
\begin{aligned}
&\supp \bigl( a_R \Hconv \phi_{r,s}\bigr)
\subseteq
T(o,1, h)  T(o, r, s) = T(o, r+1, s+h ) .
\end{aligned}
\end{equation}

First, \eqref{eq:gmax-support} implies that
$\supp (\gmaxopc^{1}(a_R)) \cap (S^*)^c = \emptyset$ and so no estimation is needed.

Second, to treat $\gmaxopc^{2}(a_R)$ and $\gmaxopc^{3}(a_R)$, we first choose an exponentially increasing family of sets $E_j$ such that, from \eqref{eq:gmax-support},
\[
\supp(\gmaxopc^{k}(a_R))
\subseteq \bigcup_{(r,s) \in \Omega_k} T(o,1+r,h+s)
\subseteq \bigcup_{j\in\N} E_j.
\]
We then take $j^*$ to be the smallest $j$ such that $E_j \cap (S^*)^c \neq \emptyset$ and $\Omega_{k,j}$ to be a subset of $\Omega_{k}$ such that $(r,s) \notin \Omega_{k.j}$ implies that $T(o,1+r,h+s) \subseteq E_{j}$.
Then, by the definitions and basic results,
\begin{equation}\label{eq:basic-calculation-gm}
\begin{aligned}
&\int_{(S^*)^c} |\gmaxopc^{k}(a_R)(g)| \wrt g \\
&\qquad= \sum_{j \geq j^*} \int_{E_{j+1}\setminus E_j} |\gmaxopc^{k}(a_R))(g)| \wrt g \\
&\qquad\leq \sum_{j \geq j^*} \abs|E_{j+1}\setminus E_j|^{1/2}
\lpar \int_{(E_j)^c} |\gmaxopc^{k}(a_R))(g)|^2 \wrt g \rpar^{1/2} \\
&\qquad= \sum_{j \geq j^*} \abs|E_{j+1}\setminus E_j|^{1/2}
\lpar   \int_{(E_j)^c} \sup_{\bsym\phi \in \familyc} \sup_{(r,s) \in \Omega_k}
\abs|a_R \Hconv \phi_{r,s}(g)|^2
\wrt g  \rpar^{1/2} \\
&\qquad= \sum_{j \geq j^*} \abs|E_{j+1}\setminus E_j|^{1/2}
\lpar \int_{(E_j)^c} \sup_{\bsym\phi \in \familyc} \sup_{(r,s) \in \Omega_{k,j}}
\abs|a_R \Hconv \phi_{r,s}(g)|^2 \wrt g \rpar^{1/2} \\
&\qquad\leq \sum_{j \geq j^*} \abs|E_{j+1}|^{1/2}
\lpar \int_{\Hn} \sup_{\bsym\phi \in \familyc} \sup_{(r,s) \in \Omega_{k,j}}
\abs|a_R \Hconv \phi_{r,s}(g)|^2
\wrt g  \rpar^{1/2} .
\end{aligned}
\end{equation}

To treat $\gmaxopc^{2}(a_R)$, we define $E_j := T(o, 2, 2 \times 2^{j}h)$ and $\Omega_{k,j} := (0,1) \times (2^jh, \infty)$ when $j \in \N$; then $\abs| E_j | \eqsim 2^jh$ and $h^* \eqsim \log_2(s^*)$.
Moreover,
\[
\begin{aligned}
a_R \Hconv \phi_{r,s}
&= s^{-2} \HLap^M b_R \Hconv \phi\one_r \Vconv (\VLap \phi\two )_s ;
\end{aligned}
\]
see \eqref{eq:S2-function}.
We define the family ${\familyo}$ to be $\{(\phi\one, \VLap\phi\two) : (\phi\one, \phi\two) \in \familyc\}$.
Then
\begin{equation}\label{eq:gmax-2-est}
\begin{aligned}
&\sup_{\bsym\phi \in \familyc} \sup_{(r,s) \in \Omega_{k,j}}
\abs|a_R \Hconv \phi_{r,s}| \\
&\qquad\leq  (2^jh)^{-2}\sup_{\bsym\phi \in \familyc} \sup_{(r,s) \in \Omega_{k,j}}
\abs| \HLap^M b_R \Hconv \phi\one_r \Vconv (\VLap \phi\two )_s| \\
&\qquad\lesssim  (2^jh)^{-2}
\gmaxopo(\HLap^M b_R) .
\end{aligned}
\end{equation}
From \eqref{eq:basic-calculation-gm}, \eqref{eq:gmax-2-est},  and Lemmas \ref{lem:poisson-and-flag-maximal-functions}, \ref{lem:flag-and-iterated-maximal-fns}, and \ref{lem:a-is-enough},
\begin{equation}\label{eq:GM-2-est}
\begin{aligned}
&\int_{(S^*)^c} |\gmaxop^{k}(a_R)(g)| \wrt g \\
&\qquad\leq \sum_{j \geq j^*} \abs|E_{j+1}|^{1/2}
\lpar \int_{\Hn} \sup_{\bsym\phi \in \familyc} \sup_{(r,s) \in \Omega_{k,j}}
\abs|a_R \Hconv \phi_{r,s}(g)|^2
\wrt g  \rpar^{1/2} \\
&\qquad\lesssim \sum_{j \geq j^*} \abs|E_{j+1}|^{1/2} (2^jh)^{-2}
\lpar  \int_{\Hn} \abs| \gmaxopo(\HLap^M b_R)(g)|^2 \wrt g  \rpar^{1/2} \\
&\qquad\lesssim \sum_{j \geq j^*} (2^jh)^{-3/2}
\norm{  \oper{M}_{\flag}(\HLap^M b_R) }_{\Leb^2(\Hn)} \\
&\qquad\lesssim 2^{-3j^*/2} h^{-3/2}
\norm{  \HLap^M b_R }_{\Leb^2(\Hn)} \\
&\qquad\lesssim \Bigl(\frac{h}{s^*}\Bigr)^{3/2} \abs|R|^{1/2} \norm{a_R }_{\Leb^2(\Hn)}^2 .
\end{aligned}
\end{equation}

To deal with $\gmaxopc^{3}(a_R)$, we define $E_j := T(o,1+2r, 2h)$ and $\Omega_{3,j} := (2^j,\infty ) \times (0,h)$ when $j \in \N$; then $\abs| E_j| \lesssim 2^{\hdim j} + 2^{2\cdim j} h$ and $j^* \eqsim \log_2(r^*)$.
Further, from \eqref{eq:S3-function},
\[
\begin{aligned}
a_R \Hconv \phi_{r,s}
&= \frac{1}{r^{2M}} \VLap b_R \Hconv [\riHLap^M \phi\one]_r \Vconv \phi\two _s;
\end{aligned}
\]
we let $\familyt$ be the family  $\{ (\riHLap^M\phi\one, \phi\two) :   (\phi\one, \phi\two) \in \familyz \}$.
Much as for the proof of \eqref{eq:GM-2-est},
\begin{equation}\label{eq:GM-3-est}
\begin{aligned}
&\int_{(S^*)^c} |\gmaxopc^{3}(a_R)(g)| \wrt g \\
&\qquad\lesssim \sum_{j \geq j^*} \abs|E_{j+1}|^{1/2} 2^{-2Mj}
\lpar  \int_{\Hn} \abs| \gmaxopt(\VLap b_R)(g)|^2 \wrt g  \rpar^{1/2} \\
&\qquad\lesssim \sum_{j \geq j^*} (2^{\hdim j} + 2^{2\cdim j} h)^{1/2} 2^{-2Mj}
\norm{  \oper{M}_{\flag}(\VLap b_R) }_{\Leb^2(\Hn)} \\
&\qquad\lesssim \Bigl(\frac{1}{r^*}\Bigr)^{2M - \hdim/2}
\abs|R|^{1/2} \norm{a_R }_{\Leb^2(\Hn)}^2 .
\end{aligned}
\end{equation}

Finally, to treat $\gmaxop^{4}(a_R)$, we use the pointwise estimate
\begin{equation*}
 \abs|a_R \Hconv \phi_{r,s} |
\lesssim_{\bsym\phi} r^{-2M-2\cdim} s^{-3} h^{5/2} \norm{ a_R }_{\Leb^2(\Hn)}
\indifn_{T(o, r+1, s+1) } ,
\end{equation*}
which is proved in a very similar but easier way to \eqref{eq:S5-est3}.

It follows that
\begin{align*}
&\sup_{(r,s) \in \Omega_4} \bigl| a_R \Hconv \phi_{r,s}  \bigr| \\
&\qquad\leq  \sum_{i,j \in \N}\sup_{\alpha^i < r < \alpha^{i+1}}  \sup_{\alpha^{2j}h < s < \alpha^{2(j+1)}h}
\norm{ a_R \Hconv \phi_{r,s} }_{\Leb^\infty(\Hn)} \indifn_{T(o, 2r+1, 2s+h) }
 \\
&\qquad\leq h^{-1/2} \norm{ a_R }_{\Leb^2(\Hn)} \sum_{i,j \in \N}
\alpha^{-2(M+\cdim)i-6j}
\indifn_{T(o, 3 \alpha^{i+1}, 3 \alpha^{2(j+1)}h) }  .
\end{align*}
In the proof of Theorem \ref{thm:Atom-implies-area}, we showed that this last function is in $\Leb^1(\Hn)$, and estimated its norm.
From this it follows that
\begin{equation}\label{eq:GM4-estimate}
\begin{aligned}
&\int_{(S^{*})^c} \gmaxop^4(a_R)(g) \wrt g \\
&\qquad\lesssim
\lpar \frac{1}{r^*} \rpar^{2(M-1)} \abs|R|^{1/2} \norm{ a_R }_{\Leb^2(\Hn)}
 + \lpar \frac{h}{ s^* }\rpar^{2} \abs|R|^{1/2} \norm{ a_R }_{\Leb^2(\Hn)}.
\end{aligned}
\end{equation}

The quotients that precede $|R|^{1/2}$ on the right hand sides of our estimates are all dominated by multiples of $\rho_{1,3/2}(R,S)$, as defined in \eqref{eq:def-rho}.
A careful examination of the proof shows that $\bsym\phi$ comes into the constants in expressions involving norms such as
$\bignorm{\riHLap^M\phi\one }_{\Leb^\infty(\Hn)}\bignorm{\VLap\phi\two }_{\Leb^\infty(\R)}$ or $\bignorm{\riHLap^M\phi\one }_{\Leb^\infty(\Hn)}\bignorm{\phi\two }_{\Leb^\infty(\R)}$,
all of which are dominated by
$\bignorm{\oper{R} \phi\one }_{(2M)}\bignorm{ \phi\two }_{(2)}$
because of the support restriction.
\end{proof}

\subsection{The inclusion $\atomHardy(\Hn) \subseteq \siphiHardy(\Hn)$}\label{ssec:atom-implies-sing-ints}

In this section, we consider a homogeneous flag singular integral operator $\oper{K}$ acting on the atomic Hardy space.

\begin{theorem}\label{thm:singular-integrals-bounded-on-atoms}
Suppose that $\phi\one$ and $\phi\two$ are Poisson bounded on $\Hn$ and on $\R$ and have mean $0$.
Then $\atomHardy(\Hn) \subseteq \siphiHardy(\Hn)$, and
\[
\norm{\oper{K}_{\boldsymbol\phi} f}_{\Leb^1(\Hn)}
\lesssim \norm{f}_{\atomHardy(\Hn)}
\qquad\forall f \in \atomHardy(\Hn).
\]

\end{theorem}

\begin{proof}
From Remark \ref{rem:compact-support-enough}, we may suppose that $\supp(\phi\one) \subseteq B\one(o,1)$ and $\supp(\phi\two) \subseteq B\two(0,1)$.
The kernel $k$ of $\oper{K}_{\bsym\phi}$ is given by
\[
k = \iint_{\R^+ \times \R^+} \phi\one_s \Vconv \phi\two_t \,\frac{ds}{s} \,\frac{dt}{t} \,.
\]

By Proposition \ref{prop:particle-bounded-function-bounded}, it suffices to study $\oper{K}_{\boldsymbol\phi}(a_R)$, where $a_R$ is a particle, and to show that 
\[
\int_{(S^*)^c} \abs| a_r \Hconv k (g) | \,dg 
\lesssim_{\bsym\phi} \rho_{2,2}(R,S) \abs|R|^{1/2} \norm{a_R}_{\Leb^2(G)}
\]
for all shards $S$ that contain $R$;  here $\rho_{2,2}$ is as defined in \eqref{eq:def-rho}.

We suppose that $a_R = \HLap \VLap b_R$, where $b_R$ is supported in the shard $R$, and, by translation and dilation invariance, that $R$ is the shard with centre $o$, width $1$ and height $h  > 1$; we may also suppose that $\cent(S) = o$, $\wid(S) =w^* > 1$ and $\heit(S) = h^* > h$.

Again, we define four regions $\Omega^j$, which partition $\R^+ \times \R^+$,  by
\[
\begin{aligned}
\Omega_1 &= \{ (r,s) \in \R^+ \times \R^+ : r \leq 1, s \leq h \} \qquad&
\Omega_2 &= \{ (r,s) \in \R^+ \times \R^+ : r \leq 1, s > h \} \\
\Omega_3 &= \{ (r,s) \in \R^+ \times \R^+ : r > 1, s \leq h\} &
\Omega_4 &= \{ (r,s) \in \R^+ \times \R^+ : r > 1, s > h\},
\end{aligned}
\]
and we write $a_R \Hconv k = a_R \Hconv k^1 + a_R \Hconv k^2 + a_R \Hconv k^2 + a_R \Hconv k^4$, where 
\[
k^i := \iint_{\Omega^i} \phi\one_s \Vconv \phi\two_t \,\frac{ds}{s} \,\frac{dt}{t} \,.
\]
We treat the four summands $a_R \Hconv k^i$ separately.
The key to our estimation is the fact that
\begin{equation}\label{eq:sing-int--support}
\begin{aligned}
&\supp \bigl( a_R \Hconv \phi_{r,s}\bigr)
\subseteq
T(o,1, h)  T(o, r, s) = T(o, r+1, s+h ) .
\end{aligned}
\end{equation}

First, by definition, $\supp (a_R \Hconv k^1) \subseteq S$, and there is nothing to consider; in any case, $a_R \Hconv k^1 \in \Leb^1(G)$.

Second, to treat $a_R \Hconv k^2$, we observe that if $a_R \Hconv \phi\one_s \Vconv \phi\two_t(g)  \neq 0$, where $g \in (S)^c$ and $(s,t) \in \Omega^2$, then $h^* \leq  h + s + 1 \leq h+2$, and so
\[
\begin{aligned}
\int_{(S)^c} \abs| a_R \Hconv k^2 (g)| \wrt g 
&= \int_{(S)^c} \abs| \int_{h^*/2}^{\infty} \int_{0}^{1} a_R \Hconv \phi\one_s \Vconv \phi\two_t  \,\frac{ds}{s} \,\frac{dt}{t} | \wrt g .
\end{aligned}
\]

Recall that $a_R = \VLap \HLap b_R$, and write, for $g \in (S)^c$,
\[
\begin{aligned}
a_R \Hconv k^2(g) 
&=  \int_{h^*/2}^{\infty} \int_{0}^{1} (\VLap \HLap b_R) \Hconv \phi\one_s \Vconv \phi\two_t \,\frac{ds}{s} \,\frac{dt}{t} \\
&=  \int_{h^*/2}^{\infty} \lpar \int_{0}^{1} (\HLap b_R) \Hconv \phi\one_s \,\frac{ds}{s}  \rpar \Vconv (\VLap\phi\two)_t \,\frac{dt}{t^3} \\
&=  \int_{h^*/2}^{\infty} c^2_R \Vconv (\VLap\phi\two)_t \,\frac{dt}{t^3}, 
\end{aligned}
\]
say, where 
\[
c^2_R = \int_0^1 (\HLap b_R) \Hconv \phi\one_s \,\frac{ds}{s}.
\]
Now $b_R$ is supported in $R$, so $c^2_R$ is supported in $R^*$, say, where $\abs|R^*| \eqsim \abs|R|$.
Moreover,   by Hölder's inequality and Lemmas \ref{lem:homogeneous-integral} and  \ref{lem:flag-Sobolev},
\[
\norm{ c^2_R }_{\Leb^1(\Hn)} 
\lesssim |R|^{1/2}  \norm{ c^2_R }_{\Leb^2(\Hn)} 
\lesssim_{\phi\one} |R|^{1/2} \norm{ \HLap b_R }_{\Leb^2(\Hn)} 
\lesssim h^2 |R|^{1/2}  \norm{ a_R }_{\Leb^2(\Hn)} .
\]
We conclude that
\[
\begin{aligned}
\int_{(S)^c} \abs| a_R \Hconv k^2 (g)| \wrt g 
&= \int_{(S)^c} \abs| \int_{h^*/2}^{\infty}  c^2_R \Vconv (\VLap\phi\two)_t(g)  \,\frac{dt}{t^3} | \wrt g \\
&\leq \int_{h^*/2}^{\infty}  \int_{\Hn } \abs|  c^2_R \Vconv (\VLap\phi\two)_t (g)| \wrt g \,\frac{dt}{t^3}  \\
&\leq \int_{h^*/2}^{\infty}  \norm{ c^2_R }_{\Leb^1(\Hn)} \norm{ (\VLap\phi\two)_t  }_{\Leb^1(\R)} \,\frac{dt}{t^3}  \\
& \eqsim_{\phi\two} \frac{1}{(h^*)^2}  \norm{ c^2_R }_{\Leb^1(\Hn)}  \\
& \lesssim_{\phi\one} \frac{h^2}{(h^*)^2}  \abs|R|^{1/2} \norm{ a_R }_{\Leb^2(\Hn)}  .
\end{aligned}
\]

Third, to treat $a_R \Hconv k^3$, we argue similarly.
First, if $a_R \Hconv \phi\one_s \Vconv \phi\two_t(g)  \neq 0$, where $g \in (S)^c$ and $(s,t) \in \Omega^3$, then $w^* \leq s+1$, whence
\[
\begin{aligned}
\int_{(S)^c} \abs| a_R \Hconv k^3 (g)| \wrt g 
&= \int_{(S)^c} \abs| \int_{w^*/2}^{\infty} \int_{h}^{\infty} a_R \Hconv \phi\one_s \Vconv \phi\two_t   \,\frac{dt}{t} \,\frac{ds}{s} | \wrt g .
\end{aligned}
\]
The next step is to write, for $g \in (S)^c$,
\[
\begin{aligned}
a_R \Hconv k^3(g) 
&=  \int_{w^*/2}^{\infty} \int_{0}^{h} (\HLap \VLap b_R) \Vconv \phi\two_t \Hconv \phi\one_s(g) \,\frac{dt}{t} \,\frac{ds}{s} \\
&=  \int_{w^*/2}^{\infty} c^3_R \Hconv (\HLap\phi\one)_s(g) \,\frac{ds}{s^3}, 
\end{aligned}
\]
say, where 
\[
c^3_R = \int_0^h (\VLap b_R) \Vconv \phi\two_t \,\frac{dt}{t} \,.
\]
Finally, 
\[
\begin{aligned}
\int_{(S)^c} \abs| a_R \Hconv k^3 (g)| \wrt g 
&= \int_{(S)^c} \abs| \int_{w^*/2}^{\infty}  c^3_R \Hconv (\riHLap\phi\one)_s(g)  \,\frac{ds}{s^3} | \wrt g \\
&\leq \int_{w^*/2}^{\infty}  \int_{\Hn } \abs|  c^3_R \Hconv (\riHLap\phi\one)_s(g) | \wrt g \,\frac{ds}{s^3}  \\
& \eqsim_{\phi\one} \frac{1}{(w^*)^2}  \norm{ c^3_R }_{\Leb^1(\Hn)}  \\
& \lesssim_{\phi\two} \frac{1}{(w^*)^2}  \abs|R|^{1/2} \norm{ a_R }_{\Leb^2(\Hn)}  .
\end{aligned}
\]

Fourth, we treat $a_R \Hconv k^4$.
If $a_R \Hconv \phi\one_s \Vconv \phi\two_t(g)  \neq 0$, where $g \in (S)^c$ and $(s,t) \in \Omega^4$, then either $s+1 \geq  w^*$, or $s+1 < w^*$ and $t + s^2 + h \geq h^*$, whence
\[
\begin{aligned}
\int_{(S)^c} \abs| a_R \Hconv k^4 (g)| \wrt g 
&\leq \int_{(S)^c} \int_{w^*/2}^{\infty} \int_{h}^{\infty} \abs|  a_R \Hconv \phi\one_s \Vconv \phi\two_t(g) |  \,\frac{dt}{t} \,\frac{ds}{s} \wrt g \\
&\qquad+  \int_{(S)^c} \int_{1}^ {w^*} \int_{h^*}^{\infty} \abs| a_R \Hconv \phi\one_s \Vconv \phi\two_t(g) |   \,\frac{dt}{t} \,\frac{ds}{s} \wrt g .
\end{aligned}
\]
Next, 
\[
\begin{aligned}
a_R \Hconv \phi\one_s \Vconv \phi\two_t 
&=  \frac{1}{s^2t^2} b_R \Hconv  (\riHLap\phi\one)_s \Vconv (\VLap\phi\two)_t .
\end{aligned}
\]
Finally, estimating much as before, we see that
\begin{align*}
&\int_{(S)^c} \int_{w^*/2}^{\infty} \int_{h}^{\infty} \abs| a_R \Hconv \phi\one_s \Vconv \phi\two_t(g) |  \,\frac{dt}{t} \,\frac{ds}{s} \wrt g \\
&\qquad\leq \int_{w^*/2}^{\infty} \int_{h}^{\infty} \int_{\Hn} \abs| b_R \Hconv (\HLap\phi\one)_s \Vconv (\VLap\phi\two)_t (g) |  \wrt g \,\frac{dt}{t^3} \,\frac{ds}{s^3}   \\
&\qquad\leq \int_{w^*/2}^{\infty} \int_{h}^{\infty} \norm{ b_R }_{\Leb^1(\Hn)} \norm{ (\HLap\phi\one)_s }_{\Leb^1(\Hn)} \norm{ (\VLap\phi\two)_t }_{\Leb^1(\R)}  \,\frac{dt}{t^3} \,\frac{ds}{s^3}  \\
&\qquad \lesssim_{\bsym\phi} \frac{1}{h^2 (w^*)^2}  \norm{ b_R }_{\Leb^1(\Hn)}  \\
&\qquad \lesssim \frac{1}{(w^*)^2}  \abs|R|^{1/2} \norm{ a_R }_{\Leb^2(\Hn)}  ,
\end{align*}
while
\[
\begin{aligned}
&\int_{(S)^c} \int_{1}^{w^*} \int_{h^*/2}^{\infty}  \abs| a_R \Hconv \phi\one_s \Vconv \phi\two_t (g) | \,\frac{dt}{t} \,\frac{ds}{s} \wrt g \\
&\qquad\leq \int_{1}^{\infty} \int_{h^*/2}^{\infty} \int_{\Hn} \abs| b_R \Hconv (\HLap\phi\one)_s \Vconv (\VLap\phi\two)_t (g) |  \wrt g  \,\frac{dt}{t^3} \,\frac{ds}{s^3} \\
&\qquad\leq \int_{1}^{\infty} \int_{h^*/2}^{\infty} \norm{ b_R }_{\Leb^1(\Hn)} \norm{ (\HLap\phi\one)_s }_{\Leb^1(\Hn)} \norm{ (\VLap\phi\two)_t }_{\Leb^1(\R)}  \,\frac{dt}{t^3} \,\frac{ds}{s^3}  \\
&\qquad \lesssim_{\bsym\phi} \frac{1}{(h^*)^2}  \norm{ b_R }_{\Leb^1(\Hn)}  \\
&\qquad \lesssim \frac{h^2}{(h^*)^2}  \abs|R|^{1/2} \norm{ a_R }_{\Leb^2(\Hn)}  .
\end{aligned}
\]
In conclusion, 
\[
\int_{(S^*)^c} \abs| a_r \Hconv k (g) | \,dg 
\lesssim_{\bsym\phi} \rho_{2,2}(R,S) \abs|R|^{1/2} \norm{a_R}_{\Leb^2(G)},
\]
as required.
\end{proof}

We note that the argument of the proof above may be improved to show that $\oper{K}_{\boldsymbol\phi}$ sends particles to $\atomHardy(\Hn)$-functions.
However, it appears to be nontrivial to then deduce that these operators also send atoms to $\atomHardy(\Hn)$-functions.
Nevertheless, this conclusion follows from our later characterisation of the Hardy space using Riesz transforms.

We also observe that straightforward modifications of the proof will deal with convolutions with distributions of the form $\sum_{j,k\in\Z} \phi^{(1)}_{2^j} \Hconv \phi^{(2)}_{2^k}$, or more generally with convolutions with distributions of the form $\sum_{j,k\in\Z} \phi^{(1,j)}_{2^j} \Hconv \phi^{(2,k)}_{2^k}$, where the $\{ \phi^{(1,j)}: j \in \Z\}$ is a uniformly bounded family of functions with mean $0$ in $\fnspace{C}^\infty(B\one(o,1))$ and $\{ \phi^{(2,k)}: k \in \Z\}$ is a uniformly bounded family of functions with mean $0$  in $\fnspace{C}^\infty(B\two(0,1))$.

\subsection{Further remarks}

The results of Sections \ref{ssec:atom-implies-area} to \ref{ssec:atom-implies-sing-ints} may easily be formulated and proved for the moment atomic Hardy space of Section \ref{ssec:moment-atom-Hardy}.


\section{The Lusin--Littlewood--Paley area function}\label{sec:Lusin}

In this section, we consider the Hardy space  defined using area functions.
We first examine the definition and properties of the space, and then show that if $f \in \areaHardyphi(\Hn)$, and $\phi\one$ and $\phi\two $  satisfy appropriate invertibility conditions, then $f \in \atomHardy(\Hn)$.

\subsection{Properties of the area function Hardy space}\label{ssec:area-props}
We show that the space is independent of some of the parameters used in its definition, and prove that the area operator (when suitably normalised) is an isometry on $\Leb^2(\Hn)$.

We begin by introducing three cones that we shall discuss.

\begin{definition}\label{def:cone}
Suppose that $g \in \Hn$ and $\beta, \gamma \in\R^+$.
The cones $\Gamma_{\beta,\gamma}(g)$, $\Gamma\one_\beta(g)$, and $\Gamma\two _\beta(g)$ are defined as follows:
\[
\begin{aligned}
\Gamma_{\beta,\gamma}(g)
&:= \{ (g',r,s) \in \Hn \times \R^{+} \times \R^{+}:
     g' \in T(g ,\beta r,\gamma s) \} ,
\\
\Gamma\one_\beta(g)
&:= \{ (g',r) \in \Hn \times \R^{+} :
     g' \in g  B\one(o,\beta r) \} ,
\\
\Gamma\two _\gamma(g'')
&:= \{ (g''',s) \in  \R \times \R^{+}:
     g''' \in g''  B\two (0,\gamma s) \} .
\end{aligned}
\]
\end{definition}

The cone $\Gamma_\beta$ defined before in \eqref{Gamma cone} corresponds to the cone $\Gamma_{\beta,\beta^2}$ above.

Take w-invertible Poisson bounded functions $\phi\one$ on $\Hn$ and $\phi\two$ on $\R$, both with mean $0$ and with w-inverses with mean $0$.
As usual we write $\phi\one_r$ and $\phi\two_s$ for their normalised dilates, and $\phi_{r,s}$ for $\phi\one_r \Hconv \phi\two_s$.
The normalised characteristic functions $|B\one(o,r)|^{-1} \indifn_{B\one(o,r)}$ and $|B\two (0,s)|^{-1} \indifn_{B\two (0,s)}$ are denoted by $\chi\one_r$ and $\chi\two _s$, and $\chi_{r,s} = \chi\one_r \Hconv \chi\two_s$.

\begin{definition}\label{def:areaHardyphi}
Suppose that $\phi\one_r$, $\phi\two _s$, $\chi\one_r$ and $\chi\two _s$ are as described above, and that $\beta, \gamma \in \R^+$.
For $f \in \Leb^1(\Hn)$, we define the \emph{Lusin--Littlewood--Paley area function} $\areaop[,\bsym\phi,\beta,\gamma](f)$ associated to $\phi\one$ and $\phi\two $ by
\begin{equation}\label{def:Lusinppbg}
\begin{aligned}
&\areaop[,\bsym\phi,\beta,\gamma] (f)(g)
:= \biggl( \iint_{\R^{+} \times\R^{+}}
\bigl| f \Hconv \phi_{r,s}   \bigr|^2 \Hconv \chi_{\beta r, \gamma s} (g)  \,\frac{dr}{r} \,\frac{ds}{s} \biggr)^{1/2}
\end{aligned}
\end{equation}
for all $g \in \Hn$, and we define the Hardy space $\areaHardyphi[,\beta,\gamma](\Hn)$ to be the set of all $f\in \Leb^1(\Hn)$ for which $\norm{ \areaop[,\bsym\phi,\beta,\gamma] (f) }_{\Leb^1(\Hn)}<\infty$, with  seminorm
\begin{equation}\label{eq:areaHardyphi-norm}
\norm{ f }_{\areaHardyphi[,\beta,\gamma](\Hn)}:=\norm{ \areaop[,\bsym\phi,\beta,\gamma] (f) }_{\Leb^1(\Hn)}.
\end{equation}
\end{definition}

In Section \ref{ssec:cone-geometry}, for all continuous $F :\Hn \times \R^+ \times \R^+ \to [0, +\infty)$, we defined
\begin{align*}
\term{I}_{\beta,\gamma}(F) &:=
\int_{\Hn} \lpar \iint_{\R^+ \times \R^+}
    (F(\cdot, r, s) \Hconv
    \chi\one_{\beta r} \Vconv \chi\two _{\gamma s})(g)
    \,\frac{dr}{r} \,\frac{ds}{s} \rpar^{1/2} \wrt g; \\
\term{J}_{\beta,\gamma}(F) &:=
\int_{\Hn}  \lpar \iiint_{\Gamma_{\beta,\gamma}(g)} \frac{F(g', r, s)}{\abs| T(o,\beta r,\gamma s)|}
\wrt g' \,\frac{dr}{r} \,\frac{ds}{s} \rpar^{1/2} \wrt g.
\end{align*}
In Lemma \ref{lem:equivalent-area-fns}, we showed that $\term{I}_{\beta,\gamma}(F) \eqsim_{\beta,\gamma,\beta'\gamma'} \term{I}_{\beta',\gamma'}(F)$ and $\term{I}_{\beta,\gamma/2}(F) \lesssim_{\beta,\gamma} \term{J}_{\beta,\gamma}(F) \lesssim_{\beta,\gamma} \term{I}_{\beta,\gamma}(F)$.
This shows that the Hardy space $\areaHardyphi[,\beta,\gamma](\Hn)$ does not depend on the parameters $\beta$ and $\gamma$.
Usually we take these parameters to be $1$ unless explicitly stated otherwise, and write $\areaHardyphi(\Hn)$.
Further, from the results of Sections \ref{ssec:atom-implies-area} (above) and \ref{ssec:area-implies-atom} (below), it follows that the Hardy space does not depend on $\bsym\phi$ either (as long as $\bsym\phi$ is w-invertible), and so we may also write $\areaHardy(\Hn)$.

Lemma \ref{lem:equivalent-area-fns} also implies that there is an equivalent definition of the area integral.

\begin{corollary}\label{cor:alternative-area-integral}
The space $\areaHardyphi(\Hn)$ is the space of all functions $f \in \Leb^1(\Hn)$ for which $\areaopphiprime(f) \in \Leb^1(\Hn)$, where
\[
\areaopphiprime(f)(g)
=\biggl( \iiint_{\Gamma(g)} \frac{1}{\abs|T(o,r',s') |}
\bigl|f \Hconv \phi_{r' ,s'} (g') \bigr|^2
        \wrt g' \,\frac{dr'}{r'} \,\frac{ds'}{s'} \biggr)^{1/2},
\]
and $\norm{\areaopphiprime(f)}_{\Leb^1(\Hn)}$ is equivalent to the $\areaHardyphi(\Hn)$ seminorm.
\end{corollary}

If $f$ is a polynomial of low degree, then $f \Hconv \phi\one_r \Vconv \phi\two _s = 0$ for all $r, s \in \R^+$, and for such $f$, it is clear that $\norm{f}_{\areaHardy(\Hn)} = 0$.
Thus it is reasonable to ask whether $\norm{ \cdot }_{\areaHardy(\Hn)}$ really is a norm.
The answer is given by the next two results.

\begin{prop}\label{prop:Lusin-L2-bounded}
Suppose that $f \in \Leb^2(\Hn)$ and that $\bsym\phi$ is w-invertible.
Then 
\[
\norm{\areaopphi(f)}_{\Leb^2(\Hn)} \eqsim \norm{ f }_{\Leb^2(\Hn)} .
\]
\end{prop}

\begin{proof}
Since $\chi\one_r$ and $\chi\two _s$ are normalised characteristic functions,
\begin{align*}
\norm{\areaopphi(f)}_{\Leb^2(\Hn)} ^2
&= \int_{\Hn} \iint_{\R^{+} \times\R^{+}}
\bigl| f \Hconv \phi_{r,s}  \bigr|^2 \Hconv \chi_{r,s} (g)  \,\frac{dr}{r} \,\frac{ds}{s} \wrt g \\
&= \int_{\Hn} \iint_{\R^{+} \times\R^{+}}
\bigl| f \Hconv \phi_{r,s}   \bigr|^2 \Hconv \chi\one_{r} \Vconv \chi\two _{s} (g) \wrt g \,\frac{dr}{r} \,\frac{ds}{s}  \\
&= \iint_{\R^{+} \times\R^{+}} \int_{\Hn}
\bigl| f \Hconv \phi_{r,s} (g)  \bigr|^2  \wrt g \,\frac{dr}{r} \,\frac{ds}{s}  \\
&\eqsim \norm{f}_{\Leb^2(\Hn)}^2,
\end{align*}
by Corollary \ref{cor:L2-norm-of-sqfnop-bounded}.
\end{proof}

\begin{lemma}\label{lem:L2-is-dense-in-H1-area}
Left translations act isometrically and continuously on $\areaHardyphi(\Hn)$.
Further, $\norm{ \cdot }_{\areaHardyphi(\Hn)}$ is a norm on $\areaHardyphi(\Hn)$.
\end{lemma}

\begin{proof}
Recall that ${}_gf$ denotes the left translate of $f$ by $g$.
If $\norm{ f }_{\areaHardyphi(\Hn)}$ is finite, then
\begin{align*}
\areaop[,\bsym\phi]({}_gf) = {}_g\areaop[,\bsym\phi](f)
\qquad\forall g \in \Hn \quad\forall f \in \Leb^1(\Hn),
\end{align*}
hence the $\Leb^1(\Hn)$ norms of $\areaop[,\bsym\phi]({}_gf)$ and $\areaop[,\bsym\phi](f)$ coincide; further,
\[
\lim_{g \to o} \norm{{}_gf - f}_{ \areaHardyphi(\Hn)} =0
\]
by Lemma \ref{lem:area-fn-cts}.

Take an approximate identity for convolution $(\psi_n)$ of $\fnspace{C}^\infty_0(\Hn)$ functions with supports shrinking to $o$ in $\Hn$.
If $f \in \areaHardyphi(\Hn)$, then, by sublinearity, $\psi_n \Hconv f \in \areaHardyphi(\Hn)$,
\begin{equation}\label{eq:H1-norm-inequality}
\norm{\psi_n \Hconv f}_{ \areaHardyphi(\Hn)}
\leq \norm{\psi_n}_{ \Leb^1(\Hn)} \norm{f}_{ \areaHardyphi(\Hn)} .
\end{equation}

If $f \in \Leb^1(\Hn)$ and $\norm{f}_{ \areaHardyphi(\Hn)} = 0$, then $\psi_n \Hconv f \in \Leb^2(\Hn) \cap \areaHardyphi(\Hn)$, and $\norm{\psi_n \Hconv f}_{ \areaHardyphi(\Hn)} =0$ from \eqref{eq:H1-norm-inequality}, so $\psi_n \Hconv f =0$ in $\Leb^2(\Hn)$ from Proposition \ref{prop:Lusin-L2-bounded}.
Thus $\psi_n \Hconv f =0$ in $\Leb^1(\Hn)$.
Further, $\psi_n \Hconv f \to f$ in $\Leb^1(\Hn)$ as $n \to \infty$ so $f = 0$ in $\Leb^1(\Hn)$.
\end{proof}

\subsection{The inclusion $\areaHardy(\Hn) \subseteq \atomHardy(\Hn)$}
\label{ssec:area-implies-atom}

Here we show that every $f\in \areaHardyphi(\Hn)$ has an atomic decomposition, and control the decomposition appropriately, provided $\phi\one$ and $\phi\two $ satisfy appropriate conditions.
To be more specific, we will show the following result.

\begin{theorem}\label{thm:area-Hardy-implies-atom-Hardy}
Suppose that $M, N \in \N^+$, and that $\phi\one$ and $\phi\two $ are Poisson bounded, as in Definition \ref{def:Poisson-bounds}, and w-invertible as in Definition \ref{def:w-invertible}, with w-inverses $\psi\one$ and $\psi\two$ of the form $\HLap^M \ancestor{\psi}\one$ and $\VLap^N \ancestor{\psi}\two$.
Then there is a constant $C$, depending on $\cdim$, $\bsym\phi$ and $\bsym\psi$, such that for all $f\in \areaHardyphi(\Hn)$, there exist numbers $\lambda_j$ and $(1, 2, M, N ,3)$ atoms $a_j$, for all $j \in \N$, such that $f \sim \sum\lambda_j a_j$, and
\[
\norm{ f }_{\atomHardy(\Hn)}
\leq \sum_{j\in\N}  |\lambda_j|
\leq C\norm{ f }_{\areaHardyphi (\Hn)}.
\]
Hence $\areaHardyphi(\Hn) \subseteq \atomHardy(\Hn)$.
\end{theorem}

\begin{proof}
From Lemma \ref{lem:L2-is-dense-in-H1-area},  without loss of generality, we may restrict attention to $f$ in $\Leb^2(\Hn) \cap \areaHardyphi(\Hn)$.

We denote by $\rect$ the collection of all shards, as in Section \ref{ssec:tiling}.
We take the enlargement parameter $\kappa$ (defined at the start of Section \ref{sec:atomic}) to be $3$, and then \eqref{eq:size-of-rectangle} implies that
\[
\begin{aligned}
R  B\one(o,q)  B\two (0,h)
&\subset g  \bar B\one(o, q/2) \bar B\two (0, (q^2 + 4h)/8)
 B\one(o,q)  B\two (0,h)  \\
&= g  B\one(o, 3q/2) B\two (0, (q^2 + 12h)/8) \\
&\subseteq g  B\one(o, \kappa q/2)  B\two (0, \kappa^2(q^2 + 4h)/8) \\
&= R^{*},
\end{aligned}
\]
when $R \in \rect$, $\cent(R) = g$, $\wid(R) = q$ and $\heit(R) = h$.
By Lemma \ref{lem:where-Ms-is-big}, for all open sets $E$ and $R \in \rect(E)$,
\[
R^* \subseteq
\lset g \in \Hn : \oper{M}_{\flag} (\indifn_E)(g) > \frac{1}{2^{2\cdim}\kappa^\hdim (5\cdim+2)} \rset .
\]

For each tile $R$ in $\rect$, the tent $\Tau(R)$ over $R$ is defined by
\[
\Tau(R)
:= \lset (g,r,s) \in \Hn \times\R^+ \times \R^+ :
    g \in R, r\in ( q/(2\cdim+1), q], s \in ( 0, q^2 ] \rset ,
\]
where $q = \wid(R)$, and for each shard $R$ in $\rect$ that is not a tile, the tent $\Tau(R)$ over $R$ is defined by
\[
\Tau(R)
:= \lset (g,r,s) \in \Hn \times\R^+ \times \R^+ :
    g \in R, r\in ( q/(2\cdim+1), q], s \in ( h/(2\cdim+1)^2, h ] \rset ,
\]
where $q = \wid(R)$ and $h = \heit(R)$.
It is evident that $\Hn \times \R^+ \times \R^+$ decomposes as a disjoint union:
\[
\Hn \times \R^+ \times \R^+
= \bigsqcup_{\substack{R\in \rect}} \Tau(R);
\]
indeed, for $(g,r,s) \in \Hn \times \R^+ \times \R^+$, the coordinates $r$ and $s$ determine the dimensions and the coordinate $g$ determines the location of a shard $R$ such that $g \in R$.

Take $f\in \Leb^2(\Hn) \cap \areaHardy(\Hn)$.
For each $\ell\in\Z$, we define
\begin{align*}
E_\ell
&:= \lset g \in \Hn: \areaopphi (f)(g) > 2^\ell \rset, \\
\rect_\ell
&:= \lset R \in \rect :  |R^*\cap E_{\ell+1}|
\leq \frac{1}{3\kappa^\hdim} |R^*| < |R^* \cap E_\ell|\rset\\
\tilde{E}_\ell
&:= \lset g \in \Hn: \oper{M}_{\flag}(\indifn_{E_\ell})(g) > \frac{1}{2^{2\cdim}\kappa^\hdim (5\cdim+2)} \rset .
\end{align*}
The reader may check that if $R \in \rect_\ell$ and $(g,r,s) \in \Tau(R)$, then $R^* \subseteq \tilde{E}_\ell$ and
\begin{equation}\label{eq:B-ell-factor}
\frac{1}{3 \kappa^\hdim} |R^*| \leq \frac{1}{2} |T(g,r,s)|,
\end{equation}

The hypothesis on the w-invertibility of $\phi\one$ and $\phi\two $ implies that
\[
f =
\iint_{\R^+ \times \R^+} f \Hconv \phi_{r,s} \Hconv \psi_{r,s} \,\frac{ds}{s}\,\frac{dr}{r} \,.
\]
For the moment, we assume also that $\psi\one$ and $\psi\two$ are supported in $B\one(o,1)$ and $B\two(0,1)$.
At the end of the proof, we remove this support assumption.

As before, we write ${}_g\psi_{r,s}$ for the left translate $g' \mapsto \psi_{r,s}(g^{-1} g')$.
It follows  from the reproducing formula above that for all $f \in \Leb^2(\Hn) \cap \areaHardyphi(\Hn)$,
\begin{equation}\label{eq:start-of-area-implies-atom}
\begin{aligned}
f
&= \int_{\Hn} \iint_{\R^+\times\R^+}
   f \Hconv \phi_{r,s}(g) \fn{}_g\psi_{r,s} \,\frac{ds}{s}\,\frac{dr}{r}\wrt g \\
&= \sum_{\ell\in\Z} \sum_{R\in \oper{B}_\ell} \iiint_{\Tau(R)}
   f \Hconv \phi_{r,s}(g) \fn{}_g\psi_{r,s} \,\frac{dr}{r}\,\frac{ds}{s}\wrt g \\
&= \sum_{\ell \in \Z} \lambda_\ell \sum_{R \in \rect_\ell} a_{\ell,R} ,
\end{aligned}
\end{equation}
say, where
\begin{gather*}
a_{\ell,R}
:= \frac{1}{\lambda_\ell} \iiint_{\Tau(R)}
   f \Hconv \phi_{r,s}(g) \fn{}_g\psi_{r,s} \wrt g\,\frac{dr}{r} \,\frac{ds}{s} \\
\noalign{\noindent{and}}
\lambda_\ell
:=\bigg\|\biggl( \sum_{R\in \rect_\ell } \iint_{\R^+\times\R^+}
\abs| f \Hconv \phi_{r,s}(\cdot) |^2 \indifn_{\Tau(R)}(\cdot) \,\frac{dr}{r}\,\frac{ds}{s}\biggr)^{1/2}\bigg\|_{\Leb^2(\Hn)}  |\tilde{E}_\ell|^{1/2}.
\end{gather*}
(If $\lambda_\ell =0$, then all $a_{\ell,R}$ are taken to be $0$, as the integral involved in the definition of $a_{\ell,R}$ vanishes.
At the end of this proof, in \eqref{eq:sum-lambda-estimate} we show that $\sum_{\ell\in\Z} \lambda_\ell$ is finite, and so these expressions make sense.)

Fix $\ell \in \Z$; we claim that $\sum_{R \in \rect_\ell} a_{\ell,R}$ is a geometric multiple of a flag atom.
From Lemma \ref{lem:grouping-particles}, it suffices to prove that there exist functions $b_{\ell,R}$ in $\Dom(\HLap^{M}\VLap^{N})$, for all $R \in \rect_\ell$,  such that
\begin{enumerate}
\item[(A1)] $a_{\ell,R}= \HLap^M \VLap^N b_R$ and $\supp b_R \subseteq R^*$;  and
\item[(A2)] for all sign sequences $\sigma: \rect_\ell \to \{\pm1\}$, the sum $\sum_{R \in \rect_\ell} \sigma_{\ell,R} \fn a_{\ell,R}$ converges in $\Leb^2(\Hn)$, to $a_{\sigma}$ say, and
$\norm{ a_\sigma }_{\Leb^2(\Hn)} \lesssim \abs|E|^{-1/2} $.
\end{enumerate}

Suppose that $\wid(R) = q$ and $\heit(R) = h$.
Evidently, for each $\ell \in \Z$ and $R \in \rect_\ell$,
\[
a_{\ell, R}: = \HLap^M \VLap^N b_{\ell, R},
\]
where
\begin{equation}\label{supp bR}
\begin{aligned}
b_{\ell,R}
:= \frac{1}{\lambda_\ell} \iiint_{\Tau(R)}
   f \Hconv \phi_{r,s}(g) \fn{}_g\ancestor\psi_{r,s} \wrt g\,\frac{dr}{r} \,\frac{ds}{s} \,.
\end{aligned}
\end{equation}

By construction, when $(g,r,s) \in \Tau(R)$, $ \supp {}_g\ancestor\psi_{r,s}
\subseteq g  T(o,r,s)
\subseteq R  T(o,q,h)$, so
\begin{equation}\label{eq1005}
\begin{aligned}
\supp  b_{\ell, R}\subseteq R^{*}.
\end{aligned}
\end{equation}

Our next step is to  take a sign sequence $\sigma: \rect_\ell \to \{\pm1\}$ and estimate the $\Leb^2(\Hn)$ norm of $\sum_{R \in \rect_\ell} \sigma_{\ell,R} \fn a_{\ell,R}$.
For all $h \in \fnspace{C}_c^\infty(\Hn)$ such that $\norm{ h }_{\Leb^2(\Hn)} \leq 1$,
\begin{align*}
&\Bigl|\int_{\Hn}
\sum_{R \in \rect_\ell} \sigma_{\ell,R} \fn a_{\ell,R}(g') h(g')\wrt g' \Bigr| \\
&\qquad = \Biggl| \frac{1}{\lambda_{\ell}}
\sum_{R \in \rect_\ell} \sigma_{\ell,R}
\iiint_{\Tau(R)} \int_{\Hn} f \Hconv \phi_{r,s}(g)  {}_g\psi_{r,s}(g') h(g') \wrt g' \,\frac{ds}{s} \frac{dr}{r} \wrt g \Biggr|\\
&\qquad = \Biggl| \frac{1}{\lambda_{\ell}}
\sum_{R \in \rect_\ell} \sigma_{\ell,R}
\iiint_{\Tau(R)}  h \Hconv \oper{R}\psi_{r,s}(g)  f \Hconv \phi_{r,s}(g)
\,\frac{ds}{s} \frac{dr}{r} \wrt g \Biggr|\\
&\qquad = \Biggl| \frac{1}{\lambda_{\ell}}
\sum_{R \in \rect_\ell} \sigma_{\ell,R}
\iiint_{\Hn \times \R^+ \times \R^+}  h \Hconv \oper{R}\psi_{r,s}(g)
f \Hconv \phi_{r,s}(g) \\
&\qquad\qquad \times \indifn_{\Tau(R)} (g,r,s)
\,\frac{ds}{s} \frac{dr}{r} \wrt g \Biggr|\\
&\qquad\leq \frac{1}{\lambda_{\ell}}
\biggl( \sum_{R \in \rect_\ell} \iiint_{\Hn \times \R^+\times \R^+}
 |h \Hconv \oper{R}\psi_{r,s}(g)|^2
    \indifn_{\Tau(R)}(g,r,s) \,\frac{ds}{s} \frac{dr}{r} \wrt g  \biggr)^{1/2}\\
&\qquad\qquad \times \biggl( \sum_{R \in \rect_\ell}
\iiint_{\Hn \times \R^+\times \R^+} \abs|f \Hconv \phi_{r,s}(g)|^2
\indifn_{\Tau(R)}(g,r,s) \,\frac{ds}{s} \,\frac{dr}{r} \wrt g \biggr)^{1/2} \\
&\qquad\leq \frac{1}{\lambda_{\ell}}
\biggl( \iiint_{\Hn \times \R^+\times \R^+}
 |h \Hconv \oper{R}\psi_{r,s}(g)|^2
  \,\frac{ds}{s} \,\frac{dr}{r}\wrt g  \biggr)^{1/2}\\
&\qquad\qquad \times \biggl( \sum_{R \in \rect_\ell}
\iiint_{\Hn \times \R^+\times \R^+} \abs|f \Hconv \phi_{r,s}(g)|^2
\indifn_{\Tau(R)}(g,r,s) \,\frac{ds}{s} \,\frac{dr}{r} \wrt g \biggr)^{1/2} \\
&\qquad\lesssim  \frac{1}{ \lambda_{\ell}}  \norm{h}_{\Leb^2(\Hn)} \norm{ \biggl( \sum_{R \in \rect_\ell} \iint_{\R^+\times \R^+}
\abs|f \Hconv \phi_{r,s}(\cdot)|^2 \indifn_{\Tau(R)}(\cdot ,r,s)
\,\frac{ds}{s} \,\frac{dr}{r} \biggr)^{1/2} }_{\Leb^2(\Hn)}\\
&\qquad\leq  |\tilde{E}_\ell|^{-{1}/{2}},
\end{align*}
by the Cauchy--Schwarz inequality, the square function estimate of Corollary \ref{cor:L2-norm-of-sqfnop-bounded} and the definition of $\lambda_\ell$.
It follows that $\norm{ \sum_{R \in \rect_\ell} \sigma_{\ell,R} \fn a_{\ell,R} }_{\Leb^2(\Hn)} \lesssim |\tilde{E}_\ell|^{-1/2}$, whence $a$ is a multiple of an atom, and the multiple depends only on $M$, $N$, $\cdim$ and $\bsym\phi$.

Finally, we verify the convergence of the series $\sum_{\ell}|\lambda_\ell|$.
To do this, we first fix $\ell\in\Z$, and observe that if $R \in \rect$ and $(g', r,s) \in \Tau(R)$, then, from \eqref{eq:size-of-rectangle},
$T(g',r,s) \subseteq R^* \subseteq \tilde{E}_\ell$, whence
\[
\abs| \tilde{E}_\ell \cap T(g',r,s)|
= \abs| T(g',r,s) |
= \abs| T(o,r,s) |,
\]
while, by definition of $\rect_\ell$ and \eqref{eq:B-ell-factor},
\[
\abs| E_{\ell+1} \cap T(g',r,s) |
\leq \abs| R^* \cap E_{\ell+1} |
\leq \frac{1}{3 \kappa^{\hdim}} \abs| R^* |
\leq \frac{1}{2} \abs| T(o,r,s) | .
\]
It follows that
\[
\frac{\abs| (\tilde{E}_\ell \setminus E_{\ell+1}) \cap
    T(g', r,s) |}
        { \abs| T(g', r,s) |}
        \geq \frac{1}{2} \,.
\]
Now by the definitions of $\rect_\ell$, the Lusin--Littlewood--Paley area function, and $E_\ell$,
\begin{align*}
&\sum_{R\in \rect_\ell}
\iiint_{\Tau(R)}\abs| f \Hconv \phi_{r,s} (g')|^2
\,\frac{ds}{s} \,\frac{dr}{r} \wrt g' \\
&\qquad\leq 2 \sum_{R\in \rect_\ell}
\iiint_{\Tau(R)} \abs| f \Hconv \phi_{r,s} (g')|^2
\frac{\abs| (\tilde{E}_\ell \setminus E_{\ell+1}) \cap
    T(g' ,r,s) |}
        { \abs| T(g',r,s) |}
\,\frac{dr}{r} \,\frac{ds}{s} \wrt g'\\
&\qquad\leq 2
\iiint_{\Hn \times \R^+ \times\R^+} \abs| f \Hconv \phi_{r,s} (g')|^2
\frac{\abs| (\tilde{E}_\ell \setminus E_{\ell+1}) \cap
    T(g' ,r,s) |}
        { \abs| T(g' ,r,s) |}
\,\frac{dr}{r} \,\frac{ds}{s} \wrt g'\\
&\qquad = 2 \int_{\tilde{E}_\ell\backslash E_{\ell+1}} \iiint_{\Hn\times\R^{+}\times\R^{+}} \abs| f \Hconv \phi_{r,s} (g')|^2
\frac{ \indifn_{
     T(o,r,s)} ( g'  g^{-1}) }
        { \abs| T(o,r,s) |}
\,\frac{dr}{r} \,\frac{ds}{s} \wrt g'\wrt g\\
&\qquad = 2 \int_{\tilde{E}_\ell\backslash E_{\ell+1}} \iiint_{\Gamma(g)} \frac{1}{\abs|T(o,r,s) |} \abs| f \Hconv \phi_{r,s} (g')|^2
\frac{ 1 }
        { \abs| T(o,r,s) |}
\,\frac{dr}{r} \,\frac{ds}{s} \wrt g'\wrt g\\
&\qquad = 2\int_{\tilde{E}_\ell\backslash E_{\ell+1}}
\abs| \areaopphi (f)(g) |^2 \wrt g\\
&\qquad \leq  2^{2\ell +3} |\tilde{E}_\ell|.
\end{align*}

Hence
\begin{equation}\label{eq:sum-lambda-estimate}
\begin{aligned}
\sum_{\ell\in\Z}|\lambda_\ell| 
&\leq \sum_{\ell\in\Z} \norm{ \biggl( \sum_{R\in \rect_\ell}
\iint_{\R^{+}\times\R^{+}} \abs|f \Hconv \phi_{r,s}(\cdot) |^2 \indifn_{\Tau(R)}(\cdot,r,s)
\,\frac{dr}{r} \frac{ds}{s} \biggr)^{1/2} }_{\Leb^2(\Hn)} \\
&\qquad\times \abs|\tilde{E}_\ell|^{1/2}\\
&= \sum_{\ell\in\Z} \biggl( \sum_{R\in \rect_\ell}\iiint_{\Tau(R)}
\abs|f \Hconv \phi_{r,s}(g') |^2
\wrt g' \,\frac{dr}{r} \,\frac{ds}{s}\biggr)^{1/2}  |\tilde{E}_\ell|^{1/2}\\
&\leq \sum_{\ell\in\Z}2^{\ell+3/2}|\tilde{E}_\ell|\lesssim \norm{ \areaoppsi (f) }_{\Leb^1(\Hn)} \\
&=\norm{ f }_{\areaHardyphi(\Hn)}.
\end{aligned}
\end{equation}

It remains to remove an assumption that we made earlier in the proof.
If $\ancestor{\psi}\one$ is not supported in $B\one(o,1)$, we use Lemma \ref{lem:molecules-are-sums} to write $\ancestor{\psi}\one$ as a sum $\sum_{i\in \N} [\ancestor{\psi}^{(1,i)}]_{2^i}$
of dilates of functions $\ancestor{\psi}^{(1,i)}$ that are supported in $B\one(o,1)$ and such that $\norm{\ancestor{\psi}^{(1,i)}}_{(Q)} \lesssim 2^{-i} \norm{\ancestor{\psi}\one}_{(Q)}$; we treat $\ancestor\psi\two$ similarly.

Then we replace \eqref{eq:start-of-area-implies-atom} by
\begin{equation*}
\begin{aligned}
f
&= \int_{\Hn} \iint_{\R^+\times\R^+}
   f \Hconv \phi_{r,s}(g) \sum_{i,j \in \N} {}_g\psi^{(i,j)}_{2^i r,2^js} \,\frac{ds}{s}\,\frac{dr}{r}\wrt g \\
&= \sum_{i,j \in \N} \sum_{\ell\in\Z} \sum_{R\in \oper{B}_\ell} \iiint_{\Tau(R)}
   f \Hconv \phi_{2^{-i}r,2^{-j}s}(g) \fn{}_g\psi^{(i,j)}_{r,s} \,\frac{dr}{r}\,\frac{ds}{s}\wrt g \\
&= \sum_{i,j \in \N} \sum_{\ell \in \Z} \lambda^{(i,j)}_\ell
   \sum_{R \in \rect_\ell} a^{(i,j)}_{\ell,R} ,
\end{aligned}
\end{equation*}
and treat each
$\sum_{\ell \in \Z} \lambda^{(i,j)}_\ell  \sum_{R \in \rect_\ell} a^{(i,j)}_{\ell,R}$
much as we treated
$\sum_{\ell \in \Z} \lambda_\ell \sum_{R \in \rect_\ell} a_{\ell,R}$ before.

This completes the proof of Theorem \ref{thm:area-Hardy-implies-atom-Hardy}. 
\end{proof}

We remark that, when $f \in \Leb^2(\Hn)$, the atomic decomposition $\sum_{\ell\in\Z} \lambda_\ell a_\ell$ in the proof above also converges to $f$ in the $\Leb^2(\Hn)$ norm.
To prove this, we only need to show that $\|\sum_{|\ell|>L} \lambda_\ell a_\ell\|_{\Leb^2(\Hn)}\rightarrow 0$ as $L$ tends to infinity.
For this, first note that
\begin{align*}
&\Bigl| \big\langle \sum_{|\ell|>L} \lambda_\ell a_\ell, h\big\rangle \Bigr| \\
&\qquad\leq \biggl| \sum_{R\in \rect_\ell} \iiint_{\Tau(R)}
h \Hconv \phi_{r,s} (g')
\sum_{|\ell|>L} \lambda_{\ell} a_{\ell} \Hconv \phi_{r,s} (g')
\wrt g' \,\frac{dr}{r} \,\frac{ds}{s} \biggr|\\
&\qquad\leq \int_{\Hn}\biggl( \sum_{R\in \rect_\ell}
\iint_{\R^+\times\R^+} \abs|h \Hconv \phi_{r,s}(g')|^2
\indifn_{\Tau(R)}(g',r,s) \,\frac{dr}{r} \frac{ds}{s} \biggr)^{1/2} \\
&\qquad\qquad \times \biggl( \sum_{R\in \rect_\ell}
\iint_{\R^+\times \R^+} \biggl| \sum_{|\ell|>L} \lambda_{\ell}
    a_{\ell} \Hconv \phi_{r,s}(g') \biggr|^2 \indifn_{\Tau(R)}(g',r,s)
\,\frac{dr}{r} \,\frac{ds}{s} \biggr)^{1/2} \wrt g'\\
&\qquad\leq C\norm{ h }_{\Leb^2(\Hn)} \\
&\qquad\qquad \times \bigg\|\biggl( \sum_{R\in \rect_\ell}
\iint_{\R^+\times \R^+} \biggl| \sum_{|\ell|>L} \lambda_{\ell}
a_{\ell} \Hconv \phi_{r,s}(g') \biggr|^2 \indifn_{\Tau(R)}(g',r,s)
\,\frac{dr}{r} \,\frac{ds}{s} \biggr)^{1/2}\bigg\|_{\Leb^2(\Hn)}\\
&\qquad\to 0
\end{align*}
when $L$ tends to $\infty$, as the area integral operator is bounded on $\Leb^2(\Hn)$.
Since
\[
\biggl\| \sum_{|\ell|>L} \lambda_\ell a_\ell \biggr\|_{\Leb^2(\Hn)}
 =\sup_{h: \norm{ h }_{\Leb^2(\Hn)=1 }}
\Bigl| \big\langle \sum_{|\ell|>L} \lambda_\ell a_\ell, h\big\rangle \Bigr|,
\]
the sum $\sum_{\ell\in\Z} \lambda_\ell a_\ell$ converges to $f$ in the $\Leb^2(\Hn)$ norm.

Finally, we observe that the same method also proves the following result.

\begin{theorem}\label{thm:area-Hardy-implies-moment-atom-Hardy}
Suppose that $M, N \in \N^+$, and that $\phi\one$ and $\phi\two $ are Poisson bounded, as in Definition \ref{def:Poisson-bounds}, and w-invertible as in Definition \ref{def:w-invertible}, with w-inverses $\psi\one$ and $\psi\two$ of the form $\oper{D} \ancestor{\psi}\one$ and $\oper{T}^{2N}\ancestor{\psi}\two$, where $\oper{D}$ is the $2M$-fold product tensor $\Hnabla \otimes \dots \otimes \Hnabla$ and $\ancestor{\phi}\one$ is a dual tensor.
Then there is a constant $C$, depending on $\cdim$ and $\bsym\phi$, such that for all $f\in \areaHardyphi(\Hn)$, there exist numbers $\lambda_j$ and $(1, 2, M, N ,3)$ moments atoms $a_j$, for all $j \in \N$, such that $f \sim \sum\lambda_j a_j$, and
\[
\norm{ f }_{\matomHardy(\Hn)}
\leq \sum_{j\in\N}  |\lambda_j|
\leq C\norm{ f }_{\areaHardyphi (\Hn)}.
\]
Hence $\areaHardyphi(\Hn) \subseteq \matomHardy(\Hn)$.
\end{theorem}

The point of this is that we may be able to show that $\phi\one$ has a w-inverse of the form in this theorem, but not of the form in Theorem \ref{thm:area-Hardy-implies-atom-Hardy}.
However, if $f \in \matomHardy(\Hn)$, then $f \in \areaHardyphi(\Hn)$, where $\phi\one \in \fnspace{A}(\Hn)$, the space mentioned at the start of Section \ref{ssec:spectral-theory}, and for such $\bsym\phi$, it is possible to find a w-inverse of the form in \ref{thm:area-Hardy-implies-atom-Hardy}, and it follows that $f \in \atomHardy(\Hn)$.
Hence $\matomHardy(\Hn)$ and $\atomHardy(\Hn)$ coincide.

\subsection{Remarks on a ``discrete area function''}
\label{ssec:discreteareaspace}
We could replace the integrals over $r$ and $s$ in the definition of the area function $\areaopphi$ (Definition \ref{def:areaHardyphi}) by sums, taking $r$ and $s$ to be $2^j$ and $2^k$, or more generally $\alpha^{j}$ and $\beta^{k}$.

The proofs that $\areaopphi(f) \in \Leb^1(\Hn)$ if $f \in \atomHardy(\Hn)$ in Section \ref{ssec:atom-implies-area} and that $f \in \atomHardy(\Hn)$ if $\areaopphi(f) \in \Leb^1(\Hn)$ in Section \ref{ssec:area-implies-atom} go through with minor modifications, and show that the atomic Hardy space may also be characterised by a ``discrete area function''.
We spare the reader the details!


\section{The square function Hardy space}\label{sec:sqfn}
We recall the definitions of the continuous and discrete Littlewood--Paley square functions.
As usual, given functions $\phi\one$ on $\Hn$ and $\phi\two $ on $\R$, we write $\phi\one_r$ and $\phi\two _s$ for their normalised dilates, and $\phi_{r,s}$ for $\phi\one_r \Vconv \phi\two _s$.

\begin{definition}\label{def:sqfnHardyphi}
Suppose that $\phi\one$ and $\phi\two $ have mean $0$ and are Poisson-bounded, as in Definition \ref{def:Poisson-bounds}.
Suppose also that $\bsym\phi$ is w-invertible (continuously or discretely, according to the Hardy space that we are going to define).
For $f \in \Leb^1(\Hn)$, we define the \emph{continuous and discrete Littlewood--Paley square functions} $\ctssqfnopphi(f)$  and $\dissqfnopphi(f)$ associated to $\phi\one$ and $\phi\two $ by
\begin{equation*}
\begin{aligned}
&\ctssqfnopphi (f)(g)
:= \biggl( \iint_{\R^{+} \times\R^{+}}
\bigl| f \Hconv \phi\one_r \Vconv \phi\two _s(g)  \bigr|^2
    \,\frac{dr}{r} \,\frac{ds}{s} \biggr)^{1/2} \\
&\dissqfnopphi (f)(g)
:= \biggl( \sum_{(m,n) \in \Z \times \Z}
\bigl| f \Hconv \phi\one_{2^m} \Vconv \phi\two _{2^n}(g)  \bigr|^2
    \biggr)^{1/2}
\end{aligned}
\end{equation*}
for all $g \in \Hn$.
We  define the \emph{square function Hardy spaces} $\ctssqfnHardyphi(\Hn)$ and $\dissqfnHardyphi(\Hn)$, often abbreviated to $\ctssqfnHardy(\Hn)$ and $\dissqfnHardy(\Hn)$, to be the set of all $f\in \Leb^1(\Hn)$ for which $\norm{ \ctssqfnopphi (f) }_{\Leb^1(\Hn)}<\infty$ or $\norm{ \dissqfnopphi (f) }_{\Leb^1(\Hn)}<\infty$, with  norms
\[
\norm{ f }_{\ctssqfnHardyphi(\Hn)}:=\norm{ \ctssqfnopphi (f) }_{\Leb^1(\Hn)}
\]
and
\[
\norm{ f }_{\dissqfnHardyphi(\Hn)}:=\norm{ \dissqfnopphi (f) }_{\Leb^1(\Hn)}.
\]
\end{definition}

We remark that $\norm{ \cdot }_{\ctssqfnHardyphi(\Hn)}$ and $\norm{ \cdot }_{\dissqfnHardyphi(\Hn)}$ are indeed norms, for much the same reason as $\norm{ \cdot }_{\areaHardy(\Hn)}$ is a norm (see the beginning of Section \ref{sec:Lusin}).

Han, Lu and Sawyer studied the square function Hardy spaces (but using left convolutions rather than right convolutions), and showed that these do not depend on the choice of $\bsym\phi$.
However, it is not clear to us that their techniques cover all the pairs $\bsym\phi$ that we consider, as at least some of their arguments assume that $\phi\one$ lies in the space $\fnspace{A}(\Hn)$ mentioned in Section \ref{ssec:Calderon}.

We showed in Lemma \ref{lem:ctssqfn-dominates-dissqfn} that every continuous square function $\ctssqfnopphi(f)$ dominates a discrete square function $\ctssqfnoppsi(f)$, and it follows that $\ctssqfnHardyphi(\Hn) \subseteq \dissqfnHardypsi(\Hn)$.
In Section \ref{sec:atomic} we established that $\atomHardy(\Hn) \subseteq \ctssqfnHardy(\Hn)$ and $\atomHardy(\Hn) \subseteq \dissqfnHardy(\Hn)$.
We may modify the proof that all $f$ in the area function Hardy space admit atomic decompositions to show that if $\dissqfnoppsi(f)$ is integrable then $f$ has an atomic decomposition.
Hence $\ctssqfnHardyphi(\Hn) \subseteq \dissqfnHardypsi(\Hn) \subseteq \atomHardy(\Hn)$.


\section{The maximal function Hardy spaces}\label{sec:maximal-Hardy-spaces}

In this section, we characterise the flag Hardy space by maximal functions.
There are two aspects of this.
First, we show that we obtain the same Hardy space irrespective of whether we define it using a radial, nontangential or grand maximal function.
Then we show that the Hardy space defined using the nontangential maximal function associated to the Poisson kernel is a subspace of the area function Hardy space.
More precisely, by combining the results of Sections \ref{sec:Lusin} and \ref{sec:atomic} with some obvious inclusions, we know that
\[
\areaHardy(\Hn) \subseteq \atomHardy(\Hn) \subseteq \gmaxHardy(\Hn)
\subseteq \nontanHardy(\Hn) \subseteq \radialHardy(\Hn).
\]
We complete the identification of these five spaces by showing that 
\[
\radialHardy(\Hn) \subseteq \nontanHardy(\Hn)
\qquad\text{and}\qquad
\nontanHardy(\Hn) \subseteq \areaHardy(\Hn).
\]
All inclusions are continuous, with corresponding norm inequalities.

We begin by recalling some relevant definitions.
Take Poisson bounded functions $\phi\one$ on $\Hn$ and $\phi\two $ on $\R$ both of which have total integral $1$, and write $\phi_{r,s}$ for the convolution $\phi\one_r \Vconv \phi\two _s$ of their normalised dilates.

\begin{definition}\label{def:radial-max-op}
Define the \emph{radial maximal operator} associated to $\bsym\phi$ by
\begin{align*}
\radialmaxphi(f)(g)
:=\sup_{r,s \in \R^+} |f \Hconv \phi_{r,s} (g)|
\qquad\forall g \in \Hn \quad \forall f \in \Leb^1(\Hn),
\end{align*}
and let $\radialHardyphi(\Hn)$ be the set of all $f\in \Leb^1(\Hn)$ such that $\norm{ \radialmaxphi(f) }_{\Leb^1(\Hn)}$ is finite, with norm
\begin{equation}\label{eq:radial-Hardy}
\norm{ f }_{\radialHardy\phi(\Hn)} :=\norm{ \radialmaxphi(f) }_{\Leb^1(\Hn)}.
\end{equation}
\end{definition}

Recall that $\Gamma_{\beta,\gamma}(g):=\{ (g',r,s) \in \Hn \times \R^{+} \times \R^{+}: g' \in g  B\one(o, \beta r)  B\two (0, \gamma s) \}$ for all $\beta, \gamma \in \R^+$.

\begin{definition}\label{def:nontan-Hardy}
Define the \emph{nontangential maximal operator} associated to $\beta$, $\gamma$ and $\bsym\phi$ by
\begin{align*}
\nontanmaxopphi[,\beta,\gamma](f)(g)
= \sup \{
\abs| f \Hconv \phi_{r,s} (g') |  : (g',r,s) \in \Gamma_{\beta,\gamma}(g) \}
\end{align*}
for all $g \in \Hn$  and all $f \in \Leb^1(\Hn)$,
and let $\nontanHardyphi[,\beta,\gamma](\Hn)$ be the linear space of all
$f\in \Leb^1(\Hn)$ such that $\norm { \nontanmaxopphi[,\beta,\gamma](f) }_{\Leb^1(\Hn)}$ is finite, with norm
\begin{equation}\label{eq:nontanHardy}
\norm{f}_{\nontanHardyphi[,\beta,\gamma](\Hn)} := \norm{ \nontanmaxopphi[,\beta,\gamma](f) }_{\Leb^1(\Hn)}.
\end{equation}
\end{definition}

We are going to see shortly that this space is independent of $\beta$ and $\gamma$, so we usually take both these to be $1$ and omit the corresponding suffices.
Later we show that $\bsym\phi$ is also irrelevant.

\begin{definition}\label{def:gmax-Hardy-2}
Let $\family$ be a Poisson bounded family of pairs of functions, as in Definition \ref{def:Poisson-bounds}, which contains at least one pair of functions $(\phi\one, \phi\two)$ such that $\int_G \phi\one(g) \wrt g =  1$ and $\int_{\R} \phi\two(g_2) \wrt g_2 = 1$.
Define the \emph{grand maximal operator} associated to $\family$ by
\begin{align*}
\gmaxop (f) (g)
:= \sup_{\bsym\phi \in \family} \sup_{r,s \in \R^+}
\abs|f \Hconv \phi_{r,s} (g)|
\qquad\forall g \in \Hn \quad \forall f \in \Leb^1(\Hn),
\end{align*}
and $\gmaxHardy(\Hn)$ to be the space of all $f\in \Leb^1(\Hn)$ such that $\norm{ \gmaxop (f) }_{\Leb^1(\Hn)}$ is finite, with norm
\begin{equation}\label{eq:gmax-Hardy-2}
\norm{ f }_{\gmaxHardy(\Hn)}:=\norm{ \gmaxop (f) }_{\Leb^1(\Hn)}.
\end{equation}
\end{definition}

If we take a ``large enough'' family $\family$, the pairs of heat kernels $(h\one_1, h\two _1)$ and Poisson kernels $(p\one_1, p\two _1)$ and all their translates by elements of a bounded subsets of $\Hn$ and of $\R$ belong to $\family$.

It is clear that the grand maximal operator dominates the nontangential maximal operator, which in turn dominates the radial maximal operator.
Further, we may dominate $\gmaxop$ by the flag maximal operator $\oper{M}_{\flag}$ (see Section \ref{ssec:heat-Poisson}), so $\gmaxop$ is $\Leb^p$ bounded for all $p \in (1, \infty]$.

We check that we are dealing with norms.
It is obvious that the expressions defined in \eqref{eq:radial-Hardy}, \eqref{eq:nontanHardy} and \eqref{eq:gmax-Hardy-2} are seminorms.
If $f \in \Leb^1(\Hn)$ and $\radialmaxphi(f) = 0$, then $f \Hconv \phi_{r,s} = 0$ for all $r,s \in \R^+$.
Now  $f = \lim_{r,s \to 0} f \Hconv \phi_{r,s}$ in $\Leb^1(G)$, so $f = 0$ by a standard approximate identity argument.

\subsection{The tangential maximal function}
We begin our study of maximal functions by showing that the behaviour of the nontangential maximal function is to a degree independent of the apertures $\beta$ and $\gamma$ of the cone  in its definition.
We take a continuous function $F: \Hn \times \R^+ \times R^+ \to [0,\infty)$, and define
\begin{equation*}
F_{\alpha,\beta}(g) = \sup\{ F(g',r,s) : g' \in T(g,\alpha r, \beta s) \} .
\end{equation*}
Then Lemma \ref{lem:apertures-maxfn} shows that, for all $\alpha, \alpha', \beta, \beta',\epsilon \in \R^+$ such that $\epsilon$ is small and $\alpha < \alpha'$ and $\beta <  \beta'$,
\begin{equation}\label{eq:cone-control}
\begin{aligned}
\| F_{\alpha,\beta} \|_{\Leb^1(\Hn)}
\leq \|F_{\alpha',\beta'} \|_{\Leb^1(\Hn)}
\lesssim_{\epsilon} \max \left\{ \left(\frac{\alpha'}{\alpha}\right)^{\hdim} , \left(\frac{\alpha'}{\alpha}\right)^{2\cdim}\frac{\beta'}{ \beta} \right\}^{1+\epsilon}   \|F_{\alpha,\beta}\|_{\Leb^1(\Hn)}  .
\end{aligned}
\end{equation}
This shows that changing the apertures $\beta$ and $\gamma$ of the flag cone in the definition of the nontangential maximal function does not change the Hardy space that arises.

This lemma enables us to control the \emph{tangential maximal operator} $\tanmaxopphi$, defined by
\[
\begin{aligned}
\tanmaxopphi(f)(g) 
&= \sup\Bigl\{ \abs| f \Hconv \phi_{r,s} (gg_1g_2)) |
\Bigl(1 + \frac{\norm{g_1}}{r}\Bigr)^{-\hdim-\epsilon}
\Bigl(1+ \frac{1+\norm{g_2}}{s}\Bigr)^{-1-\epsilon} : \\
&\qquad
r,s \in \R^+ , g_1 \in \Hn, g_2 \in \R \Bigr\};
\end{aligned}
\]
here $\epsilon $ is small and positive.
Note that, given $g_3 \in \Hn$, there are infinitely many $g_1 \in \Hn$ and $g_2 \in \R$ such that $g_3 = g_1 g_2$.
Of these, $(1 + \sfrac{\norm{g_1}}{r})^{-\hdim-\epsilon}
(1+ \sfrac{\norm{g_2}}{s})^{-1-\epsilon}$ is maximal when $\norm{g_1}$ is minimal.
As noted earlier, in our discussion of the geometry of tubes, if $g_3$ is written as $(z,t)$, then this happens when $\norm{g_1} = |z|_\infty$ and $|g_2| = 0$ if $|t| \leq |z|_\infty^2$ and $|g_2| = |t| - |z|_\infty^2$ otherwise.

\begin{corollary}\label{cor:tan-max-fun}
The following inequality holds:
\[
\norm{\tanmaxopphi(f)}_{\Leb^1(\Hn)}
\lesssim \norm{\nontanmaxopphi(f)}_{\Leb^1(\Hn)} .
\]
\end{corollary}
\begin{proof}
Observe that $\Hn = \bigcup_{i \in \N} A^{(1,i)}$ and $\R = \bigcup_{j \in \N} A^{(2,j)}$, where
\begin{align*}
A^{(1,0)} = B\one(o,1)
&\qquad\text{and}\qquad
A^{(1,i)} = B\one(o,2^i) \setminus B\one(o,2^{i-1})
\qquad\forall i \in \N^+\\
\noalign{\noindent{\text{and}}}
A^{(2,0)} = B\two (0,1)
&\qquad\text{and}\qquad
A^{(2,j)} = B\two (0,2^j) \setminus B\two (0,2^{j-1})
\qquad\forall j \in \N^+.
\end{align*}
Then
\begin{align*}
\tanmaxopphi(f)(g)
&\leq \sup\Bigl\{ \abs| f \Hconv \phi_{r,s} (gg_1g_2)) |
\Bigl(1 + \frac{\norm{g_1}}{r}\Bigr)^{-\hdim-\epsilon}
\Bigl(1+ \frac{\norm{g_2}}{s}\Bigr)^{-1-\epsilon}: \\
&\qquad  g_1 \in A^{(1,i)}(r), g_2 \in A^{(2,j)}(s), i,j \in \N ,r,s \in \R^+  \Bigr\} \\
&\leq \sup\Bigl\{ \nontanmaxopphi[,2^i,2^j] (g)
\bigl(1 + 2^{i-1} \bigr)^{-\hdim-\epsilon}
\bigl(1+ 2^{j-1} \bigr)^{-1-\epsilon}: i,j \in \N ,r,s \in \R^+  \Bigr\} \\
&\lesssim \sum_{i,j \in \N} 2^{-i(\hdim+\epsilon) -j(1+\epsilon)} \nontanmaxopphi[,2^i,2^j](f)(g)
\qquad\forall g \in \Hn,
\end{align*}
so, from \eqref{eq:cone-control}, with $F(g,r,s)$ taken to be $f \Hconv \phi_{r,s}(g)$, and a suitable choice of $\epsilon$,
\[
\norm{\tanmaxopphi(f)}_{\Leb^1(\Hn)}
\lesssim \sum_{i,j\in \N} 2^{-i(\hdim+\epsilon) -j(1+\epsilon)}  \norm{\nontanmaxopphi[,2^{i},2^j](f)}_{\Leb^1(\Hn)}
\lesssim_\epsilon \norm{\nontanmaxopphi(f)}_{\Leb^1(\Hn)} ,
\]
as required.
\end{proof}

To state the next lemma, we need a little more notation.
We write $\familycz\one$ for a family of functions in $\fnspace{C}^\infty(\Hn)$ that are supported in $B\one(o,1)$ and have mean zero, and $\familycz\two $ for a family of functions in $\fnspace{C}^\infty(\R)$ that are supported in $B\two (0,1)$ and have mean zero.

\begin{lemma}\label{lem:grad bd by non-tangential}
Suppose that $\phi\one : \Hn \to \C$ and $\phi\two : \R \to \C$ are Poisson bounded and w-invertible, with w-inverses of the form $\oper{D}\tilde\phi\one$ and $\oper{T}\tilde\phi\two$, where $\tilde{\phi}\one$ and $\tilde{\phi}\two$ are Poisson bounded and $\oper{D}$ is a differential operator of homogeneous degree $2\hdim+2$.
Suppose also that $\family$ is one of $\familycz\one \times \familycz\two $, $\familycz\one \times \{\phi\two \}$ and $ \{ \phi\one \} \times \familycz\two $.
Then
\begin{align}\label{claim grad bd by non-tangential-1}
 \gmaxop (f)(g)
\lesssim  \tanmaxopphi(f)(g)
\qquad\forall g \in \Hn
\end{align}
for all $f \in \Leb^2(\Hn)$.
\end{lemma}

\begin{proof}
The key to this lemma is the following calculation, inspired by \cite{CFefSte72}.
Take $\bsym\psi$ in $\familycz$.
If we can write
\[
\psi\one
= \int_{\R^+} \phi\one_{r} \Hconv \omega^{(1,r)} \,\frac{dr}{r}
\qquad\text{and}\qquad
\psi\two
= \int_{\R^+} \phi\two _{s} \Vconv \omega^{(2,s)} \,\frac{ds}{s},
\]
for some $\omega^{(1,r)} \in \fnspace{P}(\Hn)$, $\omega^{(2,s)} \in \fnspace{P}(\R)$ and $r, s \in \R^+$, where $C(\bsym\psi)$, given by
\begin{equation}\label{eq:C-psi}
\begin{aligned}
C(\bsym\psi) 
&:=
\lpar \int_{\R^+}\int_{\Hn} \bigabs|\omega^{(1,r)}(g_1)| \Bigl(1+\frac{\norm{g_1}}{r}\Bigr)^{\hdim+\epsilon} \wrt g_1 \,\frac{dr}{r} \rpar \\
&\qquad \times \lpar \int_{\R^+} \int_{\R} \bigabs|\omega^{(2,s)}(g_2)| \Bigl(1+\frac{\norm{g_2}}{s}\Bigr)^{1+\epsilon} \wrt g_2 \,\frac{ds}{s} \rpar,
\end{aligned}
\end{equation}
is finite, then it will follow that
\begin{align*}
&\sup_{r',s'\in\R^+} \abs| f \Hconv \psi_{r'}\one \Vconv \phi_{s'}\two (g) | \\
&\qquad= \sup_{r',s'\in\R^+} \biggl| \iiiint_{\R^+\times\R^+\times\Hn\times\R}
 (f \Hconv \phi_{r r'}\one \Vconv \phi_{s s'}\two )(g g_2^{-1} g_1^{-1})   \\
&\qquad\qquad \times \omega_{r'}^{(1,r)} (g_1) \fn \omega_{s'}^{(2,s)}(g_2)
\wrt g_2 \wrt g_1 \,\frac{dr}{r} \,\frac{ds}{s} \biggr| \\
&\qquad\leq \sup_{r',s'\in\R^+} \iiiint_{\R^+\times\R^+\times\Hn\times\R}
\abs| (f \Hconv \phi_{r r'}\one \Vconv \phi_{s s'}\two )(g g_2^{-1} g_1^{-1}) |   \\
&\qquad\qquad \times \abs|\omega_{r'}^{(1,r)} (g_1) \fn \omega_{s'}^{(2,s)}(g_2)|
\wrt g_2 \wrt g_1 \,\frac{dr}{r} \,\frac{ds}{s}  \\
&\qquad\leq
\sup_{r',s'\in\R^+} \iiiint_{\R^+\times\R^+\times\Hn\times\R}
\tanmaxopphi(f) (g) \Bigl(1 + \frac{\norm{g_1}}{rr'}\Bigr)^{\hdim+\epsilon} \Bigl(1+ \frac{\norm{g_2}}{ss'}\Bigr)^{1+\epsilon}  \\
&\qquad\qquad \times \abs|\omega_{r'}^{(1,r)} (g_1) \fn \omega_{s'}^{(2,s)}(g_2)|
\wrt g_2 \wrt g_1 \,\frac{dr}{r} \,\frac{ds}{s}  \\
&\qquad=
C(\bsym\psi) \tanmaxopphi(f) (g) ,
\end{align*}
where $C(\bsym\psi)$ is given by \eqref{eq:C-psi}.
Loosely speaking, for $C(\bsym\psi)$ to be finite, we need $\omega^{(1,r)}(g)$ and $\omega^{(2,s)}(t)$ to decay faster than $\norm{g}^{-2\hdim}$ and $|t|^{-2}$ at infinity.
We seek suitable representations of $\psi\one$ and $\psi\two $.

By hypothesis,
\[
\int_{\R^+} \phi\one_r  \Hconv \oper{D}\tilde{\phi}\one_r \,\frac{dr}{r} = \delta ,
\]
where $\oper{D}$ is a left-invariant differential operator of homogeneous degree $2\hdim+2$.
We take $\omega^{(1,r)}$ to be $\oper{D}\tilde{\phi}\one_r \Hconv \psi $, and then
\[
\psi
= \int_{\R^+} \phi\one_r  \Hconv \omega^{(1,r)} \,\frac{dr}{r} \,.
\]
Note that
\[
\begin{aligned}
&\int_{\R^+} \int_{\Hn} \bigabs|\omega^{(1,r)}(g_1)| \Bigl(1+\frac{\norm{g_1}}{r}\Bigr)^{\hdim+\epsilon} \wrt g_1 \,\frac{dr}{r} \\
&\qquad= \int_{\R^+} \int_{\Hn} \bigabs|\oper{D}\tilde{\phi}\one_r \Hconv \psi (g_1)| \Bigl(1+\frac{\norm{g_1}}{r}\Bigr)^{\hdim+\epsilon} \wrt g_1 \,\frac{dr}{r} \\
&\qquad=  \int_{\R^+} \int_{\Hn} \bigabs|\oper{D}\tilde{\phi}\one \Hconv \psi_{1/r} (g_1)| (1+\norm{g_1})^{\hdim+\epsilon} \wrt g_1 \,\frac{dr}{r} \\
&\qquad= \int_{\R^+} \int_{\Hn} \bigabs|\oper{D}\tilde{\phi}\one \Hconv \psi_{r} (g_1)| (1+\norm{g_1})^{\hdim+\epsilon} \wrt g_1 \,\frac{dr}{r} \,.
\end{aligned}
\]

On the one hand, when $r$ is small, we use Proposition \ref{prop:moments} to write $\psi = \rivf{\oper{E}} \ancestor\psi$, where $\oper{E}$ is a left invariant differential operator of degree $1$.
Then $\psi_r = r\rivf{\oper{E}} (\ancestor\psi_r)$, and
\begin{align*}
&\int_{0}^{1} \int_{\Hn} \bigabs|\oper{D}\tilde{\phi}\one \Hconv \psi_{r} (g_1)| (1+\norm{g_1})^{\hdim+\epsilon} \wrt g_1 \,\frac{dr}{r}  \\
&\qquad= \int_{0}^{1} \int_{\Hn} \bigabs|\oper{D}\tilde{\phi}\one \Hconv \rivf{\oper{E}} (\ancestor\psi_r) (g_1)| (1+\norm{g_1})^{\hdim+\epsilon} \wrt g_1 \wrt r  \\
&\qquad= \int_{0}^{1} \int_{\Hn} \abs|
\oper{E}\oper{D}\tilde{\phi}\one \Hconv \ancestor\psi_r (g_1)| (1+\norm{g_1})^{\hdim+\epsilon} \wrt g_1 \wrt r \\
&\qquad= \int_{0}^{1} \int_{\Hn} \abs|\int_{\Hn} \oper{E}\oper{D}\tilde{\phi}\one (g_1 g_0^{-1}) \fn\ancestor\psi_r (g_0) \wrt g_0| (1+\norm{g_1})^{\hdim+\epsilon} \wrt g_1 \wrt r \\
&\qquad\leq \int_{0}^{1} \int_{\Hn} \int_{\Hn} \bigabs| \oper{E}\oper{D}\tilde{\phi}\one (g_1 g_0^{-1})| \bigabs|\ancestor\psi_r (g_0) | (1+\norm{g_1})^{\hdim+\epsilon} \wrt g_0\wrt g_1 \wrt r \\
&\qquad= \int_{0}^{1} \int_{\Hn} \int_{\Hn} \bigabs| \oper{E}\oper{D}\tilde{\phi}\one (g_1)| \bigabs|\ancestor\psi_r (g_0) | (1+\norm{g_1g_0})^{\hdim+\epsilon} \wrt g_0\wrt g_1 \wrt r \\
&\qquad\leq 2^{\hdim+\epsilon} \int_{0}^{1} \int_{\Hn} \int_{\Hn} \bigabs| \oper{E}\oper{D}\tilde{\phi}\one (g_1)| \bigabs|\ancestor\psi_r (g_0) |  (1+\norm{g_1})^{\hdim+\epsilon} \wrt g_0\wrt g_1 \wrt r \\
&\qquad= \int_{\Hn} \bigabs|\ancestor\psi (g_0)| \wrt g_0 \int_{\Hn} \bigabs|\oper{E}\oper{D}\tilde{\phi}\one (g_1)| (1+\norm{g_1})^{\hdim+\epsilon} \wrt g_1 \,,
\end{align*}
since $1+\norm{g_1g_0} \leq 1+\norm{g_1} + \norm{g_0} \leq 2(1 + \norm{g_1})$ when $\norm{g_0} \leq 1$.
Both integrals are finite; in particular, since $\tilde\phi\one$ is Poisson bounded, and $\oper{E}\oper{D}$ is a differential operator of degree $2\hdim+3$,
\[
\bigabs|\oper{E}\oper{D}\tilde{\phi}\one (g_1)|
\lesssim \max\{1, \norm{g_1} \}^{-3\hdim - 4}
\qquad\forall g_1 \in \Hn.
\]

On the other hand, when $r$ is large, we may write $\oper{D}\tilde{\phi}\one = \oper{D}' \oper{D}''\tilde{\phi}\one$, where $\oper{D}'$ and $\oper{D}''$ are both left invariant differential operators of degrees $\hdim+1$.
Further, $\rivf{\oper{D}}' (\psi_r) = r^{\hdim+1} (\rivf{\oper{D}}'\psi)_r$, so
\begin{align*}
&\int_{1}^{\infty} \int_{\Hn} \bigabs|\oper{D}\tilde{\phi}\one \Hconv \psi_{r} (g_1)| (1+\norm{g_1})^{\hdim+\epsilon} \wrt g_1 \,\frac{dr}{r}  \\
&\qquad= \int_{1}^{\infty} \int_{\Hn} \bigabs|\oper{D}''\tilde{\phi}\one \Hconv \rivf{\oper{D}}' (\psi_r) (g_1)| (1+\norm{g_1})^{\hdim+\epsilon} \wrt g_1 \,\frac{dr}{r}  \\
&\qquad= \int_{1}^\infty \int_{\Hn} \abs|
\oper{D}''\tilde{\phi}\one \Hconv  (\rivf{\oper{D}}' \psi)_r (g_1)| (1+\norm{g_1})^{\hdim+\epsilon} \wrt g_1 \,\frac{dr}{r^{\hdim+2}} \\
&\qquad= \int_{1}^\infty \int_{\Hn} \abs|\int_{\Hn} \oper{D}''\tilde{\phi}\one (g_1 g_0^{-1}) \fn(\rivf{\oper{D}}' \psi)_r (g_0) \wrt g_0| (1+\norm{g_1})^{\hdim+\epsilon} \wrt g_1 \,\frac{dr}{r^{\hdim+2}} \\
&\qquad\leq \int_{1}^\infty \int_{\Hn} \int_{\Hn} \bigabs| \oper{D}''\tilde{\phi}\one (g_1 g_0^{-1})| \bigabs| (\rivf{\oper{D}}' \psi)_r (g_0) | (1+\norm{g_1})^{\hdim+\epsilon} \wrt g_0\wrt g_1 \,\frac{dr}{r^{\hdim+2}} \\
&\qquad= \int_{1}^\infty \int_{\Hn} \int_{\Hn} \bigabs| \oper{D}''\tilde{\phi}\one (g_1)| \bigabs|(\rivf{\oper{D}}' \psi)_r(g_0) | (1+\norm{g_1g_0})^{\hdim+\epsilon} \wrt g_0\wrt g_1 \,\frac{dr}{r^{\hdim+2}} \\
&\qquad= \int_{1}^{\infty }\int_{\Hn} \bigabs|(\rivf{\oper{D}}' \psi)_r(g_0)| (1+\norm{g_0})^{\hdim+\epsilon} \wrt g_0 \,\frac{dr}{r^{\hdim+2}} \\
&\qquad\qquad \times
\int_{\Hn} \bigabs|\oper{D}''\tilde{\phi}\one (g_1)| (1+\norm{g_1})^{\hdim+\epsilon} \wrt g_1
\,,
\end{align*}
since $1+\norm{g_1g_0} \leq 1+\norm{g_1} + \norm{g_0} \leq (1 + \norm{g_1})(1 + \norm{g_0})$.
Both factors are finite; indeed,
\[
\bigabs|\oper{D}''\tilde{\phi}\one (g_1)|
\lesssim \max\{1, \norm{g_1} \}^{-2\hdim - 2}
\qquad\forall g_1 \in \Hn
\]
since $\phi\one$ is Poisson bounded, and
\[
\int_{\Hn} \bigabs|(\rivf{\oper{D}}' \psi)_r(g_0)| (1+\norm{g_0})^{\hdim+\epsilon} \wrt g_0
\lesssim \int_{\Hn}  \chi\one_r (1+\norm{g_0})^{\hdim+\epsilon} \wrt g_0
\leq (1+r)^{\hdim+\epsilon}.
\]

Similar (but easier to prove) formulae hold for the $\R$ variable, and it follows that we can control $C(\bsym\psi)$, uniformly in $\bsym\psi$, thereby proving the lemma for the case where $\family = \familycz\one \times \familycz\two $.

The other two cases are similar but easier.
\end{proof}

\begin{corollary}\label{cor:grad bd by non-tangential}
Suppose that $\phi\one : \Hn \to \C$ and $\phi\two : \R \to \C$ are Poisson bounded and that $\{\phi_r\}$ is an approximate identity for convolution.
Suppose also that $\family$ is a Poisson bounded family of pairs $(\psi\one , \psi\two )$.
Then
\begin{align}\label{claim grad bd by non-tangential}
 \gmaxop (f)(g)
\lesssim  \tanmaxopphi(f)(g)
\qquad\forall g \in \Hn
\end{align}
for all $f \in \Leb^2(\Hn)$.
\end{corollary}

\begin{proof}
We take a general Poisson bounded $\psi\one$ on $\Hn$, with mean $c$.
Then $\psi\one - c \phi\one$ has mean $0$, whence from Corollary \ref{cor:Poisson decomposition} we may write $\psi\one$ as a sum:
\[
\psi\one
= c \phi\one + \sum_{i\in \N}   [{\psi}^{(1,i)}]_{2^i}
\]
where the ${\psi}^{(1,i)}$ have mean equal to $0$.
We do the same with $\phi\two $.
Then
\[
\begin{aligned}
&\abs| f \Hconv \psi\one_r \Vconv \psi\two _s(g)| \\
&\qquad\leq
\abs| f \Hconv \phi\one_r \Vconv \phi\two _s(g)| +
\sum_{j \in \N}
\abs| f \Hconv \phi\one_r \Vconv \psi^{(2,j)}_s(g)| \\
&\qquad\qquad + \sum_{i \in \N}
\abs| f \Hconv \psi^{(1,i)}_r \Vconv \phi\two _s(g)|
+ \sum_{i,j \in \N}
\abs| f \Hconv \psi^{(1,i)}_r \Vconv \psi^{(2,j)}_s(g)| .
\end{aligned}
\]
We can control the suprema of each of these four terms over $r$ and $s$, using the hypothesis for the first and Lemma \ref{lem:grad bd by non-tangential} for the others.
\end{proof}

\subsection{The inclusion $\radialHardy(\Hn) \subseteq \nontanHardy(\Hn)$}
\label{ssec:radial-implies-nontan}

We now conclude our study of the equivalence of the  maximal function Hardy spaces.
\begin{theorem}\label{thm:radial-controls-nontan}
Suppose that $\bsym\phi$ is a Poisson bounded pair.
Then
\begin{align}\label{non-tangel to radial Schwartz}
\| \nontanmaxopphi(f) \|_{\Leb^1(\Hn)}\leq C\|\radmaxopphi(f)\|_{\Leb^1(\Hn)}
\end{align}
\end{theorem}

\begin{proof}
We use a ``good-$\lambda$ inequality''.
Fix $\lambda \in \R^+$,  to be determined a little later, and define
\[
F_\lambda
= \{g\in \Hn:  \gmaxop (f) (g)\leq\lambda \nontanmaxopphi(f)(g) \},
\]
where $\gmaxop (f) (g)$ is the grand maximal function as in Definition \ref{def:grand-max-fn}.
We shall prove two inequalities: first,
\begin{align}\label{ee Flambda compliment bd by radial}
\int _{\Hn } \nontanmaxopphi(f)(g) \wrt g
\leq 2 \int _{F_\lambda } \nontanmaxopphi(f)(g)  \wrt g,
\end{align}
and second that
\begin{align}\label{non-tangel to radial Schwartz p}
\nontanmaxopphi(f)(g)
\leq C \biggl( \oper{M}_{\flag}\Bigl( \radmaxopphi(f)^\theta  \Bigr)(g) \biggr)^{1/\theta }
\qquad \forall g\in F_\lambda.
\end{align}
when $0<\theta <1$.
Once we have proved these two inequalities, the $\Leb^{1/\theta}$ boundedness of $\oper{M}_{\flag}$ implies that
\begin{align*}
\int _{\Hn } \nontanmaxopphi(f)(g) \wrt g
&\leq 2 \int _{F_\lambda } \nontanmaxopphi(f)(g)  \wrt g \\
&\leq 2C\int _{F_\lambda } \biggl( \oper{M}_{\flag}\Bigl( \radmaxopphi(f)^\theta  \Bigr)(g) \biggr)^{1/\theta } \wrt g\\
&\leq 2C\int _{\Hn} \biggl( \oper{M}_{\flag}\Bigl( \radmaxopphi(f)^\theta  \Bigr)(g) \biggr)^{1/\theta } \wrt g\\
&\leq 2C\int _{\Hn }  \radmaxopphi(f)(g) \wrt g,
\end{align*}
as required.

To prove \eqref{ee Flambda compliment bd by radial}, we observe that the definition of $F_\lambda^c$ implies that
\begin{align*}
\int _{F_\lambda^c } \radmaxopphi(f)(g) \wrt g
\leq \int _{F_\lambda^c } \frac{1}{\lambda} \gmaxop (f) (g)\wrt g
\leq \frac{1}{\lambda} \int _{\Hn} \nontanmaxopphi(f)(g) \wrt g.
\end{align*}
Further, from \eqref{claim grad bd by non-tangential},
\begin{align}\label{eq:claim grad bd by non-tangential}
\int _{\Hn } \gmaxop (f) (g)\wrt g
\leq C_1\int_{\Hn} \nontanmaxopphi(f)(g) \wrt g.
\end{align}
We take $\lambda$ to be $2C_3$ and combine these inequalities; then
\begin{align*}
\int _{F_\lambda^c } \radmaxopphi(f)(g) \wrt g
\leq \frac{C_1}{\lambda} \int_{\Hn} \nontanmaxopphi(f)(g)\wrt g
=  \frac{1}{2}\int_{\Hn} \nontanmaxopphi(f)(g)\wrt g,
\end{align*}
as desired.

To prove \eqref{non-tangel to radial Schwartz p}, take $g\in F_\lambda$.
By definition of $\nontanmaxopphi(f)(g)$, there exists $(g_1,r,s) \in \Gamma_\beta(g)$ such that
\[
|f\Hconv \phi_{r,s}(g_1)|> \frac{1}{2} \nontanmaxopphi(f)(g).
\]
We will write $g$ and $g_1$ as $(z,t)$ and $(z_1,t_1)$.

Take small $\alpha_1 \in \R^+$.
Then for all $g'=(z',t')\in B_c\one(g_1,\alpha_1r) $,
\begin{align*}
&| f\Hconv \phi_{r,s}(g_1) -f\Hconv \phi_{r,s}(g') | \\
&\qquad\leq \alpha_1 r \sup_{g'' \in B\one_c(g_1,\alpha_1r) }
|\Hnabla (f\Hconv \phi_{r,s})(g'') |\\
&\qquad\leq C_1 \alpha_1  \sup_{g'' \in B\one_c(g_1,\alpha_1r) }
| f\Hconv  r \Hnabla \phi\one_{r} \Vconv \phi\two _{s}(g'') |\\
&\qquad\leq C_1\alpha_1 \gmaxop (f) (g),
\end{align*}
assuming that the family $\family\one$ is big enough to include translates of components of the horizontal gradient of $\phi\one$ by elements of $B\one_c(g_1,\alpha_1)$.
Similarly, take small $\alpha_2 \in \R^+$.
Then for all $g'' =(0,t'')\in B\two (t',\alpha_2s) $,
\begin{align*}
&| f\Hconv \phi_{r,s}(g') -f\Hconv \phi_{r,s}(g' g'') | \\
&\qquad\leq \alpha_2 s \sup_{t''' \in B\two(t', \alpha_2s)}
|\Vnabla (f\Hconv \phi_{r,s})(z',t'') |\\
&\qquad\leq C_2 \alpha_2  \sup_{t''' \in B\two(t', \alpha_2s)}
| f\Hconv \phi\one_{r} \Vconv s\Vnabla \phi\two _{s}(z',t'') |\\
&\qquad\leq C_2 \alpha_2\gmaxop (f) (g) ,
\end{align*}
assuming that the family $\family\two $ is big enough to include certain translates of $\oper{T} \phi\two $.
Consequently,
\begin{align*}
&\nontanmaxopphi(f)(g) \\
&\qquad\leq 2  |f\Hconv \phi_{r,s}(g_1)|\\
&\qquad\leq 2  |f\Hconv \phi_{r,s}(g_1)-f\Hconv \phi_{r,s}(g')|+ 2| f\Hconv \phi_{r,s}(g') -f\Hconv \phi_{r,s}(g'' g') | \\
&\qquad\qquad + 2 \abs|f\Hconv \phi_{r,s}(g'' g') |\\
&\qquad\leq 2C_1\alpha_1\gmaxop(f) (g) + 2C_2\alpha_2\gmaxop(f) (g) + 2|f\Hconv \phi_{r,s}(g'' g') |\\
&\qquad\leq 2C_1\alpha_1\lambda \nontanmaxopphi(f)(g) + 2C_2\alpha_2\lambda \nontanmaxopphi(f)(g) + 2|f\Hconv \phi_{r,s}(g'' g') |.
\end{align*}
By choosing $\alpha_1$ and $\alpha_2$ so small that $2C_1\alpha_1\lambda<{1/4}$ and $2C_2\alpha_2\lambda<{1/4}$, we see that
\begin{align*}
\nontanmaxopphi(f)(g)
&< \frac{1}{4} \nontanmaxopphi(f)(g)
+ \frac{1}{4} \nontanmaxopphi(f)(g)
+|f\Hconv \phi_{r,s}(g'' g') |,
\end{align*}
which implies that
\begin{align*}
\nontanmaxopphi(f)(g)
&< 2|f\Hconv \phi_{r,s}(g'' g') |
\end{align*}
for all $g'' =(0,t'')\in B\two (t',\alpha_2s)$ and $g'=(z',t')\in B\one(g_1,\alpha_1r) $.

Consequently, if $\theta < 1$, then
\begin{align*}
\nontanmaxopphi(f)(g)^\theta
&< 2^\theta \frac{1}{ |B\one(g_1,\alpha_1r)||B\two (t',\alpha_2s)|} \\
&\qquad \times \int_{B\one(g_1,\alpha_1r)\cdot B\two (t',\alpha_2s)} |f\Hconv \phi_{r,s}(g'' g') |^\theta dg''dg'\\
&\leq C \oper{M}_{\flag} \bigl( \radmaxopphi(f)^\theta \bigr)(g),
\end{align*}
and so
\begin{align*}
\nontanmaxopphi(f)(g)
&\leq C\biggl( \oper{M}_{\flag}\bigl( \radmaxopphi(f)^\theta \bigr)(g)\biggr)^{1/ \theta}.
\end{align*}
This completes the proof of \eqref{eq:claim grad bd by non-tangential} and hence of Theorem \ref{thm:radial-controls-nontan}.
\end{proof}

\subsection{The inclusion $\nontanHardy(\Hn) \subseteq \areaHardy(\Hn)$}\label{ssec:nontan-implies-area}

We use the nontangential maximal function associated to the Poisson kernel, as we will use properties of harmonic functions.
In view of the results of the previous sections, this is not restrictive.

Recall that $p\one_{r}$ and $p\two _{s}$ are the Poisson kernels on $\Hn$ and on $\R$  (here $r, s \in \R^+$), and, for $f \in \Leb^1(\Hn)$, the \emph{Poisson integral} $u$ of $f$ is given by
\[
u(g,r,s):= f \Hconv p\one_{r} \Vconv p\two _{s} (g)
\qquad\forall g \in \Hn.
\]
The \emph{nontangential maximal function} of $f \in \Leb^1(\Hn)$ is defined by
\begin{align*}
u^{*}(g)
:= \sup_{(g',r,s) \in\Gamma_\beta(g)} | u(g', r, s) |,
\end{align*}
and the \emph{radial maximal function} of $f \in \Leb^1(\Hn)$ is defined by
\begin{align*}
u^{+}(g)
=\sup_{r,s \in \R^+} |u(g',r,s)|.
\end{align*}

It is easy to see that, if $f\in \Leb^p(\Hn)$, then $v$, given by $v(r,g) := f \Hconv p_r(g)$ for all $(r,g) \in \R^+ \times \Hn$, is harmonic, in the sense that $\fullHLap v = 0$, where $\fullHLap = \HLap - \partial_{r}^2$ (recall that the operator $\HLap$ was normalised to be positive).
Moreover, if $v$ is a function on $\R^+ \times \Hn$ and $\fullHLap v = 0$, then
\begin{equation}\label{green}
\begin{aligned}
\fullHLap |v|^2 (r,g) = 2 \abs|\fullHnabla v(r,g)|^2
\end{aligned}
\end{equation}
in $\R^+ \times \Hn$, where
\begin{equation}\label{gradient}
\begin{aligned}
\fullHnabla
= (\Hnabla, \partial_{r})
= (\oper{X}_1,\ldots,\oper{X}_\cdim ,\oper{Y}_1,\ldots, \oper{Y}_\cdim ,\partial_{r}).
\end{aligned}
\end{equation}
We use coordinates $(t,s)$ on $\R \times \R^+$, and write $\fullVnabla$ and $\fullVLap$ for the full gradient  $( \partial_t, \partial_{s})$ and full Laplacian $- \partial_{t}^2 - \partial_{s}^2$ on $\R \times \R+$.
As usual, we write $p_{r,s}$ for $p\one_{r} \Vconv p\two _{s}$ and $\chi_{r,s} =\chi\one_r\Vconv \chi\two _s$ where $\chi\one_r$ and $\chi\two _s$ are the normalised characteristic functions of the unit balls in $\Hn$ and $\R$. 
We also denote by $\snabla p_{r,s}$ the tensor-valued function $(\fullHnabla p\one )_r \Vconv (\fullVnabla p\two )_s$.

\begin{definition}\label{def-of-S-function-Sf L-5}
For $f\in \Leb^1(\Hn)$, the \emph{Lusin--Littlewood--Paley area integral} $\areaPoisson(f)$ of $f$ associated to the Poisson kernel is defined by
\begin{equation*}
\begin{aligned}
&\areaPoisson(f)(g) 
:=
\biggl( \iint_{\R^{+} \times\R^{+}}
\abs| f \Hconv (\snabla p)_{r,s} |^2 \Hconv \chi_{r,s}
(g)  \,\frac{dr}{r} \,\frac{ds}{s} \biggr)^{1/2} .
\end{aligned}
\end{equation*}
\end{definition}

\begin{theorem}\label{thm:nontan-implies-area}
Suppose that $u^* \in \Leb^1(\Hn)$.
Then $\areaPoisson(f) \in \Leb^1(\Hn)$ and
\begin{align}\label{nontan-implies-area}
\|\areaPoisson(f)\|_{\Leb^1(\Hn)}\lesssim \|u^* \|_{ \Leb^1(\Hn)}.
\end{align}
\end{theorem}

\begin{proof}
Take $\alpha \in \R^+$ and $f\in \Leb^1(\Hn)$ such that $\norm{u^*}_{\Leb^1(\Hn)}<\infty$, and define
\begin{align*}
E(\alpha)
:= \{ g \in \Hn : u^*(g) \leq \alpha \}
\qquad\text{and}\qquad
A(\alpha)
:= \lset g\in \Hn: \oper{M}_{\flag}  (\indifn_{E(\alpha)^c} )(g)<\frac{1}{200}\rset.
\end{align*}
By definition and the $\Leb^2$ boundedness of the flag maximal function $\oper{M}_{\flag}$,
\[
E(\alpha)^c \subseteq A(\alpha)^c
\qquad\text{and}\qquad
|A(\alpha)^c | \lesssim | E(\alpha)^c |.
\]
Recall that $\Gamma_\beta(g):=\{ (g',r,s) \in \Hn \times \R^{+} \times \R^{+}:
     g' \in g  T(o,\beta r,\beta^2 s) \}$ and $\Gamma(g) = \Gamma_1(g)$.
Define
\begin{align*}
\Omega
:= \bigcup_{g \in A(\alpha)} \Gamma(g) \qquad
\text{and}\qquad
\tilde \Omega
:=\bigcup\limits_{g\in E(\alpha) } \Gamma_\beta(g),
\end{align*}
for some large $\beta$, which we shall determine after \eqref{step2-2} below.

In order to prove \eqref{nontan-implies-area}, it suffices to prove that
\begin{equation}\label{nontan-implies-area aim-1}
\begin{aligned}
\abs| \{ g \in \Hn : \areaPoisson(f)(g) > \alpha \}|
\lesssim
\abs | E(\alpha)^c|
+ \frac{1}{\alpha^2}
\int_{E(\alpha) } \abs| u^*(g)|^2 \wrt g.
\end{aligned}
\end{equation}
For then, as in \cite[p.~163]{CFefSte72},
\[
\begin{aligned}
\|\areaPoisson(f)\|_{\Leb^1(\Hn)}
&= \int_{\R^+} \abs| \{ g \in \Hn : \areaPoisson(f)(g) > \alpha \}| \wrt \alpha \\
&\lesssim \int_{\R^+} \abs | E(\alpha)^c| \wrt \alpha
+ \int_{\R^+} \frac{1}{\alpha^2}
\int_{E(\alpha) } \abs| u^*(g)|^2 \wrt g \wrt \alpha \\
&\lesssim \|u^*\|_{\Leb^1(\Hn)} .
\end{aligned}
\]
The left-hand side of \eqref{nontan-implies-area aim-1} is equal to
\begin{align*}
&\abs|\{ g \in A(\alpha)^c : \areaPoisson(f)(g) >\alpha \}|
+ \abs|\{ g \in A(\alpha) : \areaPoisson(f)(g) >\alpha \}|.
\end{align*}
It follows straight from the definitions that
\[
\abs|\{ g \in A(\alpha)^c : \areaPoisson(f)(g) >\alpha \}|
\leq |A(\alpha)^c|
\lesssim | E(\alpha)^c |
= |\{g \in \Hn : u^*(g)>\alpha \}|,
\]
and hence it suffices to estimate the other summand.
By Chebyshev's inequality,
\begin{align*}
\abs|\{ g \in A(\alpha) : \areaPoisson(f)(g) >\alpha \}|
\leq \frac{1}{\alpha^2} \int_{A(\alpha)} \areaPoisson(f)(g)^2 \wrt g.
\end{align*}
Thus it will suffice to show that
\begin{equation}\label{nontan-implies-area aim}
\int_{A(\alpha)} \areaPoisson(f)(g)^2 \wrt g
\lesssim
\alpha^2 \abs | E(\alpha)^c|
+ \int_{E(\alpha) } \abs| u^*(g)|^2 \wrt g.
\end{equation}
Since $\chi_{r,s}$ is even, 
\begin{equation}\label{eq:key-expression}
\begin{aligned}
&\int_{A(\alpha)} \areaPoisson(f)(g')^2 \wrt g' \\
&\qquad= \int_{A(\alpha)} \int_{\Hn} \iint_{\R^+ \times \R^+}
\abs| f \Hconv (\snabla p)_{r,s}(g'g) |^2 \chi_{r,s}(g)
\,\frac{dr}{r} \,\frac{ds}{s} \wrt g \wrt g' \\
&\qquad= \int_{\Hn} \int_{\Hn} \iint_{\R^+ \times \R^+}
{J}(g'g,r,s) \abs| f \Hconv (\snabla p)_{r,s}(g'g) |^2 \chi_{r,s}(g) \\
&\qquad\qquad\,\frac{dr}{r} \,\frac{ds}{s} \wrt g' \wrt g  \\
&\qquad= \int_{\Hn} \iint_{\R^+ \times \R^+}
{J}(g',r,s) \abs| f \Hconv (\snabla p)_{r,s}(g') |^2
\,\frac{dr}{r} \,\frac{ds}{s} \wrt g' ,
\end{aligned}
\end{equation}
where ${J}: \Hn \times \R^+ \times \R^+ \to [0,\infty)$ is chosen such that ${J}(g'g,r,s)= 1$ if $g' \in A(\alpha)$ and $g \in \supp( \chi_{r,s})$.

Define
\[
V(g,r,s) := \indifn_{E(\alpha)}\Hconv p_{r,s} (g)
\qquad\forall (g,r,s) \in \Hn \times \R^+ \times \R^+ .
\]
Recall from Lemma \ref{lem:poisson-and-flag-maximal-functions} that
$p_{r,s} \gtrsim \abs|T(o,r,s)|^{-1} \indifn_{T(o,r,s)}$; it follows that
\begin{align*}
V(g',r,s)
&=  \int_{E(\alpha)}  p_{r,s}((g'')^{-1} g')  \wrt g''
= \int_{E(\alpha)}  p_{r,s}((g')^{-1} g'')  \wrt g'' \\
&\gtrsim \frac{ |E(\alpha) \cap T(g',r,s)| }{ |T(g',r,s)|}.
\end{align*}
When $(g',r,s) \in \bigcup_{g \in A(\alpha)} \Gamma(g)$, there exists $g \in A(\alpha)$ such that $g \in T(g',r,s)$, {and so}
\begin{align*}
 \frac{ |E(\alpha)^c \cap T(g',r,s)| }{ |T(g',r,s)|}
 \leq \oper{M}_{\flag}  (\indifn_{E(\alpha)^c} )(g)
\leq \frac{1}{200} \,,
\end{align*}
which implies that
\begin{align*}
\frac{ |E(\alpha) \cap T(g',r,s)| }{ |T(g',r,s)|}\geq \frac{199}{200}.
\end{align*}
Hence there is a geometric constant $C_1\in(0,1)$ such that
\begin{equation}\label{step2}
V(g',r,s) \geq C_1
\qquad\forall (g',r,s)\in \Omega.
\end{equation}

Further, we claim that there is a geometric constant $C_0$ in $(C_1.\infty)$ such that
\begin{equation}\label{step2-2}
V(g',r,s) \leq C_0
\qquad\forall (g',r,s)\in \tilde \Omega.
\end{equation}
Since $\int_{\Hn} p_{r,s}(g) \wrt g = 1$, this follows similarly if the cone aperture $\beta$ is large enough.

Now we choose a smooth function $\Eta:\R \to [0,1]$ such that
$\Eta(t)= 0$ when $t\leq C_1$ and $\Eta(t)= 1$ when $t\geq C_0$.
From \eqref{step2}, $\Eta( V(g',r,s)) = 1$ when $(g',r,s) \in \Omega$, and so, from \eqref{eq:key-expression},
\begin{equation}\label{RHS square max}
\begin{aligned}
&\int_{A(\alpha)} \areaPoisson^2(f)(g) \wrt g \\
&\qquad\leq \int_{\Hn}\iint_{\R^{+}\times \R^{+}}
\abs| f\Hconv  (\snabla p)_{r,s}(g') \fn \Eta(V(g',r,s)) |^2
\,\frac{dr}{r} \,\frac{ds}{s} \wrt g' \\
&\qquad= \int_{\Hn}\iint_{\R^{+}\times \R^{+}}
\abs| \fullHnabla U (g',r,s) \fn \Eta(V(g',r,s)) |^2
\,r\wrt r \,s\wrt s\wrt g' ,
\end{aligned}
\end{equation}
where $U(g',r,s) =  f \Vconv \fullVnabla p\two _{s}  \Hconv p\one_{r}(g') $.
We assume that $f$ is real-valued; if not, one treats its real and imaginary parts separately.
As in \eqref{green},
\begin{equation}\label{eq:derivative-formula}
\begin{aligned}
&\labs \fullHnabla U(g',r,s) \fn \Eta(V(g',r,s)) \rabs^2 \\
&\qquad=\frac{1}{2}  \fullHLap \lpar \lpar U(g',r,s)  \fn\Eta(V(g',r,s)) \rpar^2 \rpar \\
&\qquad\qquad - 4 U(g',r,s) \fn \fullHnabla U(g',r,s) \fn\Eta(V(g',r,s))
\fn\fullHnabla  \Eta(V(g',r,s)) \\
&\qquad\qquad - \labs U(g',r,s) \fn \Eta'(V(g',r,s)) \fn \fullHnabla V(g',r,s) \rabs^2\\
&\qquad\qquad - \labs U(g',r,s) \fn\Eta(V(g',r,s)) \fn\Eta''(V(g',r,s)) \fn\fullHnabla V(g',r,s)\rabs^2 ,
\end{aligned}
\end{equation}
and we write $\term{i}_1(g',r,s)$, $\term{i}_2(g',r,s)$, $\term{i}_3(g',r,s)$ and $\term{i}_4(g',r,s)$ for these four terms.

The right-hand side of \eqref{RHS square max} is bounded by $\sum_{j=1}^4 \term{I}_j$, where, for each $j$,
\[
\term{I}_j
=\labs \int_{\Hn}\iint_{\R^{+}\times \R^{+}} \term{i}_j(g',r,s)
\,r\wrt r \,s\wrt s \wrt g' \rabs.
\]

We treat $\term{I}_1$ using integration and the decay at infinity of the functions involved:
\begin{align*}
\term{I}_1
&= \biggl| \frac{1}{2} \iint_{\R^{+}\times\R^{+}} \int_{\Hn}
\sum_{m=1}^{2n} \oper{X}_j^2 \lpar\lpar U(g',r,s) \fn \Eta (V(g',r,s))  \rpar^2 \rpar \wrt g' \fn r\wrt r \fn s \wrt s \\
&\qquad + \frac{1}{2} \int_{\R^{+}} \int_{\Hn} \int_{\R^{+}} \partial_{r}^2  \lpar\lpar U(g',r,s) \fn \Eta (V(g',r,s))  \rpar^2 \rpar \fn r  \wrt r \wrt g' s\wrt s \biggr| \\
&= \frac{1}{2}  \int_{\R^{+}}   \int_{\Hn} \lpar\lpar U(g',0,s) \fn \Eta (V(g',0,s))  \rpar^2 \rpar \wrt g' \,s\wrt s ,
\end{align*}
where $U(g',0,s)$ and $V(g',0,s)$ are interpreted as the obvious limits.

For the term $ \term{I}_2,$ we use Hölder's inequality; then
\begin{align*}
\term{I}_2
&\leq 4 \biggl(\iint_{\R^{+}\times\R^{+}} \int_{\Hn}
 \labs \fullHnabla U(g',r,s) \fn \Eta( V(g',r,s)) \rabs^2
\wrt g' \,r\wrt r \,s\wrt s \biggr)^{1/2}\\
&\qquad \times  \biggl(\iint_{\R^{+}\times \R^{+}} \int_{\Hn}
\labs U(g',r,s) \fn \fullHnabla \Eta(V(g',r,s)) \rabs^2
\wrt g' \,r\wrt r \,s\wrt s \biggr)^{1/2}\\
&\leq \frac{1}{2C} \iint_{\R^{+}\times \R^{+}} \int_{\Hn}
\bigl|   \fullHnabla U(g',r,s) \fn \Eta(V(g',r,s)) \bigr|^2
\wrt g' \,r\wrt r \,s\wrt s \\
&\qquad + \frac{C}{2} \iint_{\R^{+}\times \R^{+}} \int_{\Hn}
\labs U(g',r,s) \fn \Eta'(V(g',r,s)) \fn \fullHnabla V(g',r,s) \rabs^2
\wrt g' \,r\wrt r \,s\wrt s\\
&=:\term{I}_{21}+\term{I}_{22},
\end{align*}
say, where $C$ is large.
The term $\term{I}_{21}$ is a small multiple of the right-hand side of \eqref{RHS square max}, and may be absorbed there.

It is easy to see that $\term{I}_{3}$ is dominated by $\term{I}_{22}$, and $\term{I}_{4}$ is not dissimilar to $\term{I}_{3}$.
More precisely, we take the smooth function $\Zeta$ on $\R$ such that $\Zeta(0) = 0$ and
\[
\Zeta'(t) = \bigl( (\Eta'(t))^4 + (\Eta(t)\Eta'(t))^2 \bigr)^{1/4};
\]
clearly $\Zeta'$ is supported in $[C_0,C_1]$.

The  right-hand side of \eqref{RHS square max} is bounded by a geometric multiple of
\begin{align*}
&\int_{\R^{+}}   \int_{\Hn} \lpar U(g',0,s) \fn \Eta (V(g',0,s))  \rpar^2 \wrt g' \,s\wrt s  \\
&\qquad+ \iint_{\R^{+}\times \R^{+}} \int_{\Hn}
\labs U(g',r,s) \fn \Zeta'(V(g',r,s)) \fn \fullHnabla V(g',r,s) \rabs^2
\wrt g' \,r\wrt r \,s\wrt s\\
&\qquad= \term{J}_1 + \term{J}_2,
\end{align*}
say.
We treat $\term{J}_1$ using the argument of \eqref{eq:derivative-formula} in the central variable, and see that
 \begin{align*}
\term{J}_1
&=\int_{\Hn} \int_{\R^{+}}  \bigl| \fullVnabla f\Vconv p\two _{s} (g')
\fn\Eta( V(g',0,s)) \bigr|^2 \,s\wrt s\wrt g' \\
&\leq \int_{\Hn} \bigl| f(g') \Eta(V(g',0,0))\bigr|^2 \wrt g'\\
&\qquad+  \int_{\Hn}  \int_{\R^{+}}
\bigl| f\Vconv p\two _{s}(g') \fn \Eta'( V(g',0,s) ) \fn \fullVnabla V(g',0,s) \bigr|^2
\,s\wrt s \wrt g'\\
&\qquad+  \int_{\Hn}  \int_{\R^{+}} \bigl| {f\Vconv p\two _{s}} (g') \bigr|^2
\Bigl| \Eta(V(g',0,s))  \fn \Eta''(V(g',0,s))  \Bigr| \Bigl| \fullVnabla V(g',0,s) \Bigr|^2 \\
&\qquad\qquad\,s\wrt s \wrt g'.
\end{align*}
Now $|f|\leq u^*$ and $| f\Vconv p\two _{s} |\leq  u^*(z',t')$ by definition of $u^*$.
The definitions of $E(\alpha)$ and $\Eta$ imply that $u^*(g') \leq \alpha$ for all $g'$ and $s$ such that $\Eta\bigl( ({\indifn_{E(\alpha)}\Vconv p\two _{s} }(g') \bigr)\neq 0$ and $\Eta'\bigl( ({\indifn_{E(\alpha)}\Vconv p\two _{s}}(g') \bigr) \neq 0$.
Hence
 \begin{align*}
\term{J}_1
&\leq \int_{ E(\alpha)} \bigl| u^*(g') \bigr|^2 \wrt g'
+ \alpha^2 \int_{\Hn} \int_{\R^{+}}
\bigl| s\fullVnabla(\indifn_{E(\alpha)}\Vconv p\two _{s}) (g') \bigr|^2
\,\frac{ds}{s} \wrt g' .
\end{align*}
Observe that $s\fullVnabla (1\Vconv p\two _{s}) =0$, and so
 \begin{align*}
 &\int_{\R} \int_{\R^{+}} \Bigl| s\fullVnabla (\indifn_{E(\alpha)}\Vconv p\two _{s})(z',t') \Bigr|^2  \,\frac{ds}{s} \wrt t' \\
&\qquad=    \int_{\R}  \int_{\R^{+}}  \Bigl|   {s\fullVnabla (1-\indifn_{E(\alpha)})\Vconv p\two _{s}} (z',t') \Bigr|^2\frac{ds}{s}\ dt' \\
&\qquad\lesssim \|1-\indifn_{E(\alpha)}(z',\cdot)\|_{\Leb^2(\R)}^2
\end{align*}
by Littlewood--Paley theory.
Consequently,
 \begin{align*}
\term{J}_1
&\lesssim \int_{ E(\alpha)} \bigl| u^*(g') \bigr|^2 dg'
+  \alpha^2\bigl| E(\alpha)^c \bigr|.
\end{align*}

We turn to the term $\term{J}_2$, which, up to a constant, is equal to
 \begin{align*}
&\iint_{\R^{+}\times \R^{+}} \int_{\Hn}
\labs U(g',r,s) \fn \fullHnabla \Zeta(V(g',r,s)) \rabs^2
\wrt g' \,r\wrt r \,s\wrt s  ,
\end{align*}
where $U(g',r,s) =  f \Vconv \fullVnabla p\two _{s}  \Hconv p\one_{r}(g') $.
Now $| U(g',r,s) \fn  \fullHnabla \Zeta( V(g',r,s)) |^2$ is equal to
\begin{align*}
&\frac{1}{2}  \fullVLap \biggl(u(g',r,s)^2  \fullHnabla \Zeta( V(g',r,s))^2 \biggr)\\
&\quad - 4 u(g',r,s) \fullVnabla u(g',r,s)\fullHnabla \Zeta( V(g',r,s)) \fullVnabla \fullHnabla \Zeta( V(g',r,s)) \\
&\quad- U(g',r,s)^2  \Bigl| \fullVnabla \fullHnabla \Zeta( V(g',r,s)) \Bigr|^2\\
&\quad- U(g',r,s)^2 \fullHnabla \Zeta( V(g',r,s)) \fullVnabla \fullVnabla \fullHnabla \Zeta( V(g',r,s));
\end{align*}
we write these expressions as $\term{j}_{21}(g',r,s)$, $\term{j}_{22}(g',r,s)$, $\term{j}_{23}(g',r,s)$ and $\term{j}_{24}(g',r,s)$.
It follows straight away that $\term{J}_2\leq \term{J}_{21}+\term{J}_{22}+\term{J}_{23}+\term{J}_{24}$, where
\[
\term{J}_{2i}
=\biggl|\int_{\Hn} \int_{\R^{+} \times \R^{+}} \term{j}_{2i}(g',r,s)
\,s\wrt s \,r\wrt r \wrt g'  \biggr|
\]
when $i=1,2,3,4$.

By definition,  $\Zeta'$ is bounded, and $|f\Hconv p\one_{r}(g')| \leq u^*(g') \leq \alpha$.
Further,
\[
\fullHnabla \Zeta(V(g',r,0))
= \Zeta'(V(g',r,0)) \fn \fullHnabla V(g',r,0) .
\]
We estimate the term $ \term{J}_{21}$ using integration by parts in $s$:
\begin{align*}
 \term{J}_{21}
 &\leq \frac{1}{2} \int_{\Hn}  \int_{\R^{+}}
 \bigl| u(g',r,0) \fn \fullHnabla \Zeta( V(g',r,0) )\bigr|^2
 \,r \wrt r  \wrt g'\\
 &\lesssim_{\Zeta} \alpha^2\int_{\Hn}  \int_{\R^{+}}
 \Bigl| r\fullHnabla V(g',r,0) \Bigr|^2
\,\frac{dr}{r} \wrt g'   \\
&\leq \alpha^2\int_{\Hn}  \int_{\R^{+}} \Bigl| {(1-\indifn_{E(\alpha)})\Hconv r\fullHnabla p\one_{r}}(g') \Bigr|^2
\,\frac{dr}{r} \wrt g'  \\
&\lesssim \alpha^2  \| 1-\indifn_{E(\alpha)} \|_{\Leb^2(\Hn)}^2\\
&= \alpha^2\bigl| E(\alpha)^c \bigr| .
\end{align*}

Hölder's inequality and then the argument used for $\term{I}_{2}$ show that 
\begin{align*}
\term{J}_{22}&\leq \frac{1}{2C}
\int_{\R^{+}\times\R^{+}} \int_{\Hn}
\bigl|  \fullVnabla u(g',r,s)  \fn  \fullHnabla \Zeta( V(g',r,s))  \bigr|^2
\wrt g'\,s\wrt s \,r\wrt r\\
 &\qquad\qquad +\frac{C}{2}
 \int_{\R^{+} \times\R^{+}} \int_{\Hn}
 \bigl| u(g',r,s)  \fn \fullVnabla \fullHnabla \Zeta( V(g',r,s))
 \bigr|^2 \wrt g'  \,s\wrt s \,  r \wrt r \\
&\qquad=:  \term{J}_{221}+ \term{J}_{222},
\end{align*}
say; here $C$ is large.
We see immediately that $\term{J}_{221}$ may be absorbed by $\term{J}_{2}$.
To estimate $\term{J}_{222}$, we note that $\fullVnabla \fullHnabla \Zeta( V(g',r,s))$ is equal to
\begin{align*}
 \Zeta''( V(g',r,s))  \fullHnabla V(g',r,s)\fullVnabla V(g',r,s)
+\Zeta'( V(g',r,s) ) \fullVnabla \fullHnabla V(g',r,s) .
\end{align*}
Hence, $ \term{J}_{222}$ is bounded by $\term{J}_{2221}+ \term{J}_{2222}$, where the integrands of these terms are the two summands above.
By using the support condition on $\Zeta''$ we see that 
\begin{align*}
\term{J}_{2221}&\lesssim\alpha^2\int_{\R^{+} \times \R^{+}}\int_{\Hn}
\Bigl| \fullVnabla \indifn_{E(\alpha)}\Hconv  p\two _{s} \Vconv p\one_{r}(g') \Bigr|^2 \\
&\qquad\qquad
\times \Bigl| { \fullHnabla \indifn_{E(\alpha)}\Vconv  p\one_{r}\Hconv p\two _{s}}(g') \Bigr|^2
\wrt g' \, s\wrt s \, r \wrt r\\
&\qquad\lesssim \alpha^2 \int_{\Hn} \int_{\R^{+} \times \R^{+}}
\Bigl|\oper{M}_{\Hn} \bigl( \indifn_{E(\alpha)}\Vconv s\fullVnabla p\two _{s} \bigr) (g') \Bigr|^2 \\
&\qquad\qquad
\times \Bigl| \oper{M}_{\R} \bigl(\indifn_{E(\alpha)}\Hconv  r\fullHnabla p\one_{r} \bigr)(g') \Bigr|^2
\,\frac{ds}{s} \,\frac{dr}{r} \wrt g' ,
\end{align*}
where $\oper{M}_{\Hn}$ and $\oper{M}_{\R}$ are the Hardy--Littlewood maximal functions on $\Hn$ and $\R$.
By Hölder's inequality,
\begin{align*}
 \term{J}_{2221}
&\lesssim \alpha^2\int_{\Hn}  \int_{\R^{+}}      \Bigl|  \oper{M}_{\Hn}\bigl(  s\fullVnabla (\indifn_{E(\alpha)}\Vconv p\two _{s}) \bigr) (g') \Bigr|^2 \frac{ds}{s}  \\
&\qquad \times \int_{\R^{+}} \Bigl|  \oper{M}_{\R}\bigl( r\fullHnabla (\indifn_{E(\alpha)}\Hconv  p\one_{r}) \bigr)(g') \Bigr|^2 \frac{dr}{r}  \wrt g' \\
&\leq \alpha^2\Biggl(\int_{\Hn} \biggl( \int_{\R^{+}}      \Bigl|
\oper{M}_{\Hn}\bigl(  {s\fullVnabla (\indifn_{E(\alpha)}\Vconv p\two _{s})} \bigr) (g') \Bigr|^2 \frac{ds}{s} \biggr)^2 \wrt g' \Biggr)^{1/2}\\
&\qquad \times \Biggl(\int_{\Hn} \biggl(  \int_{\R^{+}} \Bigl|  \oper{M}_{\R}\bigl( {r\fullHnabla (\indifn_{E(\alpha)}\Hconv  p\one_{r})} \bigr)(g') \Bigr|^2 \frac{dr}{r}  \biggr)^2 \wrt g' \Biggr)^{1/2}\\
&\leq \alpha^2\Biggl( \int_{\Hn} \biggl( \int_{\R^{+}}      \Bigl| {  s\fullVnabla (1-\indifn_{E(\alpha)})\Vconv p\two _{s} } (g') \Bigr|^2 \frac{ds}{s} \biggr)^2 \wrt g' \Biggr)^{1/2}\\
&\qquad \times \Biggl(\int_{\Hn} \biggl(  \int_{\R^{+}} \Bigl| {(1-\indifn_{E(\alpha)})\Hconv r\fullHnabla p\one_{r}} (g') \Bigr|^2 \frac{dr}{r}\   \biggr)^2 \wrt g' \Biggr)^{1/2}\\
&\leq \alpha^2 \|1-\indifn_{E(\alpha)}\|_{\Leb^4(\Hn)}^2 \times  \|1-\indifn_{E(\alpha)}\|_{\Leb^4(\Hn)}^2\\
&= \alpha^2 \bigl| E(\alpha)^c \bigr|,
\end{align*}
where the third inequality follows from the Fefferman--Stein vector-valued inequality \cite{CFefSte71} for the Hardy--Littlewood maximal function and then replacement of $\indifn_{E(\alpha)}$ by $1-\indifn_{E(\alpha)}$ (which is possible since both $s\fullVnabla p\two _{s} $ and $r\fullHnabla p\one_{r}$ have mean zero); the last inequality follows from Littlewood--Paley theory.
By the support condition on $\Zeta'$ and the cancellation properties of $s\fullVnabla p\two _{s} $ and $r\fullHnabla p\one_{r}$, we see that
\begin{align*}
\term{J}_{2222}
&\lesssim \alpha^2 \int_{\Cn} \int_{\R^{+} \times \R^{+}}
\Bigl| { \indifn_{E(\alpha)} \Hconv \fullHnabla p\one_{r} \Vconv \fullVnabla p\two _{s}}(g')\Bigr|^2  \,s  \wrt s \,r \wrt r \wrt g' \\
&= \alpha^2 \int_{\Cn} \int_{\R^{+} \times \R^{+}}
\Bigl|{ (1-\indifn_{E(\alpha)})\Hconv r\fullHnabla p\one_{r} \Vconv s\fullVnabla p\two _{s}}  (g')\Bigr|^2 \,s  \wrt s \,r \wrt r \wrt g'  \\
&\leq \alpha^2\|1-\indifn_{E(\alpha)}\|_{\Leb^2(\Hn)}^2\\
&=  \alpha^2\bigl| E(\alpha)^c \bigr| ;
\end{align*}
the second inequality follows from Littlewood--Paley theory, as in Corollary  \ref{cor:L2-norm-of-sqfnop-bounded}.

The term $\term{J}_{23}$ may be handled in the same way as $\term{J}_{2222}$.

Finally, we turn to $\term{J}_{24}$.
Note that
\begin{align*}
&\fullHnabla \Zeta(V(g',r,s)) \fn\fullVnabla \fullVnabla \fullHnabla \Zeta(V(g',r,s)) \\
&\qquad= \Zeta'(V(g',r,s)) \fn \Zeta'''(V(g',r,s))
|\fullHnabla V(g',r,s)|^2 | \fullVnabla V(g',r,s)|^2 \\
&\qquad\qquad + 2 \Zeta'(V(g',r,s)) \fn \Zeta''(V(g',r,s)) \fn \fullHnabla V(g',r,s) \\
&\qquad\qquad\fn \fullVnabla V(g',r,s) \fn \fullHnabla \fullVnabla V(g',r,s).
\end{align*}
Hence $\term{J}_{24}$ is bounded by $\term{J}_{241}+\term{J}_{242}$, where these terms involve the two summands above in the integrands.
The term $\term{J}_{241}$ may be handled in the same way as $\term{J}_{2221}$.
Hölder's inequality implies that $\term{J}_{242}$ may be dominated by a multiple of
\begin{align*}
&\alpha^2 \int_{\Hn} \int_{\R^{+} \times \R^{+}}
\bigl| \fullVnabla V(g',r,s) \bigr|^2
\bigl|\fullHnabla V(g',r,s) \bigr|^2  \,s\wrt s \,r\wrt r \wrt g' \\
&\qquad+ \alpha^2 \int_{\Hn} \int_{\R^{+} \times \R^{+}}
\Bigl| \fullHnabla \fullVnabla V(g',r,s) \Bigr|^2  \,s\wrt s \,r\wrt r \wrt g' ,
\end{align*}
which is bounded by $C \alpha^2\bigl| E(\alpha)^c \bigr|$.

By combining all the estimates above, we prove \eqref{nontan-implies-area aim}, and so the theorem.
\end{proof}


\section{Singular integrals and the proof of Theorem A}\label{sec:Riesz-transform}
In this section, we complete the proof of Theorem A by first showing that flag Riesz transforms also characterise the flag Hardy space, and then putting everything together.

\subsection{The inclusion $\RieszHardy(\Hn) \subseteq \dissqfnHardy(\Hn)$}\label{ssec:Riesz-implies-dissqfn}

In this section, we show that the Hardy space defined by singular integrals is a subspace of the square function Hardy space.
We remind the reader of a definition.

\begin{definition}
The (tensor-valued) \emph{flag Riesz transformations} are given by
\[
\oper{R}_{(1)} = \Hnabla \HLap^{-1/2} ,
\qquad
\oper{R}_{(2)} =  \Vnabla \VLap^{-1/2}
\qquad\text{and}\qquad
\oper{R}_{\flag} = \oper{R}_{(1)} \otimes \oper{R}_{(2)}.
\]
The space $\RieszHardy(\Hn)$ is the set of all $f\in \Leb^1(\Hn)$ such that $\oper{R}_{(1)}(f)$, $\oper{R}_{(2)}(f)$ and  $\oper{R}_{\flag}(f)$ lie in $\Leb^1(\Hn)$, with norm $\norm{ f }_{\RieszHardy(\Hn)}$ defined to be
\begin{equation}\label{eq:Riesz-Hardy}
\norm{\oper{R}_{\flag}(f) }_{\Leb^1(\Hn)}
+\norm{ \oper{R}_{(1)}(f) }_{\Leb^1(\Hn)}
+\norm{ \oper{R}_{(2)} }_{\Leb^1(\Hn)}
+ \norm{ f }_{\Leb^1(\Hn)} .
\end{equation}
\end{definition}

We are going to prove a slightly more general result, for which we need some more notation.
Let $\{ \oper{R}_j : j = 0, \dots, J\}$ be a collection of simple singular integral operators on $\Hn$ that characterise the Folland--Stein--Christ--Geller Hardy space $\FSCGHardy(\Hn)$; that is, $f \in \FSCGHardy(\Hn)$ if and only if $\oper{R}_j f \in \Leb^1(\Hn)$ for each $j$.
In particular, the identity operator together with the Riesz transformations would do.
Then
\[
\norm{f}_{\FSCGHardy(\Hn)}
\eqsim \sum_{j=0}^J \norm{\oper{R}_j f}_{\Leb^1(\Hn)},
\]
by the closed graph theorem.
For more information, see \cite{ChrGel}.
We write $\oper{H}$ for the Hilbert transformation in the central variable on $\Hn$, that is,
\[
\oper{H}f(z,t) = \int_{\R} f(z, t-s) \frac{ds}{s} \,.
\]

\begin{definition}
The space $\singHardy(\Hn)$ is the set of all functions $f \in \Leb^1(\Hn)$ for which $\oper{R}_j f \in \Leb^1(\Hn)$ and $\oper{H} \oper{R}_j f \in \Leb^1(\Hn)$ when $j = 0, \dots, J$, with norm
\begin{equation}\label{eq:sing-Hardy}
\norm{ f }_{\singHardy(\Hn)}
:= \sum_{j=0}^{J} \norm{\oper{R}_j f}_{\Leb^1(\Hn)} + \norm{H \oper{R}_j f}_{\Leb^1(\Hn)}.
\end{equation}
\end{definition}

\begin{theorem}\label{thm:Riesz-Hardy-implies-sqfn-Hardy}
Let $\oper{R}_j$ and $\oper{H}$ be as above.
If $f \in \Leb^1(\Hn)$ and $\oper{R}_j f \in \Leb^1(\Hn)$ and $\oper{H} \oper{R}_j f\in \Leb^1(\Hn)$, where $j = 0, \dots, J$, then $f \in \dissqfnHardy(\Hn)$ for some suitable $\bsym\phi$, and $\norm{f}_{\dissqfnHardy(\Hn)} \lesssim \norm{ f }_{\singHardy(\Hn)}$.
\end{theorem}

\begin{proof}
We combine the characterisation of Christ and Geller with a randomisation argument.
We remind the reader that the operators $\oper{R}_j$ and $\oper{H}$ commute.

We recall from the classical theory of Hardy spaces that $f \in \fnspace{H}^1(\R)$ if and only if both $f \in \Leb^1(\R)$ and $\oper{H} f \in \Leb^1(\R)$, and that the Hörmander--Mihlin multiplier theorem on $\R$ shows that if $\Mu$ is a bounded function on $\R$, differentiable except perhaps at $0$, such that $y \mapsto y \Mu'(y)$ is also bounded, then the associated Fourier multiplier operator $f \mapsto \oper{F}^{-1} (\Mu \hat f)$ is bounded from $\fnspace{H}^1(\R)$ to $\Leb^1(\R)$ (here $\oper{F}$ indicates the Fourier transformation $f \mapsto \hat f$ on $\R$).
For the details, see, for instance, \cite{CoiWei}.
A similar result holds for the one-parameter Hardy space on $\Hn$, except that more differentiability of $\Mu$ is required; it suffices that the functions $y \mapsto y^k \Mu^{(k)}(y)$ are bounded for $k$ between $0$ and $\cdim+2$, but more derivatives may be used without any changes other than in the value of some unimportant constants.

Take a $\fnspace{C}^\infty(\R^+)$ function $\Eta$, supported in $[1/2,2]$, such that $\sum_{n \in \Z} \Eta(2^n y) = 1$ for all $y \in \R^+$.
Then by spectral theory, there are smooth functions $\phi^{(1,m)}$ on $\Hn$ and $\phi^{(2,n)}$ on $\R$ such that $f\Hconv\phi^{(1,m)} = \Eta(2^m\HLap)f$ and $f\Vconv\phi^{(2,n)} = \Eta(2^n\VLap)f$ for all $f \in \Schwartz(\Hn)$.

In this paragraph, we consider functions on $\R$.
Let $s_n:\probspace_2 \to \C$ be a collection of independent Rademacher random variables (that is, each takes the values $\pm1$ with probability one half); then for each $\omega_2\in\probspace_2$, the operator $\oper{A}_{\omega_2}$ from $\fnspace{H}^1(\R)$ to $\Leb^1(\R)$, defined (at least formally) by
\[
\oper{A}_{\omega_2} f := \sum_{n \in \Z} s_n(\omega_2) f \Vconv \phi^{(2,n)},
\]
is bounded, with norm bounded independently of $\omega_2$, by the Mikhlin--Hörmander multiplier theorem.
That is,
\[
\norm{\sum_{n \in \Z} s_n(\omega_2) f \Vconv \phi^{(2,n)} }_{\Leb^1(\R)}
\lesssim  \norm{f}_{\Leb^1(\R)} + \norm{\oper{H} f}_{\Leb^1(\R)}
\]
for all $f \in \fnspace{H}^1(\R)$, or equivalently, for all $f \in \Leb^1(\R)$ such that $\oper{H}f \in \Leb^1(\R)$.

We take a function $f$ on $\Hn$ such that $\oper{R}_j f$ and $\oper{H} \oper{R}_j f$ lie in $\Leb^1(\Hn)$, and apply the preceding observation to the functions $\oper{R}_j f(z,\cdot)$ and $\oper{H} \oper{R}_j f(z,\cdot)$.
We then integrate over $\Cn$ and deduce that
\[
\norm{\oper{R}_j \sum_{n \in \Z} s_n(\omega_2) f \Vconv \phi^{(2,n)} }_{\Leb^1(\Hn)}
\lesssim  \norm{\oper{R}_j f}_{\Leb^1(\Hn)}
        + \norm{\oper{H} \oper{R}_j f}_{\Leb^1(\Hn)} .
\]
Now $\sum_{n \in \Z} s_n(\omega_2) f \Vconv \phi^{(2,n)} \in \FSCGHardy(\Hn)$ from the Folland--Stein--Christ--Geller characterisation \cite[Theorem A]{ChrGel},  and so, similarly, if $r_m:\probspace_1 \to \C$ is another family of independent Rademacher random variables, independent of the first family, then
\[
\begin{aligned}
&\sum_{m\in \Z} \sum_{n \in \Z} r_m(\omega_1)  s_n(\omega_2) f \Vconv \phi^{(2,n)} \Hconv \phi^{(1,m)}. \\
&\qquad\qquad = \sum_{m\in \Z} r_m(\omega_1) \sum_{n \in \Z} s_n(\omega_2) f \Vconv \phi^{(2,n)} \Hconv \phi^{(1,m)},
\end{aligned}
\]
and this function lies in $\Leb^1(\Hn)$ and
\begin{align*}
&\int_{\Hn} \left| \sum_{m\in \Z} r_m(\omega_1) \sum_{n \in \Z}
    s_n(\omega_2) f \Vconv \phi^{(2,n)} \Hconv \phi^{(1,m)} (g) \right| \wrt g \\
&\qquad\qquad \lesssim \sum_{j=1}^J \left( \norm{ \oper{R}_j f}_{\Leb^1(\Hn)}
    + \norm{\oper{H} \oper{R}_j f}_{\Leb^1(\Hn)}\right).
\end{align*}
It now follows from the orthonormality of the functions $r_m$ and $s_n$, Khinchin's inequality, Minkowski's inequality, and Khinchin's inequality that
\begin{align*}
& \left( \sum_{m\in \Z} \sum_{n \in \Z}
\left| f \Vconv \phi^{(2,n)} \Hconv \phi^{(1,m)} (g) \right|^2 \right)^{1/2}  \\
&\qquad=\left( \int_{\probspace_1} \int_{\probspace_2}
\left| \sum_{m\in \Z} \sum_{n \in \Z} r_m(\omega_1) s_n(\omega_2)
    f \Vconv  \phi^{(2,n)} \Hconv \phi^{(1,m)} (g) \right|^2 \wrt \omega_2\wrt \omega_1 \right)^{1/2}  \\
&\qquad\lesssim \left( \int_{\probspace_1} \left| \int_{\probspace_2}
    \left| \sum_{m\in \Z} \sum_{n \in \Z} r_m(\omega_1) s_n(\omega_2)
        f \Vconv \phi^{(2,n)} \Hconv \phi^{(1,m)} (g) \right| \wrt \omega_2 \right|^2
        \wrt \omega_1 \right)^{1/2} \\
&\qquad\leq \int_{\probspace_2} \left( \int_{\probspace_1}
    \left| \sum_{m\in \Z} \sum_{n \in \Z} r_m(\omega_1) s_n(\omega_2)
    f \Vconv \phi^{(2,n)} \Hconv \phi^{(1,m)} (g) \right|^2
    \wrt \omega_1  \right)^{1/2} \wrt \omega_2 \\
&\qquad\lesssim \int_{\probspace_2} \int_{\probspace_1}
    \left| \sum_{m\in \Z} \sum_{n \in \Z} r_m(\omega_1) s_n(\omega_2)
    f \Vconv \phi^{(2,n)} \Hconv \phi^{(1,m)} (g) \right| \wrt \omega_1 \wrt \omega_2  .
\end{align*}
By integrating over $\Hn$ and using Fubini's theorem and the estimates above, we conclude that
\begin{align*}
& \int_{\Hn} \left( \sum_{m\in \Z} \sum_{n \in \Z}
    \left| f \Vconv \phi^{(2,n)} \Hconv \phi^{(1,m)} (g) \right|^2 \right)^{1/2} \wrt g \\
&\qquad \lesssim \iint_{\probspace_1\times\probspace_2}\int_{\Hn}
    \left| \sum_{m\in \Z} r_m(\omega_1) \sum_{n \in \Z} s_n(\omega_2)
    f \Vconv \phi^{(2,n)} \Hconv \phi^{(1,m)} (g) \right| \wrt g \wrt \omega_1 \wrt \omega_2 \\
&\qquad\leq \iint_{\probspace_1\times\probspace_2} \sum_{j=0}^J
    \left( \norm{\oper{R}_j f}_{\Leb^1(\Hn)} +
    \norm{\oper{H} \oper{R}_j f}_{\Leb^1}\right) \wrt \omega_1 \wrt \omega_1 \\
&\qquad=  \norm{f}_{\RieszHardy(\Hn)} ,
\end{align*}
as required.
\end{proof}

The reader who is uncomfortable with our formal calculations may take sums over finite subsets of $\Z$ in the arguments above, and then allow these subsets to become arbitrarily large.
It may also be worth observing that
$f \Vconv \phi^{(2,n)} \Hconv \phi^{(1,m)} = f \Hconv \phi^{(1,m)} \Vconv \phi^{(2,n)}$ in the proof above.

\subsection{Conclusion of the proof of Theorem A and remarks}\label{ssec:conclusion-Thm-1}
We have already shown in Section \ref{sec:atomic} that various singular integral operators, including the Riesz transformations, map $\atomHardy(\Hn)$ into $\Leb^1(\Hn)$, and in Section \ref{sec:sqfn} that $\dissqfnHardy(\Hn) = \atomHardy(\Hn)$ (and corresponding norm inequalities hold).
Theorem \ref{thm:Riesz-Hardy-implies-sqfn-Hardy} therefore completes our characterisation of the Hardy space by singular integrals, and completes our proof of Theorem A.

We are now able to sharpen the results about certain singular integral operators.

\begin{corollary}\label{cor:simple-singular-integrals-bounded}
The Hilbert transformation $\oper{H}$ and simple singular integral operators, as defined in Section 2.6, are bounded on $\flagHardy(\Hn)$.
\end{corollary}

\begin{proof}
Take a family of simple homogeneous singular integral operators $\oper{R}_j$ that characterise the Folland--Stein--Christ--Geller space $\FSCGHardy(\Hn)$, and $f$ in $\flagHardy(\Hn)$.
Then $\oper{R}_j f$ and $\oper{R}_j \oper{H} f$ all lie in $\FSCGHardy(\Hn)$.
As $\oper{H}^2$ is the identity, $\oper{R}_j (\oper{H} f)$ and $\oper{R}_j \oper{H} (\oper{H} f)$ all lie in $\FSCGHardy(\Hn)$, so that $\oper{H}f \in \flagHardy(\Hn)$.

Similarly, if $\oper{K}$ is a simple homogeneous singular integral operator, and $f \in \flagHardy(\Hn)$, then $\oper{R}_j \oper{K} f \in \Leb^1(\Hn)$ since the composition $\oper{R}_j \oper{K}$ is a linear combination of the identity and another simple homogeneous singular integral operator.
Further, commutativity and the same argument shows that $\oper{R}_j \oper{H} \oper{K}f = \oper{R}_j \oper{K} \oper{H} f \in \Leb^1(\Hn)$ since $\oper{H}f \in \flagHardy(\Hn)$.
We therefore conclude that $\oper{K}f \in \flagHardy(\Hn)$.
\end{proof}

\begin{remark}\label{rem:independence-of-parameters}

On the one hand, the main theorem tells us that $\flagHardy(\Hn)$ may be characterised by homogeneous flag singular integrals.
On the other hand, the definition of an atom involves cancellations, in the sense that we impose conditions of the form $a = \HLap^M \VLap^N b$ for some function $b$, and these imply that $\int_{\Hn} a(g) p(g) \wrt g = 0$ for certain polynomials of low degree.

It should now be apparent that, provided that $M \geq 1$ and $N \geq 1$, we end up with the same atomic Hardy space.
Indeed, if $f$ belongs to the flag Hardy space $\flagHardy(\Hn)$ defined with $M$ and $N$ equal to $1$, then $\oper{K}f \in \Leb^1(\Hn)$ for all flag singular integral operators, by Theorem \ref{thm:singular-integrals-bounded-on-atoms}.
Then $f$ lies in the Hardy space defined using discrete square functions, and hence has an atomic decomposition into atoms where the particles $a_R$ are of the form $\HLap^M \VLap^N b_R$, for arbitrarily large $M$ and $N$.

In a future paper, we propose to show that even less cancellation is needed.

Similarly, the other versions of the Hardy space involve auxiliary functions (usually written $\phi\one$ and $\phi\two $).
It is now also clear that the Hardy space is largely independent of these functions, as long as they satisfy the relevant decay, cancellation and invertibility conditions.
\end{remark}


\section{Applications of the main theorem}\label{sec:applications}

The main results in this section are an isomorphism of our Hardy space with the flag Hardy space of Han, Lu and Sawyer \cite{HanLuSaw}, and the consequent identifications of the dual space of $\flagHardy(\Hn)$ and of the interpolation space between $\flagHardy(\Hn)$ and $\Leb^2(\Hn)$, and proofs of  Theorems D and E.

\subsection{Characterisation of $\flagHardy(\Hn)$ by more general heat semigroups}\label{ssec:semigroups}

In the proofs of Theorems \ref{thm:Atom-implies-area} and \ref{thm:area-Hardy-implies-atom-Hardy}, we did not need to use much about the functions $\phi\one$ and $\phi\two $.
However, we did use translation and dilation arguments to reduce estimates involving particles to estimates involving particles centred at the origin and of a certain size.
This was for convenience and to simplify the geometry.

We assert that our methods also work in settings that are not translation-invariant.
Hence, if we have self-adjoint operators $\oper{L}_1$ and $\oper{L}_2$ whose associated heat semigroups have kernels with Gaussian upper bounds, analogous estimates for derivatives, and the conservation property $\expe^{-s\oper{L}_1} 1 = 1$ and $\expe^{-t\oper{L}_2}1 = 1$, then we may define the flag area function associated to $\oper{L}_1$ and $\oper{L}_2$ by the formula
\begin{equation*}
\begin{aligned}
&\oper{S}_{\flag,\oper{L}_1,\oper{L}_2} (f)(g)
:= \biggl( \iiint_{\Gamma(g)}
\bigl| (r^2\oper{L}_1 \expe^{-r^2\oper{L}_1}s^2 \oper{L}_2 \expe^{-s^2\oper{L}_2} )f(g') \bigr|^2
\wrt g' \,\frac{\wrt r}{r} \,\frac{ds}{s} \biggr)^{1/2},
\end{aligned}
\end{equation*}
and define the flag  Hardy space $\fnspace{H}^1_{\flag,\oper{L}_1,\oper{L}_2}(\Hn)$ associated to $\oper{L}_1$ and $\oper{L}_2$ to be the completion in the norm $\|\oper{S}_{\flag,\oper{L}_1,\oper{L}_2} (f)\|_{\Leb^1(\Hn)}$ of the space $\{ f\in \Leb^2(\Hn) : \oper{S}_{\flag,\oper{L}_1,\oper{L}_2} (f) \in \Leb^1(\Hn)\}$.
Our arguments (appropriately extended) show that the Hardy space $\fnspace{H}^1_{\flag,\oper{L}_1,\oper{L}_2}(\Hn)$ coincides with our $\flagHardy(\Hn)$.

This develops the theory of singular integrals with nonsmooth kernels and function spaces associated to operators, due to Duong and McIntosh \cite{DuoMcI}, Duong and Yan \cite{DuoYan1, DuoYan2}, Hofmann and Mayboroda \cite{HofMay}, and others.
We expect that our Hardy space may be applied to the study of various types of differential equations on $\Hn$.

\subsection{Comparison with the Han--Lu--Sawyer Hardy space $\HLSHardy(\Hn)$}\label{ssec:HLS}

By definition and Theorem A, the reflection map $\oper{R}: f \mapsto f((\cdot)^{-1})$ is a linear bijection of our flag Hardy space $\flagHardy(\Hn)$ and the Han--Lu--Sawyer \cite{HanLuSaw} space $\HLSHardy(\Hn)$.
This correspondence has several interesting consequences.

\begin{prop}
The complex interpolation spaces between our flag Hardy space $\flagHardy(\Hn)$ and the Lebesgue space $\Leb^2(\Hn)$ are the Lebesgue spaces $\Leb^p(\Hn)$ where $1 < p < 2$.
\end{prop}

\begin{proof}
The map $\oper{R}$ from our flag Hardy space $\flagHardy(\Hn)$ to the Han--Lu--Sawyer \cite{HanLuSaw} space $\HLSHardy(\Hn)$ is an isometry of all the Lebesgue spaces $\Leb^p(\Hn)$.
Our interpolation theorem is now an immediate corollary of theirs.
\end{proof}

We define the space $\flagBMO(\Hn)$ to be the reflected version of the space of \cite{HanLuSaw}.

\begin{definition}\label{def:flag-BMO}
Let $\ancestor\phi\one$ on $\Hn$ and $\ancestor\phi\two $ on $\R$ be  Schwartz functions,, and define $\phi\one := \HLap^M \ancestor\phi\one$ and $\phi\two := \VLap^N \ancestor\phi\two $, where $M$ and $N$ are sufficiently large.
The flag BMO space $\flagBMO(\Hn)$ is the set of (equivalence classes of) measurable functions $b$ on $\Hn$ for which there exists a constant $C$ such that
\[
\lpar \frac{1}{|E|} \sum_{R \in \maxrect(E)}
\int_{0}^{\heit(R)} \int_{0}^{\wid(R)} \int_{R}
\abs| b \Hconv \phi\one_r \Vconv \phi\two _s (g)|
\wrt g \,\frac{dr}{r} \,\frac{ds}{s} \rpar^{1/2}
\leq C
\]
for all open subsets $E$ of finite measure in $\Hn$.
The ``norm'' $\norm{b}_{\flagBMO}$ is the least possible value of $C$ in the expression above.
\end{definition}

Note that there is a finite dimensional space of polynomials whose ``norm'' is $0$.
Note also that our conditions on $\phi\one$ and $\phi\two$ are somewhat imprecise and probably much stronger than is actually needed.
To appeal to the results of \cite{HanLuSaw}, we need to match their conditions.

\begin{prop}
The dual space of the flag Hardy space $\flagHardy(\Hn)$ is the space $\flagBMO(\Hn)$.
\end{prop}

\begin{proof}
This follows from the identification of our flag Hardy space with that of \cite{HanLuSaw}, and their proof of duality.
\end{proof}

\begin{corollary}\label{cor S Lp 2}
Suppose that $1<p<\infty$.
Then
\[
 \norm{ \areaopphi (f) }_{\Leb^p(\Hn)} \lesssim_p  \norm{ f }_{\Leb^p(\Hn)}
 \qquad\forall f\in \Leb^p(\Hn).
\]
\end{corollary}

\begin{proof}
From Theorem A and Proposition \ref{prop:Lusin-L2-bounded}, $\areaopphi$ is bounded from $\areaHardy(\Hn)$ to $\Leb^1(\Hn)$ and from $\Leb^2(\Hn)$ to $\Leb^2(\Hn)$, and hence by interpolation $\areaopphi$ is bounded on $\Leb^p(\Hn)$ when $1<p<2$.
A standard duality argument shows that $\areaopphi$ is also bounded on $\Leb^p(\Hn)$ when $2<p<\infty$.
\end{proof}

\begin{corollary}
Suppose that $(\phi\one,\phi\two)$ is a Poisson bounded pair, as in Definition \ref{def:sqfnHardyphi}.
Then the sublinear maps $\ctssqfnopphi$ and $\dissqfnopphi$ are bounded on $\Leb^p(\Hn)$ when $1<p<\infty$.
\end{corollary}
\begin{proof}
This also follows from Theorem A, interpolation and duality.
\end{proof}

\subsection{$\flagHardy(\Hn)$ is a proper subspace of the Hardy space $\FSCGHardy(\Hn)$}\label{ssec:flag-subset-FSCG}

The singular integral characterisation of our Hardy space $\flagHardy(\Hn)$ shows immediately that it is the subspace of the Folland--Stein--Christ--Geller Hardy space $\FSCGHardy(\Hn)$ of functions $f$ such that both $f$ and its Hilbert transform $\oper{H}f$ lie in $\FSCGHardy(\Hn)$.

To see that it is a proper subspace, we consider the function $a$, defined by
\[
a(z,t) = \psi(z) \fn \phi(t)
\qquad\forall (z,t)\in \Hn,
\]
where $\phi$ in $\fnspace{C}^{\infty}(\R)$ is supported in $(-1,1)$ and is nonnegative, while $\psi$ in $\fnspace{C}^{\infty}(\Cn)$ is supported in the unit ball of $\Cn$; further
\[
\int_{\Cn} \psi(z) \wrt z =0
\qquad\text{and}\qquad
\int_{\Cn} |\psi(z)| \wrt z = 1
\qquad\text{while}\qquad
\int_{\R} \phi(t) \wrt t=1.
\]
Then $a$ is a multiple of an $\FSCGHardy(\Hn)$ atom, but $\norm{a^{*}}_{\Leb^1(\Hn)} = \infty$, where $a^*$ is the radial maximal function of $a$, and hence $a \notin \flagHardy(\Hn)$.

This proves Theorem C.

\subsection{Sharp endpoint boundedness of certain flag singular integrals}
\label{ssec:flag-sing-ints}

Nagel, Ricci, Stein and Wainger \cite{NagRicSteWai1, NagRicSteWai2}  examined flag singular integral operators in considerable detail. 

\begin{definition}
A general flag singular integral kernel is a distribution of the form 
\[
\sum_{j,k=\in \Z} \phi^{(1,j)}_{2^j} \Hconv \phi^{(2,k)}_{2^k}
\qquad\text{or equivalently}\qquad
\lpar \sum_{j=\in \Z} \phi^{(1,j)}_{2^j} \rpar \Hconv \lpar \sum_{k=\in \Z}\phi^{(2,k)}_{2^k} \rpar,
\]
where the $\phi^{(1,j)}$, where $j \in \Z$, form a uniformly bounded family of functions with mean $0$ in $\fnspace{C}^\infty(\Hn)$, and the $\phi^{(2,k)}$, where $k \in \Z$, form a uniformly bounded family of functions with mean $0$ in $\fnspace{C}^\infty(\R)$.

A general flag singular integral operator is a convolution with such a distribution.
\end{definition}

It is possible to show that the collection of all flag singular integral kernels is an algebra under convolution, by showing the collection of distributions of the form $\sum_{j=\in \Z} \phi^{(1,j)}_{2^j} $ is an algebra, that the collection of distributions of the form $\sum_{k=\in \Z} \phi^{(2,k)}_{2^k} $ is also an algebra, and observing that the distributions from the first collection commute with those from the second collection.

Both the Phong--Stein and Cauchy--Szegõ operators are convolutions with general flag singular integral kernels.
Hence their boundedness on $\flagHardy(\Hn)$ is an immediate consequence of the following result.

\begin{theoremD*}
General flag singular integral operators are bounded on $\flagHardy(\Hn)$.
\end{theoremD*}

\begin{proof}
A slight modification of the proof of Theorem \ref{thm:singular-integrals-bounded-on-atoms} (see the remark following the proof of this theorem) shows that general flag singular operators map $\flagHardy(\Hn)$ to $\Leb^1(\Hn)$. 
The fact that flag singular integral operators form an algebra then implies that general flag singular operators map $\flagHardy(\Hn)$ to itself, by an argument like that of the proof of Corollary \ref{cor:simple-singular-integrals-bounded}, and proves Theorem D.
\end{proof}

\subsection{Sharp endpoint boundedness of Marcinkiewicz multipliers}\label{ssec:proof-Marcinkiewicz}

In this section, we prove Theorem~E; we begin with a lemma.

\begin{lemma}\label{lemma m 444.2}
Let $d_{(1)}$ be the gauge distance on $\Hn$ and $d_{(2)}$ be the Euclidean distance on $\R$.
Suppose that $\Phi \in \fnspace{C}^{\infty}(\R)$ and $\supp \Phi \subseteq [1/2,2]$, that $E$ and $E'$ are closed subsets of $\Hn$, and that $E_\R$ and $E'_\R$ are closed subsets of $\R$.
Then for any $s\in \R^+$ and $p\in [1,\infty]$,
\begin{align}\label{off-5}
\|\Phi(2^{-j}\HLap)\|_{\Leb^p(E)\to \Leb^p(E')}
&\lesssim_s (1+2^{j/2}d_{(1)}(E,E'))^{-s},\\
\label{off-6}
\|\Phi( 2^{-j} i\oper{T})\|_{\Leb^p(E_\R)\to \Leb^p(E'_\R)}
&\lesssim_s (1+2^{j}d_{(2)}(E_\R,E'_\R))^{-s}.
\end{align}
\end{lemma}

\begin{proof}
The convolution kernel $k_{\Phi(\HLap)}$ of $\Phi(\HLap)$ is a Schwartz function.
Hence
\[
\abs|k_{\Phi(2^{-j}\HLap)}(g)|
\lesssim_s 2^{\hdim j/2} (1+2^{j/2} \norm{g})^{-2s}
\]
for all $g \in \Hn$ and all $s\in\R^+$.
We assume that $s > \hdim$ and that $p \in (1,\infty)$.
If $f\in \Leb^p(\Hn)$ and $\supp f\subseteq E$ while $f'\in \Leb^{p'}(\Hn)$ and $\supp f' \subseteq E'$, then
\begin{align*}
&(1+2^{j/2}d_{(1)}(E,E'))^{s}\abs| \int_{\Hn}\Phi(2^{-j}\HLap)(f)(g') f'(g')\wrt g' | \\
&\qquad= (1+2^{j/2}d_{(1)}(E,E'))^{s}\abs|\iint_{\Hn\times \Hn} k_{\Phi(2^{-j}\HLap)}(g^{-1}g') f(g) f'(g) \wrt g\wrt g' | \\
&\qquad\lesssim_s  \iint_{E \times E'} \frac{2^{\hdim j/2} (1+2^{j/2}d_{(1)}(E,E'))^{s}} {(1+2^{j/2}d_{(1)}(g,g'))^{2s}} \abs|f(g)| \abs|f'(g)| \wrt g' \wrt g\\
&\qquad\leq  \iint_{E \times E'} \frac{2^{\hdim j/2}} {(1+2^{j/2}d_{(1)}(g,g'))^{s}} \abs|f(g)| \abs|f'(g)| \wrt g' \wrt g\\
&\qquad\leq  \left(\iint_{E \times E'} 2^{\hdim j/2}(1+2^{j/2}d_{(1)}(g,g'))^{-s} |f(g)|^p\wrt g'\wrt g\right)^{1/p}\\
&\qquad\qquad\times\left(\iint_{E \times E'} 2^{\hdim j/2} (1+2^{j/2}d_{(1)}(g,g'))^{-s} |f'(g)|^{p'}\wrt g'\wrt g\right)^{1/p'}\\
&\qquad\leq \left(\int_{\Hn} \int_{\Hn} 2^{\hdim j/2}(1+2^{j/2}d_{(1)}(g,g'))^{-s}\wrt g' \abs|f(g)|^p\wrt g\right)^{1/p}\\
&\qquad\qquad\times\left(\int_{\Hn} \int_{\Hn} 2^{\hdim j/2} (1+2^{j/2}d_{(1)}(g,g'))^{-s} \wrt g' \abs|f'(g)|^{p'}\wrt g\right)^{1/p'}\\
&\qquad=  \int_{\Hn} (1+d_{(1)}(e,g'))^{-s}\wrt g'
\norm{f}_{\Leb^p(\Hn)} \norm{f'}_{\Leb^{p'}(\Hn)} ,
\end{align*}
as required.
If $p=1$ or $p=\infty$ the modifications are easy.
The proof of \eqref{off-6} is similar.
\end{proof}

We restate Theorem E for convenience.
First we define the Marcinkiewicz norm of a function $\Mu: \R^+ \times \R \setminus \{0\} \to \C$.
Fix $\epsilon\in\R^+$, and set $\alpha = \cdim + \epsilon$ and $\beta = (1 + \epsilon)/2$, and set
\[
\norm{\Mu}_{\mathrm{Mar}}
:= \sup\{ \norm{ \Theta \fn\Mu(r\cdot,s\cdot) }_{L_{\alpha,\beta}^2(\R^2)} : r \in \R^+ , s \in \R\setminus \{0\} \}  .
\]

\begin{theorem}\label{thm:multiplier}
Let $\Theta$ be a nonzero compactly supported function in $\R^+ \times \R^+$.
Suppose that $\Mu:\R^+ \times \R\setminus\{0\} \to \C$ and $\norm{\Mu}_{\mathrm{Mar}}<\infty$.
Then both the operators  $\Mu(\HLap/\abs|\oper{T}|, i\oper{T})$ and $\Mu(\HLap, i\oper{T})$ are bounded on $\flagHardy(\Hn)$.
\end{theorem}

\begin{proof}
As remarked by Müller, Ricci and Stein \cite{MulRicSte2}, a partition of unity argument shows that the choice of $\Theta$ is immaterial; further, if $\Mu(\HLap/\abs|\oper{T}|, i\oper{T})$ is bounded, then so is $\Mu(\HLap, i\oper{T})$.

It is clear that if $\Mu$ satisfies the condition of the theorem, so do the functions $\Mu \indifn_{\R^+ \times \R^+}$ and $\Mu \indifn_{\R^+ \times \R^-}$.
It therefore suffices to treat functions $\Mu$ supported in the first or the second quadrant.
We treat functions $\Mu$ supported in the first quadrant; the modifications to treat the second case are very simple.

Take a smooth, $[0,1]$-valued function $\Eta$ on $\R^+$, whose support is contained in $(1/2,2)$, such that $\sum_{j \in \Z} \abs|\Eta(2^{-j} \cdot)| = 1$ in $\R^+$, and set
\[
\Mu_{j,\ell}(\lambda/\mu, \mu)
:= \Eta( 2^{-j} \lambda/\mu) \fn\Eta( 2^{-\ell}\mu) \fn\Mu(\lambda/\mu, \mu)
\qquad\forall \lambda, \mu \in \R^+ .
\]
Thus
\begin{equation}\label{eq:def-M-j-ell}
\Mu_{j,\ell}(\HLap/|\oper{T}|, \euli\oper{T})
:= \Eta( 2^{-j}\HLap/|\oper{T}|) \fn\Eta( 2^{-\ell}\euli\oper{T}) \fn\Mu(\HLap/|\oper{T}|,\euli\oper{T}) .
\end{equation}
and
\begin{equation}\label{eq:def-M-j-ell-b}
\begin{aligned}
\Mu(\HLap/|\oper{T}|, \euli\oper{T})
&= \sum_{j,\ell} \Mu_{j,\ell}(\HLap/|\oper{T}|, \euli\oper{T}) .
\end{aligned}
\end{equation}

The joint spectrum of $\HLap$ and $\oper{T}$ is the Heisenberg fan (see Section \ref{ssec:spectral-theory}), and therefore $\Eta( 2^{-j}\HLap/|\oper{T}|) = 0$ if $j$ is negative.
We assume that $j \in \N$ and $\ell \in \Z$ in what follows.

Suppose that $m ,n \in \Z$.
We define the smooth compactly supported function $\Eta^{(m)}$ on $\R^+$ by $\Eta^{(m)}(\mu) := \mu^{m} \Eta(\mu)$.
Then $\Eta^{(m)}$ and $\Eta$ satisfy similar Sobolev estimates and their supports coincide, and
\[
\lambda^m \fn\mu^n \fn\Eta( 2^{-j} \lambda/\mu) \fn\Eta( 2^{-\ell}\mu)
= 2^{m(j+\ell) + n \ell} \fn\Eta^{(m)}( 2^{-j} \lambda/\mu) \fn\Eta^{(m+n)}( 2^{-\ell}\mu) ,
\]
so by functional calculus,
\[
\HLap^{m} (i\oper{T})^{n} \Mu_{j,\ell}(\HLap/|\oper{T}|,\euli\oper{T})
=: 2^{m(j+\ell) + n \ell} \Mu_{j,\ell}^{m,n}(\HLap/|\oper{T}|,\euli\oper{T}) ,\euli\oper{T}),
\]
where 
\begin{equation}\label{eq:def-Mu-j-l-m-n}
\Mu_{j,\ell}^{m,n}(\lambda/\mu,\mu) 
= \Eta^{(m)}(2^{-j} \lambda/\mu)  \Eta^{(m+n)}(2^{-\ell} \mu) \Mu(\HLap/|\oper{T}|, \euli\oper{T}).
\end{equation}

Further, if $a_R$ is a particle associated to the shard $R$, of width $w$ and height $h$, then there is a function $c_R$ supported in $R^*$ such that $a_R = \HLap^{M}\VLap^{N} c_R$, and for all $m,n \in \N$, 
\[
\norm{ \HLap^m (\euli\oper{T})^n a _R }_{\Leb^2(\Hn)}
\lesssim_{m,n} w^{-2m} h^{-n} \norm{ a_R }_{\Leb^2(\Hn)} .
\]
It is consistent to define $\HLap^{-m} (\euli\oper{T})^{-n} a _R$ to be $\HLap^{M-m} (\euli\oper{T})^{2N-n} c _R$ when $0 \leq m \leq M$ and $0 \leq n \leq 2N$.

Hence, as the algebra of operators generated by $\HLap$ and $\euli \oper{T}$ is commutative, for all $m$ and $n$ such that $|m| \leq M$ and $|n| \leq 2N$, 
\begin{equation}\label{eq:Marcin-0a}
\begin{aligned}
&\Mu_{j,\ell}(\HLap,i\oper{T})a_R \\
&\qquad= \HLap^{m} \VLap^{n} \fn \Mu_{j,\ell} (\HLap/|\oper{T}| , \euli\oper{T}) \fn \HLap^{-m}\VLap^{-n} a_R   \\
&\qquad= \HLap^{m} \Eta( 2^{-j}\HLap/|\oper{T}|)
\fn\VLap^{n} \Eta( 2^{-\ell}i\oper{T}) \fn \Mu(\HLap/|\oper{T}| , \euli\oper{T}) \fn \HLap^{-m}\VLap^{-n} a_R   \\
&\qquad=2^{m(j+\ell) + n\ell}
\Mu_{j,\ell}^{m,n}(\HLap/|\oper{T}| , \euli\oper{T}) \HLap^{-m}\VLap^{-n} a_R    \\
&\qquad= 2^{m(j+\ell) + n\ell} (\HLap^{-m}\VLap^{-n} a_R ) \Hconv k^{m,n}_{j,\ell}    ,
\end{aligned}
\end{equation}
where $k^{m,n}_{j,\ell}$ is the kernel of the operator $\Mu_{j,\ell}^{m,n}(\HLap/|\oper{T}|, \euli\oper{T})$.
In order to exploit the factor $2^{m(j+\ell) + n\ell}$, we choose $m$ and $n$ to be positive when $j+\ell$ and $\ell$ are negative, and $0$ otherwise.

Observe that, if $a_R$ is a particle associated to a shard $R$, and $R \subset S$, then 
\begin{equation}\label{eq:weighted-kernel-a}
\begin{aligned}
&\int_{(S^*)^c} \abs| a_R \Hconv k_{j,\ell}( g) |  \wrt g \\
&\qquad= 2^{m(j+\ell) + n\ell} 
		\int_{(S^*)^c} \abs| (\HLap^{-m}\VLap^{-n} a_R ) \Hconv k^{m,n}_{j,\ell}( g) |  \wrt g  \\
&\qquad \leq 2^{m(j+\ell) + n\ell} \norm{ \HLap^{-m}\VLap^{-n} a_R }_{\Leb^1(\Hn)} 
	\sup\Bigl\{ \int_{(S^*)^c} \abs| k^{m,n}_{j,\ell}(g_1^{-1} g) |  \wrt g : g_1 \in R^*\Bigr\} ,
\end{aligned}
\end{equation}
and 
\begin{equation}\label{eq:weighted-kernel-b}
\begin{aligned}
  \norm{ \HLap^{-m}\VLap^{-n} a_R }_{\Leb^1(\Hn)} 
&\leq \abs|R|^{1/2} \norm{ \HLap^{-m}\VLap^{-n} a_R  }_{\Leb^2(\Hn)}   \\
&\lesssim \wid^{2m}(R) \heit^{2n}(R) \abs|R|^{1/2} \norm{ a_R  }_{\Leb^2(\Hn)} . 
\end{aligned}
\end{equation}

We propose to estimate the supremum in \eqref{eq:weighted-kernel-a} so as to be able to sum over $j$ and $\ell$ and apply Proposition \ref{prop:particle-bounded-function-bounded}.
Some simplifications are possible: first, by dilation and translation invariance, it suffices to suppose that the shard $R$ has centre $o$, width $1$, and height $h$, where $h > 1$.
By Proposition \ref{prop:particle-bounded-function-bounded}, it suffices to take shards $S$ with centre $o$ whose width and height are much greater than $1$ and $h$. 
By doing so, and then reducing the width of $S$ a little if necessary, we may also assume that $\heit(S) \geq 4h + 4\cdim\wid^2(S) + 4\cdim$.
Finally, we may assume that the enlargement parameter $\kappa$ is equal to $2$.

We are going to use weighted $\Leb^2$ estimates, which go back to work of Müller, Ricci and Stein \cite{MulRicSte2}, and we now introduce the relevant weight functions and a mild variation of a result of Müller, Ricci and Stein \cite{MulRicSte2}.

\begin{definition}\label{def:weights}
For $\epsilon \in \R^+$, we define the \emph{weight function} $w_{j,\ell}^{\epsilon}: \Hn \to \R^+$ by
\[
w_{j,\ell}^{\epsilon}(z,t)
:= 2^{-\cdim(j+\ell)} (1 + 2^{j+\ell}|z|^2)^{\cdim(1+\epsilon)}
    2^{-\ell} (1+2^{\ell}|t|)^{1+\epsilon}
\qquad\forall (z,t) \in \Hn.
\]
\end{definition}

\begin{lemma}\label{lemma muller stein}
Fix $\epsilon\in\R^+$, and set $\alpha = \cdim + \epsilon$ and $\beta = (1 + \epsilon)/2$. 
Let $k^{m,n}_{j,\ell}$ be the convolution kernel of the operator $\Mu^{m,n}_{j,\ell} ( \HLap/|\oper{T}|, i\oper{T})$ given by \eqref{eq:def-Mu-j-l-m-n}.
Then
\begin{equation}\label{eq:weighted-kernel-c}
\begin{aligned}
\lpar \int_{\Hn}
    \abs| k^{m,n} _{j,\ell}(g)|^2 w_{j,\ell}^{\epsilon}(g)\wrt g \rpar^{1/2}
\lesssim_{c,\epsilon,m,n} \norm{ \Mu }_{\mathrm{Mar}} .
\end{aligned}
\end{equation}
\end{lemma}

\begin{proof}
See Proposition 5.3 and Lemma 2.5 of \cite{MulRicSte2}.
\end{proof}

By Proposition \ref{prop:particle-bounded-function-bounded} and Theorem A, it suffices to find $\epsilon_1$ and $\epsilon_2$ in $\R^+$ such that
\begin{equation*}
\begin{aligned}
\int_{(S^{*})^c}
\biggl( \sum_{j,\ell} \Bigl| \Mu_{j,\ell}(\HLap/|\oper{T}|, \euli\oper{T}) a_R (g) \Bigr|^2 \biggr)^{1/2} \wrt g
\lesssim_{\Mu, \epsilon_1,\epsilon_2} \rho_{\boldsymbol\epsilon}(R,S) \abs|R|^{1/2} \norm{a_R}_{\Leb^2(\Hn)}
\end{aligned}
\end{equation*}
for all particles $a_R$ of the form $\HLap^M \VLap^N b_R$ associated to $R \in \rect$, and all $S \in \rect$ that contain $R$, where $M > \cdim/2$ and $N \geq 1$.
We may and shall assume that $R \subseteq  T(o,1/2, h/2)$ and $R^{*} = T(o,1, h)$, so $|R| \eqsim |R^*| 
\eqsim h \geq 1$, and that $S^{*} = T(o,r^*,h^*)$, where $r^* \geq 2\cdim+1$ and $h^* \geq (2\cdim+1)^2h$.
In light of the discussion leading to \eqref{eq:Marcin-0a}, it will suffice to prove that
\begin{equation}\label{eq:Marcin-0}
\begin{aligned}
\int_{(S^{*})^c}
\biggl( \sum_{j,\ell} \Bigl| c_R \Hconv k^{m,n}_{j, \ell}(g) \Bigr|^2 \biggr)^{1/2} \wrt g
\lesssim_{\Mu, \epsilon_1,\epsilon_2} \rho_{\boldsymbol\epsilon}(R,S) \abs|R|^{1/2} \norm{a_R}_{\Leb^2(\Hn)}
\end{aligned}
\end{equation}

To prove \eqref{eq:Marcin-0}, we define four regions in $\Hn$ using the Euclidean metric:
\begin{align*}
E_1 &:=  \{(z,t)\in \Hn: |z| \geq r^*, |t| \geq 8\cdim(h+|z|)\} \\
E_2 &:=  \{(z,t)\in \Hn: 2 \leq |z| \leq r^*, |t| \geq h^* \}   \\
E_3 &:=  \{(z,t)\in \Hn: |z| \geq r^*, |t|\leq 8\cdim(h+|z|)\} \\
E_4 &:=  \{(z,t)\in \Hn: |z|\leq 2, |t| \geq h^*\}.
\end{align*}
Clearly $E_1\cup E_2\cup E_3\cup E_4\cup S^*=\Hn$.
The bulk of the proof is the estimation of the integral on the left-hand side of \eqref{eq:Marcin-0} on the four regions $E_1$, $E_2$, $E_3$ and $E_4$.

We begin by estimating over $E_1$.
From the inclusion $\ell^1 \subseteq \ell^2$, we see that
\begin{equation}\label{eq:multi-1-1}
\begin{aligned}
\int_{E_1}  \biggr( \sum_{j,\ell} \Bigl|  c_R \Hconv k^{m,n}_{j, \ell}(g)  \Bigr|^2 \biggr)^{1/2} \wrt g
&\leq \sum_{j,\ell} \int_{E_1}
\Bigl| c_R \Hconv k^{m,n}_{j,\ell}(g') \Bigr|  \wrt g' .
\end{aligned}
\end{equation}

From the Cauchy--Schwarz inequality, the definition of $k^{m,n}_{j,\ell}$ after \eqref{eq:Marcin-0a},  and Lemma~\ref{lemma muller stein},
\begin{equation}\label{eq:multi-1-2}
\begin{aligned}
&  \int_{E_1}
\Bigl| c_R \Hconv k^{m,n}_{j,\ell}(g') \Bigr|  \wrt g' \\
&\qquad\leq  \int_{R^*} \int_{E_{1}}
\bigl| k^{m,n}_{j,\ell}(g^{-1}g') \bigr|
\abs| c_{R}(g)|  \wrt g' \wrt g \\
&\qquad\leq  \int_{R^*} \biggl(\int_{E_{1}}
\frac{1}{w^\epsilon_{j,\ell}(g^{-1}g')}  \wrt g' \biggr)^{1/2} \\
&\qquad\qquad\times \lpar\int_{E_{1}} w^\epsilon_{j,\ell}(g^{-1}g'). 
\abs| k^{m,n}_{j,\ell}(g^{-1}g')|^2   \wrt g' \rpar^{1/2}
\abs| c_{R}(g)| \wrt g \\
&\qquad\lesssim_{\Mu}  2^{m(j+\ell) + n\ell} \int_{R^*} \biggl(\int_{E_{1}}
 \frac{1}{w^\epsilon_{j,\ell}(g^{-1}g')}
 \wrt g' \biggr)^{1/2}\abs| c_{R}(g)| \wrt g .
\end{aligned}
\end{equation}
We will use several different versions of this calculation later.

Recall that $(z,t)^{-1} \Hprod (z',t') = (z'-z, t' - t - S(z,z'))$.
If $(z,t) \in R^* = T(0,1,h)$ and $(z',t') \in E_1$, then
\begin{equation}\label{eq:multi-1-3}
|z+z'| \geq |z'| - |z| \geq r^*-1
\qquad\text{and}\qquad
|t' - t - S(z,z')| \geq |t'| - |t| - |S(z,z')| \geq h,
\end{equation}
and so if $g \in R^*$ and $g' \in E_1$, then
\[
g^{-1}g' \in E_1^*
:= \{(z'',t'') \in\Hn : |z''| \geq r^*-1, |t''| \geq h \},
\]
whence
\begin{equation}\label{eq:multi-1-4}
\begin{aligned}
&\int_{E_{1}}
 \frac{1}{w^\epsilon_{j,\ell}(g^{-1}g')} 
 \wrt g' \\
&\qquad\leq
\int_{E_{1}^*}
 \frac{1}{w^\epsilon_{j,\ell}(g'')}  \wrt g'' \\
&\qquad= \int_{|z|>r^* - 1} 2^{\cdim(j+\ell)} (1 + 2^{j+\ell}|z|^2)^{-\cdim(1+\epsilon)} \wrt z
   \int_{|t|>h} 2^{\ell} (1+2^{\ell}|t|)^{-1-\epsilon}  \wrt t \\
&\qquad= \int_{|z|>2^{(j+\ell)/2}(r^*-1)}  (1 + |z|^2)^{-\cdim(1+\epsilon)} \wrt z
   \int_{|t|>2^\ell h} (1+|t|)^{-1-\epsilon}  \wrt t \\
&\qquad\lesssim_{\epsilon} \bigl(1 + 2^{(j+\ell)/2}r^*\bigr)^{-\cdim\epsilon}
    \bigl(1 + 2^\ell h\bigr)^{-\epsilon/2} .
\end{aligned}
\end{equation}

From inequalities \eqref{eq:multi-1-2} to \eqref{eq:multi-1-4}, we conclude that
\[
\begin{aligned}
&\int_{E_1 }  \Bigl| c_R \Hconv [\indifn_{E_1^*}k^{m,n}_{j,\ell}](g') \bigr|  \wrt g'. \\
&\qquad\lesssim_{\Mu} 2^{m(j+\ell)} \bigl(1 + 2^{(j+\ell)/2}r^*\bigr)^{-\cdim\epsilon}
    \bigl(1 + 2^\ell h \bigr)^{-\epsilon/2}  2^{n\ell} h^{n+ 1/2}
    \bignorm{ a_R }_{\Leb^2(\Hn)} ;
\end{aligned}
\]
We sum these inequalities as indicated by \eqref{eq:multi-1-1}, taking $m$ to be $M$ if $j+\ell \leq 0$ and $0$ otherwise, so that $2^{m(j+\ell)} = \min\{1, 2^{M(j+\ell)} \}$, and taking $n$ to be $N$ if $2^\ell h  \leq 1$ and $0$ otherwise, so that $2^{n\ell} = \min\{1, 2^{N\ell} \}$.  
Then
\begin{equation}\label{eq:multi-1-5}
\begin{aligned}
&\int_{E_1}  \biggr(
\sum_{j,\ell} \Bigl|  c_R \Hconv k^{m,n}_{j,\ell} (g) \Bigr|^2 
\biggr)^{1/2} \wrt g   \\
&\qquad\leq \sum_{j,\ell} \int_{E_1}
\Bigl| c_R \Hconv k^{m,n}_{j,\ell}(g') \Bigr|  \wrt g' \\
&\qquad\lesssim_{\epsilon,M, N} \Bigl(\frac{1}{r^*}\Bigr)^{2M\cdim\epsilon/(2M+\cdim\epsilon)}
\abs|R|^{1/2} \norm{a_R}_{\Leb^2(\Hn)} .
\end{aligned}
\end{equation}
this is an estimate of the required form.
We estimated the sum in the above expression as follows: define
\[
\begin{gathered}
J_1 := \{ j : 2^{j + \ell} < 1/r^*\}, 
\qquad J_2 := \{ j : 1/r^* \leq 2^{j + \ell} < 1\}
\qquad
J_3 := \{ j : 2^{j +\ell} \geq 1\},
\\
L_1 := \{ \ell : 2^\ell h < 1 \}
\qquad
L_2 := \{ \ell : 2^\ell h \geq 1 \};
\end{gathered}
\]
then 
\begin{equation}\label{eq:multi-1-5-a}
\begin{aligned}
&\sum_{j} \min\{ 1, 2^{M(j+\ell)}\}
(1 + 2^{(j+\ell)/2}r^*)^{-\cdim\epsilon} \} \\
&\qquad\eqsim \lpar
\sum_{j \in J_1} 2^{M(j+\ell)}
+ \sum_{j \in J_2} 2^{M(j+\ell)} (2^{(j+\ell)/2}r^*)^{-\cdim\epsilon}
+ \sum_{j \in J_3} (2^{(j+\ell)/2}r^*)^{-\cdim\epsilon} \rpar \\
&\qquad\eqsim_{M,\epsilon} \Biglpar
(r^*)^{-2M} + (r^*)^{-\cdim\epsilon} + (r^*)^{-\cdim\epsilon} \Bigrpar ,
\end{aligned}
\end{equation}
whence 
\begin{equation}\label{eq:multi-1-5-b}
\begin{aligned}
&\sum_{j,\ell} \min\{ 1, 2^{M(j+\ell)}\}
(1 + 2^{(j+\ell)/2}r^*)^{-\cdim\epsilon} 
( 1 + 2^\ell h)^{-\epsilon/2} \min\{1, (2^{\ell}h)^N \}
\\
&\qquad\eqsim_{M,\epsilon}
 (r^*)^{-\cdim\epsilon}
\lpar \sum_{\ell \in L_1} (2^{\ell}h)^N
+ \sum_{\ell \in L_2} (2^\ell h)^{-\epsilon/2}  \rpar\\
&\qquad\eqsim_{N,\epsilon} \Bigl(\frac{1}{r^*}\Bigr)^{\cdim\epsilon},
\end{aligned}
\end{equation}

The next step of the proof is to estimate the integral over $E_2$.
The argument is very similar to that used to treat the region $E_1$.
Indeed, apart from replacing $E_1$ by $E_2$, we replace the argument for \eqref{eq:multi-1-3} and \eqref{eq:multi-1-4} by the observation that if $(z,t) \in R^* = T(0,1,h)$ and $(z',t') \in E_2$, then
\begin{equation*}
|z' - z| \geq |z'| - |z| \geq 1
\qquad\text{and}\qquad
|t'-t-S(z,z')| \geq |t'| - |t| - |S(z,z')| \geq h^*/4,
\end{equation*}
so if $g \in R^*$ and $g' \in E_2$, then
\[
g^{-1}g' \in E_2^*
:= \{(z'',t'') \in\Hn : |z''| \geq 1, |t''| \geq h^*/4\},
\]
whence
\begin{equation*}
\begin{aligned}
\int_{E_{2}}  \frac{1}{w^\epsilon_{j,\ell}(g^{-1}g')}  \wrt g'
&\leq
\int_{E_{2}^*}  \frac{1}{w^\epsilon_{j,\ell}(g'')}  \wrt g'' \\
&= \int_{|z|>1} 2^{\cdim(j+\ell)} (1 + 2^{j+\ell}|z|^2)^{-\cdim(1+\epsilon)} \wrt z
   \int_{|t|>h^*/4} 2^{\ell} (1+2^{\ell}|t|)^{-1-\epsilon}  \wrt t \\
&\lesssim_\epsilon \bigl(1 + 2^{(j+\ell)/2}\bigr)^{-2\cdim\epsilon}
    \bigl(1 + 2^\ell h^*\bigr)^{-\epsilon} .
\end{aligned}
\end{equation*}
By making these modifications and summing much as before, we conclude that
\begin{equation*}
\int_{E_2} \biggl(
\sum_{j,\ell} \Bigl|  c_R \Hconv k^{m,n}_{j,\ell} (g') \Bigr|^2
\biggr)^{1/2}  \wrt g'
\lesssim \Bigl(\frac{h}{h^*}\Bigr)^{\epsilon} \abs|R|^{1/2} \norm{a_R}_{\Leb^2(\Hn)},
\end{equation*}
which is again of the required form.

Now we estimate the integral over $E_3$.
We decompose this region into dyadic pieces $E_{3,k}$, given by
\begin{align*}
E_{3,k} &:=\{(z,u): z \in A_k  , |u|\leq 8\cdim(h+|z|)\}.\\
\noalign{\noindent{where}}
A_k &:= \{z \in \Cn : 2^{k}r^*<|z|\leq 2^{k+1} r^* \}.
\end{align*}
For later purposes, we define $A_k^* := \{z \in \Cn : 2^{k}r^*-1<|z|\leq 2^{k+1} r^* + 1 \}$.

We take a smooth compactly supported real-valued function $\Phi$ on $\R\setminus\{0\}$ such that $\Phi \Eta = \Eta$, so that $\Eta(2^{ \ell}i\oper{T})  = \Phi(2^{ \ell}i\oper{T}) \Eta(2^{ \ell}i\oper{T})$, and define
\begin{equation}\label{eq:cRell-def}
c_{R,\ell} := \Phi(2^{ \ell}i\oper{T}) c_R ;
\end{equation}
we do not make explicit the dependence of $c_{R,\ell}$ on choices of $m$ and $n$.
From spectral theory,
\begin{equation}\label{eq:cRell-LP}
\biggl( \sum_{\ell} \norm{c_{R,\ell}}_{\Leb^2(\Hn)}^2  \biggr)^{1/2}
\lesssim_{\Phi} \norm{c_R}_{\Leb^2(\Hn)}.
\end{equation}
The operator $\Phi(2^{ \ell}i\oper{T})$ is a convolution in the central variable with a Schwartz function $\phi_\ell$, say, so $c_{R,\ell} = c_R \Vconv \phi_{\ell}$, and so $c_{R, \ell}$ is supported in $\{(z,t): |z| \leq 1\}$.
Now, from \eqref{eq:Marcin-0a},
\begin{align*}
&\int_{E_{3,k}} \biggl( \sum_{j,\ell} \abs|  c_R \Hconv k^{m,n}_{j,\ell} (g') |^2 \biggr)^{1/2} \wrt g'. \\
&\qquad=\int_{E_{3,k}} \biggl( \sum_{j,\ell} \abs|  c_{R,\ell} \Hconv k^{m,n}_{j,\ell} (g') |^2 \biggr)^{1/2} \wrt g' \\
&\qquad=\int_{E_{3,k}} \biggl( \sum_{j,\ell} \abs|  [\indifn_{G}c_{R,\ell}] \Hconv k^{m,n}_{j,\ell} (g') |^2 \biggr)^{1/2} \wrt g'
\\&\qquad\qquad+
\int_{E_{3,k}} \biggl( \sum_{j,\ell} \abs|  [\indifn_{C(g')}c_{R,\ell}] \Hconv k^{m,n}_{j,\ell} (g') |^2 \biggr)^{1/2} \wrt g' \\
&\qquad\qquad+
\int_{E_{3,k}} \biggl( \sum_{j,\ell} \abs|  [\indifn_{D(g')}c_{R,\ell}] \Hconv k^{m,n}_{j,\ell} (g') |^2 \biggr)^{1/2} \wrt g' \\
&\qquad =:  \term{I}^{G}_{k} +  \term{I}^{C}_{k} + \term{I}^{D}_{k},
\end{align*}
say, where we have split the region $\{(z,t) : |z| \leq 1, t\in \R\}$ in which $c_{R,\ell}$ is supported into two subsets, $G$ and $B$, and then split $B$ into two subsets which depend on $g'$, as follows:
\begin{equation} \label{eq:defGBCD}
\begin{aligned}
G&:= \{(z,t) : |z| \leq 1, |t| \geq 2h\}  \\
B&:= \{(z,t) : |z| \leq 1, |t| < 2h\}  \\
C(g')&:=\{g \in B: P_{\R}(g^{-1}g') \notin [-2h,2h] \} \\
D(g')&:=\{g \in B: P_{\R}(g^{-1}g') \in [-2h,2h]\}.
\end{aligned}
\end{equation}
Here $P_\R$ denotes the (noncanonical) projection of $\Hn$ onto $\R$ given by $(z,t) \mapsto t$.

First we consider the terms $\term{I}^{G}_{k}$.
From the inclusion $\ell^1 \subseteq \ell^2$, and a calculation like \eqref{eq:multi-1-2},
\begin{equation}\label{eq:multi186}
\begin{aligned}
\term{I}^{G}_{k}
&\leq\sum_{j,\ell} \int_{E_{3,k}} \abs|  [\indifn_{G}c_{R,\ell}] \Hconv k^{m,n}_{j,\ell} (g') | \wrt g' \\
&\leq \sum_{j,\ell}  \int_{\Hn} \int_{E_{3,k}}
\bigl| k^{m,n}_{j,\ell}(g^{-1}g') \bigr|
\abs| \indifn_{G} c_{R,\ell}(g)|  \wrt g' \wrt g \\
&\lesssim_{\Mu} \sum_{j,\ell}  2^{m(j+\ell) + n\ell} \int_{\Hn} \biggl(\int_{E_{3,k}}
 \frac{1}{w^\epsilon_{j,\ell}(g^{-1}g')}
 \wrt g' \biggr)^{1/2}\abs| \indifn_{G} c_{R,\ell}(g)| \wrt g .
\end{aligned}
\end{equation}
When $g \in G$ and $g' \in E_{3,k}$, $g^{-1}g' \in P^{-1}A^*_k$, so, much as argued from \eqref{eq:multi-1-3} to \eqref{eq:multi-1-4},
\begin{equation*}
\begin{aligned}
\biggl(\int_{E_{3,k}}
 \frac{1}{w^\epsilon_{j,\ell}(g^{-1}g')}
 \wrt g' \biggr)^{1/2}
&\leq
\biggl(\int_{P^{-1}\region{A}_{k}^*}
 \frac{1}{w^\epsilon_{j,\ell}(g')}
 \wrt g' \biggr)^{1/2}. \\
&\lesssim \bigl(1 + 2^{(j+\ell+2k)/2}r^*\bigr)^{-\cdim\epsilon}
\qquad\forall g \in G.
\end{aligned}
\end{equation*}
Since $\phi$ is a Schwartz function, much as shown in Lemma~\ref{lemma m 444.2}, for all $S \in \N$,
\[
\begin{aligned}
\int_{G} \abs| c_{R,\ell}(g)| \wrt g
&\leq \int_{\R\setminus [-2h,2h]} \int_{\Cn} \int_{[-h,h]} \abs| c_R(z,t') |
\abs| \phi_\ell(t - t')| \wrt t' \wrt z \wrt t \\
&\leq \int_{\R\setminus [-h,h]} \int_{\Cn}\int_{[-h,h]} \abs| c_R(z,t') | \abs|\phi_\ell(t'')| \wrt t'\wrt z \wrt t'' \\
&= \int_{\R\setminus [-2^\ell h,2^\ell h]}  \abs| \phi(t'')| \wrt t'' \int_{\Cn}\int_{[-h,h]} \abs| c_R(z,t')| \wrt t' \wrt z \\
&\lesssim_{\phi} (1+ 2^{\ell}h)^{-S} \norm{c_R}_{\Leb^1(\Hn)} \\
&\leq (1+ 2^{\ell}h)^{-S} h^N h^{1/2} \norm{a_R}_{\Leb^2(\Hn)}
\end{aligned}
\]
by Lemma \ref{lem:a-is-enough}.
To conclude, by taking $m$ to be $M$ if $j+\ell \leq 0$ and $0$ otherwise and  $n$ to be $N$ if $2^\ell h \leq 1$ and $0$ otherwise, and summing as in \eqref{eq:multi-1-5-a} and \eqref{eq:multi-1-5-b}, we deduce that
\begin{equation*}
\begin{aligned}
\sum_{k\in\N}\term{I}^{G}_{k}
&\lesssim \sum_{k\in\N}(2^kr^*)^{-\cdim\epsilon} h^{1/2} \norm{a_R}_{\Leb^2(\Hn)}
\lesssim_\epsilon \Bigl(\frac{1}{r^*}\Bigr)^{\cdim\epsilon} |R|^{1/2}\norm{a_R}_{\Leb^2(\Hn)} .
\end{aligned}
\end{equation*}

Next we estimate the terms $\term{I}^{C}_{k}$, given by
\[
\term{I}^{C}_{k} =
\int_{E_{3,k}} \biggl( \sum_{j,\ell} \abs|  [\indifn_{C(g')}c_{R,\ell}] \Hconv k^{m,n}_{j,\ell} (g') |^2 \biggr)^{1/2} \wrt g'.
\]
The argument used to prove \eqref{eq:multi186} shows that
\begin{equation*}
\begin{aligned}
\term{I}^{C}_{k}
&\lesssim_{\Mu} \sum_{j,\ell}  2^{m(j+\ell) + n\ell} \int_{G} \biggl(\int_{E_{3,k}}
 \frac{1}{w^\epsilon_{j,\ell}(g^{-1}g')}
 \wrt g' \biggr)^{1/2}\abs| \indifn_{C(g')} c_{R,\ell}(g)| \wrt g .
\end{aligned}
\end{equation*}
Let $E_{3,k}^*$ be the set $\{(z,t) \in \Hn : z \in A_k^*, |t| \geq 2h \}$.
By the definition of $C(g')$ (see \eqref{eq:defGBCD}), if $g \in G$ and $g' \in E_{3,k}$, then $g^{-1}g' \in E^*_{3,k}$, so
\[
\begin{aligned}
\biggl(\int_{E_{3,k}}
 \frac{1}{w^\epsilon_{j,\ell}(g^{-1}g')}
 \wrt g' \biggr)^{1/2} 
&\leq
\biggl(\int_{E^*_{3,k}}
 \frac{1}{w^\epsilon_{j,\ell}(g')}
 \wrt g' \biggr)^{1/2} \\
&\lesssim (1 + 2^{(j+\ell)/2}2^kr^*)^{-\cdim\epsilon} (1 + 2^{\ell}h)^{-\epsilon/2}
\end{aligned}
\]
for all $g \in G$.
Once again, we choose $m$ to be $M$ when $j+\ell$ is small and $0$ otherwise, and $n$ to be $N$ when $\ell$ is small and $0$ otherwise, and sum.
This leads to the conclusion that
\begin{equation*}
\begin{aligned}
\term{I}^{C}_{k}
&\lesssim \Bigl(\frac{1}{2^kr^*}\Bigr)^{\cdim\epsilon}|R|^{1/2}\norm{a_R}_{\Leb^2(\Hn)},
\end{aligned}
\end{equation*}
which in turn gives the desired estimate
\[
\int_{E_{3}} \biggl( \sum_{j,\ell} \abs|  [\indifn_{C(g')}c_{R,\ell}] \Hconv k^{m,n}_{j,\ell} (g') |^2 \biggr)^{1/2} \wrt g'
\lesssim \Bigl(\frac{1}{r^*}\Bigr)^{\cdim\epsilon}
\abs|R|^{1/2}\norm{a_R}_{\Leb^2(\Hn)}.
\]

Finally, we deal with the terms $\term{I}^{D}_{k}$, given by
\[
\term{I}^{D}_{k} =
\int_{E_{3,k}} \biggl( \sum_{j,\ell} \abs|  [\indifn_{D(g')}c_{R,\ell}] \Hconv k^{m,n}_{j,\ell} (g') |^2 \biggr)^{1/2} \wrt g'.
\]
We take $n = 0$ throughout this estimate, and assume for the moment that $r^* \geq h$.

The projection $PD(g')$ of $D(g')$ is a subset of $\Cn$, and we claim that
\begin{equation}\label{eq:multi144}
\begin{aligned}
|PD(g')| \lesssim \frac{h}{2^kr^*}.
\end{aligned}
\end{equation}
To see this, we fix $(z', t') \in E_{3,k}$, and take $(z,t) \in D(g')$.
Then $|z|_\infty \leq 1$ and $|t| \leq 2h$ since $(z,t) \in B$, and further $|t'-t-S(z',z)|\leq 2h$, that is,
$-4h \leq t -2h \leq S(z',z) - t' \leq t + 2h \leq 4h $.
Now $S(z',z)$ is a euclidean inner product $Jz' \cdot z$, where the euclidean norm of $Jz'$ is $8\cdim$ multiplied by the euclidean norm of $z'$, which is bounded below by a multiple of $2^{k}r^*$.
So these inequalities describe a slice of $\Cn$ of thickness of the order of $h/2^{k}r^*$, and the measure of the intersection of this slice with $\{ z \in \Cn: |z|_\infty \leq 1\}$ is as claimed.

By the Cauchy--Schwarz inequality,
\begin{equation}\label{eq:multi-3-1}
\begin{aligned}
\term{I}^{D}_{k} \leq \abs|E_{3,k}|^{1/2}
\biggl( \sum_{j,\ell} \int_{E_{3,k}}  \abs|  [\indifn_{D(g')}c_{R,\ell}] \Hconv k^{m,0}_{j,\ell} (g') |^2  \wrt g' \biggr)^{1/2}.
\end{aligned}
\end{equation}
From the Cauchy--Schwarz inequality, \eqref{eq:multi144}, and changes of variables and order of integration, 
\begin{align*}
&\int_{E_{3,k}} \Bigl| [\indifn_{D(g')}
 c_{R,\ell}] \Hconv k^{m,0}_{j,\ell}(g') \Bigr|^2  \wrt g' \\
&\qquad\leq\int_{E_{3,k}}
\biggl( \int_{\Hn} \abs| [\indifn_{D(g')}
 c_{R,\ell}] (g)| \fn \abs| k^{m,0}_{j,\ell}(g^{-1} g')| \wrt g \biggr)^2   \wrt g'\\
&\qquad= \iint_{E_{3,k}} \biggl( \int_{\Hn} \int_{\R}
\abs| [\indifn_{D(z',t')} c_{R,\ell}] (z,t) |
\fn \abs| k^{m,0}_{j,\ell} ((-z,-t)(z',t'))| \wrt t \wrt z \biggr)^2   \wrt t'  \wrt z'\\
&\qquad\leq \iint_{E_{3,k}} \biggl( \int_{PD(z',t')} \Bigl(\int_{\R}
\frac{\abs| [\indifn_{D(z',t')} c_{R,\ell}] (z,t) |^2}{w_{j,\ell}^{\epsilon}((-z,-t)(z',t'))}  \wrt t\Bigr)^{1/2} \\
&\qquad\qquad\times
\Bigl(\int_{\R} w_{j,\ell}^{\epsilon}((-z,-t)(z',t'))) \abs| k^{m,0}_{j,\ell}((-z,-t)(z',t'))|^2 \wrt t \Bigr)^{1/2}
\wrt z \biggr)^2 \wrt t'  \wrt z' \\
&\qquad\leq \int_{A_k} \int_{\R} \abs|PD(z',t')| \biggl( \int_{PD(z',t')} \Bigl(\int_{\R}
\frac{\abs| [\indifn_{D(z',t')} c_{R,\ell}] (z,t) |^2}{w_{j,\ell}^{\epsilon}((-z,-t)(z',t'))}  \wrt t\Bigr) \\
&\qquad\qquad\times
\Bigl(\int_{\R} w_{j,\ell}^{\epsilon}((-z,-t)(z',t')) \abs| k^{m,0}_{j,\ell}((-z,-t)(z',t'))|^2 \wrt t \Bigr)
\wrt z \biggr)  \wrt t'  \wrt z' \\
&\qquad\lesssim \frac{h}{2^kr^*} \int_{A_k} \int_{\R}  \int_{\Cn} \Bigl(\int_{\R}
\frac{\abs| [\indifn_{B} c_{R,\ell}] (z,t) |^2}{w_{j,\ell}^{\epsilon}((-z,-t)(z',t'))}  \wrt t\Bigr) \\
&\qquad\qquad\times
\Bigl(\int_{\R} w_{j,\ell}^{\epsilon}((-z,-t)(z',t')) \abs| k^{m,0}_{j,\ell}((-z,-t)(z',t'))|^2 \wrt t \Bigr)
 \wrt z  \wrt t'  \wrt z' \\
&\qquad= \frac{h}{2^kr^*} \int_{A_k} \int_{\R}  \int_{\Cn} \Bigl(\int_{\R}
\frac{\abs| [\indifn_{B} c_{R,\ell}] (z,t) |^2 }{w_{j,\ell}^{\epsilon}((-z,-t)(z',t'))} \wrt t\Bigr) \\
&\qquad\qquad\times
\Bigl(\int_{\R} w_{j,\ell}^{\epsilon}(z'-z,t) \abs| k^{m,0}_{j,\ell}(z'-z,t)|^2 \wrt t \Bigr)
\wrt z  \wrt t'  \wrt z' \\
&\qquad= \frac{h}{2^kr^*} \int_{A_k}  \int_{\Cn} \int_{\R}  \Bigl(\int_{\R}
\frac{\abs| [\indifn_{B} c_{R,\ell}] (z,t) |^2  }{w_{j,\ell}^{\epsilon}(z'-z,t'-t-S(z,z'))}
\wrt t\Bigr) \wrt t' \\
&\qquad\qquad\times
\Bigl(\int_{\R} w_{j,\ell}^{\epsilon}(z'-z,t) \abs| k^{m,0}_{j,\ell}(z'-z,t)|^2 \wrt t \Bigr)
\wrt z \wrt z' \\
&\qquad= \frac{h}{2^kr^*} \int_{A_k}  \int_{\Cn} \Bigl(\int_{\R} \int_{\R}
\frac{\abs| [\indifn_{B} c_{R,\ell}] (z,t) |^2}{w_{j,\ell}^{\epsilon}(z'-z,t')} \wrt t'  \wrt t\Bigr) \\
&\qquad\qquad\times
\Bigl(\int_{\R} w_{j,\ell}^{\epsilon}(z'-z,t) \abs| k^{m,0}_{j,\ell}(z'-z,t)|^2 \wrt t \Bigr) \wrt z   \wrt z' .
\end{align*}
Now when $(z,t) \in B$ and $(z',t') \in E_{3,k}$, $\abs|z-z'| \geq \abs|z'| - \abs|z| \gtrsim 2^k r^*$, and so
\[
\begin{aligned}
\int_{\R}  \frac{1}{w_{j,\ell}^{\epsilon}(z'-z,t')} \wrt t'
&= \int_{\R} 2^{\cdim(j+\ell)} (1 + 2^{j+\ell}|z'-z|^2)^{-\cdim(1+\epsilon)}
    2^{\ell} (1+2^{\ell}|t'|)^{-1-\epsilon} \wrt t' \\
&\lesssim  2^{\cdim(j+\ell)} (1 + 2^{k+(j+\ell)/2}r^*)^{-2\cdim(1+\epsilon)}. \\
\end{aligned}
\]
Hence by further changes of order of integration and of variables, and Lemma \ref{lemma muller stein},
\begin{align*}
&\int_{E_{3,k}} \Bigl| [\indifn_{D(g')}
 c_{R,\ell}] \Hconv k^{m,0}_{j,\ell}(g') \Bigr|^2  \wrt g'\\
&\qquad\lesssim \frac{h}{2^kr^*} \int_{\Cn} \int_{A_k}  \Bigl(
2^{\cdim(j+\ell)} (1 + 2^{k+(j+\ell)/2}r^*)^{-2\cdim(1+\epsilon)}
\int_{\R} \abs|c_{R,\ell}(z,t)|^2 \wrt t \Bigr)  \\
&\qquad\qquad\times  \Bigl( \int_{\R} w_{j,\ell}^{\epsilon}(z'-z,  t)
\fn \abs|k^{m,0}_{j,\ell}(z'-z, t)|^2 \wrt t \Bigr) \wrt z' \wrt z \\
&\qquad= \frac{h}{2^kr^*} 2^{\cdim(j+\ell)}
(1 + 2^{k+(j+\ell)/2}r^*)^{-2\cdim(1+\epsilon)}
\int_{\Cn}
\int_{\R} \abs|c_{R,\ell}(z,t)|^2 \wrt t \wrt z  \\
&\qquad\qquad\times  \int_{A_k} \int_{\R} w_{j,\ell}^{\epsilon}(z',  t)
\fn \abs|k^{m,0}_{j,\ell}(z', t)|^2 \wrt t  \wrt z' \\
&\qquad\lesssim_{\Mu} \frac{h}{2^kr^*} 2^{\cdim(j+\ell)}
(1 + 2^{k+(j+\ell)/2}r^*)^{-2\cdim(1+\epsilon)} 2^{2m(j+\ell)}
\norm{ c_{R,\ell}}_{\Leb^2(\Hn)}^2    .
\end{align*}
We sum this estimate over $j$ and $\ell$, much as in \eqref{eq:multi-1-5-a} and \eqref{eq:multi-1-5-b}, using \eqref{eq:cRell-def} and \eqref{eq:cRell-LP} and Lemma \ref{lem:a-is-enough}, and  taking $m$ to be $M$ if $j+\ell \leq 0$ and $0$ otherwise and $n$ to be $0$.
We recall that $c_{R,\ell}$ also depends on $m$, and hence on $j$ and $\ell$, and here write $c^m_{R,\ell}$ instead of $c_{R,\ell}$.
Then, assuming that $2M > \epsilon$, we see that
\begin{equation*}
\begin{aligned}
&\sum_{j,\ell}\int_{E_{3,k}} \Bigl| [\indifn_{D(g')}
 c^m_{R,\ell}] \Hconv k^{m,0}_{j,\ell}(g') \Bigr|^2  \wrt g'\\
&\qquad\lesssim_{\Mu} \sum_{j,\ell}\frac{h}{2^kr^*} 2^{\cdim(j+\ell)}
(1 + 2^{k+(j+\ell)/2}r^*)^{-2\cdim(1+\epsilon)}
\min\{1, 2^{2M(j+\ell)} \} 
\norm{ c^m_{R,\ell}}_{\Leb^2(\Hn)}^2   \\
&\qquad\eqsim_{M,\epsilon} \frac{h}{2^kr^*} \sum_{\ell}
\lpar  \frac{1}{ (2^{k}r^*)^{4M+2\cdim} }
+ \frac{1}{ (2^{k}r^*)^{2\cdim(1+\epsilon)}}
+  \frac{1}{ (2^{k}r^*)^{2\cdim(1+\epsilon)}} \rpar 
 \norm{ c^m_{R,\ell}}_{\Leb^2(\Hn)}^2   \\
&\qquad\lesssim_{M,\epsilon} \frac{h}{2^{2k\cdim(1+\epsilon)+k} (r^*)^{2\cdim(1+\epsilon)+1}}
\sum_{\ell} \lpar \norm{ c^M_{R,\ell}}_{\Leb^2(\Hn)}^2  +
 \norm{ c^0_{R,\ell}}_{\Leb^2(\Hn)}^2  \rpar \\
&\qquad\lesssim_{\Phi,M} \frac{h}{2^{2k\cdim(1+\epsilon)+k} (r^*)^{2\cdim(1+\epsilon)+1}} \norm{ a_{R}}_{\Leb^2(\Hn)}^2  .
\end{aligned}
\end{equation*}
By combining this last inequality with \eqref{eq:multi-3-1} and the estimate $\abs|E_{3,k}| \lesssim (2^k r^*)^{2\cdim + 1}$ (which holds since $h \leq r^*$), we see that
\begin{equation}\label{eq:multi-3-3}
\begin{aligned}
\sum_{k\in\N} \term{I}^{D}_{k}
&\lesssim \sum_{k\in\N} \abs|E_{3,k}|^{1/2}
\biggl( \sum_{j,\ell} \int_{E_{3,k}}  \abs|  [\indifn_{D(g')}c_{R,\ell}] \Hconv k^{m,0}_{j,\ell} (g') |^2  \wrt g' \biggr)^{1/2} \\
&\lesssim_{\Mu,\Phi,\epsilon} \sum_{k\in\N} (2^k r^*)^{\cdim + 1/2}
\frac{h^{1/2}}{2^{k\cdim(1+\epsilon)+k/2} (r^*)^{\cdim +\cdim\epsilon+1/2}} 
\norm{ a_{R}}_{\Leb^2(\Hn)} \\
&\eqsim_{\epsilon} \Bigl(\frac{1}{r^*}\Bigr)^{\cdim\epsilon/2}
\abs|R|^{1/2} \norm{ a_{R}}_{\Leb^2(\Hn)}.
\end{aligned}
\end{equation}
Thus the required estimate for the integral over $E_3$ is proved when $r^* \geq h$.

If $r^* \geq h$, we modify the argument above slightly.
We use the inclusion $D(g') \subseteq B$ and replace $PD(g')$ by $PB$ in the first stage of our calculations.
The expression $(h \abs|E_{3,k}| / 2^kr^*)^{1/2}$ that appears in \eqref{eq:multi-3-3} is then replaced by $\abs|E_{3,k}|^{1/2}$, and with our alternative assumption, $\abs|E_{3,k}|^{1/2} \lesssim (2^kr^*)^{2\cdim}h$, which leads us to the same conclusion.

Finally, we estimate the integral over $E_{4}$.
If $\Eta(\lambda/\mu) \neq 0$ and $\Eta(\mu) \neq 0$, then $\mu, \lambda/\mu \in (1/2,2)$ and so $\lambda \in (-1/4,4)$.
Hence there is a smooth, compactly supported function $\Phi$ on $\R^+$ such that
\[
\Phi(\lambda)
= \Phi(\lambda)
\fn\Eta(  \lambda/\mu) \fn\Eta( \mu)
\qquad\forall \lambda, \mu \in \R^+ ,
\]
whence
\[
\Phi(2^{-j-\ell}\lambda)
= \Phi(2^{-j-\ell}\lambda)
\fn\Eta( 2^{-j} \lambda/\mu) \fn\Eta( 2^{-\ell}\mu)
\qquad\forall \lambda, \mu \in \R^+ .
\]
Taking $\mu$ to be $1$ shows that
\[
\Phi(2^{-j}\lambda)
= \Phi(2^{-j}\lambda)
\fn\Eta( 2^{-j} \lambda)
\qquad\forall \lambda \in \R^+ .
\]
Then $\Mu_{j,\ell}(\HLap/|\oper{T}| , \euli\oper{T}) \Phi(2^{-j-\ell}\HLap) = \Mu_{j,\ell}(\HLap/|\oper{T}| , \euli\oper{T})$.
We define
\begin{equation*}
c_{R,j+\ell} := \Phi(2^{-j-\ell}\HLap) c_R ,
\end{equation*}
keeping implicit the dependence of $c_{R,j+\ell}$ on $m$ and $n$.
From spectral theory,
\begin{equation*}
\biggl( \sum_{j,\ell} \norm{c_{R,j+\ell}}_{\Leb^2(\Hn)}^2  \biggr)^{1/2}
\lesssim_{\Phi} \norm{c_R}_{\Leb^2(\Hn)}.
\end{equation*}

We divide the region $E_{4}$ into dyadic pieces:
\[
E_{4}=\bigcup_{k\in\N} E_{4,k}
:=\bigcup_{k\in\N}\{(z,t): |z|\leq 2, 2^{k}h^*<|t|\leq 2^{k+1} h^*\},
\]
and define sets
\begin{gather*}
A_{{k}}
:=\{g \in \Hn : d_{(1)}(g, R) \geq 2^k h^* \}
\qquad\text{and}\qquad
B_{{k}}:= (A_{{k}})^c \setminus R^*.
\end{gather*}
Then $\Hn = A_{{k}} \cup B_{{k}} \cup R^*$.
Much as in the estimate over $E_3$, we see that
\begin{equation*}
\begin{aligned}
&\int_{E_{4}}  \biggl( \sum_{j,\ell} \Bigl|  c_R \Hconv k_{j,\ell}^{m,n}(g')(g') \Bigr|^2 \biggr)^{1/2} \wrt g' \\
&\qquad\leq \sum_{k\in\N}\int_{E_{4,k}}\biggl( \sum_{j,\ell} \Bigl|  [\indifn_{A_{{k}}}c_{R,j+\ell}] \Hconv k_{j,\ell}^{m,n} (g') \Bigr|^2 \biggr)^{1/2} \wrt g'\\
&\qquad\qquad+ \sum_{k\in\N}\int_{E_{4,k}}\biggl( \sum_{j,\ell} \Bigl|  [\indifn_{B_{{k}}}c_{R,j+\ell}] \Hconv k_{j,\ell}^{m,n} (g') \Bigr|^2 \biggr)^{1/2} \wrt g'\\
&\qquad\qquad+ \sum_{k\in\N}\int_{E_{4,k}}\biggl( \sum_{j,\ell} \Bigl|  [\indifn_{R^*}c_{R,j+\ell}] \Hconv k_{j,\ell}^{m,n} (g') \Bigr|^2 \biggr)^{1/2} \wrt g'\\
&\qquad=: \sum_{k\in\N}\term{I}_{4,k}^A + \sum_{k\in\N}\term{I}_{4,k}^B + \sum_{k\in\N}\term{I}_{4,k}^R,
\end{aligned}
\end{equation*}
say.

To treat the term $\term{I}_{4,k}^A$, we observe that convolution with $k_{j,\ell}^{m,n}$ is bounded on $\Leb^2(\Hn)$ independently of $j$, $\ell$, $m$ and $n$ by spectral theory, so that taking $m$ to be $M$ if $j+\ell \leq 0$ and $0$ otherwise and $n$ to be $N$ if $2^\ell h\leq 1$ and $0$ otherwise shows that
\begin{equation*}
\begin{aligned}
\term{I}_{4,k}^A
&\leq \abs|E_{4,k}|^{1/2}
\biggl(\sum_{j,\ell} \int_{E_{4,k}}\Bigl|
\indifn_{A_{{k}}} c_{R,j+\ell} \Hconv k_{j,\ell}^{m,n}(g') \Bigr|^2  \wrt g' \biggr)^{1/2} \\
&\lesssim_{\Mu,M,N} \abs|E_{4,k}|^{1/2}
\biggl(\sum_{j,\ell} \min\{1,2^{j+\ell}\}^M\min\{1,2^\ell h\}^N \int_{A_{{k}}}\abs| c_{R,j+\ell} (g)|^2  \wrt g\biggr)^{1/2} .
\end{aligned}
\end{equation*}
Since $d_{(1)}(A_{{k}},R)\geq (2^{k}h^*)^{1/2}$, it follows from Lemma~\ref{lemma m 444.2} that for all $S \in (\hdim,\infty)$,
\begin{equation*}
\begin{aligned}
\int_{A_{{k}}}\abs| c_{R,j+\ell} (g)|^2  \wrt g
&= \int_{A_{{k}}}\abs| \Phi(2^{-j-\ell}\HLap) a_R (g) |^2 \wrt g \\
&\lesssim_{\Phi,S} (1 + 2^{j+\ell+k}h^*)^{-S}\}
\abs|R| \norm{a_R}_{\Leb^2(\Hn)}^2.
\end{aligned}
\end{equation*}
We recall that $j$ is nonnegative, combine the last two inequalities, and sum over $k$, much as above.
Assuming that $S > M > N \geq 1$, we see that $\sum_{k} \term{I}_{4,k}^A$ is dominated by a multiple of 
\begin{equation*}
\begin{aligned}
&\sum_{k} \abs|E_{4,k}|^{1/2} \biggl(\sum_{j,\ell}
\frac{ \min\{1,2^{j+\ell}\}^M \min\{1,2^\ell h\}^N  }{ ( 1 + 2^{j+\ell+k}h^*)^{S} }
\biggr)^{1/2}  \abs|R|^{1/2} \norm{a_R}_{\Leb^2(\Hn)} \\
&\qquad\lesssim_{\epsilon} \Bigl( \frac {h}{h^*} \Bigr)^{(M-1)/2} |R|^{1/2}\norm{a_R}_{\Leb^2(\Hn)}  .
\end{aligned}
\end{equation*}

To treat $\term{I}_{4,k}^B$, we note that, if $g'=(z',t')\in E_{4,k}$ and $g=(z,t)\in (A_{{k}})^c$, then
\[
|t'-t-S(z',z)|
\geq 2^kh^* - (2^{k-2}h^*)^{1/2} - 16\cdim(2^{k-2}h^*)^{1/2}
\geq 2^{k-2}h^* ,
\]
whence
\[
g^{-1}g'\in E_{4,k}^* := \{(z'',t'') \in \Hn: |t''|>2^{k-2}h^*\}.
\]
Now
\begin{equation*}
\begin{aligned}
\int_{E^*_{4,k}} \frac{1}{w_{j,\ell}^{\epsilon}(g')} \wrt g' 
&\leq \int_{\Cn} \int_{|t''|>2^{k-2}h^*}
\frac{2^{(j+\ell)\cdim}2^{\ell}}{(1+2^{j+\ell}|z''|^2)^{\cdim(1+\epsilon)} (1+2^\ell|t''|)^{1+\epsilon} } \wrt t'' \wrt z'' \\
&\lesssim_{\epsilon} (1+2^k2^\ell h^*)^{-\epsilon}.
\end{aligned}
\end{equation*}
From this inequality,  Lemma~\ref{lemma muller stein}, the Cauchy--Schwarz inequality and Lemma \ref{lemma m 444.2},
\begin{equation*}
\begin{aligned}
\term{I}_{4,k}^B
&\leq \sum_{j,\ell} \int_{E_{4,k}}\Bigl|  \indifn_{B_{{k}}} c_{R,j+\ell} \Hconv k_{j,\ell}^{m,n}(g') \Bigr|  \wrt g'\\
&\leq \sum_{j,\ell} \int_{B_{{k}}} \abs| c_{R,j+\ell}(g) | \int_{E_{4,k}}\Bigl| k_{j,\ell}^{m,n}(g^{-1}g')\Bigr| \wrt g' \wrt g\\
&\leq \sum_{j,\ell} \int_{B_{{k}}} \abs| c_{R,j+\ell}(g) | \int_{E_{4,k}^*}\Bigl| k_{j,\ell}^{m,n}(g')\Bigr| \wrt g' \wrt g\\
&\lesssim_{\Mu} \sum_{j,\ell} 2^{m(j+\ell) + n\ell} \int_{B_{{k}}} \bigl| c_{R,j+\ell}(g) \bigr|\wrt g \Biglpar\int _{E^*_{4,k}}\frac{1}{w_{j,\ell}^{\epsilon}(g')} \wrt g'\Bigrpar^{1/2} \\
&\lesssim_{\epsilon} \sum_{j,\ell} 2^{m(j+\ell) + n\ell}
(1+ 2^k2^\ell h^*)^{-\epsilon/2}
\int_{B_{{k}}} \bigl| c_{R,j+\ell}(g) \bigr|\wrt g \\
&\lesssim_{\Phi} \sum_{j,\ell} 2^{m(j+\ell) + n\ell}
(1+2^k2^\ell h^*)^{-\epsilon/2} (1+ 2^{j+\ell})^{-\epsilon}|R|^{1/2}
\norm{a_R}_{\Leb^2(\Hn)}.
\end{aligned}
\end{equation*}
We take $m$ to be $M$ if $j+\ell \leq 0$ and $0$ otherwise and $n$ to be $N$ if $2^\ell h\leq 1$ and $0$ otherwise, and sum much as in \eqref{eq:multi-1-5-a}  and \eqref{eq:multi-1-5-b}.
Then
\begin{equation*}
\begin{aligned}
\term{I}_{4,k}^B
&\lesssim  2^{-k\epsilon/2} \lpar \frac {h}{h^*} \rpar^{\epsilon/2}|R|^{1/2}\norm{a_R}_{\Leb^2(\Hn)}.
\end{aligned}
\end{equation*}

To treat the term $\term{I}_{4,k}^R$, we first observe that if $g\in R^*$ and $g'\in E_{4,k}$, then
\[
g^{-1}g' \in E_{4,k}^* := \{(z'',u'') \in \Hn : |z''| \leq 3, (2^k-16\cdim)h^* < |t| < (2^{k+1} + 16\cdim)h^*\}.
\]
By the Cauchy--Schwarz inequality,
\begin{equation*}
\begin{aligned}
&\int_{E_{4,k}} \bigglpar \sum_{j,\ell}
\labs [\indifn_{R^*} c_{R,j+\ell}] \Hconv k_{j,\ell}^{m,n}(g')\rabs^2\biggrpar^{1/2}
\wrt g' \\
&\qquad=
 |E_{4,k}|^{1/2} \bigglpar \sum_{j,\ell} \norm{
[[\indifn_{R^*} c_{R,j+\ell}] \Hconv k_{j,\ell}^{m,n}] \indifn_{E_{4,k}}}_{L^2(\Hn)} ^2 \biggrpar^{1/2}.
\end{aligned}
\end{equation*}

We set $ w_{1,j,\ell}^{\epsilon}(g)
:=2^{-\cdim(j+\ell)} (1+2^{j+\ell}|z|^2)^{\cdim(1+\epsilon)}
$ and $
w_{2,j,\ell}^{\epsilon}(g):=2^{-\ell} (1+2^{\ell}|t|)^{1+\epsilon}$.
For all $f\in L^2(\Hn)$,
\begin{align*}
&\labs \int_{\Hn} [[[\indifn_{R^*} c_{R,j+\ell}] \Hconv k_{j,\ell}^{m,n}] \indifn_{E_{4,k}}](g') \fn f(g') \wrt g' \rabs\\
&\qquad \leq
\int_{R^*} \int_{E_{4,k}} \labs c_{R,j+\ell}(g)\rabs
\biglabs k_{j,\ell}^{m,n}(g^{-1}g') \bigrabs \labs f(g') \rabs \wrt g' \wrt g \\
&\qquad \leq
\lpar \int_{R^*} \labs c_{R,j+\ell}(g)\rabs^2 \wrt g \rpar^{1/2}
\biggl(\int_{R^*}\biggl( \int_{E_{4,k}} \biglabs k_{j,\ell}^{m,n}(g^{-1}g') \bigrabs
\labs f(g') \rabs \wrt g' \biggr)^{2} \wrt g \biggr)^{1/2} .
\end{align*}
Furthermore,
\begin{align*}
&\biggl(\int_{R^*}\biggl( \int_{E_{4,k}} \biglabs k_{j,\ell}^{m,n}(g^{-1}g') \bigrabs
\labs f(g') \rabs \wrt g' \biggr)^{2} \wrt g \biggr)^{1/2}  \\
&\qquad \leq
\biggl(\int_{R^*} \biggl( \int_{E_{4,k}} \biglabs k_{j,\ell}^{m,n}(g^{-1}g') \bigrabs^2 \frac{w^\epsilon_{j,\ell}(g^{-1}g')}{w^\epsilon_{2,j,\ell}(g^{-1}g')}
 \wrt g'
\int_{E_{4,k}}
\frac{\labs f(g') \rabs^2}{w^\epsilon_{1,j,\ell}(g^{-1}g')} \wrt g' \biggr)  \wrt g \biggr)^{1/2}  \\
&\qquad \lesssim
\biggl(\int_{R^*}\biggl( \int_{E_{4,k}}
\frac{2^\ell}{(1+2^\ell 2^k h^*)^{1+\epsilon}}
\biglabs k_{j,\ell}^{m,n}(g^{-1}g') \bigrabs^2 w^\epsilon_{j,\ell}(g^{-1}g')
 \wrt g' \\
&\qquad\qquad\qquad \times
\int_{E_{4,k}}
\frac{\labs f(g') \rabs^2}{w^\epsilon_{1,j,\ell}(g^{-1}g')} \wrt g' \biggr) \wrt g \biggr)^{1/2} \\
&\qquad \lesssim
 \biggl(\int_{R^*}
\frac{2^{2m(j+\ell) + (2n+1)\ell}}{(1+2^\ell 2^k h^*)^{1+\epsilon}} \biggl(
\int_{E_{4,k}}
\frac{\labs f(g') \rabs^2}{w^\epsilon_{1,j,\ell}(g^{-1}g')} \wrt g' \biggr) \wrt g \biggr)^{1/2}  \\
&\qquad =
\frac{2^{m(j+\ell) + (n+1/2)\ell}}{(1+2^\ell 2^k h^*)^{1+\epsilon}}
\biggl( \int_{E_{4,k}} \int_{R^*}
\frac{\labs f(g') \rabs^2}{w^\epsilon_{1,j,\ell}(g^{-1}g')} \wrt g \wrt g' \biggr)^{1/2}\\
&\qquad \lesssim_{\epsilon}
\frac{2^{m(j+\ell) + (n+1/2)\ell}}{(1+2^\ell 2^k h^*)^{1+\epsilon}} h^{1/2}
\biggl( \int_{E_{4,k}}
\labs f(g') \rabs^2 \wrt g' \biggr)^{1/2}.
\end{align*}
By combining the last two inequalities, we deduce that
\begin{align*}
&\labs \int_{\Hn} [[[\indifn_{R^*} c_{R,j+\ell}] \Hconv k_{j,\ell}^{m,n}] \indifn_{E_{4,k}}](g') \fn f(g') \wrt g' \rabs\\
&\qquad \lesssim_{\epsilon}
\frac{2^{m(j+\ell) + (n+1/2)\ell}}{(1+2^\ell 2^k h^*)^{1+\epsilon}} h^{1/2}
\lpar \int_{R^*} \labs c_{R,j+\ell}(g)\rabs^2 \wrt g \rpar^{1/2}
\biggl( \int_{E_{4,k}}
\labs f(g') \rabs^2 \wrt g' \biggr)^{1/2}
\end{align*}
for all $f\in L^2(\Hn)$. 
By the converse of the Cauchy--Schwarz inequality,
\[
\norm{
[[\indifn_{R^*} c_{R,j+\ell}] \Hconv k_{j,\ell}^{m,n}] \indifn_{E_{4,k}}}_{L^2(\Hn)}
\lesssim
\frac{2^{m(j+\ell) + (n+1/2)\ell}}{(1+2^\ell 2^k h^*)^{1+\epsilon}}
h^{1/2}
\norm{ [\indifn_{R^*} c_{R,j+\ell}] }_{L^2(\Hn)}  .
\]

We take $m$ to be $M$ if $j+\ell <0$ and $0$ otherwise, and $n$ to be $1$ if $\ell < 0$ and $0$ otherwise.
Summing as before, it is now straightforward to see that
\[
\begin{aligned}
&\sum_{k\in\N} |E_{4,k}|^{1/2} \lpar \sum_{j,\ell} \norm{
[[\indifn_{R^*} c_{R,j+\ell}] \Hconv k_{j,\ell}^{m,n}] \indifn_{E_{4,k}}}_{L^2(\Hn)} ^2 \rpar^{1/2} \\
&\qquad\lesssim  \lpar \frac {h}{h^*} \rpar^{\epsilon/2}|R|^{1/2}\norm{a_R}_{\Leb^2(\Hn)}
\end{aligned}
\]
as required, and the proof of Theorem~E is complete.
\end{proof}


\end{document}